\documentclass[11pt,reqno]{amsart}
\usepackage[latin1]{inputenc}
\usepackage{esint}
\usepackage{bbm}
\usepackage{amssymb}
\usepackage{mathrsfs}
\usepackage{amsfonts}
\usepackage{amsmath}
\usepackage{amsthm}
\usepackage{url}
\usepackage{enumerate}
\usepackage{xcolor}
\usepackage{booktabs}
\usepackage{tabularx}

\newcommand{\ds}{\displaystyle}
\newtheorem{theorem}{Theorem}[section]
\newtheorem{lemma}[theorem]{Lemma}
\newtheorem{proposition}[theorem]{Proposition}
\newtheorem{corollary}[theorem]{Corollary}
\theoremstyle{definition}
\newtheorem{definition}[theorem]{Definition}
\newtheorem{assumption}{Assumption}
\newtheorem{remark}[theorem]{Remark}
\numberwithin{equation}{section}

\newtheorem{example}[theorem]{Example}

\usepackage[left=2.0cm, right=2.0cm, top=2.0cm, bottom=2.0cm]{geometry}

\usepackage{cite}
\usepackage{hyperref}
\hypersetup{
	colorlinks = true,%
	citecolor = blue,
	filecolor=red,%
	linkcolor = [rgb]{0.65,0.0,0.0},%
	anchorcolor = red,
	pagecolor = red,
	urlcolor= [rgb]{0.65,0.0,0.0}
	linktocpage=true,
	pdfpagelabels=true,
	bookmarksnumbered=true,
}

\DeclareMathOperator{\id}{Id}
\DeclareMathOperator{\diam}{Diam}
\DeclareMathOperator{\vol}{vol}

\DeclareMathOperator{\loc}{loc}

\DeclareMathOperator{\di}{div}

\DeclareMathOperator{\dd}{d}

\DeclareMathOperator{\supp}{supp}

\DeclareMathOperator{\Lip}{FLip}

\DeclareMathOperator{\Geo}{Geo}

\DeclareMathOperator{\FAC}{\mathsf{FAC}}

\DeclareMathOperator{\AC}{\mathsf{AC}}

\DeclareMathOperator{\BAC}{\mathsf{BAC}}
\DeclareMathOperator{\APE}{\mathsf{FEPMMS}}
\DeclareMathOperator{\Chc}{\mathsf{Ch}}
\DeclareMathOperator{\intt}{int}

\DeclareMathOperator{\m}{\mathfrak{m}}
\DeclareMathOperator{\dm}{d\mathfrak{m}}

\DeclareMathOperator{\Ht}{\mathsf{H}_{\it t}}
\DeclareMathOperator{\Htt}{\mathsf{H}_{\it t}}

\DeclareMathOperator{\osc}{osc}

\DeclareMathOperator{\RL}{BLip}

\DeclareMathOperator{\Tt}{\mathcal{T}}

\DeclareMathOperator{\res}{\mathsf{restr}}

\DeclareMathOperator{\g}{\mathfrak{g}}

\DeclareMathOperator{\I}{\mathfrak{T}}

\DeclareMathOperator{\SL}{Lip}

\DeclareMathOperator{\ee}{{e}}

\DeclareMathOperator{\CD}{\mathsf{CD}}
\DeclareMathOperator{\Po}{\mathscr{P}}

\DeclareMathOperator{\X}{\mathcal {X}}

\DeclareMathOperator{\Mod}{\mathsf{Mod}}

\DeclareMathOperator{\Vv}{\mathsf{pV}}
\DeclareMathOperator{\eV}{\mathsf{eV}}

\author{Alexandru Krist\'aly}
\address{Department of Economics\\
	Babe\c s-Bolyai University\\
	400591 Cluj-Napoca, Romania \&  Institute of Applied Mathematics\\
 \'Obuda University\\
 1034 Budapest, Hungary}
  \email{alex.kristaly@econ.ubbcluj.ro; kristaly.alexandru@nik.uni-obuda.hu}

  \thanks{A.~Krist\'aly is supported by the
  	Excellence Researcher Program  of \'Obuda University, Hungary.}

\author{Shin-ichi Ohta}
\address{
Department of Mathematics\\
Osaka University\\
560-0043 Osaka\\
Japan\\
and RIKEN Center for Advanced Intelligence Project (AIP)\\
1-4-1 Nihonbashi \\ Tokyo 103-0027\\
Japan}
\email{s.ohta@math.sci.osaka-u.ac.jp}
\thanks{S.~Ohta is supported by the JSPS Grant-in-Aid for Scientific Research (KAKENHI)
22H04942, 24K00523, 24K21511.}

\author{Wei Zhao}
\address{
School of Mathematics\\
East China University of Science and Technology\\
200237 Shanghai, China}
\email{szhao\underline{ }wei@yahoo.com}
\thanks{W.~Zhao is supported by  Natural Science Foundation of China (No. 12471045), Ningbo  Natural Science Foundation (No. 2024J017) and Science and Technology Project of Xinjiang Production and Construction Corps (No. 2023CB008-12).}

\date{\today}

\keywords{Asymmetric metric space; heat flow; Laplacian; Sobolev space; Finsler manifold; Funk metric.}

\subjclass[2010]{Primary 49Q20, Secondary 35K05, 53C21, 58J35}
\begin{document}

\title[Analysis on asymmetric metric measure spaces]
{Analysis on asymmetric metric measure spaces:\\ $q$-heat flow, $q$-Laplacian and Sobolev spaces}

\begin{abstract}
We investigate geometric analysis on metric measure spaces equipped with asymmetric distance functions.
Extending concepts from the symmetric case,
we introduce upper gradients and corresponding $L^q$-energy functionals
as well as $q$-Laplacian in the asymmetric setting.
Along the lines of gradient flow theory,
we then study $q$-heat flow as the $L^2$-gradient flow for $L^q$-energy.
We also study the Sobolev spaces over asymmetric metric measure spaces,
extending some results of Ambrosio--Gigli--Savar\'e to asymmetric distances.
A wide class of asymmetric metric measure spaces is provided by irreversible Finsler manifolds,
which serve to construct various model examples by pointing out genuine differences between the symmetric and asymmetric settings.
\end{abstract}

\maketitle
\tableofcontents

\section{Introduction}

We often encounter asymmetric quantities (transport costs, action functionals, etc.) in our daily life.
Among others, irreversible Finsler manifolds and, more generally, asymmetric distance functions provide
an appropriate framework for studying such asymmetric phenomena.
Relevant examples include the Zermelo navigation problem (see Bao--Robles--Shen \cite{BRS}),
the Matsumoto mountain slope model describing the law of walking on an inclined plane under the action of gravity
(see Matsumoto \cite{Matsumoto}),
and modeling light propagation in anisotropic and inhomogeneous media (see Antonelli--Ingarden--Matsumoto \cite{AIM}).
An important model example is the Funk metric on the $n$-dimensional Euclidean unit ball $\mathbb B^n$,
where the distance function is given by
\begin{equation}\label{dist-reversfunk-00}
d_{F}(x_1,x_2) =\log \Biggl[
 \frac{\sqrt{\|x_1-x_2\|^2-(\|x_1\|^2\|x_2\|^2-\langle x_1,x_2\rangle^2)}-\langle x_1,x_2-x_1\rangle}{\sqrt{\|x_1-x_2\|^2-(\|x_1\|^2\|x_2\|^2-\langle x_1,x_2\rangle^2)}-\langle x_2,x_2-x_1\rangle} \Biggr],
 \quad x_1,x_2\in \mathbb{B}^n
\end{equation}
(see, e.g., Shen \cite{Sh1}), where $\|\cdot\|$ and $\langle\cdot, \cdot\rangle$ denote
the Euclidean norm and inner product, respectively.
Clearly $d_F$ is asymmetric ($d_F(x_1,x_2) \neq d_F(x_2,x_1)$); moreover,
\begin{equation}\label{reversfunk-00}
	\lim_{\|x\|\to1^-}d_F(\mathbf{0},x)=\infty, \qquad \lim_{\|x\|\to1^-} d_F(x,\mathbf{0})=\log2.
\end{equation}
Such asymmetric metric spaces can be efficiently discussed within the theory of Finsler manifolds;
in particular, due to the asymmetric version of the Busemann--Mayer theorem \cite{BM} (see also Bao--Chern--Shen \cite[Exercise~6.3.4]{BCS}),
the Funk metric $d_F$ in \eqref{dist-reversfunk-00} is associated with the Finsler metric:
\begin{equation}\label{Finsler-funk}
	F(x,y) =\frac{\sqrt{\|y\|^2-(\|x\|^2\|y\|^2-\langle x,y\rangle^2)}}{1-\|x\|^2}+\frac{\langle x,y\rangle}{1-\|x\|^2},
	\quad x \in \mathbb{B}^n,\ y\in T_x\mathbb{B}^n=\mathbb R^n.
\end{equation}
We notice that the unboundedness of the ratio $d_F(\mathbf{0},x)/d_F(x,\mathbf{0})$
(called the reversibility; see Definition~\ref{reversibilitydef}) observed in \eqref{reversfunk-00}
deeply affects the validity of many of the expected geometric/analytic results on this structure.

The primordial aim of this article is to elaborate the theory of $q$-heat flow and $q$-Laplacian for $q \in (1,\infty)$
on asymmetric metric measure spaces, as a continuation of our previous works
on optimal transport theory  and gradient flows in asymmetric metric spaces (see Krist\'aly--Zhao \cite{KZ} and Ohta--Zhao \cite{OZ}).
Precisely, in \cite[\S 4.4]{OZ}, we constructed usual ($2$-)heat flow as gradient flow for the relative entropy
in the $L^2$-Wasserstein space over a Finsler manifold.
Instead, we shall construct $q$-heat flow as gradient flow for an appropriate energy functional
(called the (weak) $q$-Cheeger energy) on the $L^2$-space.
In the symmetric setting, heat flow on metric measure spaces has been deeply studied.
Among them, the most relevant one to us is Ambrosio--Gigli--Savar\'e's \cite{AGS4,AGS2}
based on the gradient flow theory developed in \cite{AGS},
systematically integrating seminal works by Cheeger~\cite{Ch} and Shanmugalingam~\cite{Sh}
in terms of upper gradients (see also Heinonen \cite{He} and Heinonen--Koskela--Shanmugalingam--Tyson \cite{HKST}).
Precisely, they studied the case of $q=2$ and $q \neq 2$ in \cite{AGS2} and \cite{AGS4}, respectively,
in the symmetric setting.
We also refer to Kell~\cite{Ke} for the construction of $q$-heat flow as gradient flow
for the R\'enyi entropy in the $L^p$-Wasserstein space with $q \neq 2$ and $p^{-1}+q^{-1}=1$,
again in the symmetric case.

We intend to extend the notions of weak upper gradients and relaxed gradients as well as
corresponding Cheeger-type energies and Laplacians to the asymmetric setting with $q \neq 2$,
emphasizing at the same time relevant differences between the usual symmetric and the present asymmetric cases.
By  working  in an appropriate framework, our results are formally similar to the symmetric setting to some extent.
However, pathological situations may occur due to the asymmetric character of the distance function,
which will be exemplified by, e.g., the aforementioned Funk metric \eqref{dist-reversfunk-00} or \eqref{Finsler-funk}.
For instance, we need to distinguish forward and backward curves;
a constant speed minimal geodesic in the forward sense may not be even absolutely continuous in the backward sense
(see Example~\ref{forwardabsoucuveextend}).

We begin with the study of a weak upper gradient $G$ of a function $f$
on an asymmetric metric measure space $(X,d,\m)$, requiring
\[
\int_{\partial\gamma}f \leq \int_{\gamma}G
\]
along curves $\gamma$ in a certain class $\I$ (Definition~\ref{weakuppgrad}).
Depending  on  the choice of forward and backward curves,
we arrive at the notion of minimal weak forward and backward upper gradients $|D^\pm f|_{w,\I_p}$
(Definition~\ref{df:min-wug}).
This leads us to the weak forward $q$-Cheeger energy
\begin{equation}\label{Cheeger-0}
	\Chc^+_{w,q}(f) :=\frac1q \int_X |D^+ f|^q_{w,\I_p} \,{\dm}
\end{equation}
(Definition~\ref{weakcheegerenergy}),
while the backward energy satisfies $\Chc^-_{w,q}(f)=\Chc^+_{w,q}(-f)$.
It turns out that $\Chc^+_{w,q}$ is convex and lower semi-continuous in the Hilbert space $L^2(X,\m)$;
therefore, one can construct $q$-heat flow as gradient flow for $\Chc^+_{w,q}$
as well as the $q$-Laplacian $\Delta_q$ as the subdifferential of $\Chc^+_{w,q}$ with the least $L^2$-norm.
Thus, given any initial datum $f \in L^2(X,\m)$,
we have $f_t=\Ht(f)$ solving the $q$-heat equation:
\[
\frac{\dd\ }{{\dd}t^+}f_t=\Delta_q f_t, \qquad f_t \to f \text{ in } L^2(X,\m) \text{ as } t \to 0
\]
(Theorems~\ref{existstheoremft}, \ref{realchgraevslop}, Corollary~\ref{heatflowequationproper-00}).
We also establish Kuwada's lemma in our setting
stating an upper bound of the backward metric speed $|\mu'_-|(t)$ for a curve $\mu_t=\Ht(f_0)\m$ of probability measures
obeying the $q$-heat equation (Lemma \ref{wssvolemm2}), utilizing the Hopf--Lax semigroup studied in  Appendix \ref{hopflaxsemig}.
The appearance of the backward speed is an interesting phenomenon apparent only in the asymmetric setting
(we refer to Ohta \cite{Oham}, Ohta--Sturm \cite{OS} and Ohta--Suzuki \cite{OSuz} for related results).

We then introduce relaxed gradients via an approximation by Lipschitz functions in $L^q(X,\m)$
along the lines of Cheeger's construction \cite{Ch} (Definition~\ref{relaxedgradient}).
Then, similarly to \eqref{Cheeger-0}, the forward $q$-Cheeger energy $\Chc^+_q(f)$
is defined by using minimal relaxed ascending gradients (Definition~\ref{df:Cheeger}).
By definition, we have $\Chc^+_{w,q}(f) \le \Chc^+_q(f)$ in general (Proposition~\ref{bascieastime}).
Analyzing $\Chc^+_q$ requires more care since it is defined via an approximation in $L^q(X,\m)$
but gradient flow is constructed in $L^2(X,\m)$;
for this reason we assume $q \ge 2$ or $\m[X]<\infty$ to obtain gradient flow for $\Chc^+_q$
(Theorems~\ref{chgradinetflwo}, \ref{compprinconvr12}).
Nonetheless, thanks to Kuwada's lemma for relaxed gradients (Lemma~\ref{wssvolemm}),
we establish the coincidence of weak upper and relaxed gradients under $q \ge 2$ or $\m[X]<\infty$
(Theorem~\ref{twogradientcoindced}),  where an asymmetric Lisini-type result is indispensable; see \cite{Li, Li2}.
This  {\color{blue} extends} \cite[Theorem~7.3]{AGS4} in two respects:
the distance function can be asymmetric, and we do not assume $\m[X]<\infty$ when $q \ge 2$.
We also expand the equivalence to include relaxed upper gradients and upper gradients
(Theorem~\ref{conincdiegradee}).

We can define the Sobolev space associated with the $q$-Cheeger energy over asymmetric metric measure spaces.
Then, however, the forward energy is not suitable since the resulting Sobolev space may not be a linear space;
precisely, $\Chc^+_q(f)<\infty$ does not imply $\Chc^-_q(f)<\infty$
(see Example~\ref{Examphalfsobolevspace} concerning  the Funk metric \eqref{dist-reversfunk-00}).
Therefore, we naturally consider the Sobolev space $W^{1,q}_*(X,d,\m)$
in terms of the absolute $q$-Cheeger energy (Definition~\ref{standardsobolespaces}),
and show that it coincides with $W^{1,q}_0(\X,d_F,\m)$ for measured Finsler manifolds $(\X,F,\m)$
(Theorem~\ref{sobolevfinsler}).

  The Appendices are 
 devoted to a Lisini-type theorem  for forward extended Polish metric spaces, as well as 
to the Hopf--Lax semigroup and the Hamilton--Jacobi equation
in the asymmetric setting; these results   play a crucial role in the proof of Kuwada's lemma.

Our results are elaborated under various assumptions, each of them reflecting the appropriate geometric/analytic framework where a corresponding result holds (see also the table below which provides a rough overview of these connections):  

	\begin{enumerate}[\ \ $\bullet$]
\item   Assumption \ref{strongerstassumptiontheta}  (at the beginning of Section \ref{section-3})
requires the existence of a nonnegative, Borel, forward Lipschitz function $V$ on $X$, which is locally bounded and satisfies $\int_X {\ee}^{-V^2} \dm\leq 1$.

\item  Assumption \ref{newassumpt32liambdafinite}  (in Subsection \ref{ssc-3.2})
is stronger than Assumption
\ref{strongerstassumptiontheta}, as it additionally requires every compact
set to admit a backward neighborhood of finite $\m$-measure.

\item Assumption \ref{compactreversbiliey}  (in Subsection \ref{ssc-4.2})
requires either $(X,d)$ to have finite
reversibility or $(X,d)$ to be geodesic with every compact set having finite
diameter.

\item Assumption \ref{newassumpt32liambdafinite22} (in Subsection \ref{strongcheegerflow})
requires either finite reversibility together with
Assumption~\ref{strongerstassumptiontheta}, or finite total mass,
geodesicity of $(X,d)$, and finite diameter of every compact set.

\end{enumerate}

In particular, Assumption \ref{newassumpt32liambdafinite22} is the strongest among these four assumptions and implies all the other three.
In the symmetric case, Assumption \ref{compactreversbiliey} is automatically
satisfied, while Assumptions \ref{newassumpt32liambdafinite} and
\ref{newassumpt32liambdafinite22} both reduce to
Assumption \ref{strongerstassumptiontheta}.
In the asymmetric setting, Finsler structures provide a natural and
important class of examples.
For a measured  forward complete Finsler
manifold $(\X,F,\m)$ equipped with its natural topology, Assumption
\ref{compactreversbiliey} is automatically satisfied. Moreover, if either
$\m[\X]<\infty$ or $\lambda_F(\X)<\infty$ together with
${\rm Ric}_\infty\geq K$ (i.e., $\CD(K,\infty)$), then all the assumptions from
\ref{strongerstassumptiontheta} through
\ref{newassumpt32liambdafinite22} are satisfied.


We now
summarize the main results and the corresponding assumptions required
for each of them:

\begin{table}[htbp]
\centering
\small
\begin{tabularx}{\textwidth}
{>{\raggedright\arraybackslash}p{0.45\textwidth}
 @{\hspace{1cm}}
 X}
\toprule
\textbf{Main result} & \textbf{Required assumption} \\
\midrule

  General results \& Weak upper gradients 
&   Assumption \ref{strongerstassumptiontheta}  \\

 Weak Cheeger energy and its $L^2$-gradient flow
& Assumption \ref{newassumpt32liambdafinite} \\

Kuwada estimate by weak  upper gradient
& Assumption \ref{newassumpt32liambdafinite} \\

 Relaxed  gradients  
&   Assumption \ref{compactreversbiliey} \\

Relaxed Cheeger energy and its $L^2$-gradient flow
& Assumption \ref{newassumpt32liambdafinite22} and
($\m[X]<\infty$ or $q\geq 2$) \\

Kuwada estimate by relaxed   gradient
&   Assumption   \ref{newassumpt32liambdafinite22} and
($\m[X]<\infty$ or $q\geq 2$) \\

Identification of weak and relaxed gradients
& Assumption    \ref{newassumpt32liambdafinite22} and
($\m[X]<\infty$ or $q\geq 2$) \\

Identification of abstract and classical Sobolev spaces on Finsler manifolds
& Assumption  \ref{newassumpt32liambdafinite22} \\

\bottomrule
\end{tabularx}
\end{table}

We also remark that the spaces considered in this paper are required
to have locally finite reversibility (see Definition
\ref{forwcondt}), thereby ruling out cases in which the asymmetry
of the metric becomes arbitrarily large locally.
A typical example 
is the Wasserstein space $(\Po(\mathbb{B}^n),W_p)$ over the Funk
metric space $(\mathbb{B}^n,d_F)$ defined by \eqref{dist-reversfunk-00}.
Although the reversibility of $(\mathbb{B}^n,d_F)$ is locally finite,
there exists a sequence $(\mu_k)_{k=1}^\infty\subset\Po(\mathbb{B}^n)$
such that
\[
\lim_{k\to\infty}W_p(\mu_\infty,\mu_k)=0,
\qquad
W_p(\mu_k,\mu_\infty)=\infty
\quad\text{for all }k\in\mathbb{N}
\]
(see Example \ref{wasserteinspaceisasymmetric}).
Consequently, the reversibility of $(\Po(\mathbb{B}^n),W_p)$ is
infinite at $\mu_\infty$, and the forward and backward topologies
are incompatible. In such spaces, all the main results are no longer applicable.

A natural further problem is to show that the $q$-heat flow constructed in this article
coincides with gradient flow for the R\'enyi entropy in the $L^p$-Wasserstein space,
generalizing Kell's result \cite{Ke} to the asymmetric setting.
See also Remark~\ref{rm:further} for some possible applications of the analysis of Hopf--Lax semigroups.

\section{Preliminaries}\label{sc:prel}

\subsection{Asymmetric metric spaces}\label{asymsec}

\subsubsection{Setting}

We will follow and extend the setting of \cite{KZ,OZ}.

\begin{definition}[Asymmetric (extended) metric spaces]\label{generalsapcedef}
Let $X$ be a set and $d:X\times X\rightarrow [0,\infty]$.
We call $(X,d)$ {an} {\it asymmetric extended metric space} if for any $x,y,z \in X$:
\begin{enumerate}[{\rm (1)}]
\item $d(x,y)\geq 0$, with equality if and only if $x=y$;
\item $ d(x,z)\leq d(x,y)+d(y,z)$.
\end{enumerate}
We call $d$ a {\it metric} on $X$.
If $d$ is always finite, then $(X,d)$ is called an {\it asymmetric metric space}.
\end{definition}

Due to the asymmetry of $d$, we could have $d(x,y)<\infty$ while $d(y,x)=\infty$.
Moreover,  there are two kinds of balls.
The \emph{forward open ball} $B^+_x(r)$ and \emph{backward open ball} $B_x^-(r)$
of radius $r>0$ centered at $x \in X$ are defined as
\[
B^+_x(r):=\{y\in X \mid d(x,y)<r\}, \qquad B^-_x(r):=\{y\in X \mid d(y,x)<r\}.
\]
Let $\mathcal {T}_+$ (resp., $\mathcal {T}_-$) denote the topology induced by forward (resp., backward) open balls,
and let $\widehat{\mathcal {T}}$ be the topology induced by both forward and backward open balls.
We summarize some fundamental facts (see Mennucci\cite[\S 3.1]{Me} and Wilson \cite{Wi}).

\begin{theorem}\label{asymmetrtopol}
Let $(X,d)$ be an asymmetric extended metric space.
Then we have the following.
\begin{enumerate}[{\rm (i)}]
\item
$d:X\times X\rightarrow [0,\infty]$ is continuous under $\widehat{\mathcal {T}}\times \widehat{\mathcal {T}}$,
and $(X,d)$ is a Hausdorff space.

\item
The topology $\widehat{\mathcal {T}}$ is exactly the one induced from the symmetrized metric
\begin{equation}\label{symmmetricde}
\hat{d}(x,y) :=\frac12 \{ d(x,y)+d(y,x) \}.
\end{equation}

\item\label{equconv}
A sequence $(x_i)_{i \ge 1}$ converges to $x$ in $(X,\widehat{\mathcal {T}})$
if and only if both $d(x,x_i)\to 0$ and $d(x_i,x)\to 0$ hold as  $i\to \infty$.
\end{enumerate}
\end{theorem}

Note that $\widehat{\mathcal {T}}$ coincides with the topology induced from $\tilde{d}(x,y):=\max\{d(x,y), d(y,x) \}$
(see \cite[\S 3.1]{Me}), which is equivalent to $\hat{d}$ since $\hat{d} \leq \tilde{d} \leq 2\hat{d}$.

\begin{definition}[Completeness]\label{df:cmplt}
Let $(X,d)$ be {an} asymmetric extended metric space.
\begin{enumerate}[{\rm (1)}]
\item
A sequence $(x_i)_{i \ge 1}$ in $X$ is called
a {\it forward} (resp., {\it backward}) {\it Cauchy sequence} if, for each $\varepsilon>0$,
there is $N \ge 1$ such that $d(x_i,x_j)<\varepsilon$ (resp., $d(x_j,x_i)<\varepsilon$) holds for all $j \geq i \ge N$.

\item  $(X, d)$ is said to be {\it forward} (resp., {\it backward}) {\it complete} if
every forward (resp., backward) Cauchy sequence in $X$ converges with respect to $\widehat{\Tt}$.

\item We say that $(X,d)$ is {\it complete} if it is both forward and backward complete.
\end{enumerate}
\end{definition}

The following notion will play an important role.

\begin{definition}[Reversibility]\label{reversibilitydef}
Let $(X,d)$ be {an}  asymmetric extended metric space.
For a nonempty set $A\subset X$, define
\begin{equation*}
\lambda_d(A):=\inf\left\{\lambda\geq1 \mid d(x,y)\leq \lambda d(y,x) \text{ for all }x,y\in A  \right\}.
\end{equation*}
We call $\lambda_d(X)$ the {\it reversibility} of $(X,d)$.
\end{definition}

Clearly, $d$ is symmetric if and only if $\lambda_d(X)=1$.
We consider a decomposition of an asymmetric extended metric space $(X,d)$ into the disjoint union of asymmetric metric spaces.
Setting
\begin{equation}\label{equasetfintex}
X_{[x]}:=\{y\in X \mid d(x,y)<\infty\}, \qquad \overleftarrow{X}_{[x]}:=\{y\in X \mid d(y,x)<\infty\},\quad x\in X,
\end{equation}
if we have
\[
X=\bigsqcup_{\alpha\in \mathscr{A}} X_{[x_\alpha]}, \qquad X=\bigsqcup_{\beta\in \mathscr{B}} \overleftarrow{X}_{[x_\beta]}, \quad x_\alpha,x_\beta\in X
\]
for some index sets $\mathscr{A},\mathscr{B}$, then we call them \emph{forward} and \emph{backward finite decompositions} of $(X,d)$.

\begin{definition}[Forward (extended) metric space]\label{forwcondt}
An asymmetric extended metric space $(X,d)$ is called a {\it forward extended metric space} if
there are a forward finite decomposition
\begin{equation}\label{decompsiofx}
X=\bigsqcup_{\alpha\in \mathscr{A}} X_{[x_\alpha]}, \quad x_\alpha\in X,
\end{equation}
and a family of non-decreasing functions $\Theta_\alpha:(0,\infty)\rightarrow [1,\infty)$, $\alpha\in \mathscr{A}$,
such that
\begin{equation}\label{thecontric}
\lambda_d\big( B^+_{x_\alpha}(r)  \big)\leq \Theta_\alpha(r) \quad \text{for all}\ r>0.
\end{equation}
Moreover, if additionally $d$ is finite (so that $X=X_{[x]}$ for any $x \in X$), then $(X,d)$ is called a {\it forward metric space}.
\end{definition}

Note that \eqref{thecontric} implies the equivalence between $d(x,y)<\infty$ and $d(y,x)<\infty$, which means $X_{[x]}= \overleftarrow{X}_{[x]}$ for every $x\in X$.
We refer to  \cite{KZ,OZ}
for the theories of Gromov--Hausdorff topology, optimal transport and gradient flows for forward metric spaces.
By a similar argument to \cite[Theorem 2.6]{KZ}, we have the following.

\begin{proposition}\label{topoconditiondeq}
For a forward extended metric space $(X,d)$,
we have $\mathcal {T}_- \subset \mathcal {T}_+ =\widehat{\mathcal {T}}$.
\end{proposition}

Owing to Theorem \ref{asymmetrtopol} and Proposition \ref{topoconditiondeq},
it is convenient to endow every forward extended metric space $(X,d)$ with the forward topology $\mathcal {T}_+$.

\begin{remark}\label{forwardpointspaceandbackwardones}
\begin{enumerate}[(a)]
\item \label{pfms-b}
For a forward extended metric space $(X,d)$, given another forward finite decomposition
\[
X=\bigsqcup_{\alpha'\in \mathscr{A}'} X_{[x_{\alpha'}]},
\]
there always exists a family of non-decreasing functions $\widetilde{\Theta}_{\alpha'}$, $\alpha'\in \mathscr{A}'$, with
\begin{equation*}
\lambda_d\big( B^+_{x_{\alpha'}}(r)   \big)\leq \widetilde{\Theta}_{\alpha'}(r)\quad \text{for all}\ r>0.
\end{equation*}
In fact, for every $\alpha'\in \mathscr{A}'$, there exists some $\alpha\in \mathscr{A}$
such that $x_{\alpha'}\in X_{[x_\alpha]}$.
Then, we have $X_{[x_{\alpha'}]}=X_{[x_\alpha]}$ with the help of \eqref{thecontric} and
\[
\lambda_d\bigl( B^+_{x_{\alpha'}}(r) \bigr)
\leq \lambda_d\bigl( B^+_{x_{\alpha}}\left( d(x_\alpha,x_{\alpha'})+r \right) \bigr)
\leq \Theta_\alpha \big( d(x_\alpha,x_{\alpha'})+r \big).
\]
Hence, we can choose $\widetilde{\Theta}_{\alpha'}(r):=\Theta_\alpha( d(x_\alpha,x_{\alpha'})+r)$.

\item \label{pfms-c}
One can similarly introduce a \emph{backward extended metric space} $(X,d)$
by a backward finite decomposition $X=\bigsqcup_{\beta\in \mathscr{B}} \overleftarrow{X}_{[x_\beta]}$
and a family of non-decreasing functions $\Theta_\beta:(0,\infty)\rightarrow [1,\infty)$, $\beta\in \mathscr{B}$,
satisfying $\lambda_d(B^-_{x_\beta}(r))\leq \Theta_\beta(r)$ for all $r>0$.
Note that a backward extended metric space may not be a forward extended metric space;
see Example \ref{funkmetricsapce} below for instance.
However, since $(X,d)$ is a backward extended metric space
if and only if $(X, \overleftarrow{d})$ is a forward extended metric space
with the \emph{reverse metric} $\overleftarrow{d}(x,y):=d(y,x)$,
in the sequel we will focus only on  forward extended metric spaces.
\end{enumerate}
\end{remark}

\begin{lemma}\label{pfms-d}
An asymmetric extended metric space $(X,d)$ is a forward extended metric space
if and only if there exists a function $\Theta:X\times (0,\infty)\rightarrow [1,\infty)$ such that, for every $x\in X$,
\begin{itemize}
\item the function $\Theta_x(r):=\Theta(x,r)$ is non-decreasing in $r$,
\item $\lambda_d(B^+_x(r))\leq \Theta_x(r)$ for any $r>0$.
\end{itemize}
\end{lemma}

\begin{proof}
The ``if" part is clear by taking $\Theta_{\alpha} :=\Theta_{x_\alpha}$.
For the ``only if" part, define $\Theta:X\times (0,\infty)\rightarrow [1,\infty)$ as
$\Theta(x,r) := \Theta_\alpha (d(x_\alpha,x)+r)$, where $x \in X_{[x_\alpha]}$.
Note that, by \eqref{decompsiofx}, there is {\color{blue}a} unique $\alpha\in \mathscr{A}$ with $x \in X_{[x_\alpha]}$.
Then, for any $r>0$,
we have $\lambda_d(B^+_x(r)) \leq \lambda_d(B^+_{x_\alpha}(d(x_\alpha ,x)+r)) \leq \Theta_x(r)$.
\end{proof}

\begin{remark}\label{setreversiblitfuntion}
Given a set $K\subset X$, define $d(K,x):=\inf_{z\in K}d(z,x)$, $d(x,K):=\inf_{z\in K}d(x,z)$, and
\[
B^+_K(r):=\left\{x\in X \mid d(K,x)<r \right \},\qquad B^-_K(r):=\left\{x\in X \mid d(x,K)<r \right \}.
\]
If $K$ has
finite diameter, i.e.,
$\diam_d(K):=\sup_{x,y\in K}d(x,y)<\infty$,
then we can introduce
\begin{equation}\label{setreversbility}
\Theta_K(r) := \Theta_\alpha\bigl( d(x_\alpha,K)+\diam_d(K)+r \bigr), \quad
\text{where}\ K \subset X_{[x_\alpha]}.
\end{equation}
Observe that $\lambda_d(B^+_K(r))\leq \Theta_K(r)$ for any $r>0$.
\end{remark}

Inspired by  \cite{AGS2}, we introduce the following definition.

\begin{definition}[Forward (extended) Polish spaces]\label{weaktopolgytau}
A triple $(X,\tau,d)$ is called a {\it forward $($extended$)$ Polish space} if the following hold.
\begin{enumerate}[{\rm (1)}]
\item \label{weakt1-1} $(X,d)$ is a forward complete forward (extended) metric space
(endowed with the topology $\mathcal {T}_+$).

\item \label{weakt1-2} $\tau$ is a Polish topology of $X$, i.e.,
there is some {symmetric metric} $\rho$ such that $(X,\rho)$ is a complete separable metric space,
whose metric topology coincides with $\tau$.

\item \label{weakt1-3} For any sequence $(x_i)_{i \ge 1}$ in $X$,
the backward convergence $d(x_i,x)\to 0$ implies the $\tau$-convergence $x_i \to x$.

\item \label{weakt1-4} $d$ is lower semi-continuous in $X\times X$ with respect to the ($\tau\times \tau$)-topology.
\end{enumerate}
\end{definition}

Some remarks are in order.

\begin{remark}\label{exposlipspaceexplain}
\begin{enumerate}[(a)]
\item\label{rmpolish-a}
In \eqref{weakt1-1} we required that $(X,d)$ is forward complete.
Thus, it follows from Theorem \ref{asymmetrtopol} and Proposition \ref{topoconditiondeq}
that the symmetrized space $(X,\hat{d})$ is complete, and vice versa.
There is a stronger assumption that $(X,d)$ is backward complete,
in which case not only $(X,\hat{d})$ but also $(X,d)$ is complete, and vice versa
(cf.\ \cite[Proposition 2.8]{KZ}).

\item\label{rmpolish-b}
In \eqref{weakt1-2}, the symmetry of $\rho$ was required merely for convenience.
In fact, one can equivalently assume that $(X,\rho)$ is a forward complete, separable, forward metric space,
because in this case the symmetrized space $(X,\hat{\rho})$ is also a Polish space
by Theorem \ref{asymmetrtopol} and Proposition \ref{topoconditiondeq}.

\item\label{rmpolish-c}
In \eqref{weakt1-3}, although $x_i\to x$ with respect to $\widehat{\Tt}$ means
$d(x,x_i)\to 0$ and $d(x_i,x)\to 0$ (recall Theorem \ref{asymmetrtopol}\eqref{equconv}),
Proposition \ref{topoconditiondeq} points out that it is also equivalent to $d(x,x_i)\to 0$ only, i.e.,
$d(x,x_i)\to 0$ implies $d(x_i,x)\to 0$ (but not vice versa),
and then $x_i\overset{\tau}{\longrightarrow}x$ by \eqref{weakt1-3}.
Hence, $\tau\subset \widehat{\Tt}=\Tt_+$.
\end{enumerate}
\end{remark}

An important example is the following.

\begin{example}[Minkowski spaces]\label{ex:Randers}
Let $(\mathscr{H},\langle\cdot,\cdot\rangle)$ be a separable Hilbert space
and $\|x\|:=\sqrt{\langle x,x \rangle}$. Set
\[
X:=\overline{\mathbb{B}_\mathbf{0}(1)}:=\{x\in \mathscr{H}\mid \|x\|\leq 1 \}.
\]
Choose $\omega \in X$ with $\|\omega\|<1$ and define a function $d:X\times X \rightarrow [0,\infty)$ by
\[
 d(x,y) :=\|y-x\|+\langle \omega, y-x \rangle.
\]
Then $(X,d)$ is a complete forward metric space
with $\Theta_x \equiv (1+\|\omega\|)/(1-\|\omega\|)$ for any $x\in X$,
and $\mathcal{T}_+=\mathcal{T}_-$ coincides with the (strong) topology of $(X,\langle\cdot,\cdot\rangle)$.
Now, let $\tau$ denote the weak topology of $\mathscr{H}$ on $X$.
It follows from \cite[\S 5.1.2]{AGS} that $\tau$ can be induced from the norm
\[
\|x\|_\varpi:=\left(\sum_{i=1}^\infty \frac1{i^2}\langle x, e_i \rangle^2 \right)^{1/2},
\]
where $(e_i)_{i \ge 1}$ is an orthonormal basis of $\mathscr{H}$.
Set $\rho(x,y):=\|y-x\|_\varpi$.
Since $X$ is bounded with respect to $\|\cdot\|$,
the compactness of $(X,\rho)$ is seen by the argument before \cite[Definition 5.1.11]{AGS}.
Moreover, since
\[
\rho(x,y)\leq \left(\sum_{i=1}^\infty \frac1{i^2}\right)^{1/2}\|y-x\|
 \leq \left(\sum_{i=1}^\infty \frac1{i^2}\right)^{1/2} \frac{d(x,y)}{1-\|\omega\|}
 \quad \text{for all}\ x,y\in X,
\]
the separability of $(X,\rho)$ follows from that of $(X,\|\cdot\|)$.
Thus, $(X,\tau, d)$ is a forward extended Polish space.
\end{example}

In the sequel, unless other topology (e.g., $d$) is explicitly mentioned, all the topological notations
(e.g., compact sets, Borel sets, continuous functions) are referred to the topology $\tau$.

\begin{lemma}\label{twoclosedballcoincides}
Let $(X,\tau,d)$ be a forward extended Polish space.
Then, every forward or backward closed ball $\overline{B^\pm_{x_0}(r)}^d$ is closed, where
\[
\overline{B^+_{x_0}(r)}^d:=\{ x\in X \mid d(x_0,x)\leq r\}, \qquad
\overline{B^-_{x_0}(r)}^d:=\{ x\in X \mid d(x,x_0)\leq r\}.
\]
\end{lemma}

\begin{proof}
For any sequence $(x_i)_{i \ge 1}$ in $\overline{B^+_{x_0}(r)}^d$ with $x_i \overset{\tau}{\rightarrow }x$,
Definition \ref{weaktopolgytau}\eqref{weakt1-4} furnishes
\[
r\geq \liminf_{i \to \infty}d(x_0,x_i)\geq d(x_0,x),
\]
thereby $x \in \overline{B^+_{x_0}(r)}^d$ and $\overline{B^+_{x_0}(r)}^d$ is $\tau$-closed.
Similarly one can show the case of backward balls.
\end{proof}

\begin{lemma}
Let  $(X,\tau,d)$ be a forward extended Polish space.
Then, the Borel sets of $\widehat{\Tt}$ and $\tau$,
denoted by $\mathcal {B}_{\widehat{\Tt}}(X)$ and $\mathcal {B}(X)$, coincide.
\end{lemma}

\begin{proof}
On the one hand, every $d$-forward closed ball is $\tau$-closed by Lemma~\ref{twoclosedballcoincides}.
Since every $d$-forward open ball is realized as the union of $d$-forward closed balls, it is Borel in $\tau$;
thereby $\mathcal{B}_{\widehat{\Tt}}(X)\subset \mathcal{B}(X)$.
On the other hand, recall from Remark \ref{exposlipspaceexplain}\eqref{rmpolish-c} that $\tau \subset \widehat{\Tt}$.
Thus, $\mathcal{B}(X)\subset \mathcal {B}_{\widehat{\Tt}}(X)$.
\end{proof}

Recall that, on a Polish space, every Borel probability measure is tight.

\begin{proposition}\label{tightmeasper}
Let $(X,\tau,d)$ be a forward extended Polish space.
Then, every Borel probability measure $\mu$ on $X$ is tight with respect to $\tau$.
\end{proposition}

\subsubsection{Finsler manifolds}\label{Finslergeometry}

In this article, Finsler manifolds serve to construct various model examples.
Here we recall some definitions and properties; for details see \cite{BCS,Obook,Sh1}, etc.

Let $\mathcal {X}$ be an $n$-dimensional connected $C^\infty$-manifold without boundary
and $T{\X}=\bigcup_{x \in \X}T_{x} {\X}$ be its tangent bundle.
We say that $(\X,F)$ is a \emph{Finsler manifold} if $F:T{\X}\to [0,\infty)$ satisfies:
\begin{enumerate}[(1)]
\item\label{finslerpr1} $F\in C^{\infty}(T{\X}\setminus\{ 0 \});$
\item\label{finslerpr2} $F(x,\lambda y)=\lambda F(x,y)$ for all $\lambda\geq 0$ and $(x,y)\in T{\X};$
\item\label{finslerpr3} $g_{ij}(x,y) :=[\frac12F^{2}]_{y^{i}y^{j}}(x,y)$
is positive definite for all $(x,y)\in T{\X}\setminus\{ 0 \}$,
where $y=\sum_{i=1}^n y^i\frac{\partial}{\partial x^i}|_x$.
\end{enumerate}


The {\it reversibility} of $A \subset {\X}$ is defined as
\[
\lambda_F(A):=\sup_{x\in A} \sup_{y \in T_x{\X} \setminus \{0\}} \frac{F(x,-y)}{F(x,y)};
\]
see Rademacher \cite{Ra1, Ra2}.
We have $\lambda_F({\X})\geq 1$ with equality if and only if $F$ is {\it reversible} (i.e., $F(x,-y)=F(x,y)$).
Note also that $\lambda_F(x):=\lambda_F(\{x\})$ is a continuous function.

Denote by $T^*{\X}$ the cotangent bundle, and define the {\it dual metric} $F^*$ by
\begin{equation*}\label{dualFinslermetric}
F^*(x,\xi) :=\sup_{y \in T_x{\X} \setminus \{0\}} \frac{\xi(y)}{F(x,y)}
\end{equation*}
for $\xi \in T_x^* \X$.
Note that
\[
\xi(y) \leq F^*(x,\xi) F(x,y) \quad \text{for all}\ y\in T_x{\X},\, \xi \in T^*_x{\X}.
\]
The {\it Legendre transformation} $\mathfrak{L} : T_x{\X} \rightarrow T_x^*{\X}$ is defined by
\begin{equation*}
\mathfrak{L}(x,y):=\sum_{i,j=1}^n g_{ij}(x,y) y^i \,{\dd}x^j \ \text{ if } y\neq0,
\qquad
\mathfrak{L}(x,y):=0 \ \text{ if } y=0.
\end{equation*}
Then, $\mathfrak{L}:T{\X}\setminus\{0\}\rightarrow T^*{\X}\setminus\{0\}$ is a diffeomorphism
and $F^*(\mathfrak{L}(x,y))=F(x,y)$ for any $(x,y)\in T{\X}$.
For a $C^1$-function $f : M \rightarrow \mathbb{R}$,
the {\it gradient vector field} of $f$ is defined as $\nabla f := \mathfrak{L}^{-1}({\dd}f)$.
We remark that $\nabla$ is nonlinear (unless $F$ is Riemannian), i.e.,
$\nabla(f+h)\neq\nabla f+\nabla h$.

Let $C^\infty_{\loc}([0,1];{\X})$ denote the set of piecewise smooth curves defined on $[0,1]$.
Given $\gamma\in C^\infty_{\loc}([0,1];{\X})$, its \textit{length} is defined as
\[
L_F(\gamma):=\int^1_0 F\bigl(\gamma(t), \gamma'(t) \bigr) \,{\dd}t.
\]
Define the {\it distance function} $d_F:\X\times \X \rightarrow [0,\infty)$ by
\[
d_{F}(x_0,x_1):=\inf\{L_F(\gamma) \mid \gamma\in C^\infty_{\loc}([0,1];{\X}),\,
\gamma(0)=x_0,\, \gamma(1)=x_1\}.
\]
Then $({\X},d_F)$ is an asymmetric metric space,
and the forward topology (as well as the backward topology) coincides with the original topology of the manifold.
In the Finsler setting, we always choose the topology $\tau$ as the original topology of $\X$,
and hence, $\tau={\Tt}_\pm=\widehat{\Tt}$.

A \emph{geodesic} $\gamma:[0,1] \rightarrow \X$ is understood as a constant speed, locally minimizing curve,
in the sense that $d_F(\gamma(s),\gamma(t)) =[(t-s)/(b-a)] \cdot d_F(\gamma(a),\gamma(b))$
for sufficiently short intervals $[a,b]$ and any $a \le s<t \le b$.
A Finsler manifold $(\X,F)$ is said to be {\it forward} (resp., {\it backward}) {\it complete}
if so is $(\X,d_F)$, in which case every geodesic defined on $[0,1]$ (resp., $[-1, 0]$)
can be extended to $[0, \infty)$ (resp., $(-\infty, 0]$).
According to the Hopf--Rinow theorem (see \cite{BCS,Obook,Sh1}),
the closure of a forward ball (with a finite radius) is compact if $(\X,F)$ is forward complete.
However, this is not the case for backward balls (see Example \ref{funkmetricsapce}).

According to \cite[Theorem 2.23]{KZ}, we have the following.

\begin{proposition}\label{finslermetricspace}
Any asymmetric  metric space $({\X},d_F)$ induced from a forward complete Finsler manifold $({\X},F)$
is a forward metric space with
\[
\lambda_{d_F}\bigl( B^+_x(r) \bigr) \leq \Theta_x(r)
 :=\lambda_F\Bigl( B_x^+\bigl(2r+ \lambda_F\bigl( B^+_x(r) \bigr) r \bigr) \Bigr),
 \quad x\in \X,\, r>0.
\]
\end{proposition}

Let $\m$ be a smooth positive measure on $\X$.
In a local coordinate system ($x^i$), $\m$ can be expressed as
\begin{equation*}
\m({\dd}x) = \sigma\, {\dd}x^1 \cdots {\dd}x^n.
\end{equation*}
For instance, the {\it Busemann--Hausdorff measure}
$\m_{\mathrm{BH}}$ is defined by
\[
\m_{\mathrm{BH}}({\dd}x):=\frac{\vol(\mathbb{S}^{n-1})}{\vol(S_x\X)}\, {\dd}x^1 \cdots {\dd}x^n,
\]
where $\mathbb{S}^{n-1} \subset \mathbb{R}^n$ is the $(n-1)$-dimensional Euclidean
sphere and $S_x\X:=\{y\in T_xM\,|\, F(x,y)=1\}$ identified with a subset of $\mathbb{R}^n$ via the chart $(x^i)$.
The {\it distortion} $\tau$ and the {\it S-curvature} $\mathbf{S}$ of $(\X,F,\m)$ are defined  by
\begin{equation*}
\tau(x,y):= \log \Biggl[ \frac{\sqrt{\det g_{ij}(x,y)}}{\sigma(x)} \Biggr], \qquad
 \mathbf{S}(x,y):=\left.\frac{\dd}{{\dd}t}\right|_{t=0}\tau\bigl( \gamma_y(t), {\gamma}'_y(t) \bigr), \quad y\in T_x\X \setminus \{0\},
\end{equation*}
respectively, where  $t\mapsto \gamma_y(t)$ is the geodesic with $\gamma'_y(0)=y$.

According to Ohta \cite{Oint}, the {\it weighted Ricci curvature} $\mathbf{Ric}_N$ is defined as follows:
given $N\in [n,\infty]$, for any unit vector $y\in S\X:=\bigcup_{x\in \X}S_x\X$,
\[
\mathbf{Ric}_N(y):=\begin{cases}
\mathbf{Ric}(y)+\left.\frac{\dd}{{\dd}t}\right|_{t=0}\mathbf{S}(\gamma'_y(t))-\frac{\mathbf{S}^2(y)}{N-n} & \text{ for }N\in (n, \infty),\\
\displaystyle\lim_{L\downarrow n} \mathbf{Ric}_L(y) & \text{ for }N=n, \\
\mathbf{Ric}(y)+\frac{\dd}{{\dd}t}\big|_{t=0} \mathbf{S}(\gamma'_y(t)) & \text{ for }N=\infty.
\end{cases}
\]
Weighted Ricci curvature enables us to generalize many results in Riemannian geometry to the Finsler setting (see \cite{Oint,Obook}).
For example, $\mathbf{Ric}_N(y) \ge KF^2(y)$ for all $y \in T\X$ is equivalent to the \emph{curvature-dimension condition} $\CD(K,N)$.

The following is a model example of a forward complete forward metric space we have in mind.

\begin{example}\label{funkmetricsapce}
Let $\mathbb{B}^n=\{x\in \mathbb R^n \mid \|x\|<1\}$ be the open unit ball,
where $\|\cdot\|$ and $\langle\cdot, \cdot\rangle$ denote the Euclidean norm and inner product, respectively.
The \textit{Funk metric} (see, e.g., \cite{Sh1})
$F:\mathbb{B}^n\times \mathbb R^{n}\to \mathbb R$ is defined  by
\begin{equation*}\label{Funckmeatirc}
F(x,y) =\frac{\sqrt{\|y\|^2-(\|x\|^2\|y\|^2-\langle x,y\rangle^2)}}{1-\|x\|^2}+\frac{\langle x,y\rangle}{1-\|x\|^2},
 \quad x \in \mathbb{B}^n,\ y\in T_x\mathbb{B}^n=\mathbb R^n.
\end{equation*}
The distance function associated to $F$ is explicitly written as
\[
d_{F}(x_1,x_2)=\log \Biggl[ \frac{\sqrt{\|x_1-x_2\|^2-(\|x_1\|^2\|x_2\|^2-\langle x_1,x_2\rangle^2)}-\langle x_1,x_2-x_1\rangle}{\sqrt{\|x_1-x_2\|^2-(\|x_1\|^2\|x_2\|^2-\langle x_1,x_2\rangle^2)}-\langle x_2,x_2-x_1\rangle} \Biggr],
\quad x_1,x_2\in \mathbb{B}^n.
\]
Clearly $d_F$ is asymmetric, moreover,
\begin{equation}\label{reversfunk}
\lim_{\|x\|\to1^-}d_F(\mathbf{0},x)=\infty, \qquad \lim_{\|x\|\to1^-} d_F(x,\mathbf{0})=\log2.
\end{equation}
By Proposition~\ref{finslermetricspace},
$(\mathbb{B}^n,d_F)$ is a forward complete forward metric space with $\Theta_{\mathbf{0}}(r):=2{\ee}^r-1$.
However, it is not a backward metric space because \eqref{reversfunk} implies
$\lambda_{d_F}(B^-_\mathbf{0}(r))=\lambda_{d_F}(\mathbb{B}^n)=\infty$ for $r\geq \log 2$.
Moreover, we have $\mathbf{Ric}(y)=-(n-1)/4$ and, with respect to the Lebesgue measure $\m_{\mathrm{L}}$,
$\mathbf{S}(y)=(n+1)/2$ for any $y\in S\mathbb{B}^n$.
Hence, $(\mathbb{B}^n,d_F,\m_{\mathrm{L}})$ satisfies $\CD\left(-\frac{n-1}{4}-\frac{(n+1)^2}{4(N-n)},N\right)$ for any $N\in (n,\infty]$ (see Ohta \cite{Ofunk} as well as Example~\ref{ex-1} below).
\end{example}

\begin{example}
Let $({\X}_i,F_i)$, $i\in \mathbb{N}$, be a sequence of forward complete Finsler manifolds.
Set ${\X}:=\bigsqcup_i {\X}_i$ and $d|_{\X_i\times \X_i}:=d_{F_i}$, $d|_{\X_i\times \X_j}:=\infty$ for $i\neq j$.
Then, $({\X},d)$ is a forward extended metric space with $\Tt_+=\Tt_-=\widehat{\Tt}$.
\end{example}

\subsection{Absolutely continuous curves}\label{asconsec}

\begin{definition}[Forward absolutely continuous curves]\label{absoldef}
Let $(X,d)$ be an asymmetric extended metric space and $I$ be an interval in $[-\infty,\infty]$.
Given $p\in [1,\infty]$, a curve $\gamma:I\rightarrow X$ is said to be
{\it forward} (resp., \emph{backward}) \emph{$p$-absolutely continuous}
if there is a nonnegative function $f\in L^p(I)$ such that
\[
d\bigl( \gamma(t_1),\gamma(t_2) \bigr) \leq \int^{t_2}_{t_1}f(s) \,{\dd}s \quad
\biggl(\text{resp.,}\ d\bigl( \gamma(t_2),\gamma(t_1) \bigr) \leq \int^{t_2}_{t_1}f(s) \,{\dd}s \biggr)
\]
for any $t_1,t_2\in I$ with $t_1\leq t_2$.
The set of forward (resp., backward) $p$-absolutely continuous curves is denoted by
$\FAC^p(I;X)$ (resp., $\BAC^p(I;X)$).
We say that $\gamma$ is {\it $p$-absolutely continuous} if
\[ \gamma\in \AC^p(I;X) := \FAC^p(I;X) \cap \BAC^p(I;X). \]
We denote $\FAC^1(I;X)$, $\BAC^1(I;X)$ and $\AC^1(I;X)$
by $ \FAC(I;X)$, $\BAC(I;X)$ and $\AC(I;X)$, respectively.
\end{definition}

In the symmetric case, absolutely continuous curves defined on open intervals
are continuously extended to closed intervals.
In the asymmetric case, however, it is not the case as in the next example
(see \cite[\S 2.2]{OZ} for a further discussion).

\begin{example}\label{forwardabsoucuveextend}
Let $(\mathbb{B}^n,d_F)$ be the Funk metric space as in Example \ref{funkmetricsapce},
and $\gamma:(0,1] \rightarrow \mathbb{B}^n$ be the constant speed minimal geodesic
such that $\gamma(1)=\mathbf{0}$ and $\gamma(t) \to (-1,0,\ldots,0)$ in $\mathbb{R}^n$ as $t \to 0$.
Then, $d_F(\gamma_s,\gamma_t)=(t-s)\log 2$ for $0<s\leq t\leq 1$,
thus $\gamma$ is forward absolutely continuous.
However, $\gamma$ is not backward absolutely continuous and cannot be extended to $t=0$.
\end{example}

\begin{remark}\label{lengfinite}
In an asymmetric metric space, forward absolutely continuous curves could be discontinuous (in $\widehat{\Tt}$).
For example, consider $X=\mathbb{R}$ endowed with the asymmetric distance:
\[
d(x,y) :=y-x\ \text{ for } x \le y, \qquad d(x,y) :=1\ \text{ for } x>y.
\]
Then $\gamma(t):=t$ ($t\in [0,1]$) is forward absolutely continuous,
but is discontinuous in $\widehat{\Tt}$ (recall Theorem \ref{asymmetrtopol}\eqref{equconv}).
In fact, $X$ is a discrete space under $\widehat{\Tt}$.
%
\end{remark}

For a forward extended metric space,
the relationship between forward and backward absolutely continuous curves reads as follows
(see Zhang--Zhao \cite[Proposition 5.3, Remark 8]{ZZ}).

\begin{proposition}\label{extabslcurv}
Let $(X,d)$ be a forward extended metric space and $I$ be a bounded closed interval.
Then we have
\[
\AC^p(I;X) =\FAC^p(I;X) \subset \BAC^p(I;X) \cap C(I;X)
\]
for all $p\in [1,\infty]$,
where $C(I;X)$ denotes the set of continuous curves $($in $\Tt_+=\widehat{\Tt})$ defined on $I$.
Moreover, if $\BAC^p(I;X) \subset C(I;X)$, then
\[
\AC^p(I;X) =\FAC^p(I;X) =\BAC^p(I;X).
\]
\end{proposition}

In particular, the set of $p$-absolutely continuous curves for the symmetrized distance
$\hat{d}$ (recall \eqref{symmmetricde}) coincides with $\AC^p(I;X)$.
Let $\mathscr{L}^n$ denote the Lebesgue measure of $\mathbb{R}^n$.
According to Rossi--Mielke--Savar\'e\cite[Proposition 2.2]{RMS}, we have the following.

\begin{proposition}[Forward metric derivative]\label{speedofabscc}
Let $(X,d)$ be an asymmetric extended metric space and $I$ be an interval in $[-\infty,\infty]$.
For any $\gamma\in \FAC^p(I;X)$ $($resp., $\gamma\in \BAC^p(I;X))$ with $p\in [1,\infty]$,
the \emph{forward} $($resp., \emph{backward}$)$ \emph{metric derivative}
\begin{align*}
|\gamma'_+|(t)&:=\lim_{h\to 0^+}\frac{d(\gamma(t),\gamma(t+h))}{h}=\lim_{h\to 0^+}\frac{d(\gamma(t-h),\gamma(t))}{h}\\
 \biggl( \text{resp.,\ } |\gamma'_-|(t)&:=\lim_{h\to 0^+}\frac{d(\gamma(t+h),\gamma(t))}{h}=\lim_{h\to 0^+}\frac{d(\gamma(t),\gamma(t-h))}{h} \biggr)
\end{align*}
exists for $\mathscr{L}^1$-a.e.\ $t\in \intt(I)$.
Furthermore, $|\gamma'_+|$ $($resp., $|\gamma'_-|)$ is in $L^p(I)$ and satisfies
\[
d\bigl( \gamma(t_1),\gamma(t_2) \bigr) \leq \int^{t_2}_{t_1}|\gamma'_+|(s) \,{\dd}s \quad
\biggl( \text{resp.,\ } d\bigl( \gamma(t_2),\gamma(t_1) \bigr) \leq \int^{t_2}_{t_1}|\gamma'_-|(s) \,{\dd}s \biggr)
\]
for any $t_1,t_2\in I$ with $t_1\leq t_2$.
\end{proposition}

\begin{definition}[Rectifiable curves]
Let $(X,d)$ be an asymmetric extended metric space and $\gamma:[0,1]\rightarrow X$ be a continuous curve.
Consider a partition $Y=\{t_0,\ldots ,t_N\}$ of $[0,1]$ such that $0=t_0\leq t_1\leq \cdots \leq t_N=1$.
The supremum of the sum
$\Sigma(Y):=\sum_{i=1}^N d(\gamma(t_{i-1}),\gamma(t_i))$ 
over all partitions $Y$ is called the {\it length} of $\gamma$ and denoted by $L_d(\gamma)$.
A curve is said to be {\it rectifiable} if its length is finite.
\end{definition}

As in Burago--Burago--Ivanov \cite[Theorem 2.7.6]{DYS}, we see that
\begin{equation}\label{lengthabscur}
L_d(\gamma)=\int^1_0 |\gamma'_+|(t) \,{\dd}t <\infty \quad \text{for all}\ \gamma\in \FAC([0,1];X).
\end{equation}
For a forward extended Polish space $(X,\tau,d)$,
denote by $C_\tau ([0,1];X)$ the collection of $\tau$-continuous curves defined on $[0,1]$.
It follows from Definition \ref{weaktopolgytau}\eqref{weakt1-3} that $\AC([0,1];X)\subset C_\tau ([0,1];X)$.
We will endow $C_\tau ([0,1];X)$ with the {\it uniform distance}
$d^*(\gamma,\tilde{\gamma}):=\sup_{t\in [0,1]}d(\gamma(t),\tilde{\gamma}(t))$
and the {\it compact-open topology} $\tau^*$ generated by a subbase
$\{ \gamma\in C_\tau ([0,1];X) \mid \gamma(K)\subset U \}$,
where $K \subset [0,1]$ is compact and $U\in \tau$.

\begin{proposition}\label{curvetopology}
Let $(X,\tau,d)$ be a forward extended Polish space.

\begin{enumerate}[{\rm  (i)}]
\item \label{curvespace-2}
For each $t\in [0,1]$, the \emph{evaluation map} $e_t:(C_\tau ([0,1];X),\tau^*)\rightarrow(X,\tau)$,
$\gamma \mapsto \gamma(t)$, is continuous.

\item \label{curvespace-1}
Suppose that the symmetric distance $\rho$ inducing $\tau$ satisfies $\rho \le d$.
Then, the following hold.
\begin{itemize}
\item $(C_\tau ([0,1];X),\tau^*)$ is a Polish space and $d^*$ is lower semi-continuous in $\tau^*$.

\item If additionally $\Theta_K(r)$ defined in \eqref{setreversbility} is finite
for any compact set $K\subset (X,\tau)$ and $r>0$,
then $(C_\tau ([0,1];X),\tau^*,d^*)$ is a forward extended Polish space.
\end{itemize}
\end{enumerate}
\end{proposition}

\begin{proof}
\eqref{curvespace-2} is straightforward from the definition of $\tau^*$.
For \eqref{curvespace-1}, note that the compact-open topology $\tau^*$ is exactly the topology
induced from the uniform distance
$\rho^*(\gamma,\tilde{\gamma}):=\sup_{t\in [0,1]} \rho(\gamma(t),\tilde{\gamma}(t))$.
Thus, $(C_\tau ([0,1];X),\tau^*)$ is a Polish space since $(X,\rho)$ is Polish (see Garling \cite[Theorem 3.3.12]{Gar}).
The lower semi-continuity of $d^*$ directly follows from that of $d$
(recall Definition \ref{weaktopolgytau}\eqref{weakt1-4}).

For the latter part of \eqref{curvespace-1}, given $\gamma\in C_\tau ([0,1];X)$,
since $\gamma([0,1])$ is compact in $\tau$,
the function $\Theta_{\gamma}(r):=\sup_{t\in [0,1]}\Theta_{\gamma(t)}(r)$ is well-defined.
Then, for any
\[
\eta_1,\eta_2 \in \mathfrak{B}^+_{\gamma}(r)
 :=\left\{ \xi \in C_\tau ([0,1];X) \mid d^*(\gamma,\xi)<r  \right\}
\]
and $t\in [0,1]$, since $\eta_1(t),\eta_2(t)\in {B}^+_{\gamma(t)}(r)$,
we infer that $d(\eta_1(t),\eta_2(t)) \leq \Theta_{\gamma(t)}(r) \cdot d(\eta_2(t),\eta_1(t))$.
Hence, $d^*(\eta_1,\eta_2)\leq \Theta_{\gamma}(r)\cdot d^*(\eta_2,\eta_1)$,
and $(C_\tau ([0,1];X),d^*)$ is a forward extended metric space by Lemma~\ref{pfms-d}.
Moreover, together with the hypothesis $\rho\leq d$, we find that
$d^*(\gamma_i,\gamma) \to 0$ implies $\gamma_i \overset{\tau^*}{\longrightarrow }\gamma$.

It remains to show the forward completeness.
Given a forward Cauchy sequence $\gamma_i\in C_\tau ([0,1];X)$,
there is $N_0 \ge 1$ such that $\gamma_i\in \mathfrak{B}^+_{\gamma_{N_0}}(1)$ for all $i\geq N_0$.
Moreover, for any $\varepsilon>0$, there is $N=N(\varepsilon)>N_0$ such that
$d^*(\gamma_i,\gamma_j)<\varepsilon/\Theta_{\gamma_{N_0}}(1)$ for any $j\geq i\geq N$.
On the other hand, the forward completeness of $(X,d)$ yields a curve $\gamma:[0,1]\rightarrow X$
such that $\lim_{j\to\infty}d(\gamma(t),\gamma_j(t))=0$ for every $t\in [0,1]$.
Then we have
\begin{align*}
d\bigl( \gamma(t),\gamma_i(t) \bigr)
&\leq d\bigl( \gamma(t),\gamma_j(t) \bigr) +d\bigl( \gamma_j(t),\gamma_i(t) \bigr)
 \leq d\bigl( \gamma(t),\gamma_j(t) \bigr) +\Theta_{\gamma_{N_0}}(1)\cdot d^* (\gamma_i,\gamma_j) \\
&\leq d\bigl( \gamma(t),\gamma_j(t) \bigr) +\varepsilon
\end{align*}
for all $t \in [0,1]$.
By letting $j\to \infty$, we obtain $d(\gamma(t),\gamma_i(t))\leq \varepsilon$,
thereby $d^*(\gamma,\gamma_i)\leq\varepsilon$ for all $i \ge N$.
Thus, $\gamma$ is continuous  under $\rho$ (i.e., $\tau$), and the forward completeness holds.
\end{proof}

We remark that the assumption $\Theta_K(r)<\infty$ holds true when $\tau=\widehat{\Tt}$.


\begin{definition}[$p$-energy]\label{pengergycurve}
Let $(X,d)$ be an asymmetric extended metric space.
For $p \in (1,\infty)$, the ({\it forward}) \emph{$p$-energy} is defined by
\[
\mathcal {E}_p(\gamma) :=\int_{0}^1 |\gamma'_+|^p(t) \,{\dd}t\,\ \text{ if } \gamma\in \FAC^p([0,1];X),
\qquad
\mathcal {E}_p(\gamma) :=\infty\ \text{ otherwise.}
\]
\end{definition}


\begin{lemma}\label{poslishspeccurves}
Let $(X,\tau,d)$ be a forward extended Polish space.
Then,  for $p \in (1,\infty)$, $\mathcal {E}_p$ is $\tau^*$-lower semi-continuous
and $\AC^p([0,1];X)$ is a Borel set in $C_\tau ([0,1];X)$.
\end{lemma}

\begin{proof}
Observe that, in view of Proposition \ref{extabslcurv},
$\AC^p([0,1];X) =\FAC^p([0,1];X) =\mathcal {E}_p^{-1}([0,\infty))$.
Hence, it suffices to show that $\mathcal {E}_p$ is $\tau^*$-lower semi-continuous
(cf.\ \cite[\S 2.2]{AGS2}).

Let $(\gamma_i)_{i \ge 1} \subset \AC^p([0,1];X)$ be a sequence $\tau^*$-converging to $\gamma$.
Without loss of generality, we assume $\liminf_{i \to \infty} \mathcal {E}_p(\gamma_i) <\infty$.
Then, the reflexivity of the $L^p$-space yields that (a subsequence of)
$f_i :=|(\gamma_i)'_{+}|$ weakly converges to some nonnegative function $f\in L^p([0,1])$.
Hence,
\[
\liminf_{i \to \infty} \int_0^1 f_i(r) \,{\dd}r =\int_0^1 f(r) \,{\dd}r\leq \|f\|_{L^p([0,1])}<\infty.
\]
By the $\tau$-lower semi-continuity of $d$, we have, for any $0\leq s<t \leq 1$,
\[
d\bigl( \gamma(s),\gamma(t) \bigr)
\leq \liminf_{i \to \infty} d\bigl( \gamma_i(s),\gamma_i(t) \bigr)
\leq \liminf_{i \to \infty} \int_{s}^t |(\gamma_i)'_{+}|(r) \,{\dd}r,
\]
which implies $d(\gamma(s),\gamma(t)) \leq \int^t_s f(r) \,{\dd}r$,
and hence $|\gamma'_+|\leq f$ $\mathscr{L}^1$-a.e.\ on $[0,1]$.
Thus, $|\gamma'_+|\in L^p([0,1])$ and $\gamma\in \AC^p([0,1];X)$.
Finally, by the lower semi-continuity of the norm $\|\cdot\|_{L^p([0,1])}$ under weak convergence, we obtain
\[
\liminf_{i \to\infty}\mathcal {E}_p(\gamma_{i})
 =\liminf_{i \to \infty}\|f_i\|^p_{L^p([0,1])}
 \ge \|f\|^p_{L^p([0,1])}
 \geq \bigl\| |\gamma'_+| \bigr\|_{L^p([0,1])}^p
 =\mathcal {E}_p(\gamma).
\]
Therefore, $\mathcal {E}_p$ is lower semi-continuous.
\end{proof}


\subsection{Slopes and Lipschitz functions}\label{sloplifunsec}

For $a \in \mathbb{R}$, we set $a^+:=\max\{a,0\}$ and $a^-:=\max\{-a,0\}=[-a]^+$
(thus $a=a^+ -a^-$).

\begin{definition}
Let $(X,d)$ be a forward extended metric space.
For $f:X\rightarrow [-\infty,\infty]$, denote by $\mathfrak{D}(f):=f^{-1}(\mathbb{R})$ the {\it effective domain of $f$}.
The {\it local Lipschitz constant} at $x\in \mathfrak{D}(f)$ is defined by
\[
|Df|(x):=\limsup_{y\to x}\frac{|f(y)-f(x)|}{d(x,y)}.
\]
We also define the {\it descending} and {\it ascending slopes} at $x\in \mathfrak{D}(f)$ as
\[
|D^- f|(x):=\limsup_{y\to x}\frac{[f(x)-f(y)]^+}{d(x,y)},\qquad
|D^+ f|(x):=\limsup_{y\to x}\frac{[f(y)-f(x)]^+}{d(x,y)},
\]
respectively.
When $x\in \mathfrak{D}(f)$ is an isolated point of $X$, we set
$|Df|(x)=|D^-f|(x)=|D^+f|(x):=0$.
All slopes are set to be $\infty$ on $X\setminus \mathfrak{D}(f)$.
\end{definition}


\begin{example}\label{forwardbackwardslop}
Let $({\X},d_F)$ be an asymmetric metric space induced by a Finsler manifold $({\X},F)$.
For any $C^1$-function $f$ on $\X$,
a direct calculation implies $|D^\pm f|(x)=F(x,\nabla(\pm f))=F^*(x,\pm {\dd}f)$
and $|Df|(x)=\max\{|D^\pm f|(x)\}$.
Note that $|D^+ f|(x) \neq |D^- f|(x)$ in general due to the asymmetry of $d_F$.
\end{example}

The same argument as in \cite[\S 2.2]{AGS2} yields the following,
with the convention $0 \cdot \infty=\infty$.

\begin{lemma}\label{derivatnorm}
Let $(X,d)$ be a forward extended metric space and $f,g:X\rightarrow [-\infty,\infty]$.
\begin{enumerate}[{\rm (i)}]
\item\label{derivatnorm1}
For all $x\in \mathfrak{D}(f)$, we have
$|D f|(x)=\max\{ |D^+ f|(x),|D^- f|(x) \}$ and $|D^-f|(x)=|D^+(-f)|(x)$.

\item\label{derivatnorm2}
On $\mathfrak{D}(f)\cap \mathfrak{D}(g)$, for all $\alpha,\beta\in \mathbb{R}$, we have
\[
|D(fg)|\leq |f||Dg|+|g||Df|, \qquad |D(\alpha f+\beta g)|\leq |\alpha||D f|+|\beta||Dg|.
\]
Moreover, for any $\alpha,\beta>0$,
\begin{equation}\label{postivegradnormineq}
|D^\pm (\alpha f)|=\alpha|D^\pm f|,\qquad
 |D^\pm (\alpha f+\beta g)|\leq \alpha |D^\pm f|+\beta |D^\pm g|.
\end{equation}

\item\label{derivatnorm3}
For $\chi:X\rightarrow [0,1]$, we have
\[
\bigl| D^\pm \bigl( \chi f+(1-\chi)g \bigr) \bigr| \leq \chi |D^\pm f|+(1-\chi)|D^\pm g|+|D\chi||f-g|.
\]
\end{enumerate}
\end{lemma}

We remark that $|D^+(fg)| \leq f|D^+ g| +g|D^+ f|$ also holds if $f$ and $g$ are nonnegative.

\begin{definition}[Lipschitz functions]\label{forwadefsepaere}
Let $(X,d)$ be a forward extended metric space.
\begin{enumerate}[(1)]
\item
The {\it reversibility} at $x\in X$ is defined as $\lambda_d(x):=\lim_{r\to 0^+}\lambda_d({B}^+_x(r))$.

\item
A function $f:X\rightarrow \mathbb{R}$ is said to be {\it forward Lipschitz}
if there is some constant $C>0$ such that $f(y)-f(x)\leq Cd(x,y)$ for all $x,y\in X$.
The smallest constant $C$ is denoted by $\Lip(f)$.

\item
We say that $f$ is \emph{backward Lipschitz} if $-f$ is forward Lipschitz,
and then set $\RL(f):=\Lip(-f)$.

\item
A function $f:X\rightarrow \mathbb{R}$ is said to be {\it Lipschitz}
if there is some constant $C>0$ such that $|f(y)-f(x)|\leq Cd(x,y)$ for all $x,y\in X$.
The smallest constant $C$ is denoted by $\SL(f)$.
\end{enumerate}
\end{definition}

We remark that forward Lipschitz functions are called Lipschitz functions in some literature
(e.g., \cite{Obook}).
In our notation, the distance function $d(x_0,\cdot)$ from $x_0$ is forward Lipschitz
but not necessarily backward Lipschitz.
Observe also that $f$ is Lipschitz if and only if it is both forward and backward Lipschitz. See also Daniilidis--Jaramillo--Venegas \cite{DJV} for more discussions on Lipschitz functions in the Finsler setting.

\begin{lemma}\label{Lipscontin}
Let $(X,d)$ be a forward extended metric space.
\begin{enumerate}[{\rm (i)}]
\item\label{Lispro1}
For every forward Lipschitz function $f$, we have
\[
|D^+f|(x)\leq \Lip(f),\qquad |D^-f|(x)\leq \lambda_d(x) \Lip(f)
 \qquad \text{for all $x \in X$.}
\]
In particular, $|Df|(x)<\infty$ for all $x\in X$.

\item\label{Lispro2}
If $f:X\rightarrow \mathbb{R}$ is forward or backward Lipschitz, then it is continuous.
\end{enumerate}
\end{lemma}

\begin{proof}
\eqref{Lispro1} follows from the definition of $|D^+ f|(x)$.
To see \eqref{Lispro2}, it suffices to consider forward Lipschitz functions.
For any $x,y\in X$ with $d(x,y)<\infty$, we have $f(y)-f(x) \leq \Lip(f) d(x,y)$ and,
by Lemma~\ref{pfms-d},
\begin{equation*}
f(y)-f(x) \ge -\Lip(f)d(y,x) \ge -\Lip(f) \Theta_x\bigl( d(x,y) \bigr) d(x,y).
\end{equation*}
Letting $y \to x$ in $\mathcal{T}_+ =\widehat{\mathcal {T}}$ implies the continuity of $f$.
\end{proof}

A forward extended metric space $(X,d)$ is called a {\it geodesic space} if, for any $x,y\in X$ with $d(x,y)<\infty$,
there is a curve $\gamma$ from $x$ to $y$ with $L_d(\gamma)=d(x,y)$.
Such $\gamma$ of constant speed is called a \emph{minimal geodesic}.

\begin{lemma}\label{forwardlipsislpsc}
Let $(X,d)$ be a forward extended metric space and $f:X\rightarrow \mathbb{R}$ be forward Lipschitz.
Then, $f$ is Lipschitz if one of the following conditions holds.
\begin{enumerate}[{\rm (a)}]
\item\label{forwardislip1} The reversibility of $(X,d)$ is finite.
\item\label{forwardislip2} $(X,d)$ is a geodesic space and $\supp(f)$ is of finite diameter.
\end{enumerate}
\end{lemma}

\begin{proof}
Clearly \eqref{forwardislip1} implies the claim.
Thus, we assume \eqref{forwardislip2}.
Put $D:=\diam_d(\supp(f))+1<\infty$ and fix a point $\star\in \supp(f)$.
Then, for any $x,y\in  B^+_{\star}(D)$ ($\supset\supp(f)$), we have
\begin{equation}\label{fintechies}
f(y)-f(x) \ge -\Lip(f) d(y,x) \ge -\Lip(f) \Theta_\star(D) d(x,y).
\end{equation}
Note that \eqref{fintechies} also holds for $x,y\notin B^+_{\star}(D)$ since then $f(x)=f(y)=0$.
We shall show that \eqref{fintechies} remains valid for all $x,y\in X$,
which completes the proof.
Without loss of generality, suppose $d(x,y)<\infty$.

When $x\notin  B^+_{\star}(D)$ and $y\in B^+_{\star}(D)$, since $(X,d)$ is a geodesic space,
there exists a minimal geodesic $\gamma(t)$, $t\in [0,1]$, from $x$ to $y$ with $L_d(\gamma)=d(x,y)$.
Thus, we find $x_0\in \gamma((0,1))$ such that $x_0\in B^+_\star (D)\setminus \supp(f)$
and $d(x_0,y)<d(x,y)$.
Since $x_0,y\in B^+_\star(D)$ and $f(x)=f(x_0)=0$,
we deduce from \eqref{fintechies} that
\begin{align*}
f(y)-f(x) =f(y)-f(x_0)\geq -\Lip(f) \Theta_\star(D) d(x_0,y) >-\Lip(f) \Theta_\star(D) d(x,y).
\end{align*}
For $x\in  B^+_{\star}(D)$ and $y\notin B^+_{\star}(D)$,
we again take a minimal geodesic $\gamma$ from $x$ to $y$ with $L_d(\gamma)=d(x,y)$.
We find $y_0$ on $\gamma$ such that $y_0\in B^+_\star (D)\setminus \supp(f)$ and $d(x,y_0)<d(x,y)$,
and then
\begin{align*}
f(y)-f(x) =f(y_0)-f(x) \geq -\Lip(f) \Theta_\star(D) d(x,y_0) >-\Lip(f) \Theta_\star(D) d(x,y),
\end{align*}
which concludes the proof.
\end{proof}

The following example points out that the conditions in Lemma \ref{forwardlipsislpsc} are necessary.

\begin{example}\label{lipschitz-example}
Let $(\mathbb{B}^n,d_F)$ be the Funk metric space as in Example \ref{funkmetricsapce}.
Let $r(x):=d_F(\mathbf{0},x)$ be the distance function from the origin,
which is forward $1$-Lipschitz by the triangle inequality.
However, it is not backward Lipschitz since, in view of \eqref{reversfunk},
there does not exist a constant $C \geq 0$ such that
\[
d_F(\mathbf{0},y)=r(y)-r(\mathbf{0})\leq C d_F(y,\mathbf{0}) \quad  \text{for all }y\in \mathbb{B}^n.
\]
\end{example}

The next proposition is inspired by \cite[Proposition~4.1]{AGS2}.

\begin{proposition}\label{densfolLIPS}
Let $(X,\tau,d,\m)$ be a forward extended Polish space equipped with
a nonnegative Borel measure $\m$ on $(X,\tau)$ such that
\begin{equation}\label{Kcondition1}
\forall K\subset (X,\tau)\colon \text{compact}, \qquad \exists\, r=r(K)>0, \qquad
 \m\Bigl[ \overline{B^+_K(r)\cup B^-_K(r)}^d  \Bigr] <\infty.
\end{equation}
Denote by $\mathscr{C}$ the class of bounded, Borel and forward Lipschitz functions $f$ such that
there is a compact set $K$ with $\supp(f) \subset \overline{B^+_K(r)\cup B^-_K(r)}^d$
for $r=r(K)$ in \eqref{Kcondition1}.
Then, $\mathscr{C}$ is dense in $L^q(X,\m)$ for all $q\in [1,\infty)$.
\end{proposition}

Note that every $f \in \mathscr{C}$ enjoys $f \in L^q(X,\m)$ and $|D^+ f|\in L^q(X,\m)$
(recall Lemma \ref{Lipscontin}).

\begin{proof}
Since $L^q(X,\m)=\overline{C_0(X)}^{\|\cdot\|_{L^q}}$,
 where $C_0(X)$ denotes the set of all continuous functions on $X$ with compact support,
it suffices to prove that, for any $\varphi\in C_0(X)$, there is a sequence $\varphi_i \in \mathscr{C}$
such that $\varphi_i \to \varphi$ in $L^q(X,\m)$.
Let $K \subset X$ be a compact set including $\supp(\varphi)$.

\textbf{Step 1.}
Firstly we assume that $\varphi$ is nonnegative and set $S:=\max_K\varphi\geq 0$.
For $i \ge 1$, define
\[
\varphi_i (x) :=\sup_{y\in K}\left[ \varphi(y) -id(x,y) \right].
\]
Clearly $0 \le \varphi \le \varphi_i \le S$.
We claim that $\varphi_i$ is forward Lipschitz with $\Lip(\varphi_i)\leq i$
and $\tau$-upper semi-continuous.

For any $x_1,x_2\in X$ and $\varepsilon>0$, we take $y_\varepsilon\in K$ such that
\[
\varphi_i(x_2) \leq \varphi(y_\varepsilon)-id(x_2,y_\varepsilon) +\varepsilon,
\]
which furnishes
\[
\varphi_i(x_2)-\varphi_i(x_1)
 \leq \varphi(y_\varepsilon)-id(x_2,y_\varepsilon) +\varepsilon
 -\{ \varphi(y_\varepsilon)-id(x_1,y_\varepsilon) \}
 \leq i d(x_1,x_2)+\varepsilon.
\]
Hence, $\varphi_i$ is $i$-forward Lipschitz.
To show the $\tau$-upper semi-continuity,
for any sequence $(x_k)_{k \ge 1}$ converging to $x$ under $\tau$,
take $y_k\in K$ such that
\[
\varphi_i(x_k) \leq \varphi(y_k) -id(x_k,y_k) +k^{-1}.
\]
By the compactness of $K$, we may assume that $(y_k)_{k \ge 1}$ converges to some $y_\infty\in K$.
Then, the $\tau$-lower semi-continuity of $d$ (Definition \ref{weaktopolgytau}\eqref{weakt1-4}) yields
\begin{align*}
\limsup_{k\to \infty}\varphi_i(x_k)
\leq \limsup_{k\to \infty} [\varphi(y_k)-i d(x_k,y_k)]
\leq \varphi(y_\infty)-id(x,y_\infty) \leq \varphi_i(x).
\end{align*}
Thus, $\varphi_i$ is $\tau$-upper semi-continuous.

Fix $x\in X$ with $d(x,K) \ge S/i$.
Then, since $d(x,y)\geq d(x,K) \ge S/i \ge \varphi(y)/i$ for any $y \in K$, we have $\varphi_i(x)=0$.
Therefore, $\supp(\varphi_i) \subset \overline{B^-_K(S/i)}^d$ and
$|D\varphi_i|=0$ on $X \setminus \overline{B^-_K (S/i)}^d$.
For $r>0$ given in \eqref{Kcondition1} and for all $i>S/r$, we find
\[
\supp(\varphi_i) \subset {\overline{B^-_K(r)}}^d, \qquad
 \m\Bigl[ \overline{B^-_K(r)}^d \Bigr] <\infty.
\]
Hence, $\varphi_i$ satisfies all the requirements and belongs to $\mathscr{C}$.
Finally, since $0\leq \varphi \leq \varphi_i \leq S$ and $\varphi_i(x) \downarrow \varphi(x)$
for every $x\in X$, we have $\varphi_i \to \varphi$ in $L^q(X,\m)$
by the dominated convergence theorem.

\textbf{Step 2.}
Next we assume that $\varphi$ is nonpositive.
In this case, employing
\[
\varphi_i(x) :=\inf_{y\in K} [\varphi(y)+id(y,x)] =-\sup_{y\in K} [-\varphi(y)-id(y,x)],
\]
one can follow the same lines as above and show that
$-\varphi_i$ is backward Lipschitz with $\RL(-\varphi_i)\leq i$ and $\tau$-upper semi-continuous,
equivalently, $\varphi_i$ is forward Lipschitz with $\Lip(\varphi_i)\leq i$ and $\tau$-lower semi-continuous.
Moreover, setting $S:=-\inf_K \varphi \ge 0$, for $r>0$ as in \eqref{Kcondition1} and all $i>S/r$, we obtain
\[
\supp(\varphi_i) \subset {\overline{B^+_K(r)}}^d, \qquad
 \m\Bigl[ \overline{B^+_K(r)}^d\Bigr] <\infty.
\]
Note also that $\varphi_i \to \varphi$ in $L^q(X,\m)$.

\textbf{Step 3.}
Finally, for general $\varphi$, we decompose it into the positive and negative parts
and approximate each part as in the previous steps.
\end{proof}

Recall that $f$ is backward Lipschitz if and only if $-f$ is forward Lipschitz.
Thus, Proposition \ref{densfolLIPS} combined with Lemma \ref{forwardlipsislpsc} yields the following.

\begin{corollary}\label{denseitybackwardmea}
Let $(X,\tau,d,\m)$ be as in Proposition $\ref{densfolLIPS}$.
\begin{enumerate}[{\rm (i)}]
\item\label{backwarddesne}
The class of bounded, Borel and backward Lipschitz functions $f \in L^q(X,\m)$
with $|D^- f|\in L^q(X,\m)$ is dense in $L^q(X,\m)$ for all $q \in [1,\infty)$.

\item If the reversibility of $(X,d)$ is finite, then the class of bounded, Borel and Lipschitz functions
$f \in L^q(X,\m)$ with $|Df|\in L^q(X,\m)$ is dense in $L^q(X,\m)$ for all $q \in [1,\infty)$.

\item\label{backwarddesne3}
If $(X,d)$ is a geodesic space such that, for every compact set $K$,
the set $\overline{B^+_K(r)\cup B^-_K(r)}^d$ in \eqref{Kcondition1} is of finite diameter,
then the class of bounded, Borel and Lipschitz functions $f \in L^q(X,\m)$
with $|Df| \in L^q(X,\m)$ is dense in $L^q(X,\m)$ for all $q \in [1,\infty)$.
\end{enumerate}
\end{corollary}

\begin{example}\label{approxinfinslercase}
Let $({\X},d_F)$ be a forward metric space induced from a forward complete Finsler manifold,
and let $\m$ be a smooth positive measure on $\X$.
Then, the forward completeness and the Hopf--Rinow theorem imply \eqref{Kcondition1}.
Denote by $\Lip_0(\X)$ the collection of forward Lipschitz functions with compact support,
and we similarly define $\RL_0(\X)$ and $\SL_0(\X)$.
Since $({\X},d_F)$ is a geodesic space, it follows from Lemma \ref{forwardlipsislpsc} that
$\Lip_0(\X)=\RL_0(\X)=\SL_0(\X)$.
Note also that, since $C^\infty_0(\X)\subset \Lip_0(\X)$ and $L^p({\X},\m)=\overline{C^\infty_0(\X)}^{L^p}$,
we have
\[
L^p({\X},\m)=\overline{\Lip_0(\X)}^{L^p}=\overline{\RL_0(\X)}^{L^p}=\overline{\SL_0(\X)}^{L^p},
\]
which corresponds to Proposition \ref{densfolLIPS} and
Corollary \ref{denseitybackwardmea}\eqref{backwarddesne}, \eqref{backwarddesne3}.
\end{example}

In the symmetric case, if a metric measure space $(X,d,\m)$ is doubling and $d$ is finite,
then we have $|D^+f|=|D^-f|$ $\m$-a.e.\ for any Lipschitz function $f$ (see \cite[Proposition 2.7]{AGS2}).
Although this is not the case in the asymmetric setting,
we can obtain a modified estimate depending on the reversibility.

\begin{example}\label{finitreverrnexa}
Consider a Finsler manifold $(\mathbb{R}^n,F)$ (of both Berwald and Randers type) given by
\[
F(x,y):=\sqrt{\sum_{i=1}^n (y^i)^2} +\frac12 y^1,\qquad
 y=\sum_{i=1}^n y^i\frac{\partial}{\partial x^i}.
\]
Let $\m$ be the Busemann--Hausdorff measure, explicitly given by $\m=2^{-(n+1)/2} \mathscr{L}^n$.
Note that line segments are geodesics and the doubling condition holds; more precisely,
\[
\m\bigl[ B^+_x(2r) \bigr] = 2^n \m\bigl[ B^+_x(r) \bigr] \quad
 \text{for all}\ x \in \mathbb{R}^n,\ r>0.
\]
We readily observe $\lambda_{d_F}(x)=3$ for all $x\in \mathbb{R}^n$.
For $f\in C^1(\mathbb{R}^n)$, one can calculate (cf.\ \cite[Example 3.2.1]{Sh1}) that
\[
|D^\pm f| =F \bigl( \nabla(\pm f) \bigr) =\frac{2}{3} \left[
 \sqrt{4\left( \frac{\partial f}{\partial x^1} \right)^2 +3\sum_{i=2}^{n}\left( \frac{\partial f}{\partial x^i} \right)^2}
 \mp\frac{\partial f}{\partial x^1} \right].
\]
Thus, we have $|D^+f|\neq |D^-f|$ but
\begin{equation}\label{bascisloprelation}
|D^\pm f|(x)\leq \lambda_{d_F}(x)\cdot|D^\mp f|(x).
\end{equation}
In fact, along the lines of \cite[Proposition 2.7]{AGS2},
one can show that \eqref{bascisloprelation} remains valid $\m$-a.e.\ for every Lipschitz function
on a doubling forward metric measure space.
\end{example}

By the same argument as in \cite[Lemma 2.5]{AGS2},
with the help of $[a+b]^+ \le a^+ +b^+$, we have the following.

\begin{lemma}\label{lipconvexfundd}
Let $(X,d)$ be a forward extended metric space, $f,g:X\rightarrow \mathbb{R}$ be forward Lipschitz,
and let $\phi:\mathbb{R}\rightarrow \mathbb{R}$ be $C^1$ with $0\leq \phi'\leq 1$.
Set
\[
\widetilde{f}:=f+\phi(g-f),\qquad \widetilde{g}:=g-\phi(g-f).
\]
Then, for any convex non-decreasing function $\psi:[0,\infty)\rightarrow \mathbb{R}$,
we have for all $x \in X$
\begin{align*}
\psi \bigl( |D^+\widetilde{f}|(x) \bigr) +\psi\bigl( |D^+\widetilde{g}|(x) \bigr)
 &\leq \psi\left( |D^+ f|(x)  \right)+\psi\left( |D^+ g|(x)  \right).
\end{align*}
\end{lemma}

The following lemma will be useful in Section \ref{rexaghher}.

\begin{lemma}\label{postivenegativepartoffowrdlips}
Let $(X,d)$ be a forward extended metric space.
Then, for every forward Lipschitz function $f:X \rightarrow \mathbb{R}$,
$f^+$ and $-f^-$ are forward Lipschitz with
 \begin{equation}\label{expressionofD+}
|D^+f|=|D^+ f^+|+|D^+ (-f^-)|.
\end{equation}
\end{lemma}

\begin{proof}
Fix $x \in X$.
On the one hand, for $y \in X$ with $f(y) \ge 0$, we observe
\[
f^+(y) -f^+(x) \le f(y) -f(x) \le \Lip(f) d(x,y).
\]
On the other hand, if $f(y)<0$, then we also have
\[
f^+(y) -f^+(x) \le 0 \le \Lip(f) d(x,y).
\]
Thus, $f^+$ is forward Lipschitz.
Since $-f$ is backward Lipschitz, we similarly see that $(-f)^+ =f^-$ is backward Lipschitz,
and hence $-f^-$ is forward Lipschitz.

To show \eqref{expressionofD+}, recall from Lemma \ref{Lipscontin} that $f$ is continuous.
Hence, if $f(x)>0$, then $f^+=f$ on a neighborhood of $x$ and we obtain
$|D^+ f|(x)=|D^+ f^+|(x)$ as well as $|D^+ (-f^-)|(x)=0$.
Similarly, when $f(x)<0$, we deduce that $|D^+ f|(x)=|D^+ (-f^-)|(x)$ and $|D^+ f^+|(x)=0$.
Finally, if $f(x)=0$, then $|D^+ (-f^-)|(x)=0$ again holds since $-f^- \le 0$, and
\[
|D^+f|(x) =\limsup_{y\to x}\frac{[f(y)]^+}{d(x,y)}=\limsup_{y\to x}\frac{[f^+(y) -f^+(x)]^+}{d(x,y)}
 =|D^+ f^+|(x),
\]
which concludes the proof.
\end{proof}

\begin{lemma}\label{ref'nproofandother}
Let $(X,\tau,d)$ be a forward extended Polish space and $K\subset (X,\tau)$ be a compact set.
Suppose that one of the following conditions holds:
\begin{enumerate}[{\rm (a)}]
\item\label{case-a}
The reversibility of $(X,d)$ is finite;
\item\label{case-b}
$(X,d)$ is a geodesic space and every compact set in $(X,\tau)$ is of finite diameter.
\end{enumerate}
Given $r>0$, choose a non-increasing Lipschitz function $\phi_r:\mathbb{R}\rightarrow [0,1]$
such that $\phi_r \equiv1$ in $[0,r/3]$ and $\phi_r \equiv0$ in $[2r/3,\infty)$,
and set $\chi_r(x):=\phi_r(d(K,x))$.
Then, we have the following.
\begin{enumerate}[{\rm (i)}]
\item\label{chiislips}
$\chi_r$ is a $\tau$-upper semi-continuous Lipschitz function.

\item \label{producetchi}
If $f$ is a bounded, forward Lipschitz function, then $\chi_r f$ is also forward Lipschitz with
\begin{equation}\label{df'nlipscontroll}
|D^+(\chi_r f)| \leq \chi_r  |D^+ f|+|f|\SL(\phi_r)\Theta_K(r).
\end{equation}

\item\label{chilinercombinislip}
If $f,g$ are bounded, forward Lipschitz functions, then $\chi_r f+(1-\chi_r)g$ is also forward Lipschitz.
\end{enumerate}
\end{lemma}

\begin{proof}
We shall consider only the case \eqref{case-b};
the other case \eqref{case-a} can be handled in the same way.

\eqref{chiislips}
To show the $\tau$-upper semi-continuity of $\chi_r$,
it suffices to prove that $d(K,\cdot)$ is $\tau$-lower semi-continuous.
Let $x_i\to x$ in $\tau$.
For any $\varepsilon>0$, there is $y_i \in K$ such that $d(K,x_i)\geq d(y_i,x_i)-\varepsilon$.
Since $K$ is compact in $\tau$, we may assume that $y_i$ converges to $y \in K$ in $\tau$.
Hence, Definition \ref{weaktopolgytau}\eqref{weakt1-4} yields
\begin{align*}
\liminf_{i \to \infty} d(K,x_i)
\geq \liminf_{i \to \infty} d(y_i,x_i)-\varepsilon
\geq d(y,x)-\varepsilon
\geq d(K,x)-\varepsilon.
\end{align*}
Thus, $d(K,\cdot)$ is $\tau$-lower semi-continuous.

To show that $\chi_r$ is Lipschitz,
we first consider $x,y \in X$ with $\chi_r(x)-\chi_r(y)\geq 0$, i.e., $d(K,x) \leq d(K,y)$.
In this case, we immediately deduce that
\[
\chi_r(x)-\chi_r(y)
 \leq \SL(\phi_r) \bigl( d(K,y)-d(K,x) \bigr) \leq \SL(\phi_r) d(x,y).
\]
Next, when $\chi_r(y)-\chi_r(x)\geq 0$, a similar calculation yields
$\chi_r(y)-\chi_r(x) \leq \SL(\phi_r) d(y,x)$.
If $x,y\in B^+_K(r)$, then Remark \ref{setreversiblitfuntion} implies $d(y,x) \leq \Theta_K(r)d(x,y)$,
thereby,
\[
\chi_r(y)-\chi_r(x) \leq \SL(\phi_r)\Theta_K(r)d(x,y).
\]
We claim that this inequality remains valid for all $x,y\in X$.
This is obviously true if $y \notin B^+_K(r)$ since $\chi_r(y)=0$.
When $x\notin B^+_K(r)$ and $y\in B^+_K(r)$, since $(X,d)$ is a geodesic space,
there is a minimal geodesic $\gamma$ from $x$ to $y$.
Thus, we find $x_0$ on $\gamma$ such that $x_0\in B^+_K(r)\setminus B^+_K(2r/3)$ and $d(x_0,y)<d(x,y)$.
It follows from $x_0,y\in B^+_K(r)$ and the definition of $\chi_r$ that
\[
\chi_r(y)-\chi_r(x) =\chi_r(y)-\chi_r(x_0)\leq \SL(\phi_r)\Theta_K(r)d(x_0,y)<\SL(\phi_r)\Theta_K(r)d(x,y).
\]
Therefore, the claim holds true and $\chi_r$ is $[\SL(\phi_r)\Theta_K(r)]$-Lipschitz.

\eqref{producetchi}
Let $f_r :=\chi_r f$ and $|f| \le C$.
For any $x,y\in X$, we infer from \eqref{chiislips} that
\begin{align*}
f_r(y)-f_r(x)
&= \chi_r(y)\bigl( f(y)-f(x) \bigr) +f(x)\bigl( \chi_r(y)- \chi_r(x) \bigr) \\
&\leq \Lip(f) d(x,y) +C\SL(\phi_r)\Theta_K(r)d(x,y).
\end{align*}
Hence, $f_r$ is forward Lipschitz.
Moreover,
\begin{align*}
\limsup_{y\to x} \frac{[f_r(y)-f_r(x)]^+}{d(x,y)}
&\leq \chi_r(x)\limsup_{y\to x}\frac{[f(y)-f(x)]^+}{d(x,y)}
 +|f(x)| \limsup_{y\to x} \frac{|\chi_r(y)- \chi_r(x)|}{d(x,y)} \\
&\leq \chi_r(x) |D^+ f|(x)+|f(x)|\SL(\phi_r)\Theta_K(r),
\end{align*}
which furnishes \eqref{df'nlipscontroll}.

\eqref{chilinercombinislip}
Let $h_r:=1-\chi_r$.
We deduce from \eqref{chiislips} that $h_r$ is $[\SL(\phi_r)\Theta_K(r)]$-Lipschitz,
and then the proof of  \eqref{producetchi} yields that $h_r g$ is forward Lipschitz.
This completes the proof.
\end{proof}

Let $(X,\tau,d)$ be a forward extended Polish space,
and $\Po(X)$ be the set of Borel probability measures on $X$.
Denote by $\mathcal {B}^*(X)$ the collection of \emph{universally measurable} sets,
which is the $\sigma$-algebra of sets that are $\mu$-measurable for any $\mu\in \Po(X)$.
Then we have the following along the lines of \cite[Lemma 2.6]{AGS2}.

\begin{lemma}\label{borlstarmea}
Let $(X,\tau,d)$ be a forward extended Polish space.
If $f:X\rightarrow [-\infty,\infty]$ is Borel in $\tau$,
then its slopes $|D^\pm f|$ and $|Df|$ are $\mathcal {B}^*(X)$-measurable in $\mathfrak{D}(f)$.
In particular, if $\gamma:[0,1]\rightarrow X$ is a continuous curve with $\gamma(t) \in \mathfrak{D}(f)$
for a.e.\ $t\in [0,1]$, then $|D^\pm f|\circ\gamma$ and $|Df| \circ \gamma$ are Lebesgue measurable.
\end{lemma}

\subsection{Strong upper gradients and curves of maximal slope}\label{uppgrasec}

Let $(X,d)$ be a forward extended metric space in this subsection.

\begin{definition}[Strong upper gradients]\label{gradefe}
Let $f:X\rightarrow [-\infty,\infty]$ be a \emph{proper} function
(i.e., $\mathfrak{D}(f)=f^{-1}(\mathbb{R}) \neq \emptyset$).
A function $\mathfrak{g}:X\rightarrow [0,\infty]$ is called a  {\it strong upper gradient} of $f$
if, for any curve $\gamma\in \FAC([0,1];X)$, $s\mapsto \mathfrak{g}(\gamma(s))|{\gamma}'_+|(s)$
is $\mathscr{L}^1$-measurable in $[0,1]$ (with the convention $0\cdot\infty=0$)
and
\begin{equation}\label{graddef}
f\bigl( \gamma(t_2) \bigr) -f\bigl( \gamma(t_1) \bigr)
\leq \int^{t_2}_{t_1} \mathfrak{g}\bigl( \gamma(s) \bigr) |\gamma'_+|(s) \,{\dd}s \quad
\text{for all}\ 0\le t_1 <t_2 \le 1.
\end{equation}
\end{definition}

\begin{remark}\label{uppergraordef}
Since $\gamma\in \FAC([0,1];X)$ is arbitrary, \eqref{graddef} can be replaced with
\[
\int_{\partial\gamma}f :=f\bigl( \gamma(1) \bigr) -f\bigl( \gamma(0) \bigr)
\leq \int_\gamma \g := \int^{1}_{0} \g \bigl( \gamma(s) \bigr) |\gamma'_+|(s) \,{\dd}s.
\]
Moreover, if $d$ is symmetric, then \eqref{graddef} is equivalent to
\begin{equation}\label{reveruppergradein}
\bigl| f\bigl( \gamma(t_2) \bigr) -f\bigl( \gamma(t_1) \bigr) \bigr|
\leq \int^{t_2}_{t_1} \mathfrak{g}\bigl( \gamma(s) \bigr) |\gamma'_+|(s) \,{\dd}s.
\end{equation}
In an asymmetric setting of a Finsler manifold $(\X,F)$, given  $f\in C^1({\X})$,
\eqref{graddef} holds with $\g=|D^+ f|=F(\nabla f)$,
while \eqref{reveruppergradein} holds for $\g=|Df|=\max\{ F(\nabla f),F(\nabla(-f)) \}$.
\end{remark}

\begin{lemma}\label{uppergradofLispcf}
If $f:X \rightarrow \mathbb{R}$ is forward Lipschitz, then $|D^+ f|$ is its strong upper gradient.
\end{lemma}

\begin{proof}
Given any $\gamma\in\FAC([0,1];X)$, since $L_d(\gamma)<\infty$ (see \eqref{lengthabscur}),
we deduce from Lemma~\ref{pfms-d} that
\[
\bigl| f\bigl( \gamma(t) \bigr) -f\bigl( \gamma(s) \bigr) \bigr|
\leq \Lip(f) \Theta_{\gamma(0)} \bigl( L_d(\gamma) \bigr) d\bigl( \gamma(s),\gamma(t) \bigr)
\leq \Lip(f) \Theta_{\gamma(0)} \bigl( L_d(\gamma) \bigr) \int^t_s |\gamma'_+|(r) \,{\dd}r
\]
for all $0\le s<t \le 1$.
Thus, $h(t):=f(\gamma(t))$ is absolutely continuous, and hence,
$h'(t)$ exists for $\mathscr{L}^1$-a.e.\ $t\in (0,1)$.
For $t$ where $h'(t)$ and $|\gamma'_+|(t)$ exist, we obtain
\begin{align*}
h'(t)
&= \lim_{s \to t^+}\frac{h(s)-h(t)}{s-t}
 \leq \lim_{s \to t^+} \frac{[h(s)-h(t)]^+}{s-t}
 \leq \limsup_{s \to t^+}
 \left( \frac{[h(s)-h(t)]^+}{d(\gamma(t),\gamma(s))} \frac{d(\gamma(t),\gamma(s))}{s-t} \right) \\
&\leq |D^+ f| \bigl( \gamma(t) \bigr) |\gamma'_+|(t).
\end{align*}
Therefore, $\int_{\partial \gamma} f \leq \int_{\gamma} |D^+ f|$ holds as desired.
\end{proof}

\begin{definition}[Geodesically convex functions]
Denote by $\Geo(X)$ the collection of minimal geodesics $\gamma:[0,1] \rightarrow X$.
For $k\in \mathbb{R}$, we say that $\phi: X\rightarrow (-\infty,\infty]$ is {\it $k$-geodesically convex}
if, for any $x_0,x_1\in \mathfrak{D}(\phi)$, there exists $\gamma\in \Geo(X)$
such that $\gamma(0)=x_0$, $\gamma(1)=x_1$ and
\[
\phi \bigl( \gamma(t) \bigr)
\leq (1-t) \phi(x_0)+t \phi(x_1) -\frac{k}2t(1-t) d^2(x_0,x_1) \quad \text{for all}\ t\in [0,1].
\]
\end{definition}

Similarly to \cite[Theorem~2.4.9]{AGS} and \cite[Theorem~4.8, Corollary~4.9]{OZ},
we have the following.

\begin{proposition}\label{upperconvefun}
Let $\phi$ be a $k$-geodesically convex, lower semi-continuous function.
Then, $|D^- \phi|$ is a strong upper gradient of $-\phi$, and we have
\[
|D^-\phi|(x)=\sup_{y\in X\setminus\{x\}}\left[ \frac{\phi(x)-\phi(y)}{d(x,y)}+\frac{k}2 d(x,y) \right]^+
\quad \text{for all}\ x \in \mathfrak{D}(\phi).
\]
\end{proposition}

\begin{definition}[Curves of maximal slope]\label{df:maxslope}
Let $\phi$ be a $k$-geodesically convex, lower semi-continuous function,
and $p\in (1,\infty)$.
\begin{enumerate}[(1)]
\item
Define $\FAC^p_{\loc}((0,\infty);X)$ as the set of $\gamma:(0,\infty) \rightarrow X$ such that
$\gamma|_{[s,t]} \in \FAC^p([s,t];X)$ for all $0<s<t<\infty$.

\item
A curve $\gamma\in \FAC^p_{\loc}((0,\infty);X)$ is called a {\it curve of $p$-maximal slope for $\phi$}
if $\phi\circ \gamma$ is $\mathscr{L}^1$-a.e.\ equal to a non-increasing function $\varphi$ satisfying
\[
 \varphi'(t) \leq -\frac{1}{p}|\gamma'_+|^p(t)-\frac{1}{q}|D^-\phi|^q \bigl( \gamma(t) \bigr)
 \quad \text{for $\mathscr{L}^1$-a.e.\ $t\in (0,\infty)$},
\]
where $1/p+1/q=1$.

\item
We say that $\gamma\in \FAC^p_{\loc}((0,\infty);X)$ satisfies the \emph{$p$-dissipation inequality} if
\[
 \phi\bigl( \gamma(0) \bigr) \geq
 \phi \bigl( \gamma(t) \bigr) +\frac1p \int^t_0 |\gamma'_+|^p(s) \,{\dd}s
 +\frac1q \int^t_0 |D^-\phi|^q \bigl( \gamma(s) \bigr) \,{\dd}s \quad \text{for all}\ t>0. 
\]
\end{enumerate}
\end{definition}

\begin{remark}\label{pcurveofmarxsloppssio}
It follows from Proposition \ref{extabslcurv} that $\FAC^p_{\loc}((0,\infty);X)=\AC^p_{\loc}((0,\infty);X)$.
Moreover, owing to \cite[Proposition 2.24]{OZ}, if $\gamma$ is a curve of $p$-maximal slope for $\phi$,
then we have
\begin{align*}
|\gamma'_+|^p(t)
&=|D^-\phi|^q \bigl( \gamma(t) \bigr) \quad \text{for $\mathscr{L}^1$-a.e.}\ t\in (0,\infty), \\
\phi\bigl( \gamma(t) \bigr)
&= \phi\bigl( \gamma(s) \bigr) -\int^t_s |D^-\phi| \bigl( \gamma(r) \bigr) |\gamma'_+|(r) \,{\dd}r
 \quad \text{for any } 0 \le s<t<\infty.
\end{align*}
On the other hand, if these equations hold,
then $t\mapsto \phi(\gamma(t))$ is non-increasing and
\[
[\phi \circ \gamma]'(t) =-|D^-\phi|(\gamma(t))|\gamma'_+|(t)
 =-\frac{1}{p} |\gamma'_+|^p(t) -\frac1q |D^-\phi|^q \bigl( \gamma(t) \bigr)
 \quad \text{for $\mathscr{L}^1$-a.e.}\ t\in (0,\infty).
\]
Hence, $\gamma$ is a curve of $p$-maximal slope for $\phi$.
\end{remark}

\subsection{Wasserstein spaces}\label{Wasssect}

In this subsection, let $(X,\tau,d)$ be a forward extended Polish space and $p\in [1,\infty)$.
For $\mu,\nu\in \Po(X)$, the {\it Kantorovich--Wasserstein distance of order $p$} is defined by
\begin{equation}\label{wassdisdefin}
W_p(\mu,\nu):=\left( \inf_{\pi\in \Pi(\mu,\nu)} \int_{X\times X}d^p(x,y) \,\pi({\dd}x\,{\dd}y) \right)^{1/p},
\end{equation}
where $\Pi(\mu,\nu) \subset \Po(X \times X)$ is the set of \emph{transference plans} (or \emph{couplings}) $\pi$,
i.e., $\mathfrak{p}^1_\sharp \pi=\mu$ and $\mathfrak{p}^2_\sharp \pi=\nu$
for the projections $\mathfrak{p}^i:X \times X \rightarrow X$ to the $i$-th component,
where  $\mathfrak{p}^i_\sharp \pi$ denotes the push-forward measure of $\pi$ by $\mathfrak{p}^i$.
According to Villani \cite[Theorem 4.1]{Vi}, the infimum can be attained in \eqref{wassdisdefin},
and we call such $\pi$ an \emph{optimal} transference plan from $\mu$ to $\nu$.

In view of Definition \ref{reversibilitydef}, the standard theory (cf.\ \cite{Vi,KZ}) yields the following.

\begin{theorem}\label{forwstrucwass}
For each $p>1$, $(\Po(X),W_p)$ is an asymmetric extended metric space.
Moreover, if $\lambda_d(X)<\infty$, then $(\Po(X),W_p)$ is a forward extended metric space
with $\lambda_{W_p}(\Po(X))= \lambda_d(X)$.
\end{theorem}

If $\lambda_d(X)=\infty$, then the forward and backward topologies of $(\Po(X),W_p)$
may be different (see \cite[Example 5]{ZZ}).

\begin{example}\label{wasserteinspaceisasymmetric}
Let $(\mathbb{B}^n,d_F)$ be the Funk metric space defined in Example \ref{funkmetricsapce}.
Then, for every $p>1$, $(\Po(\mathbb{B}^n),W_p)$ is not a forward extended metric space.
Precisely, $\mu_k :=(\gamma_k)_{\sharp}\mathscr{L}^1 |_{[0,1]}$ ($k \ge 1$)
and $\mu:=\gamma_{\sharp}\mathscr{L}^1 |_{[0,1]}$ given by
\begin{equation*}
\gamma_k(t)= \begin{cases}
(1-{\ee}^{t/(t-1)},0,\ldots,0) & \text{for}\ t\in [0,(k-1)/k),\\
\mathbf{0} & \text{for}\ t\in [(k-1)/k,1],
\end{cases}
\quad
\gamma(t)=\begin{cases}
(1-{\ee}^{t/(t-1)},0,\ldots,0) & \text{for}\ t\in [0,1),\\
\mathbf{0} & \text{for}\ t=1,
\end{cases}
\end{equation*}
satisfy $\lim_{k \to \infty}W_p(\mu,\mu_k)=0$ but $W_p(\mu_k,\mu)=\infty$ for all $k$.
\end{example}

In the sequel, $\mathscr{P} (X)$ is equipped with the weak topology
(against bounded $\tau$-continuous functions) and
$(\mathscr{P}(X),W_p)$ is endowed with the topology induced from the symmetrized metric
(as in \eqref{symmmetricde}):
\[
\widehat{W}_p(\mu,\nu):=\frac12 \bigl\{ {W}_p(\mu,\nu)+{W}_p(\nu,\mu) \bigr\}.
\]

\begin{lemma}\label{wasstopologequaivel}
If a sequence $(\mu_i)_{i \ge 1}$ converges to $\mu$ in $(\mathscr{P}( X),W_p)$,
then $\mu_i$ converges weakly to $\mu$.
\end{lemma}

\begin{proof}
Along the lines of \cite[Proposition 4.7]{ZZ} (see also \cite[Section 6]{Vi}),
we infer that $\mu_i$ weakly converges to $\mu$ in $\hat{d}$.
Now, since $\tau\subset \Tt_+$ (recall Remark \ref{exposlipspaceexplain}\eqref{rmpolish-c}),
$\tau$-continuous functions are $\hat{d}$-continuous.
Thus, $\mu_i$ weakly converges to $\mu$ also in $\tau$.
\end{proof}

\begin{definition}[$c_p$-convex functions]
Set $c_p(x,y):=d^p(x,y)/p$ for $x,y \in X$.
\begin{enumerate}[(1)]
\item
A function $\psi:X\rightarrow (-\infty,\infty]$ is said to be {\it $c_p$-convex}
if it is not identically $\infty$ and there is $\phi:X\rightarrow [-\infty,\infty]$ such that
\[
\psi(x) =\sup_{y \in X}\bigl\{ \phi(y)-c_p(x,y) \bigr\} \quad \text{for all}\ x\in X.
\]

\item
For a function $\psi:X\rightarrow [-\infty,\infty]$, we define
\[
\psi^{c_p} (y):= \inf_{x\in X} \bigl\{ \psi(x)+c_p(x,y) \bigr\}, \quad y\in X.
\]
\end{enumerate}
\end{definition}

Owing to  \cite[Theorem 5.10]{Vi}, we have the following \emph{Kantorovich duality},
whose validity is independent of the finiteness of the reversibility.

\begin{theorem}[Kantorovich duality]\label{Kantorduth1}
Given $\mu,\nu\in \Po(X)$, we have the following.
\begin{enumerate}[{\rm (i)}]
\item
We have the duality$:$
\begin{equation}\label{convexpwp}
\begin{split}
\frac{1}{p} W_p(\mu,\nu)^p
&= \sup_{\psi\in L^1(\mu)} \left( \int_{X} \psi^{c_p} \,{\dd}\nu -\int_{X} \psi \,{\dd}\mu \right) \\
&= \sup_{\{(\psi,\phi)\in C_b(X)\times C_b(X) \mid \phi-\psi\leq c_p\}}
 \left( \int_{X} \phi \,{\dd}\nu -\int_{X} \psi \,{\dd}\mu \right) \\
&= \sup_{\{(\psi,\phi)\in L^1(\mu)\times L^1(\nu) \mid \phi-\psi\leq c_p\}}
 \left( \int_{X} \phi \,{\dd}\nu -\int_{X} \psi \,{\dd}\mu \right),
\end{split}
\end{equation}
and in the suprema above we can also assume that $\psi$ is $c_p$-convex.

\item
If $W_p(\mu,\nu)<\infty$ and $d$ is finite,
then, for any $\pi\in \Pi(\mu,\nu)$, the following are equivalent.
\begin{enumerate}[{\rm (a)}]
\item $\pi$ is optimal.
\item There is a $c_p$-convex function $\psi$ such that $\psi^{c_p}(y)-\psi(x)=c_p(x,y)$ for $\pi$-a.e.\ $(x,y)$.
\item There are $\psi:X\rightarrow (-\infty,\infty]$ and $\phi:X\rightarrow [-\infty,\infty)$
such that $\phi(y)-\psi(x)\leq c_p(x,y)$ for all $(x,y)$, with equality $\pi$-a.e.
\end{enumerate}

\item
If $W_p(\mu,\nu)<\infty$ and $d$ is finite with $d(x,y) \leq \alpha(x)+\beta(y)$
for some $(\alpha,\beta) \in L^1(\mu) \times L^1(\nu)$,
then the suprema in \eqref{convexpwp} are attained by a $c_p$-convex function $\psi$ and $\phi=\psi^{c_p}$.
\end{enumerate}
\end{theorem}

\begin{corollary}\label{corokan}
Given $\mu,\nu\in \Po(X)$, we have
\[
\frac1p W_p(\mu,\nu)^p
 =\sup_{\psi\in C_b(X)} \left( \int_X \psi^{c_p} \,{\dd}\nu -\int_X \psi \,{\dd}\mu \right).
\]
\end{corollary}

\begin{proof}
For $(\psi,\phi)\in C_b(X)\times C_b(X)$ with $\phi-\psi\leq c_p$,
we have $\phi \le \psi^{c_p}$ by the definition of $\psi^{c_p}$.
Together with \eqref{convexpwp}, we find
\begin{align*}
\frac1pW_p(\mu,\nu)^p
 \leq \sup_{\psi\in C_b(X)} \left( \int_X \psi^{c_p} \,{\dd}\nu -\int_X \psi \,{\dd}\mu \right)
 \leq \sup_{\psi\in L^1(\mu)} \left(\int_X \psi^{c_p} \,{\dd}\nu -\int_X \psi \,{\dd}\mu \right)
=\frac1pW_p(\mu,\nu)^p.
\end{align*}
This completes the proof.
\end{proof}

The Kantorovich duality with $p=1$ has a particular form.

\begin{proposition}[Kantorovich--Rubinstein duality]\label{W1KANrUBDIS}
For any $\mu,\nu\in \Po(X)$, we have
\[
W_1(\mu,\nu)
 =\sup_{\psi \in {\Lip}_1(X)\cap L^1(\mu)} \left( \int_X \psi \,{\dd}\nu -\int_X \psi \,{\dd}\mu \right),
\]
where ${\Lip}_1(X)$ denotes the set of forward Lipschitz functions $f$ with $\Lip(f) \le 1$.
\end{proposition}

\begin{proof}
One can readily see that $\psi \in {\Lip}_1(X)$ if and only if it is $c_1$-convex,
and then $\psi^{c_1}=\psi$ (cf.\  \cite[Particular Case 5.4]{Vi}).
Thus, \eqref{convexpwp} yields the claim.
\end{proof}

\begin{remark}
If there exists some $x_0\in X$ such that $\widehat{W}_1(\delta_{x_0},\mu)<\infty$,
then ${\Lip}_1(X)\subset  L^1(\mu)$. In fact, for $f\in {\Lip}_1(X)$, we have
$
f(x_0)-d(y,x_0)\leq f(y)\leq f(x_0)+d(x_0,y)$ $ \text{for all}\ y\in X,
$
which implies $f\in L^1(\mu)$ since
$
f(x_0)-W_1(\mu,\delta_{x_0})\leq \int_X f \,{\dd}\mu \leq f(x_0)+W_1(\delta_{x_0},\mu).
$
\end{remark}

The following  Lisini-type theorem will be proven in Appendix \ref{Lisinitheoremproof}.

\begin{theorem}\label{Lisinitheorem}
Suppose $p>1$ 
and that a curve $(\mu_t) \in \FAC^p([0,1];(\mathscr{P}(X),W_p))$ satisfies one of the following conditions:
\begin{enumerate}[{\rm (a)}]
\item\label{lisptotalvariation} $(\mu_t)$ is uniformly continuous with respect to the total variation norm;
\item\label{backwwardabs} $(\mu_t) \in \BAC([0,1];(\mathscr{P}(X),W_1))$.
\end{enumerate}
Then, there exists ${\eta} \in \mathscr{P}(C_\tau([0,1];X))$ such that
\begin{enumerate}[{\rm (i)}]
\item\label{cond1measre1} ${\eta}$ is concentrated on $\AC^p([0,1];X)$,
\item\label{cond1measre2} $\mu_t=(e_t)_\sharp {\eta}$ for any $t\in [0,1]$,
\item\label{cond1measre3} $|\mu'_+|^p(t) =\int_{C_\tau([0,1];X)} |\gamma'_+|^p(t) \,{\eta}({\dd}\gamma)$
for $\mathscr{L}^1$-a.e.\ $t\in (0,1)$.
\end{enumerate}
\end{theorem}


\begin{remark}
Note that, if $(X,d)$ is a forward Polish space and
$\Theta_\star^{p/(p-1)}(r)$ is a concave function for some point $\star\in X$,
then \eqref{backwwardabs} holds because of \cite[Proposition 4.11]{ZZ}. On the other hand, Theorem \ref{Lisinitheorem} is not valid when $p=1$; see Abedi--Li--Schultz \cite[Example 1.1]{ALS} for  a counterexample  in the symmetric case.
\end{remark}

\begin{remark}\label{totalvariationint}
If $(X,\tau,d,\m)$ is a $\APE$ as in Definition \ref{df:APE} below and
$\mu_i=\rho_i\m\in \mathscr{P}(X)$, $i=1,2$,
then the total variation norm of the signed measure $\mu_1-\mu_2$ is given by
\begin{equation*}
|\mu_1-\mu_2|(X) =\int_X |\rho_1-\rho_2| \dm.
\end{equation*}
\end{remark}


\section{Weak upper gradients, weak Cheeger energy and its $L^2$-gradient flow}\label{section-3}

For convenience, we introduce the following definition.

\begin{definition}[Forward extended Polish metric measure spaces]\label{df:APE}
We call $(X,\tau,d,\m)$ a {\it forward extended Polish metric measure space} ($\APE$ for short)
if $(X,\tau,d)$ is a forward extended Polish space
and $\m$ is a nonnegative, Borel and $\sigma$-finite measure on $X$.
\end{definition}

Recall the following fact (cf.\ \cite[(4.1)]{AGS2}).

\begin{lemma}\label{fintieboumeasure}
There exists a bounded Borel function $\vartheta:X\rightarrow (0,\infty)$ such that $\int_X \vartheta \,{\dm}\leq 1$.
\end{lemma}

Note that, although $\m[X]$ could be infinite,
the finite measure $\vartheta\m$ and $\m$ share the same class of negligible sets.
The following mild assumption is standard.

\begin{assumption}\label{strongerstassumptiontheta}
There is a Borel, forward Lipschitz function $V: X\rightarrow [0,\infty)$ such that
$V$ is bounded on each compact set $K\subset (X,\tau)$ and $\int_X {\ee}^{-V^2} \dm\leq 1$.
We set $\widetilde{\m}:={\ee}^{-V^2}\m$.
\end{assumption}

Such a function $V$ always exists on $\CD(K,\infty)$-spaces (see \cite[Remark~9.2]{AGS2}).

\begin{example}\label{Finslercdkn}
According to \cite[Theorem 5.14]{KZ} (see also \cite{SturmI}),
if $(X,d,\m)$ is a forward metric measure space satisfying the (weak) $\CD(K,\infty)$ condition,
then for any $x_0\in X$ and $K'<K$, we have
\[
\int_X \exp\biggl( \frac{K'}{2}d^2(x_0,x) \biggr) \m({\dd}x) <\infty.
\]
Hence, $\int_X {\ee}^{-V^2} \dm \leq 1$ is achieved by choosing
\[
 V(x) =\biggl( \frac{|K|+1}{2} d^2(x_0,x) +\log C \biggr)^{1/2}, \qquad
 C:=\int_X \exp\biggl( -\frac{|K|+1}{2} d^2(x_0,x) \biggr) \m({\dd}x).
\]
\end{example}

\subsection{Test plans and weak upper gradients}

In this subsection, let $(X,\tau,d,\m)$ be a $\APE$ satisfying Assumption \ref{strongerstassumptiontheta}.
Recall that for $t\in [0,1]$ the evaluation map $e_t:C_\tau([0,1];X)\rightarrow X$
is defined by $e_t(\gamma)=\gamma(t)$.

\begin{definition}[Test plans]\label{df:testplan}
Let $p \in [1,\infty)$.
\begin{enumerate}[(1)]
\item A probability measure $\eta \in \Po(C_\tau([0,1];X))$ is called a {\it $p$-test plan} if
\begin{equation}\label{testcondi}
\eta \bigl[ C_\tau([0,1];X)\setminus \AC^p([0,1];X) \bigr] =0, \qquad
 (e_t)_\sharp \eta \ll \m\ \text{ for all }t\in [0,1].
\end{equation}

\item
A $p$-test plan $\eta$ is said to have {\it bounded compression} on the sublevels of $V$
if for every $M\geq 0$ there exists $C=C(\eta,M)\in [0,\infty)$ satisfying
\begin{equation}\label{bouncomprecon}
\bigl( (e_t)_\sharp\eta \bigr) \bigl[ \{x\in B \mid V(x)\leq M\}  \bigr] \leq C\m[B] \quad
 \text{for all}\ B\in \mathcal{B}(X),\ t\in [0,1].
\end{equation}

\item A collection $\mathfrak{T}$ of $p$-test plans is said to be {\it stretchable} if,
for any $\eta \in \mathfrak{T}$ and $0\leq t_1\leq t_2\leq 1$,
we have $({\res}^{t_2}_{t_1})_{\sharp}\eta\in \mathfrak{T}$,
where ${\res}^{t_2}_{t_1}:C_\tau([0,1];X)\rightarrow C_\tau([0,1];X)$ is the {\it restriction map} defined by
${\res}^{t_2}_{t_1}(\gamma)(s):=\gamma((1-s)t_1+s t_2)$.
\end{enumerate}
\end{definition}

If $\m[X]<\infty$, then we can take $V \equiv 0$
and the condition \eqref{bouncomprecon} is independent of $M$.

\begin{remark}\label{testpconverex}
A ``trivial" $p$-test plan with bounded compression on the sublevels of $V$ is given as follows
(cf.\ Ambrosio--Gigli--Savar\'e \cite{AGS3}).
Let $\iota:X\rightarrow \AC^p([0,1];X)$ be defined by $\iota(x) \equiv x$ and set
\[
\eta:=\iota_\sharp \left(\frac{{\ee}^{-V^2} \m}{\int_X {\ee}^{-V^2} \dm} \right) \in \Po\bigl( C_\tau([0,1];X) \bigr).
\]
Clearly, $\eta$ is concentrated in $\AC^p([0,1];X)$ (in fact, constant curves).
Since $e_t\circ\iota=\id_X$, we have
\[
\bigl( (e_t)_\sharp \eta \bigr) \bigl[ B\cap  \{V \leq M\} \bigr]
 =\frac{\int_{B\cap  \{V \leq M\}} {\ee}^{-V^2}\dm}{\int_X {\ee}^{-V^2}\dm}
 \leq \left( \int_X {\ee}^{-V^2}\dm \right)^{-1} \m[B].
\]
\end{remark}


\begin{example}\label{exboundedcompression}
Let $(\X, d_F)$ be a forward metric space induced from a forward complete Finsler manifold $(\X,F)$,
and let $\m$ be a smooth positive measure with $\m[\X]=\infty$.
Suppose that $(\X,d_F,\m)$ satisfies $\CD(K,\infty)$ (equivalently, $\mathrm{Ric}_{\infty} \ge K$).
As in Example \ref{Finslercdkn}, there exists a forward Lipschitz function
$V:\X\rightarrow [0,\infty)$ satisfying Assumption \ref{strongerstassumptiontheta}.
Now,  given $p>1$, consider a curve $(\mu_t) \in \FAC^p([0,1];(\Po(\X),W_p))$ of the form
$\mu_t =f_t \widetilde{\m}=f_t {\ee}^{-V^2}\m$, where $f_t$ satisfies
$0\leq f_t\leq C_1$ and $|f_t-f_s|\leq C_2|t-s|$ for some $C_1,C_2>0$ and all $s,t \in [0,1]$.
Then, Remark \ref{totalvariationint} implies that $\mu_t$ is uniformly continuous
with respect to the total variation norm.
Hence, by Theorem \ref{Lisinitheorem}, we obtain a measure $\eta \in \Po(C([0,1];\X))$
concentrated in $\AC^p([0,1];\X)$ such that
$(e_t)_\sharp \eta =\mu_t =f_t {\ee}^{-V^2} \m$, and then
\[
\bigl( (e_t)_\sharp \eta \bigr) [B] \leq C_1\m[B] \quad \text{for all}\ B\in \mathcal{B}(\X).
\]
Hence, $\eta$ is a $p$-test plan with bounded compression.
\end{example}

Typical examples of stretchable collections are the families of all $p$-test plans with bounded compression,
as well as those concentrated in curves of finite $p$-energy or geodesics.

\begin{lemma}\label{stricpreccssclass}
Let $\mathfrak{T}_p$ be the collection of all $p$-test plans $\eta \in \Po(C_\tau([0,1];X))$
with bounded compression on the sublevels of $V$.
Then, $\mathfrak{T}_p$ is stretchable.
\end{lemma}

\begin{proof}
Given any $\eta \in \mathfrak{T}_p$ and $0 \leq t_1 \le t_2\leq 1$,
we need to show that $({\res}^{t_2}_{t_1})_{\sharp}\eta$ satisfies \eqref{testcondi} and \eqref{bouncomprecon}.
The former condition in \eqref{testcondi} is obvious since ${\res}^{t_2}_{t_1}(\AC^p([0,1];X)) \subset \AC^p([0,1];X)$
and $\eta$ is concentrated in $\AC^p([0,1];X)$.
The latter condition in \eqref{testcondi} and \eqref{bouncomprecon} follow from
$e_s \circ {\res}_{t_1}^{t_2} =e_{(1-s)t_1 +st_2}$.
%
\end{proof}

\begin{definition}[Negligible sets of curves]
Let $\mathfrak{T}$ be a stretchable collection of $p$-test plans,
and let $P$ be a statement about continuous curves $\gamma \in C_\tau([0,1];X)$.
We say that $P$ holds for \emph{$\mathfrak{T}$-almost every curve} if,
for any $\eta \in \mathfrak{T}$, $P$ holds $\eta$-a.e.\ $\gamma \in C_\tau([0,1];X)$.
\end{definition}

By the definition of $p$-test plans, we only need to consider $\gamma \in \AC^p([0,1];X)$.

\begin{definition}[Weak upper gradients]\label{weakuppgrad}
Let $\I$ be a stretchable collection of $p$-test plans.
Consider a function $f:X\rightarrow \mathbb{R}$ and an $\m$-measurable function $G:X\rightarrow [0,\infty]$.
\begin{enumerate}[(1)]
\item
$G$ is called a {\it $\I$-weak forward upper gradient} of $f$ if
\[
\int_{\partial\gamma}f \leq \int_{\gamma}G<\infty \quad
 \text{for $\I$-almost every $\gamma \in C_\tau([0,1];X)$.}
\]

\item
$G$ is called a {\it $\I$-weak backward upper gradient} of $f$
if it is a $\I$-weak forward upper gradient of $-f$.

\item
$G$ is called a {\it $\I$-weak upper gradient} of $f$ if
\[
\left|\int_{\partial\gamma}f \right| \leq \int_{\gamma}G<\infty \quad
 \text{for $\I$-almost every $\gamma \in C_\tau([0,1];X)$.}
\]
\end{enumerate}
\end{definition}

In view of Proposition \ref{extabslcurv},
a strong upper gradient in the sense of Definition \ref{gradefe} is always a weak forward upper gradient for any $\I$.
Note that, on the one hand, if $G_1$ (resp., $G_2$) is a $\I$-weak forward (resp., backward) upper gradient,
then $G:=\max\{G_1,G_2\}$ is a $\I$-weak upper gradient.
On the other hand, a $\I$-weak upper gradient is clearly a $\I$-weak forward and backward upper gradient.


\begin{remark}\label{weakuppergradientidentti}
The definition of weak upper gradients enjoys natural invariance
with respect to modifications in $\m$-negligible sets.
For example, the measurability of $s\mapsto G(\gamma_s)$ in $[0,1]$
for $\I$-almost every $\gamma$ is a direct consequence of the $\m$-measurability of $G$.
Indeed, for a Borel modification $\widetilde{G}$ of $G$ and an $\m$-negligible set $A$
including $\{G\neq\widetilde{G}\}$, any $p$-test plan $\eta$ satisfies
\[
\eta \bigl[ \{ \gamma \mid \gamma(t) \in A \} \bigr]
 =\bigl( (e_t)_\sharp \eta \bigr)[A] =0
\]
by the latter condition in \eqref{testcondi}.
Thus, we deduce from Fubini's theorem that
\begin{align*}
0 =\int^1_0 \int_{C_\tau([0,1];X)} \mathbbm{1}_{\{ \gamma \mid \gamma(t) \in A \}} \,\eta({\dd}\gamma) \,{\dd}t
 =\int_{C_\tau([0,1];X)} \int^1_0 \mathbbm{1}_{\{ \gamma \mid \gamma(t) \in A \}} \,{\dd}t \,\eta({\dd}\gamma),
\end{align*}
which yields $\int^1_0 \mathbbm{1}_{\{ \gamma \mid \gamma(t) \in A \}} \,{\dd}t=0$,
thereby $G(\gamma(t))=\widetilde{G}(\gamma(t))$ for $\mathscr{L}^1$-a.e.\ $t\in [0,1]$,
for $\eta$-a.e.\ $\gamma$.
\end{remark}

\begin{definition}[Sobolev along almost every curve]\label{weakdef}
We say that an $\m$-measurable function $f:X\rightarrow \mathbb{R}$ is
{\it Sobolev along $\I$-almost every curve} if, for $\I$-almost every curve $\gamma$,
$f\circ\gamma$ coincides with an absolutely continuous function $f_\gamma:[0,1]\rightarrow \mathbb{R}$
at $\{0,1\}$ and $\mathscr{L}^1$-a.e.\ in $(0,1)$.
\end{definition}

The existence of a $\I$-weak upper gradient yields the above Sobolev regularity
(cf.\ \cite[Proposition 5.7]{AGS2}).

\begin{proposition}\label{uppergradsobo}
Let $\I$ be a stretchable collection of $p$-test plans and $f:X\rightarrow \mathbb{R}$ be $\m$-measurable.
\begin{enumerate}[{\rm (i)}]
\item\label{requweak-1}
If $G$ is a $\I$-weak forward upper gradient of $f$, then $f$ is Sobolev along $\I$-almost every curve and,
for $\I$-almost every $\gamma\in \AC^p([0,1];X)$,
\begin{equation}\label{derivuppeest}
f'_\gamma \leq G\circ \gamma \cdot |\gamma'_+|
 \quad\ \mathscr{L}^1\text{-a.e.\ in }[0,1].
\end{equation}

\item\label{requweak-2}
If $G$ is a $\I$-weak upper gradient of $f$, then,
for $\I$-almost every $\gamma\in \AC^p([0,1];X)$,
\begin{equation}\label{derivuppeest3}
|f'_\gamma| \leq G\circ \gamma \cdot |\gamma'_+|
 \quad\ \mathscr{L}^1\text{-a.e.\ in }[0,1].
\end{equation}
\end{enumerate}
\end{proposition}

\begin{proof}
\eqref{requweak-1}
Let $\triangle:=\{(t_1,t_2) \mid 0<t_1\leq t_2<1\}\subset (0,1)\times (0,1)$.
Given $\eta \in \I$, it follows from Fubini's theorem applied to the product measure
$\mathscr{L}^2 \times \eta$ on $\triangle\times C_\tau([0,1];X)$ that, for $\eta$-a.e.\ $\gamma$,
\begin{equation}\label{finweaest}
f\bigl( \gamma(t_2) \bigr) -f\bigl( \gamma(t_1) \bigr)
 \leq \int^{t_2}_{t_1}G(\gamma)|\gamma'_+| \,{\dd}t \quad
 \text{for $\mathscr{L}^2$-a.e.\ $(t_1,t_2)\in \triangle$}
\end{equation}
(recall that $\res^{t_2}_{t_1}(\gamma)\in \I$ by the stretchability).
We similarly obtain, for $\mathscr{L}^1$-a.e.\ $s\in (0,1)$,
\[
f\bigl( \gamma(s) \bigr) -f\bigl( \gamma(0) \bigr) \leq \int^s_0 G(\gamma) |\gamma'_+| \,{\dd}t, \qquad
f\bigl( \gamma(1) \bigr) -f\bigl( \gamma(s) \bigr) \leq \int^1_s G(\gamma) |\gamma'_+| \,{\dd}t.
\]
Now, for any $\gamma\in \AC([0,1];X)$,
since $\overline{\gamma}(t):=\gamma(1-t)$ is also absolutely continuous with
\[
\theta^{-1} |\gamma'_+| \leq |\overline{\gamma}'_+| \leq \theta |\gamma'_+|,
 \qquad \theta:=\Theta_{\gamma(0)} \bigl( L_d(\gamma) \bigr),
\]
we further observe
\begin{align}
&\bigl| f\bigl( \gamma(t_2) \bigr) -f\bigl( \gamma(t_1) \bigr) \bigr|
 \leq \theta \int^{t_2}_{t_1}G(\gamma) |\gamma'_+| \,{\dd}t \quad
 \text{for $\mathscr{L}^2$-a.e.\ $(t_1,t_2)\in (0,1)^2$}, \notag\\
&\bigl| f\bigl( \gamma(s) \bigr) -f\bigl( \gamma(0) \bigr) \bigr|
 \leq \theta \int^s_0 G(\gamma) |\gamma'_+| \,{\dd}t, \qquad
\bigl| f\bigl( \gamma(1) \bigr) -f\bigl( \gamma(s) \bigr) \bigr|
 \leq \theta \int^1_s G(\gamma) |\gamma'_+| \,{\dd}t. \label{boundary01}
\end{align}
Then it follows from \cite[Lemma 2.11]{AGS2} that $f\circ\gamma\in W^{1,1}(0,1)$,
and \eqref{finweaest} implies
$(f\circ\gamma)' \leq G\circ \gamma \cdot |\gamma'_+|$ $\mathscr{L}^1$-a.e.\ in $[0,1]$.
Hence, for $\I$-almost every $\gamma$,
$f\circ\gamma\in W^{1,1}(0,1)$ and it admits an absolutely continuous representative $f_\gamma$
satisfying \eqref{derivuppeest}.
Moreover, we deduce from \eqref{boundary01} that $f(\gamma(t))=f_{\gamma}(t)$ at $t=\{0,1\}$.
This completes the proof of \eqref{requweak-1}.

\eqref{requweak-2} is seen by applying \eqref{requweak-1} to $f$ and $-f$.
\end{proof}

We have the following immediate corollary (cf.\ \cite[Remark 5.8]{AGS2}).

\begin{corollary}\label{equivweakupp}
Let $\I$ be a stretchable collection of $p$-test plans, and $f$ be Sobolev along $\I$-almost every curve.
Then, a function $G$ satisfying $\int_\gamma G<\infty$ for $\I$-almost every $\gamma$
is a $\I$-weak forward upper gradient $($resp., $\I$-weak upper gradient$)$ of $f$ if and only if,
for $\I$-almost every $\gamma$, $f_\gamma$ in Definition $\ref{weakdef}$
satisfies \eqref{derivuppeest} $($resp., \eqref{derivuppeest3}$)$.
\end{corollary}

\subsection{Calculus with weak upper gradients}\label{ssc-3.2}

Throughout this subsection, without otherwise indicated,
let $(X,\tau,d,\m)$ be a $\APE$ satisfying Assumption~\ref{strongerstassumptiontheta},
and $\I$ be a stretchable collection of $p$-test plans.
When $p>1$, we denote its H\"older conjugate by $q$ (i.e., $p^{-1}+q^{-1}=1$).
We begin with two fundamental properties of $\I$-weak forward upper gradients.

\begin{lemma}\label{miniweakuppergradi}
Let $f:X\rightarrow \mathbb{R}$ be $\m$-measurable.
If $G_1,G_2$ are $\I$-weak forward upper gradients of $f$,
then $\min\{G_1,G_2\}$ is also a $\I$-weak forward upper gradient of $f$.
\end{lemma}

\begin{proof}
It follows from Proposition \ref{uppergradsobo} that $f$ is Sobolev along $\I$-almost every $\gamma$ and
\[
f'_\gamma \leq \min\{G_1,G_2\}\circ \gamma \cdot |\gamma'_+| \quad
 \text{$\mathscr{L}^1$-a.e.\ in $[0,1]$}
\]
for $\I$-almost every $\gamma\in \AC^p([0,1];X)$.
This completes the proof by Corollary \ref{equivweakupp}.
\end{proof}

\begin{lemma}\label{weakuppniegsets}
Suppose that  $f,\widetilde{f}:X\rightarrow \mathbb{R}$ and $G,\widetilde{G}:X\rightarrow [0,\infty]$
are $\m$-measurable such that $f=\widetilde{f}$ and $G=\widetilde{G}$ hold $\m$-a.e.
If $G$ is a $\I$-weak forward upper gradient of $f$,
then $\widetilde{G}$ is a $\I$-weak forward upper gradient of $\widetilde{f}$.
\end{lemma}

\begin{proof}
Let $\eta \in \I$.
It suffices to prove that, for $\eta$-a.e.\ $\gamma$,
we have $f(\gamma(0))=\widetilde{f}(\gamma(0))$, $f(\gamma(1))=\widetilde{f}(\gamma(1))$,
and $\int_\gamma G=\int_\gamma \widetilde{G}$.
The first two equations follow from $(e_0)_\sharp \eta \ll \m$ and $(e_1)_\sharp \eta \ll \m$, respectively.
For the last one, recall the argument in Remark \ref{weakuppergradientidentti} to see that,
for $\eta$-a.e.\ $\gamma$, we have $G(\gamma(t))=\widetilde{G}(\gamma(t))$
for $\mathscr{L}^1$-a.e.\ $t \in [0,1]$.
Therefore, $\int_\gamma G= \int_{\gamma}\widetilde{G}$ for such $\gamma$.
\end{proof}

The counterparts to the above lemmas for weak (backward) upper gradients clearly hold true.

\begin{remark}
Thanks to Lemma~\ref{weakuppniegsets}, we can also consider extended real valued functions $f$,
provided $\m[\{|f|=\infty\}]=0$.
Indeed, for $\I$-almost every $\gamma$, we have $\gamma(0), \gamma(1) \not\in \{|f|=\infty\}$
and $\int_{\partial\gamma}f$ is well-defined.
\end{remark}

\begin{definition}[Minimal weak upper gradients]\label{df:min-wug}
Let $f:X\rightarrow \mathbb{R}$ be an $\m$-measurable function admitting a $\I$-weak forward upper gradient.
Then the {\it minimal $\I$-weak forward upper gradient} $|D^+f|_{w,{\I}}$ of $f$
is a $\I$-weak forward upper gradient such that, for every $\I$-weak forward upper gradient $G$ of $f$,
we have
\[
|D^+ f|_{w,\I}\leq G \quad \text{$\m$-a.e. in $X$}.
\]
We similarly define the {\it minimal $\I$-weak backward upper gradient} $|D^- f|_{w,\I}$
and the {\it minimal $\I$-weak upper gradient} $|D f|_{w,\I}$.
\end{definition}

\begin{remark}\label{existsupperweakgradient}
By definition and Lemma \ref{miniweakuppergradi}, we have a unique
minimal $\I$-weak forward upper gradient (up to $\m$-negligible sets).
To see the existence, consider a minimizing sequence $(G_i)_{i \ge 1}$ for the quantity
$\Xi(G):=\int_X \arctan(G)\vartheta \,\dm$ among $\I$-weak forward upper gradients of $f$,
for $\vartheta$ as in Lemma \ref{fintieboumeasure}.
Thanks to Lemma \ref{miniweakuppergradi}, we may assume that $G_{i+1} \leq G_i$.
Then, the monotone convergence theorem yields that $G:=\inf_i G_i$ exists $\m$-a.e.,
which is a minimizer of $\Xi$ and enjoys
\[
\int_{\partial\gamma}f \leq \lim_{i \to \infty} \int_\gamma G_i
 =\int_{\gamma} \lim_{i \to\infty}G_i =\int_\gamma G
\]
for $\I$-almost every $\gamma$.
Hence, $G$ is the unique minimal $\I$-weak forward upper gradient.
\end{remark}

\begin{remark}
If $\I_1\subset \I_2$ are stretchable collections of $p$-test plans
and a function $f:X\rightarrow \mathbb{R}$ is Sobolev along $\I_2$-almost every curve,
then $f$ is also Sobolev along $\I_1$-almost every curve and $|D^+f|_{w,\I_1}\leq |D^+f|_{w,\I_2}$.
Indeed, a larger class of $p$-test plans induces a smaller class of weak forward upper gradients,
and hence, a larger minimal weak forward upper gradient
(see \cite[Remark 5.13]{AGS2}).
\end{remark}

\begin{example}\label{finslercaseweakupgra}
Let $(\X, d_F,\m)$ be a forward metric measure space
induced from a forward complete Finsler manifold endowed with a smooth positive measure
satisfying either $\m[\X]<\infty$ or $\CD(K,\infty)$.
Recall from Example \ref{Finslercdkn} that Assumption \ref{strongerstassumptiontheta} is satisfied.
Given $f\in \Lip(\X)$,
the Rademacher theorem implies that $f$ is differentiable $\m$-a.e.
For every stretchable collection $\I$ of $p$-test plans, we observe from Corollary \ref{equivweakupp} that
\[
|D^+f|_{w,\I} \le F^*({\dd}f), \quad |D^-f|_{w,\I} \le F^*(-{\dd}f), \quad
 |Df|_{w,\I} \le \max\{F^*(\pm{\dd}f)\} \quad \text{$\m$-a.e.}
\]
We have equality in those inequalities when $\I$ is sufficiently rich.
\end{example}

We summarize some further fundamental properties of minimal weak forward upper gradients.

\begin{lemma}\label{basicweakgradiprop}
Let $f,g:X\rightarrow \mathbb{R}$ be $\m$-measurable functions having $\I$-weak forward upper gradients.
Then, we have
\begin{enumerate}[{\rm (i)}]
\item\label{weakcondpla1} $|D^+(\lambda f+c)|_{w,\I} =\lambda |D^+ f|_{w,\I}$ for any $\lambda\geq 0$ and $c\in \mathbb{R}$,
\item\label{weakcondpla2} $|D^+ (f+g)|_{w,\I} \leq |D^+ f|_{w,\I} + |D^+ g|_{w,\I}$,
\item\label{weakcondpla3} $|D^+ f|_{w,\I}=|D^-(-f)|_{w,\I}$.
\end{enumerate}
\end{lemma}

\begin{lemma}\label{lm:wD^pm}
An $\m$-measurable function $f:X\rightarrow \mathbb{R}$ has a $\I$-weak upper gradient
if and only if it has both $\I$-weak forward and backward upper gradients, and then
we have $|Df|_{w,\I} =\max\{ |D^\pm f|_{w,\I} \}$ $\m$-a.e.
\end{lemma}

\begin{proof}
This is straightforward from the observation after Definition~\ref{weakuppgrad}.
\end{proof}

Next, we establish the stability of weak forward upper gradients (cf.\ \cite[Theorem 5.14]{AGS2}).
Note that the assumption \eqref{etbxcondition} below is weaker than the bounded compression \eqref{bouncomprecon}.

\begin{proposition}\label{stabilityofweakgradient}
Let $p>1$ and
suppose that, for any $\eta \in \I$ and $M\geq 0$, there is $C=C(\eta,M) \ge 0$ such that
\begin{equation}\label{etbxcondition}
\int^1_0 \bigl( (e_t)_\sharp \eta \bigr) \bigl[ \{x \in B \mid V(x)\leq M\} \bigr] \,{\dd}t
 \leq C\m[B] \quad \text{for all}\ B \in \mathcal{B}(X).
\end{equation}
Let $(f_i)_{i \ge 1}$ be a sequence of $\m$-measurable functions
and $G_i$ be a $\I$-weak forward upper gradient of $f_i$.
If $f_i(x) \to f(x)$ for $\m$-a.e.\ $x \in X$ and
$G_i$ weakly converges to $G$ in $L^q(\{V\leq M\},\m)$ for all $M\geq 0$,
then $G$ is a $\I$-weak forward upper gradient of $f$.
\end{proposition}

\begin{proof}
Fix $\eta \in \I$.
By approximation, we may assume that, for some $L,M>0$,
$\eta$-a.e.\ $\gamma$ satisfies $\mathcal {E}_p(\gamma)\leq L$ (recall Definition~\ref{pengergycurve})
and $\gamma \subset \{V\leq M\}$.
Since $G_i$ weakly converges to $G$ in $L^q(\{V\leq M\},\m)$,
Mazur's lemma yields convex combinations
\[
H_k =\sum_{i=N_k +1}^{N_{k+1}} \alpha_i G_i \quad
 \text{with }\alpha_i \geq 0,\ \sum_{i=N_k +1}^{N_{k+1} }\alpha_i=1,\quad \lim_{k \to \infty} N_k =\infty,
\]
which strongly converges to $G$ in $L^q(\{V\leq M\},\m)$.
Note that $H_k$ is a $\I$-weak forward upper gradient of
$\widetilde{f}_k :=\sum_{i=N_k +1}^{N_{k+1}} \alpha_i f_i$ and that $\widetilde{f}_k \to f$ $\m$-a.e.

Now, for every nonnegative Borel function $\varphi:X\rightarrow [0,\infty]$,
we deduce from \eqref{etbxcondition} that
\[
\int^1_0 \int_{\{V\leq M\}} \varphi^q \,{\dd}[(e_t)_\sharp \eta] \,{\dd}t
 \leq C(\eta,M) \int_{\{V\leq M\}} \varphi^q \,{\dm}.
\]
Thus, setting $\mathfrak{C}:=C^{1/q}L^{1/p}$, we have
\begin{align*}
&\int_{C_\tau([0,1];X)} \biggl( \int_{\gamma \cap \{ V\leq M \}} \varphi \biggr) \,\eta({\dd}\gamma)
 =\int_{C_\tau([0,1];X)}
 \biggl( \int^1_0 \mathbbm{1}_{\{V\leq M\}}(\gamma) \varphi(\gamma) |\gamma'_+| \,{\dd}t \biggr)
 \,\eta({\dd}\gamma) \\
&\leq \biggl( \int_{C_\tau([0,1];X)} \int^1_0 \mathbbm{1}_{\{V\leq M\}}(\gamma) \varphi^q(\gamma) \,{\dd}t \,\eta({\dd}\gamma)
 \biggr)^{1/q} \biggl( \int_{C_\tau([0,1];X)} \int^1_0 |\gamma'_+|^p \,{\dd}t \,\eta({\dd}\gamma) \biggr)^{1/p} \\
&= \biggl( \int^1_0 \int_{\{V\leq M\}} \varphi^q \,{\dd}[(e_t)_\sharp \eta] \,{\dd}t \biggr)^{1/q}
 \biggl( \int_{C_\tau([0,1];X)} \mathcal {E}_p(\gamma) \,\eta({\dd}\gamma) \biggr)^{1/p} \\
&\leq \biggl( C \int_{\{V\leq M\}} \varphi^q \,{\dm} \biggr)^{1/q} L^{1/p}
 = \mathfrak{C} \|\varphi\|_{L^q(\{V\leq M\},\m)}. 
\end{align*}
Substituting $\varphi=|H_k -G|$ yields
\[
\int_{C_\tau([0,1];X)} \biggl( \int_{\gamma\cap \{ V\leq M \}} |H_k -G| \biggr) \,\eta({\dd}\gamma)
 \leq \mathfrak{C} \|H_k -G\|_{L^q(\{V\leq M\},\m)}\to 0 \quad (k \to \infty).
\]
Hence, there exists a subsequence, again denoted by $H_k$,
such that $\int_\gamma |H_k -G| \to 0$ for $\eta$-a.e.\ $\gamma$.

Since $\widetilde{f}_k \to f$ $\m$-a.e., it follows from \eqref{testcondi} that
$\widetilde{f}_k (\gamma(t))\to f(\gamma(t))$ at $t=0,1$ for $\eta$-a.e.\ $\gamma$.
Therefore, as the limit of $\int_{\partial \gamma} \widetilde{f}_k \le \int_\gamma H_k$,
we obtain $\int_{\partial \gamma} f \le \int_\gamma G$ for $\eta$-a.e.\ $\gamma$.
This completes the proof.
\end{proof}

\begin{corollary}\label{weakconvegestability}
Let $f_i \in L^2(X,\m)$ be weakly convergent to $f$ in $L^2(X,\m)$, and
$G_i \in L^q(X,\m)$ be a $\I$-weak forward upper gradient of $f_i$ weakly converging to $G$ in $L^q(X,\m)$.
Then, $G$ is a $\I$-weak forward upper gradient of $f$  provided that $\I$ satisfies \eqref{etbxcondition}.
\end{corollary}

\begin{proof}
By Mazur's lemma, we have convex combinations
$\widetilde{f}_k=\sum_{i=N_k +1}^{N_{k+1}} \alpha_i f_i$ strongly converging to $f$,
and $H_k :=\sum_{i=N_k +1}^{N_{k+1}} \alpha_i G_i$ is a $\I$-weak forward upper gradient of $\widetilde{f}_k$
weakly converging to  $G$.
Thus, by Proposition \ref{stabilityofweakgradient}, $G$ is a $\I$-weak forward upper gradient of $f$.
\end{proof}

The following chain rule follows essentially from that in the Euclidean case
(cf.\ \cite[Proposition 5.16]{AGS2}).
Compare \eqref{weakpro3} with Lemma \ref{lipconvexfundd}.

\begin{proposition}\label{properweakuppergr}
Let $p>1$ and suppose that $\I$ satisfies \eqref{etbxcondition}.
If  $\m$-measurable functions $f,g:X\rightarrow \mathbb{R}$ have $\I$-weak forward upper gradients, 
then we have the following.
\begin{enumerate}[{\rm (i)}]
\item\label{weakpro1}
For any $\mathscr{L}^1$-negligible Borel set $N\subset \mathbb{R}$,
we have $|D^+f|_{w,\I}=0$ $\m$-a.e.\ on $f^{-1}(N)$.

\item\label{weakpro2}
For any non-decreasing function $\phi$ which is locally Lipschitz on an interval including the image of $f$,
we have $|D^+(\phi(f))|_{w,\I}=\phi'(f)|D^+f|_{w,\I}$ $\m$-a.e.\ on $X$.

\item\label{weakpro3}
For any non-decreasing contraction $($i.e., $|\phi(x)-\phi(y)|\leq c|x-y|$ for some $c \in [0,1))$, we have
\[
\bigl| D^+ \bigl( f+\phi(g-f) \bigr) \bigr|^q_{w,\I} +\bigl| D^+ \bigl( g-\phi(g-f) \bigr) \bigr|^q_{w,\I}
 \leq |D^+f|^q_{w,\I} +|D^+g|^q_{w,\I} \quad \m\text{-a.e.\ in }X.  
\]
\end{enumerate}
\end{proposition}

\begin{proof}
\eqref{weakpro1}
It follows from  Proposition \ref{uppergradsobo}\eqref{requweak-1} that,
for $\I$-almost every $\gamma$,
$f\circ\gamma$ coincides with an absolutely continuous function $f_\gamma:[0,1]\rightarrow \mathbb{R}$
at $\{0,1\}$ and $\mathscr{L}^1$-a.e.\ in $(0,1)$, and
\begin{equation*}
f'_\gamma \leq |D^+f|_{w,\I} \circ \gamma \cdot |\gamma'_+| \quad \mathscr{L}^1\text{-a.e.\ in }[0,1].
\end{equation*}
Since $f_\gamma$ is absolutely continuous,
we have $f'_\gamma(t) =0$ for $\mathscr{L}^1$-a.e.\ $t\in f_\gamma^{-1}(N)$.
This yields that $f'_\gamma(t) =0$ for $\mathscr{L}^1$-a.e.\ $t$ with $f(\gamma(t)) \in N$.
Then, set $G(x) :=|D^+f|_{w,\I}(x)$ if $f(x) \in \mathbb{R}\setminus N$
and $G(x):=0$ otherwise.
The above argument implies
\[
f'_\gamma \leq G \circ \gamma \cdot |\gamma'_+| \quad \mathscr{L}^1\text{-a.e.\ in }[0,1].
\]
Therefore, $G$ is a $\I$-weak forward upper gradient of $f$ by Corollary \ref{equivweakupp}.
It follows that $|D^+f|_{w,\I} \le G$ $\m$-a.e., and hence $|D^+f|_{w,\I}=0$ $\m$-a.e.\ on $f^{-1}(N)$.

\eqref{weakpro2}
For $f_\gamma$ as in \eqref{weakpro1},
$\phi(f_\gamma)$ is absolutely continuous by the hypothesis on $\phi$.
Since $(\phi\circ f)_\gamma:=\phi(f_\gamma)$ satisfies
\[
(\phi\circ f)'_\gamma =\phi'(f_\gamma) f'_\gamma
 \le \phi'(f_\gamma) \cdot |D^+ f|_{w,\I} \circ \gamma \cdot |\gamma'_+|
 \quad \text{for $\mathscr{L}^1$-a.e.\ } t \in [0,1],
\]
$\phi'(f)\,|D^+f|_{w,\I}$ is a $\I$-weak forward upper gradient of $\phi(f)$
and $|D^+(\phi(f))|_{w,\I} \leq \phi'(f)|D^+f|_{w,\I}$ $\m$-a.e.

Now, by scaling, we may assume that $\phi$ is $1$-Lipschitz.
Then, $t \mapsto t-\phi(t)$ is also non-decreasing and we find
$|D^+(f-\phi(f))|_{w,\I} \leq (1-\phi'(f))|D^+f|_{w,\I}$ $\m$-a.e.
Hence,  by Lemma \ref{basicweakgradiprop}\eqref{weakcondpla2}, 
\begin{align*}
|D^+ f|_{w,\I} &\le \bigl| D^+ \bigl( \phi(f) \bigr) \bigr|_{w,\I} +\bigl| D^+ \bigl( f-\phi(f) \bigr) \bigr|_{w,\I}
 \le \phi'(f)|D^+f|_{w,\I} +\bigl(1-\phi'(f) \bigr) |D^+f|_{w,\I} \\
&= |D^+ f|_{w,\I},
\end{align*}
and equality necessarily holds $\m$-a.e.

\eqref{weakpro3}
Approximating $\phi$ by a sequence $(\phi_i)_{i \ge 1}$ of non-decreasing $C^1$-contractions
(on an interval including the image of $f$)
such that $\phi_i \to \phi$ and $\phi'_i \to \phi'$ pointwise $\m$-a.e., we may assume that $\phi$ is $C^1$.

Put $\tilde{f}:=f+\phi(g-f)$ and note that, for $\I$-almost every $\gamma$,
$\tilde{f}_\gamma :=f_\gamma + \phi(g_\gamma-f_\gamma)$ is absolutely continuous.
Hence, for $0 \le t_1 < t_2 \le 1$, the Lagrange mean value theorem yields $t \in (t_1,t_2)$ such that
\[
\phi\bigl( (g_\gamma -f_\gamma)(t_2) \bigr) -\phi\bigl( (g_\gamma -f_\gamma)(t_1) \bigr)
 =\phi' \bigl( (g_\gamma-f_\gamma)(t) \bigr) \bigl\{ (g_\gamma-f_\gamma)(t_2)- (g_\gamma-f_\gamma)(t_1) \bigr\}.
\]
Then, we have
\begin{align*}
\tilde{f}_\gamma(t_2) -\tilde{f}_\gamma(t_1)
&= f_\gamma(t_2)-f_\gamma(t_1)
 +\phi' \bigl( (g_\gamma-f_\gamma)(t) \bigr) \bigl\{ (g_\gamma-f_\gamma)(t_2)- (g_\gamma-f_\gamma)(t_1) \bigr\} \\
&= \bigl\{ 1- \phi'\bigl( (g_\gamma-f_\gamma)(t) \bigr) \bigr\} \bigl( f_\gamma(t_2)-f_\gamma(t_1) \bigr)
 +\phi'\bigl( (g_\gamma-f_\gamma)(t) \bigr) \bigl( g_\gamma(t_2)-g_\gamma(t_1) \bigr) \\
&\le \bigl\{ 1- \phi'\bigl( (g_\gamma-f_\gamma)(t) \bigr) \bigr\} \int_{\gamma|_{[t_1,t_2]}} |D^+f|_{w,\I}
 +\phi'\bigl( (g_\gamma-f_\gamma)(t) \bigr) \int_{\gamma|_{[t_1,t_2]}}|D^+g|_{w,\I}.
\end{align*}
Since $\phi$ is $C^1$, this implies
\[
\tilde{f}'_\gamma
 \leq \bigl[ \bigl( 1-\phi'(g-f) \bigr) |D^+f|_{w,\I} \bigr] \circ\gamma \cdot |\gamma'_+|
 +\bigl[ \phi'(g-f) |D^+g|_{w,\I} \bigr] \circ\gamma \cdot |\gamma'_+|
\]
$\mathscr{L}^1$-a.e.\ on $(0,1)$.
Setting $h:= \phi'(g-f) \in [0,1]$, we find
\[
|D^+\tilde{f}|_{w,\I} \leq (1-h)|D^+f|_{w,\I} +h|D^+g|_{w,\I} \quad \text{$\m$-a.e.}
\]
By the convexity of $s \mapsto s^q$ for $s \ge 0$, we obtain
\[
|D^+\tilde{f}|_{w,\I}^q \leq (1-h)|D^+f|^q_{w,\I} +h|D^+g|^q_{w,\I} \quad \text{$\m$-a.e.}
\]
Similarly, for $\tilde{g}:=g-\phi(g-f)$, we deduce that
\[
|D^+\tilde{g}|_{w,\I}^q \leq (1-h)|D^+g|_{w,\I}^q +h|D^+f|_{w,\I}^q \quad \text{$\m$-a.e.}
\]
Combining these completes the proof of \eqref{weakpro3}.
\end{proof}

\begin{remark}\label{derivbakwarduppergradient}
If $f$ has a $\I$-weak backward upper gradient,
then applying \eqref{weakpro2} to $\bar{\phi}(s):=-\phi(-s)$ yields
\[
\bigl|D^- \bigl( \phi(f) \bigr) \bigr|_{w,\I} =\bigl| D^+ \bigl( \bar{\phi}(-f) \bigr) \bigr|_{w,\I}
 =\bar{\phi}'(-f) |D^+(-f)|_{w,\I} =\phi'(f) |D^- f|_{w,\I} \quad \text{$\m$-a.e.}
\]
\end{remark}

The next assumption plays an important role to study the Cheeger energy, heat flow and $q$-Laplacian.

\begin{assumption}\label{newassumpt32liambdafinite}
Assume that $(X,\tau,d,\m)$ is a $\APE$ satisfying Assumption \ref{strongerstassumptiontheta}
and that, for every compact set $K\subset (X,\tau)$, there is $r=r(K)>0$
such that $\m\Bigl[ \overline{B^-_K(r)}^d \Bigr] <\infty$.
\end{assumption}

\begin{remark}\label{meaningassmupt2}
Since $V$ is bounded on $K$ and forward Lipschitz by Assumption \ref{strongerstassumptiontheta},
we have ${\ee}^{-V^2} \ge c>0$ on $\overline{B^+_K(r)}^d$.
Combining this with $\int_X {\ee}^{-V^2} \dm<\infty$ yields $\m\Bigl[ \overline{B^+_K(r)}^d \Bigr] <\infty$.
Thus, \eqref{Kcondition1} holds under Assumption \ref{newassumpt32liambdafinite}.
We also remark that, when $\lambda_d(X)<\infty$ or in the Finsler setting,
Assumption \ref{newassumpt32liambdafinite} follows from Assumption \ref{strongerstassumptiontheta}.
\end{remark}

The following lemma follows the lines of \cite[Lemma 5.17]{AGS2}.

\begin{lemma}\label{bcaslemmaabso}
Let $p>1$, $(X,\tau,d,\m)$ be a $\APE$ satisfying Assumption $\ref{newassumpt32liambdafinite}$,
and ${\I_p}$ be the collection of all $p$-test plans with bounded compression on the sublevels of $V$.
Suppose that $(\mu_t) \in \AC^p([0,T];(\Po(X),W_p))$  and
a convex function $\phi:[0,\infty)\rightarrow \mathbb{R}$ satisfy the following:
\begin{enumerate}[{\rm (a)}]
\item\label{ass-a}
$\mu_t$ has the uniformly bounded density $f_t={\dd}\mu_t/{\dm}$ for all $t\in [0,T]$;
\item\label{ass-b}
$f_t$ is Sobolev along ${\I_p}$-almost every curve for $\mathscr{L}^1$-a.e.\ $t\in (0,T)$;
\item\label{ass-c}
$\phi(0)=0$ and $\phi'$ is locally Lipschitz in $(0,\infty)$;
\item\label{ass-d}
$H^\pm, G^\pm\in L^q(0,T)$, where
\[
H^\pm(t):=\biggl( \int_X |D^\pm f_t|^q_{w,{\I_p}} \dm \biggr)^{1/q}, \qquad
G^\pm(t):=\biggl( \int_{\{f_t>0\}} \bigl( \phi''(f_t)|D^\pm f_t|_{w,{\I_p}} \bigr)^qf_t \dm \biggr)^{1/q};
\]
\item\label{ass-e}
$\ds \int_X |\phi(f_0)| \dm<\infty$.
\end{enumerate}
Then, $t\mapsto \int_X |\phi(f_t)| \dm$ is bounded in $[0,T]$,
$\Phi(t):=\int_X \phi(f_t) \dm$ is absolutely continuous in $[0,T]$, and
\begin{equation}\label{dergphit}
-G^-(t) |\mu'_+|(t) \leq \Phi'(t) \leq G^+(t) |\mu'_+|(t) \quad \text{for $\mathscr{L}^1$-a.e. $t\in (0,T)$}.
\end{equation}
Moreover, for every  $t\in [0,T]$,
 we have the pointwise estimates:
\begin{align}\label{supinfdergp}
\begin{split}
\limsup_{s\to t^+}\frac{\Phi(s)-\Phi(t)}{s-t}
&\leq G^+(t) \limsup_{s\to t^+}\frac{1}{s-t}\int^s_t |\mu'_+| \,{\dd}r, \\
 \liminf_{s\to t^+} \frac{\Phi(s)-\Phi(t)}{s-t}
&\leq G^+(t) \liminf_{s\to t^+} \frac{1}{s-t} \int^s_t |\mu'_+| \,{\dd}r.
\end{split}
\end{align}
\end{lemma}



\begin{proof}
Let $T=1$  without loss  of generality,
and take $C>0$ such that $\mu_t \leq C\m$ for all $t\in [0,1]$ by \eqref{ass-a}.
Then we have
\begin{equation}\label{ftislp}
\|f_t\|_{L^q} =\biggl( \int_X f^{q-1}_t \,{\dd}\mu_t \biggr)^{1/q}
 \leq C^{(q-1)/q} =C^{1/p} <\infty.
\end{equation}
Note that $W_1\leq W_p$ implies $(\mu_t) \in \AC([0,1];(\mathscr{P}(X),W_1))$
and $W_1(\mu_s,\mu_t) \to 0$ as $s \to t$.
Thus, Proposition \ref{W1KANrUBDIS} yields
\begin{equation}\label{weaupconver}
\lim_{s \to t} \int_X \varphi f_s \dm = \int_X \varphi f_t \dm
 \qquad \text{for all}\ \varphi\in \Lip_1(X) \cap C_b(X).
\end{equation}
Given $\varphi \in L^p(X,\m)$,
thanks to the denseness as in Proposition \ref{densfolLIPS} (recall also Remark \ref{meaningassmupt2}),
there is a sequence $(\varphi_i)_{i \ge 1}$ of bounded and forward Lipschitz functions in $L^p(X,\m)$
such that $\|\varphi_i -\varphi\|_{L^p} \to 0$.
Then we deduce from \eqref{ftislp} that, for any $t \in [0,1]$,
\[
\biggl| \int_X \varphi_i f_t \dm - \int_X \varphi f_t \dm \biggr|
 \leq \|\varphi_i -\varphi\|_{L^p} \| f_t \|_{L^q}
 \leq C^{1/p} \|\varphi_i -\varphi\|_{L^p}.
\]
Combining this with $\int_X \varphi_i f_s \dm \to \int_X \varphi_i f_t \dm$
from \eqref{weaupconver}, we find
\begin{equation}\label{weakcontinft}
\lim_{s \to t} \int_X \varphi f_s \dm = \int_X \varphi f_t \dm
 \qquad \text{for all}\ \varphi\in L^p(X,\m).
\end{equation}

By approximation, we may assume that $\phi'$ is locally Lipschitz on $[0,\infty)$
(instead of $(0,\infty)$ as in \eqref{ass-c}; see \cite[Lemma 5.17]{AGS2} for details).
Moreover, by replacing $\phi(r)$ with $\phi(r)-\phi'(0)r$ if necessary
(then $\Phi(t)$ is replaced with $\Phi(t)-\phi'(0)$),
we can assume that $\phi'(0)=0$.
Then, for any $t\in (0,1]$, we have
\[
|\phi(f_t)| =|\phi(f_t)-\phi(0)| \leq \bigl( |\phi'(0)|+\SL(\phi'|_{[0,f_t]}) f_t \bigr) f_t
 \leq C\SL(\phi'|_{[0,C]}) f_t,
\]
which implies $\int_X |\phi(f_t)| \,{\dm}  \le C\SL(\phi'|_{[0,C]})<\infty$.
By the convexity of $\phi$, we have
\begin{equation}\label{convephift}
\phi(f_s)-\phi(f_t)\geq \phi'(f_t)(f_s -f_t) \qquad \text{for}\ s,t \in [0,1].
\end{equation}
Since $|\phi'(f_t)| \le \SL(\phi'|_{[0,C]}) f_t$ and $\|f_t\|_{L^p}<\infty$ as in \eqref{ftislp},
we find $\phi'(f_t)\in L^p(X,\m)$.
Thus, \eqref{weakcontinft} yields
\[
\liminf_{s \to t} \bigl( \Phi(s) -\Phi(t) \bigr)
\geq \lim_{s \to t} \int_X \phi'(f_{t})(f_s -f_t) \,{\dm} =0,
\]
thereby $\Phi$ is lower semi-continuous.

Owing to Theorem \ref{Lisinitheorem}, there exists $\eta \in \Po(C_\tau([0,1];X))$
concentrated in $\AC^p([0,1];X)$ with $\mu_t=(e_t)_\sharp \eta$ for every $t \in [0,1]$ and
\begin{equation}\label{strcacwass}
|\mu'_+|^p(t) =\int_{C_\tau([0,1];X)} |\gamma'_+|^p(t) \,\eta({\dd}\gamma)
 \quad \text{for $\mathscr{L}^1$-a.e.\ $t\in (0,1)$}.
\end{equation}
Note that $\eta \in {\I_p}$ since $f_t\leq C$.
According to \eqref{ass-b} and \eqref{ass-d}, $f_t$ is Sobolev along $\eta$-a.e.\ curve
and $H^{\pm}(t) <\infty$ for $\mathscr{L}^1$-a.e.\ $t\in (0,1)$.
We set
\[
h_t :=\phi'(f_t), \qquad
g^\pm_t :=|D^\pm h_t|_{w,{\I_p}} =\phi''(f_t)|D^\pm f_t|_{w,{\I_p}},
\]
where we used $\phi'' \ge 0$ (recall Proposition \ref{properweakuppergr}\eqref{weakpro2}).
Together with Remark \ref{derivbakwarduppergradient}, this furnishes
\begin{equation}\label{newweakuppergra}
 h_t \bigl( \gamma(t) \bigr) -h_t\bigl( \gamma(s) \bigr)
 \le \int^t_s g^+_t(\gamma)|\gamma'_+| \,{\dd}r, \qquad   h_s(\gamma(s))-h_s(\gamma(t))\leq   \int^t_s g_s^-(\gamma)|\gamma'_+| \,{\dd}r,
\end{equation}
for $\eta$-a.e.\ $\gamma$ and all $0\leq s\leq t\leq 1$.
Moreover, since $\phi'$ is locally Lipschitz and $f_t\leq C$,
there is a constant $C'>0$ such that $g^\pm_t\leq C'|D^\pm f_t|_{w,{\I_p}}$.
Hence, $g_t^{\pm} \in L^q(X,\m)$ by \eqref{ass-d}.

For every $0 \le s<t \le 1$,
we deduce from \eqref{convephift}, \eqref{newweakuppergra} and \eqref{strcacwass} that
\begin{align}
\begin{split}\label{keyconstPhi}
&\Phi(t)-\Phi(s)
 \leq \int_X \phi'(f_t)(f_t-f_s) \,{\dm}
 =\int_X h_t \,{\dd}\mu_t -\int_X h_t \,{\dd}\mu_s \\
&= \int_{C_\tau([0,1];X)}
 \Bigl( h_t \bigl( \gamma(t)\bigr) -h_t\bigl( \gamma(s) \bigr) \Bigr) \,\eta({\dd}\gamma)
 \leq \int^t_s \biggl( \int_{C_\tau([0,1];X)} g_t^+(\gamma) |\gamma'_+| \,\eta({\dd}\gamma) \biggr) \,{\dd}r \\
&\leq \int^t_s \biggl( \int_{C_\tau([0,1];X)} |g_t^+|^q(\gamma) \,\eta({\dd}\gamma) \biggr)^{1/q}
 \biggl( \int_{C_\tau([0,1];X)} |\gamma'_+|^p \,\eta({\dd}\gamma) \biggr)^{1/p} \,{\dd}r \\
&= \int^t_s \biggl( \int_X |g_t^+|^q f_r \,{\dm} \biggr)^{1/q} |\mu'_+|(r) \,{\dd}r
 \leq C^{1/q} C' \int^t_s \biggl( \int_X |D^+f_t|_{w,{\I_p}}^q \,{\dm} \biggr)^{1/q} |\mu'_+|(r) \,{\dd}r \\
&= C^{1/q} C'H^+(t) \int^t_s|\mu'_+|(r) \,{\dd}r.
\end{split}
\end{align}
A similar argument yields
\[
\Phi(t)-\Phi(s)\leq C^{1/q} C'H^-(t) \int^s_t|\mu'_+|(r) \,{\dd}r
\]
for $0 \le t<s \le 1$, thereby we obtain
\[
\Phi(t) -\Phi(s) \leq C^{1/q} C' \max\{H^\pm (t)\} \left|\int^t_s|\mu'_+|(r) \,{\dd}r\right|
\quad \text{for all}\ s,t \in [0,1].
\]
Then, since $H^{\pm} \in L^q(0,1)$ from \eqref{ass-d},
it follows from \cite[Lemma 2.9]{AGS2} that $\Phi$ is absolutely continuous.
Moreover, we infer from the above calculation \eqref{keyconstPhi} that
\begin{align*}
\frac{\Phi(t)-\Phi(s)}{t-s}\leq \frac{1}{t-s} \int^t_s \biggl( \int_X |g^+_{t}|^q \,{\dd}\mu_r \biggr)^{1/q} |\mu'_+|(r) \,{\dd}r, & \quad \text{ for } s<t,\\
-\frac{1}{s-t} \int_t^s \biggl( \int_X |g^-_{t}|^q \,{\dd}\mu_r \biggr)^{1/q} |\mu'_+|(r) \,{\dd}r
 \le \frac{\Phi(s)-\Phi(t)}{s-t}, &\quad  \text{ for } s>t.
\end{align*}
Taking the limit as  $t \to s^+$ in the former inequality  yields \eqref{supinfdergp} with the help of \eqref{weakcontinft}.
 Together with the latter inequality, we have
\eqref{dergphit} when $\Phi$ is differentiable at $t$ and $t$ is a Lebesgue point for $|\mu'_+|$.
\end{proof}

\begin{remark}\label{inveraweakupper}
All the concepts introduced so far are invariant when we replace $\m$ with the finite measure
$\widetilde{\m}:={\ee}^{-V^2}\m$ as in Assumption \ref{strongerstassumptiontheta}.
Indeed, we readily see that the bounded compression \eqref{bouncomprecon} with respect to $\widetilde{\m}$
on the sublevels of $\widetilde{V} \equiv 0$ implies that for $\m$ and $V$.
On the other hand, if $\eta$ is a $p$-test plan with bounded compression for $\m$ and $V$,
then we can approximate it with
\[
\eta_K :=\eta[\{ \gamma \mid \gamma \subset K \}]^{-1} \cdot \eta|_{\{\gamma \mid \gamma \subset K \}}
\]
for (large) compact sets $K \subset (X,\tau)$,
and $\eta_K$ is of bounded compression for $\widetilde{\m}$ and $\widetilde{V}$.
\end{remark}

\subsection{$L^2$-gradient flow of weak Cheeger energy}\label{L^2cheegerengerystrong}

Throughout this subsection,
let $(X,\tau,d,\m)$ be a $\APE$ satisfying Assumption \ref{newassumpt32liambdafinite},
$\I_p$ be the collection of all $p$-test plans with bounded compression on the sublevels of $V$
for some $p\in (1,\infty)$, and $p^{-1}+q^{-1}=1$.

\begin{definition}[Weak $q$-Cheeger energy]\label{weakcheegerenergy}
The {\it weak forward $q$-Cheeger energy} is a functional
defined in the class of $\m$-measurable functions $f:X\rightarrow [-\infty,\infty]$ by
\[
\Chc^+_{w,q}(f) :=\frac1q \int_X |D^+ f|^q_{w,\I_p} \dm
\]
if $f$ has a $\I_p$-weak forward upper gradient in $L^q(X,\m)$, and $\Chc^+_{w,q}(f) :=\infty$ otherwise.
\end{definition}

The {\it weak backward $q$-Cheeger energy} is defined in the same way,
and then $\Chc^-_{w,q}(f)=\Chc^+_{w,q}(-f)$ by Lemma \ref{basicweakgradiprop}\eqref{weakcondpla3}.
Hence, in the sequel, we focus on the forward one.
Observe from Lemma \ref{weakuppniegsets} that
$\Chc^+_{w,q}(f)$ is invariant under  modifications of $f$ in an $\m$-negligible set.

\begin{lemma}\label{constlowerseimcheegenergestronger}
$\Chc^+_{w,q}$ is convex and sequentially lower semi-continuous
with respect to the $\m$-a.e.\ pointwise convergence.
\end{lemma}

\begin{proof}
The convexity follows from Lemma \ref{basicweakgradiprop}{\eqref{weakcondpla1}}, \eqref{weakcondpla2}.
To show the lower semi-continuity,
let $(f_i)_{i \ge 1}$ be a sequence of $\m$-measurable functions converging to $f$ $\m$-a.e.
Without loss of generality, we may assume $\liminf_{i \to \infty}\Chc^+_{w,q}(f_i)<\infty$.
Then, $(|D^+ f_i|_{w,\I_p})_{i \ge 1}$ includes a subsequence bounded in $L^q(X,\m)$.
Thus, by passing to a subsequence,
$(|D^+ f_i|_{w,\I_p})_{i \ge 1}$ weakly converges to some $G\in L^q(X,\m)$.
Proposition \ref{stabilityofweakgradient} then implies that $G$ is a $\I_p$-weak forward upper gradient of $f$,
thereby $G\geq |D^+f|_{w,\I_p}$.
By the lower semi-continuity of the $L^q$-norm under weak convergence, we obtain
\[
\liminf_{i \to \infty}\Chc^+_{w,q}(f_i)
 \geq \frac1q \int_X G^q \dm \geq \Chc^+_{w,q}(f).
\]
This completes the proof.
\end{proof}

In what follows, we restrict $\Chc^+_{w,q}$ to $L^2(X,\m)$.
Let us still denote it by $\Chc^+_{w,q}$ for convenience,
and denote by $\mathfrak{D}(\Chc^+_{w,q})$ the domain of $\Chc^+_{w,q}$ in $L^2(X,\m)$,
i.e., $f\in \mathfrak{D}(\Chc^+_{w,q})$ if $f\in L^2(X,\m)$ and $|D^+ f|_{w,\I_p}\in L^q(X,\m)$.
The following is a consequence of Proposition \ref{densfolLIPS} (recall also Remark \ref{meaningassmupt2}).

\begin{lemma}\label{lm:Cheeger}
$\mathfrak{D}(\Chc^+_{w,q}) \cap L^2(X,\m)$ is dense in $L^2(X,\m)$.
\end{lemma}

Since $L^2(X,\m)$ is a Hilbert space, we can apply the standard theory in  \cite{AGS}.
Let $\partial^-\Chc^+_{w,q}(f)\subset L^2(X,\m)$ denote
the {\it subdifferentials} at $f\in \mathfrak{D}(\Chc^+_{w,q})$,
i.e., $\ell\in \partial^- \Chc^+_{w,q}(f)$ if
\[
\liminf_{h \to f\, \text{in}\, L^2}
 \frac{\Chc^+_{w,q}(h)-\Chc^+_{w,q}(f)- \langle h-f, \ell \rangle_{L^2}}{\|h-f\|_{L^2}}
 \geq 0.
\]
Under Assumption \ref{newassumpt32liambdafinite},
since $\Chc^+_{w,q}$ is convex and lower semi-continuous in $L^2(X,\m)$,
it follows from \cite[Proposition 1.4.4]{AGS} that $\ell\in \partial^-\Chc^+_{w,q}(f)$ if and only if
\begin{equation}\label{subdiffche}
\langle h-f, \ell \rangle_{L^2} \leq \Chc^+_{w,q}(h)-\Chc^+_{w,q}(f)
 \qquad \text{ for every } h \in L^2(X,\m).
\end{equation}
Note also that $\partial^-\Chc^+_{w,q}(f)$ is a closed convex set.

Define $\partial^\circ\Chc^+_{w,q}(f) \in \partial^- \Chc^+_{w,q}(f)$
as the unique element of the least $L^2$-norm,
and set $\|\partial^\circ\Chc^+_{w,q}(f)\|_{L^2}:=\infty$
if $ \partial^-\Chc^+_{w,q}(f) =\emptyset$.
The next proposition is a direct consequence of \cite[Proposition 1.4.4]{AGS}.

\begin{proposition}\label{estidesloandsubdiff}
The descending slope $|D^-\Chc^+_{w,q}|$ is a strong upper gradient for $-\Chc^+_{w,q}$,
and
\[
|D^-\Chc^+_{w,q}|(f)=\|\partial^\circ\Chc^+_{w,q}(f)\|_{L^2}
 \qquad \text{for all}\ f\in L^2(X,\m).
\]
\end{proposition}

We summarize some more properties derived from \cite[Theorems 2.4.15, 4.0.4]{AGS}.

\begin{theorem}\label{existstheoremft}
Given any $f_0 \in L^2(X,\m)$, there exists a unique curve $(f_t)_{t \ge 0}$
of $2$-maximal slope for $\Chc^+_{w,q}$ with respect to $|D^-\Chc^+_{w,q}|$
satisfying the following.
\begin{enumerate}[{\rm (i)}]
\item\label{graflowmaxislo1}
$(f_t)_{t>0}$ is locally Lipschitz in $L^2(X,\m)$.

\item\label{graflowmaxislo3}
$f_t\in \mathfrak{D}(|D^-\Chc^+_{w,q}|)\subset \mathfrak{D}(\Chc^+_{w,q})$ for any $t>0$.

\item\label{graflowmaxislo2}
The right derivative ${\dd}f_t/{\dd}t^+$ exists for every $t>0$ with
\[
\frac{\dd\ }{{\dd}t^+}\Chc^+_{w,q}(f_t)
 =-|D^-\Chc^+_{w,q}|^2 \big( f_t \big)
 =-\biggl\| \frac{\dd\ }{{\dd}t^+}f_t \biggr\|^2_{L^2}.
\]

\item\label{graflowmaxislo4}
For any $t>0$, we have the following regularizing effects:
\begin{align*}
|D^-\Chc^+_{w,q}|^2(f_t)
&\leq |D^-\Chc^+_{w,q}|^2(g) +\frac{\|f_0-g\|^2_{L^2}}{t^2}
 \quad \text{for all}\ g\in\mathfrak{D}(|D^-\Chc^+_{w,q}|), \\
\Chc^+_{w,q}(f_t)
&\leq \Chc^+_{w,q}(g) +\frac{\|f_0 -g\|_{L^2}^2}{2t}
 \quad \text{for all}\ g\in \mathfrak{D}(\Chc^+_{w,q}).
\end{align*}

\item\label{graflowmaxislo5}
For any $f_0,g_0\in L^2(X,\m)$, the corresponding curves $(f_t)_{t \ge 0},(g_t)_{t \ge 0}$ of $2$-maximal slope satisfy
\[
\|f_t-g_t\|_{L^2}\leq \|f_0-g_0\|_{L^2}
 \qquad \text{for all}\ t \ge 0.
\]
\end{enumerate}
\end{theorem}

\begin{proposition}\label{baceschder}
Let $I\subset \mathbb{R}$ be an open interval and $(f_t) \in \AC^2(I;L^2(X,\m))$
with $f_t \in \mathfrak{D}(\Chc^+_{w,q})$ for all $t \in I$.
If
\begin{equation}\label{inetergassumption}
\int_I |D^-\Chc^+_{w,q}|(f_t) \biggl\| \frac{\dd}{{\dd}t}f_t \biggr\|_{L^2}{\dd}t<\infty,
\end{equation}
then $t \mapsto {\Chc^+_{w,q}}(f_t)$ is absolutely continuous in $I$
and, for $\mathscr{L}^1$-a.e.\ $t \in I$,
\begin{equation*}
 \frac{\dd}{{\dd}t}{\Chc^+_{w,q}}(f_t)
 = \int_X \ell \frac{\dd}{{\dd}t}f_t \dm
 \qquad \text{for all}\ \ell \in \partial^-{\Chc^+_{w,q}}(f_t).
\end{equation*}
\end{proposition}

\begin{proof}
Since $|D^-\Chc^+_{w,q}|$ is a strong upper gradient (Proposition \ref{estidesloandsubdiff})
and the metric derivative of $(f_t)$ is exactly $\|{\dd}f_t/{\dd}t\|_{L^2}$ (see, e.g., \cite[Remark 1.1.3]{AGS}),
\eqref{inetergassumption} yields the absolute continuity of ${\Chc^+_{w,q}}(f_t)$.
Thus, both $\Chc^+_{w,q}(f_t)$ and $f_t$ are differentiable at $\mathscr{L}^1$-a.e.\ $t\in I$.
Given such $t$, we have
\[
\Chc^+_{w,q}(f_{t+\varepsilon})-\Chc^+_{w,q}(f_t)
 \geq \int_X  \ell (f_{t+\varepsilon}-f_t) \dm
\]
for any $\ell\in \partial^-\Chc^+_{w,q}(f_t)$,
which furnishes (by letting $\varepsilon \downarrow 0$ and $\varepsilon \uparrow 0$)
\begin{equation}\label{leftrightderivch}
\frac{\dd\ }{{\dd}t^+}\Chc^+_{w,q}(f_t) \geq \int_X \ell \frac{\dd}{{\dd}t^+}f_t \dm,
 \qquad \frac{\dd\ }{{\dd}t^-}\Chc^+_{w,q}(f_t) \leq \int_X \ell \frac{\dd}{{\dd}t^-}f_t \dm.
\end{equation}
This implies the claim.
\end{proof}

\begin{definition}[Gradient curves]\label{df:GF}
Given $f_0\in L^2(X,\m)$, a curve $(f_t) \in \AC^2_{\loc}((0,\infty);L^2(X,\m))$
is called a \emph{gradient curve} for $\Chc^+_{w,q}$
if $f_t \in \mathfrak{D}(\Chc^+_{w,q})$ for $t>0$,
$f_t\to f_0$ in $L^2(X,\m)$ as $t\to 0$ and
\[
\frac{\dd}{{\dd}t}f_t  \in -\partial^-{\Chc^+_{w,q}}(f_t)
 \quad \text{for $\mathscr{L}^1$-a.e.}\ t \in (0,\infty).
\]
\end{definition}

\begin{remark}\label{basicfunctionfochragd}
If $(f_t)$ is a gradient curve for $\Chc^+_{w,q}$, then we deduce from Proposition \ref{baceschder} that
\[
\frac{\dd}{{\dd}t}{\Chc^+_{w,q}}(f_t)
 = -\int_X \left( \frac{\dd}{{\dd}t}f_t\right)^2 \dm > -\infty
 \quad \text{for $\mathscr{L}^1$-a.e.}\ t \in (0,\infty),
\]
which yields the energy identity:
\[
\Chc^+_{w,q}(f_{t_1})-\Chc^+_{w,q}(f_{t_2})
 =\int^{t_2}_{t_1} \biggl\| \frac{\dd}{{\dd}t}f_t\biggr\|_{L^2}^2 {\dd}t
 \quad \text{ for all }0<{t_1}\leq t_2.
\]
\end{remark}

\begin{theorem}\label{realchgraevslop}
For any $f_0\in L^2(X,\m)$, a curve $(f_t) \in \AC^2_{\loc}((0,\infty);L^2(X,\m))$
with $f_t \to f_0$ in $L^2(X,\m)$ as $t \to 0$ is a curve of $2$-maximal slope
for $\Chc^+_{w,q}$ with respect to $|D^-\Chc^+_{w,q}|$
if and only if it is a gradient curve for $\Chc^+_{w,q}$.
Moreover, for every $t>0$, we have $-{\dd}f_t/{\dd}t^+ =\partial^\circ \Chc^+_{w,q}(f_t)$.
\end{theorem}

\begin{proof}
First, suppose that $(f_t)_{t>0}$ is a gradient curve for $\Chc^+_{w,q}$.
Since $|D^- \Chc^+_{w,q}|$ is a strong upper gradient for $-\Chc^+_{w,q}$
by Proposition \ref{estidesloandsubdiff},
it follows from  Remark \ref{basicfunctionfochragd} and \eqref{graddef} that
\[
\biggl\| \frac{\dd}{{\dd}t}f_t \biggr\|_{L^2}^2 =-\frac{\dd}{{\dd}t}{\Chc^+_{w,q}}(f_t)
 \leq |D^- \Chc^+_{w,q}|(f_t) \biggl\| \frac{\dd}{{\dd}t}f_t \biggr\|_{L^2}
 \leq \biggl\| \frac{\dd}{{\dd}t}f_t \biggr\|_{L^2}^2
\]
for $\mathscr{L}^1$-a.e.\ $t \in (0,\infty)$,
where the latter inequality follows from $|D^-\Chc^+_{w,q}|(f_t)=\|\partial^\circ\Chc^+_{w,q}(f_t)\|_{L^2}$
from Proposition \ref{estidesloandsubdiff} and ${\dd}f_t/{\dd}t \in -\partial^-{\Chc^+_{w,q}}(f_t)$.
Hence, equality holds in both inequalities:
\begin{equation}\label{basisiqu2}
 \frac{{\dd}}{{\dd}t}\Chc^+_{w,q}(f_t)
 =-\frac12 \biggl\| \frac{\dd}{{\dd}t}f_t \biggr\|_{L^2}^2 -\frac12 |D^- {\Chc^+_{w,q}}|^2(f_t), \qquad
 |D^- {\Chc^+_{w,q}}|(f_t) =\biggl\| \frac{\dd}{{\dd}t}f_t \biggr\|_{L^2}.
\end{equation}
Thus, $(f_t)_{t>0}$ is a curve of $2$-maximal slope for ${\Chc^+_{w,q}}$
(recall Definition~\ref{df:maxslope}).

Conversely, if $(f_t)_{t>0}$ is a curve of $2$-maximal slope,
then we have \eqref{basisiqu2} for $\mathscr{L}^1$-a.e.\ $t \in (0,\infty)$.
By Theorem \ref{existstheoremft}\eqref{graflowmaxislo3},
we can take $\ell_t :=-\partial^\circ \Chc^+_{w,q}(f_t)$, which satisfies
$\|\ell_t\|_{L^2}=|D^-\Chc^+_{w,q}|(f_t) <\infty$.
Then it follows from  \eqref{basisiqu2} and Proposition \ref{baceschder} that, for $\mathscr{L}^1$-a.e.\ $t \in (0,\infty)$,
\[
\biggl\| \frac{\dd}{{\dd}t}f_t \biggr\|_{L^2}^2
 =-\frac{\dd}{{\dd}t}{\Chc^+_{w,q}}(f_t)
 =\int_X \ell_t  \frac{\dd}{{\dd}t}f_t \dm
 \leq \|\ell_t\|_{L^2} \biggl\| \frac{\dd}{{\dd}t}f_t \biggr\|_{L^2}
 =\biggl\| \frac{\dd}{{\dd}t}f_t \biggr\|_{L^2}^2,
\]
which implies ${\dd}f_t/{\dd}t =\ell_t$.
Hence, $(f_t)_{t>0}$ is a gradient curve for $\Chc^+_{w,q}$.

To show the last assertion, we infer from Theorem \ref{existstheoremft}\eqref{graflowmaxislo2} that
\begin{equation*}
\frac{\dd\ }{{\dd}t^+}\Chc^+_{w,q}(f_t)
 = -\biggl\| \frac{\dd\ }{{\dd}t^+}f_t \biggr\|^2_{L^2}
 = -|D^-\Chc^+_{w,q}|^2(f_t)
\end{equation*}
for all $t>0$.
Combining this with
\begin{equation*}
\frac{\dd\ }{{\dd}t^+}\Chc^+_{w,q}(f_t)
 \geq \int_X -\ell_t \frac{\dd\ }{{\dd}t^+}f_t\dm
 \geq -\|\ell_t\|_{L^2} \biggl\| \frac{\dd\ }{{\dd}t^+}f_t \biggr\|_{L^2}
 =-|D^-\Chc^+_{w,q}|(f_t) \biggl\| \frac{\dd\ }{{\dd}t^+}f_t \biggr\|_{L^2}
\end{equation*}
from \eqref{leftrightderivch}, we obtain
${\dd}f_t/{\dd}t^+ =\ell_t =-\partial^\circ {\Chc^+_{w,q}}(f_t)$.
This completes the proof.
\end{proof}

\subsection{$q$-Laplacian}\label{oldLaplcian}

We continue to consider a $\APE$ $(X,\tau,d,\m)$ satisfying Assumption \ref{newassumpt32liambdafinite}
and the collection $\I_p$ of all $p$-test plans with bounded compression on the sublevels of $V$
for some $p \in (1,\infty)$.
Before introducing the $q$-Laplacian (with $p^{-1}+q^{-1}=1$) in our context, we review the Finsler case.

\begin{example}[Finsler case]\label{LaplacininRiemacase}
Let $(\X, d_F,\m)$ be a forward metric measure space
induced from a forward complete Finsler manifold endowed with a smooth positive measure,
satisfying either $\m[\X]<\infty$ or $\CD(K,\infty)$.
Recall from Remark \ref{meaningassmupt2} that Assumption \ref{newassumpt32liambdafinite} is satisfied.
Given $f \in C_0^{\infty}(\X)$, we define the {\it $q$-Laplacian} by
\[
\Delta_q f:=\di_{\m} \bigl (F^{q-2}(\nabla f)\nabla f \bigr)
\]
in the distributional sense:
\[
\int_{\X} \varphi \Delta_q f \dm
 := -\int_{\X} F^{q-2}(\nabla f) \cdot {\dd}\varphi(\nabla f)  \dm
 \qquad \text{for all}\ \varphi\in C^\infty_0(\X).
\]
Note that, for every $\varphi \in C^\infty_0(\X)$,
\[
\lim_{\varepsilon \to 0}
\frac{\Chc^+_{w,q}(f +\varepsilon \varphi) -\Chc^+_{w,q}(f)}{\varepsilon}
 = \int_{\X} F^{q-2}(\nabla f) \cdot {\dd}\varphi(\nabla f) \dm
 = -\int_{\X} \varphi \Delta_q f \,{\dm}.
\]
Thus, by approximation, we find $\Delta_q f \in -\partial^- \Chc^+_{w,q}(f)$
with $\|\Delta_q f\|_{L^2} = |D^- \Chc^+_{w,q}|(f)$.
Together with Proposition \ref{estidesloandsubdiff}, we have $\Delta_q f =-\partial^\circ \Chc^+_{w,q}(f)$.
\end{example}

In view of Theorem \ref{realchgraevslop} and Example \ref{LaplacininRiemacase},
we introduce the following (cf.\ \cite[Definition 4.13]{AGS2}).

\begin{definition}[$q$-Laplacian]\label{lapldef}
For $f\in L^2(X,\m)$ with $\partial^-\Chc^+_{w,q}(f) \neq \emptyset$,
we define the {\it $q$-Laplacian} by $\Delta_q f := -\partial^\circ \Chc^+_{w,q}(f)$.
\end{definition}

Note that, by Proposition \ref{estidesloandsubdiff},
the domain $\mathfrak{D}(\Delta_q)$ of $\Delta_q$ coincides with $\mathfrak{D}(|D^-\Chc^+_{w,\I_q}|)$.

\begin{lemma}\label{remarkfordefineioflapalce}
\begin{enumerate}[{\rm (i)}]
\item\label{lapace-2}
$\mathfrak{D}(\Delta_q)$ is a dense subset of $\mathfrak{D}(\Chc^+_{w,q})$ with respect to the $L^2$-norm.

\item\label{lapace-3}
For any gradient curve $(f_t)_{t>0}$ for $\Chc^+_{w,q}$, we have
\[
 \frac{\dd\ }{{\dd}t^+}f_t=\Delta_q f_t \qquad \text{for all}\ t>0.
\]
\end{enumerate}
\end{lemma}

\begin{proof}
\eqref{lapace-2} follows from the denseness of $\mathfrak{D}(|D^-\Chc^+_{w,q}|)$
in $\mathfrak{D}(\Chc^+_{w,q})$ seen in the same way as \cite[Lemma 3.1.3]{AGS}
(see also \cite[Lemma 3.20]{OZ}).
\eqref{lapace-3} is a direct consequence of Theorem \ref{realchgraevslop}.
\end{proof}

In the sequel, we will denote by $(\Ht(f))_{t>0}$
the gradient curve for $\Chc^+_{w,q}$ emanating from $f$.
By Lemma \ref{remarkfordefineioflapalce}\eqref{lapace-3},
$\Ht$ coincides with the \emph{$q$-heat flow} giving solutions to the \emph{$q$-heat equation}
${\dd}u_t/{\dd}t =\Delta_q u_t$.
The next corollary summarizes the outcomes of Theorems \ref{existstheoremft} and \ref{realchgraevslop}.

\begin{corollary}\label{heatflowequationproper-00}
\begin{enumerate}[{\rm (i)}]
\item\label{heatflowequation}
For any $f \in L^2(X,\m)$, there exists a unique solution
$(\Ht(f)) \in \AC^2_{\loc}((0,\infty);L^2(X,\m))$ to the $q$-heat equation
\[
\frac{\dd}{{\dd}t} u_t =\Delta_q u_t \quad \text{for $\mathscr{L}^1$-a.e.\ $t\in (0,\infty)$},
\qquad \lim_{t\to 0^+}u_t=f \,\ \text{in}\ L^2(X,\m).
\]
Moreover, for all $t>0$, the right derivative ${\dd}{\Ht}(f)/{\dd}t^+$ exists and satisfies
\[
 \frac{\dd\ }{{\dd}t^+}{\Ht}(f) =\Delta_q \Ht(f).
\]

\item
For $f_0 \in L^2(X,\m)$ and a curve $(f_t) \in \AC^2_{\loc}((0,\infty);L^2(X,\m))$
such that $f_t \to f_0$ in $L^2(X,\m)$, the following are equivalent:
\begin{itemize}
\item $(f_t)$ is a gradient curve for $\Chc^+_{w,q}$;
\item $(f_t)$ is a solution to the $q$-heat equation;
\item $(f_t)$ is a curve of $2$-maximal slope for $\Chc^+_{w,q}$ with respect to $|D^-\Chc^+_{w,q}|$.
\end{itemize}
\end{enumerate}
\end{corollary}

\begin{remark}\label{rm:revheat}
Unlike in the symmetric setting, the weak forward $q$-Cheeger energy
is generally not even, i.e., $\Chc^+_{w,q}(-f)\neq \Chc^+_{w,q}(f)$.
Consequently, the associated semigroup does not necessarily satisfy $\Ht(-f)=-\Ht(f)$.
Indeed, if we equip $X$ with the reverse metric
$\overleftarrow{d}(x,y):=d(y,x)$, then $-\Ht(-f)$ is precisely
the $q$-heat flow on the reverse space
$(X,\overleftarrow{d},\m)$ starting from $f$.
This reflects an essential difference between the forward
$q$-heat flow in the asymmetric setting and the classical
$q$-heat flow in the symmetric setting.
For the corresponding phenomenon in the Finsler setting,
we refer to \cite{Obook,OS}.
\end{remark}

\begin{lemma}[Homogeneity]\label{potentialoflinearlity}
The $q$-Laplacian $\Delta_q$ and the corresponding gradient flow $\Ht$
are positively $(q-1)$-homogenous, namely
\[
\Delta_q (\lambda f) =\lambda^{q-1} \Delta_q f, \qquad \Ht(\lambda f) =\lambda \mathsf{H}_{\lambda^{q-2}t}(f)
 \qquad \text{for all}\ \lambda \ge 0.
\]
Moreover, if $\m[X]<\infty$, then
\[
\Delta_q (f+c) =\Delta_q f, \qquad \Ht(f+c) =\Ht(f)+c \qquad \text{for all}\ c \in \mathbb{R}.
\]
\end{lemma}

For $\lambda<0$, the above homogeneity does not hold due to the asymmetry of $d$.

\begin{proof}
The former assertion is straightforward from the positive $q$-homogeneity of $\Chc^+_{w,q}$:
$\Chc^+_{w,q}(\lambda f) =\lambda^q \Chc^+_{w,q}(f)$ for $\lambda \ge 0$.
Then, in view of \eqref{subdiffche}, $\ell \in \partial^-\Chc^+_{w,q}(f)$ if and only if
$\lambda^{q-1} \ell \in \partial^-\Chc^+_{w,q}(\lambda f)$ for $\lambda>0$.
The latter assertion follows from $\Chc^+_{w,q}(f+c)=\Chc^+_{w,q}(f)$
(recall Lemma \ref{basicweakgradiprop}\eqref{weakcondpla1}).
\end{proof}

We summarize some more properties of the $q$-Laplacian
(cf.\ \cite[Proposition 4.15]{AGS2}, \cite[Proposition~6.5]{AGS4}, \cite[Proposition 3.2]{Ke}).

\begin{proposition}\label{basciproerfoLapla}
\begin{enumerate}[{\rm (i)}]
\item\label{proflap-1}
For any $f \in \mathfrak{D}(\Delta_q)$ and $h \in \mathfrak{D}(\Chc^+_{w,q})$, we have
\[
-\int_X h \Delta_q f \dm \leq \int_X |D^+ h|_{w,\I_p} |D^+ f|^{q-1}_{w,\I_p} \,{\dm}.
\]

\item\label{proflap-2}
For any $f \in \mathfrak{D}(\Delta_q)$ and Lipschitz function $\phi:I \rightarrow \mathbb{R}$
on a closed interval $I \subset \mathbb{R}$ including the image of $f$
$($with $\phi(0)=0$ if $\m[X]=\infty)$, we have
\[
-\int_X \phi(f)\Delta_q f \dm =\int_X \phi'(f) |D^+ f|_{w,\I_p}^q \,{\dm}.
\]

\item\label{proflap-3}
For any $f,h \in \mathfrak{D}(\Delta_q)$ with $-f\in \mathfrak{D}(\Chc^+_{w,q})$
and non-decreasing Lipschitz function $\phi:\mathbb{R}\rightarrow \mathbb{R}$
$($with $\phi(0)=0$ if $\m[X]=\infty)$, we have
\[
\int_X \phi(h-f) (\Delta_q h -\Delta_q f) \,{\dm} \leq 0.
\]
\end{enumerate}
\end{proposition}

\begin{proof}
\eqref{proflap-1}
Observe from $-\Delta_q f\in \partial^-\Chc^+_{w,q}(f)$ and \eqref{subdiffche} that
\[ 
-\int_X \varepsilon h \Delta_q f \,{\dm}
\leq \Chc^+_{w,q}(f+\varepsilon h)-\Chc^+_{w,q}(f)
\]
for all $\varepsilon \in \mathbb{R}$.
For $\varepsilon>0$, we infer from Lemma \ref{basicweakgradiprop}\eqref{weakcondpla2} that
\[
q\Chc^+_{w,q}(f+\varepsilon h)
 =\int_X |D^+(f+\varepsilon h)|_{w,\I_p}^q \,{\dm}
 \leq \int_X (|D^+ f|_{w,\I_p}+\varepsilon|D^+ h|_{w,\I_p})^q \,{\dm}.
\]
Combining them, we obtain
\[
-\int_X \varepsilon h\Delta_q f \,{\dm}
\leq \frac1q \int_X \bigl\{ (|D^+ f|_{w,\I_p}+\varepsilon|D^+ h|_{w,\I_p})^q-|D^+ f|_{w,\I_p}^q \bigr\} \,{\dm}.
\]
Dividing both sides by $\varepsilon$ and letting $\varepsilon\downarrow 0$ shows \eqref{proflap-1}.

\eqref{proflap-2}
Since $\phi$ is Lipschitz, for $\varepsilon \in \mathbb{R}$ close to $0$,
the function $t \mapsto t+\varepsilon\phi(t)$ is non-decreasing.
Thus, Proposition \ref{properweakuppergr}\eqref{weakpro2} yields
$|D^+(f+\varepsilon\phi(f))|_{w,\I_p} =(1+\varepsilon\phi'(f))|D^+f|_{w,\I_p}$ $\m$-a.e.
Hence, we have
\begin{align*}
\Chc^+_{w,q} \bigl( f+\varepsilon\phi(f) \bigr) -\Chc^+_{w,q}(f)
&= \frac1q \int_X |D^+f|^q_{w,\I_p} \bigl\{ \bigl(1+\varepsilon\phi'(f) \bigr)^q-1 \bigr\} \,{\dm} \\
&= \varepsilon \int_X \phi'(f) |D^+f|^q_{w,\I_p} \,{\dm} + o(\varepsilon).
\end{align*}
Therefore, for any $\ell \in \partial^-\Chc^+_{w,q}(f)$, we find
\begin{align*}
\varepsilon \int_X \ell \phi(f) \,{\dm}
 \leq \Chc^+_{w,q} \bigl( f+\varepsilon\phi(f) \bigr) -\Chc^+_{w,q}(f)
 =\varepsilon\int_X \phi'(f) |D^+f|^q_{w,\I_p} \,{\dm} + o(\varepsilon).
\end{align*}
Replacing $\varepsilon$ with $-\varepsilon$ yields the reverse inequality, thereby
\[
\varepsilon \int_X \ell \phi(f) \,{\dm}
 =\varepsilon\int_X \phi'(f) |D^+f|^q_{w,\I_p} \,{\dm} + o(\varepsilon)
\]
holds.
Choosing $\ell=-\Delta_q f$ completes the proof.

\eqref{proflap-3}
Setting $\varphi:=\phi(h-f)$, we find from the assumption $-f\in \mathfrak{D}(\Chc^+_{w,q})$
and Proposition \ref{properweakuppergr}\eqref{weakpro2} that
$\varphi \in \mathfrak{D}(\Chc^+_{w,q})$.
Since $-\Delta_q f \in \partial^- \Chc^+_{w,q}(f)$ and $-\Delta_q h \in \partial^- \Chc^+_{w,q}(h)$,
for any $\varepsilon>0$, we have
\[
-\varepsilon \int_X \varphi (\Delta_q f-\Delta_q h) \,{\dm}
 \leq \Chc^+_{w,q}(f+\varepsilon \varphi)-\Chc^+_{w,q}(f)
 +\Chc^+_{w,q}(h-\varepsilon \varphi)-\Chc^+_{w,q}(h).
\]
Now, for sufficiently small $\varepsilon>0$, $\varepsilon\phi$ is a contraction
and Proposition \ref{properweakuppergr}\eqref{weakpro3} implies that the RHS is nonpositive,
which completes the proof.
\end{proof}

Next, we consider properties of gradient curves for $\Chc^+_{w,q}$
(cf.\ \cite[Theorem~4.16]{AGS2}, \cite[Proposition~6.6]{AGS4}, \cite[Theorem 3.4]{Ke}).

\begin{theorem}[Comparison principle, contraction and mass preservation]\label{compprinconvr1}
Let $f_t=\Htt(f_0)$, $h_t=\Htt(h_0)$ be gradient curves for $\Chc^+_{w,q}$
starting from $f_0,h_0\in L^2(X,\m)$, respectively, with $-h_t\in \mathfrak{D}(\Chc^+_{w,q})$.
Let ${e}:\mathbb{R}\rightarrow [0,\infty]$ be a convex, lower semi-continuous function
and set $E(f):=\int_X {e}(f) \dm$.
\begin{enumerate}[{\rm (i)}]
\item\label{specpro-1}
If $f_0\leq C$ $($resp., $f_0\geq C)$ for some $C \in \mathbb{R}$,
then $f_t\leq C$ $($resp., $f_t\geq C)$ for all $t\geq 0$.
Similarly, if $f_0\leq h_0+C$ for some $C\in \mathbb{R}$, then $f_t\leq h_t+C$ for all $t \ge 0$.

\item\label{specpro-2}
The functional $f \mapsto E(f)$ is convex and lower semi-continuous in $L^2(X,\m)$ and satisfies
\begin{equation}\label{Eestimates}
E(f_t)\leq E(f_0), \quad E(f_t-h_t)\leq E(f_0-h_0) \quad \text{ for all }t\geq 0.
\end{equation}
In particular, if $f_0, h_0 \in L^\theta(X,\m)$ for some $\theta \in [1,{\color{blue}\infty]}$,
then $f_t, h_t \in L^\theta(X,\m)$ and $\Htt$ satisfies the \emph{contraction property:}
\begin{equation}\label{lpl2consgra}
\| f_t -h_t \|_{L^\theta}\leq \|f_0-h_0\|_{L^\theta} \quad \text{ for all }t\geq 0.
\end{equation}

\item\label{specpro-3}
If ${e}'$ is locally Lipschitz and $E(f_0)<\infty$, then we have
\begin{equation}\label{eftcontral}
E(f_t) +\int^t_0 \int_X {e}''(f_s) |D^+f_s|^q_{w,\I_p} \dm {\dd}s
 =E(f_0) \quad \text{ for all } t\geq 0.
\end{equation}

\item\label{specpro-4}
When $\m[X]<\infty$, $\Htt$ also satisfies the \emph{mass preservation:}
\[
\int_X f_t \dm=\int_X f_0\dm \quad \text{ for all } t\geq 0.
\]
\end{enumerate}
\end{theorem}

\begin{proof}
\eqref{specpro-1}
This can be reduced to \eqref{specpro-2} by choosing ${e}(r)=\max\{r-C,0\}$
or ${e}(r)=\max\{C-r,0\}$.

\eqref{specpro-2}
The convexity of $E$ is straightforward from that of ${e}$.
The lower semi-continuity of $E$ follows from that of ${e}$ and Fatou's lemma.
Note also that the first inequality in \eqref{Eestimates} follows from the second one by choosing $h_0 \equiv 0$.

Let us first assume that ${e}'$ is bounded and Lipschitz.
Then, for $a,b \in \mathbb{R}$, we have
\begin{align}
|{e}'(a)| &\leq |{e}'(0)| +\SL({e}')|a|, \notag\\
|{e}(b)-{e}(a)-{e}'(a)(b-a)| &\leq \frac12 \SL({e}')|b-a|^2, \notag\\
|{e}(b)-{e}(a)| &\leq \bigl( |{e}'(0)| +\SL({e}')(|a|+|b-a|) \bigr) |b-a|, \label{confxlfore}
\end{align}
where \eqref{confxlfore} follows from the former two inequalities.
We remark that, if $\m[X]=\infty$, then we can assume ${e}'(0)={e}(0)=0$.
Indeed, $f_0-h_0\in L^2(X,\m)$ and $E(f_0-h_0)<\infty$ (otherwise the proof is trivial) imply ${e}(0)=0$,
and then ${e}'(0)=0$ since ${e}$ is nonnegative and $C^1$.
Recall from Theorem \ref{existstheoremft}\eqref{graflowmaxislo1} that
$(f_t)_{t>0}$ and $(h_t)_{t>0}$ are locally Lipschitz in $L^2(X,\m)$,
thereby $t\mapsto {e}(f_t-h_t) \in L^1(X,\m)$ is locally Lipschitz by \eqref{confxlfore}.
Hence, for a.e.\ $t>0$, we infer from Proposition \ref{basciproerfoLapla}\eqref{proflap-3} that
\begin{align*}
\frac{\dd}{{\dd}t}E(f_t-h_t)
 =\int_X \frac{\dd}{{\dd}t}{e}(f_t-h_t) \dm
 =\int_X {e}'(f_t-h_t) (\Delta_q f_t -\Delta_q h_t)\dm \leq 0,
\end{align*}
which shows $E(f_t-h_t) \leq E(f_0-h_0)$.

For general ${e}$, we truncate ${e}'$ and then replace ${e}$
with its Yosida approximation (see Showalter \cite[pp.~161--163]{S}),
yielding $E(f_t-h_t) \leq E(f_0-h_0)$.
Finally, \eqref{lpl2consgra} is obtained by choosing ${e}(r)=|r|^\theta$.  When $\theta=\infty$, since $f_t,h_t\in L^2(X,\m)$, we have   $\lim_{p\rightarrow \infty}\|f_t-h_t\|_{L^p}= \|f_t-h_t\|_{L^\infty}$. 

\eqref{specpro-3}
Suppose again that ${e}'$ is bounded and Lipschitz.
Then, \eqref{eftcontral} follows from Proposition \ref{basciproerfoLapla}\eqref{proflap-2} as
\[ 
\frac{\dd}{{\dd}t}E(f_t) =\int_X {e}'(f_t) \Delta_q f_t \,{\dm}
 =-\int_X {e}''(f_t) |D^+f_t|^q_{w,\I_p} \,{\dm}.
\]
For general ${e}$, replacing it with
${e}_\varepsilon(r):={e}(r) +\varepsilon(1+r^2)^{1/2}-\varepsilon$ if necessary,
we may assume that ${e}$ takes its minimum value at some $r_0\in \mathbb{R}$
($r_0=0$ with $\varepsilon=0$ if $\m[X]=\infty$, since $E(f_0)<\infty$).
Then, we approximate ${e}$ by the convex functions
\[
{e}_k(r) :={e}(r_0) +\int^r_{r_0} w_k(s) \,{\dd}s, \qquad
 w_k:=\min \bigl\{ k,\max\{{e}',-k \} \bigr\}.
\]
Since ${e}_k'=w_k$ is bounded and Lipschitz, we have
\[
\int_X {e}_k(f_t) \dm +\int^t_0 \int_X {e}_k''(f_s) |D^+f_s|^q_{w,\I_p} \dm {\dd}s
 =\int_X {e}_k(f_0) \,{\dm}.
\]
By construction, ${e}''_k \uparrow {e}''$, ${e}'_k \to {e}'$ and ${e}_k \uparrow {e}$ as $k \to \infty$.
Thus, the monotone convergence theorem yields \eqref{eftcontral}.

\eqref{specpro-4}
By choosing $\phi \equiv 1$ in Proposition \ref{basciproerfoLapla}\eqref{proflap-2},
we obtain
\[
\frac{\dd}{{\dd}t}\int_X f_t \dm =\int_X \Delta_q f_t \dm =0,
\]
which concludes the proof.
\end{proof}

We can also relax the nonnegativity of ${e}$ as follows,
along the lines of \cite[Corollary 3.5]{Ke} (see also \cite[Proposition 4.22]{AGS2}).

\begin{corollary}\label{longequationcrestim}
Let $f_t=\Ht(f_0)$ be a gradient curve for $\Chc^+_{w,q}$ with $f_0 \ge 0$,
and ${e}:[0,\infty)\rightarrow \mathbb{R}$ be a convex $C^2$-function such that
\begin{itemize}
\item ${e}(0)=0$ and ${e}(r)\leq Cr^2$ for some $C\geq 0$,
\item there is some interval $I=[0,a]$ such that ${e}(f_t)\in L^1(X,\m)$ for every $t \in I$ and
\[
\int_I \int_{X} {e}''(f_s) |D^+f_s|_{w,\I_p}^q \dm {\dd}s<\infty.
\]
\end{itemize}
If $\int_X f_t \dm$ is  finite and  constant in $I$, then $E(f_t)=\int_X {e}(f_t) \dm$ satisfies
\begin{equation*}
E(f_t) +\int^t_0 \int_{X} {e}''(f_s) |D^+f_s|_{w,\I_p}^q \dm{\dd}s
 =E(f_0) \quad\ \text{for all}\ t\in I.
\end{equation*}
In particular, $t \mapsto E(f_t)$ is absolutely continuous in $I$ and
\[
\frac{{\dd}}{{\dd}t}E(f_t)
 =-\int_{X} {e}''(f_t) |D^+f_t|_{w,\I_p}^q \dm \quad\ \text{for $\mathscr{L}^1$-a.e.}\ t\in I.
\]
\end{corollary}

\begin{proof}
For each $\varepsilon\in (0,1)$, define a regularization of ${e}$ by
\[
{e}_\varepsilon(r) :=
\begin{cases}
{e}'(\varepsilon) r & \text{ for }r\in [0,\varepsilon],\\
{e}(r)-{e}(\varepsilon)+\varepsilon {e}'(\varepsilon) & \text{ for } r\in (\varepsilon,\infty).
\end{cases}
\]
Owing to the assumptions on ${e}$, we have ${e}_\varepsilon \in C^{1,1}$,
${e}'_\varepsilon(r) =\max\{{e}'(r),{e}'(\varepsilon) \}$,
and ${e}_\varepsilon(r) \downarrow {e}(r)$ as $\varepsilon\downarrow 0$.
Moreover, we infer from ${e}(0)=0$ and ${e}(r) \le C r^2$ that
${e}(r)\leq {e}_\varepsilon(r)\leq C_1r+C_2 r^2$ for some constants $C_1,C_2>0$ independent of $\varepsilon$.
Hence, we have
$|{e}_\varepsilon (r)| \leq |{e}(r)|+C_1r+C_2r^2$ and, thanks to $e(f_t) \in L^1(X,\m)$,
\[
\int_X |{e}_\varepsilon (f_t)| \,{\dm}
 \leq \int_X |{e}(f_t)| \,{\dm} +C_1\int_X f_t \,{\dm} +C_2\int_X f_t^2 \,{\dm} <\infty.
\]
Therefore, the monotone convergence theorem yields
\begin{equation}\label{conlftepsilon}
\lim_{\varepsilon \downarrow 0} \int_X {e}_\varepsilon (f_t) \dm =E(f_t).
\end{equation}

We now consider a convex $C^{1,1}$-function
\[
\widetilde{{e}}_\varepsilon(r) :={e}_\varepsilon(r)-{e}'(\varepsilon)r =
\begin{cases}
0 & \text{ for }r\in [0,\varepsilon],\\
{e}(r)-{e}'(\varepsilon) r-{e}(\varepsilon)+\varepsilon {e}'(\varepsilon)
 & \text{ for } r\in (\varepsilon,\infty).
\end{cases}
\]
Note that $\widetilde{{e}}_\varepsilon \ge 0$ by the convexity of $e$,
thereby we deduce from \eqref{eftcontral} that
\[
\int_X \widetilde{{e}}_\varepsilon(f_t) \dm
 +\int^t_0 \int_X \widetilde{{e}}''_\varepsilon(f_s) |D^+f_s|^q_{w,\I_p} \dm {\dd}s
 =\int_X \widetilde{{e}}_\varepsilon(f_0) \dm \quad\ \text{for all}\ t \in I.
\]
Since $\int_X f_t\dm$ is constant by hypothesis, this implies
\[
\int_X {e}_\varepsilon(f_t) \dm
 +\int^t_0 \int_X {e}''_\varepsilon(f_s) |D^+f_s|^q_{w,\I_p} \dm {\dd}s
 =\int_X {e}_\varepsilon(f_0) \dm \quad\ \text{for all}\ t \in I.
\]
Letting $\varepsilon \downarrow 0$ and using \eqref{conlftepsilon} completes the proof.
\end{proof}

\begin{lemma}\label{L2caclu}
For any Lipschitz curves $a:[0,T] \rightarrow L^p(X,\m)$ and $b:[0,T] \rightarrow L^q(X,\m)$
with $T<\infty$, the function $t\mapsto \int_X a_t b_t \dm$ is Lipschitz and satisfies
\[
\frac{\dd}{{\dd}t} \left[ \int_X a_t b_t \dm \right]
 =\int_X b_t \cdot \partial_t a_t \dm + \int_X a_t \cdot \partial_t b_t \dm\quad
 \text{ for $\mathscr{L}^1$-a.e.\ $t\in (0,T)$}.
\]
\end{lemma}

\begin{proof}
Let $C(a)$ and $C(b)$ be the Lipschitz constants of $t\mapsto a_t$ and $t\mapsto b_t$.
For any $s,t \in [0,T]$, we have
\begin{align*}
\int_X a_t b_t \dm -\int_X a_s b_s \dm
&=\int_X a_t (b_t -b_s) \dm +\int_X b_s (a_t -a_s) \dm \\
&\leq \|a_t\|_{L^p} \|b_t -b_s \|_{L^q} +\|b_s \|_{L^q} \|a_t -a_s \|_{L^p} \\
&\leq \biggl( \sup_{t\in[0,T]}\|a_t\|_{L^p} +\sup_{t\in[0,T]}\|b_t\|_{L^q} \biggr) \max\{C(a),C(b)\} \cdot |s-t|.
\end{align*}
Hence, $t\mapsto \int_X a_tb_t \dm$ is Lipschitz and differentiable $\mathscr{L}^1$-a.e.,
and then a direct calculation shows the claim.
\end{proof}

The next lemma is a key ingredient for the identification result of $q$-heat flow,
corresponding to \cite[Proposition 3.7]{GKO},  \cite[Lemma~6.1]{AGS2}, \cite[Lemma~7.2]{AGS4}
and \cite[Lemma 4.1]{Ke}.
We also refer to \cite[Section 7]{OS} for the case of Finsler manifolds;
we remark that the appearance of the backward metric speed $|\mu'_-|(t)$
is consistent with it (see Remark~\ref{rm:revLap} below as well)
and can be compared with the Varadhan formula in \cite{OSuz}.
We also refer to \cite{Oham} for a more general study of Hamiltonian systems.

\begin{lemma}[Kuwada's lemma]\label{wssvolemm2}
Let $f_t=\Htt(f_0)$ be a gradient curve for $\Chc^+_{w,q}$ in $L^2(X,\m)$ satisfying
\[
f_0 \ge 0, \qquad \int_X f_t \dm =1, \qquad
 \int^t_0 \int_{\{f_s>0 \}}\frac{|D^+f_s|^q_{w,\I_p}}{f_s^{p-1}} \dm {\dd}s <\infty
 \quad\ \text{for all } t\geq 0.
\]
Then, the curve $\mu_t:=f_t\m$ is in $\BAC^p_{\loc}((0,\infty);(\Po(X),W_p))$
and $|\mu'_-|$ satisfies
\[
|\mu'_-|^p(t) \leq \int_{\{f_t>0\}}\frac{|D^+f_t|^q_{w,\I_p}}{f_t^{p-1}} \dm
 \quad\ \text{for $\mathscr{L}^1$-a.e.\ $t\in (0,\infty)$}.
\]
\end{lemma}


\begin{proof}
Recall from Corollary \ref{corokan} that, for  $\mu,\nu\in \Po(X)$,
\begin{equation}\label{Q1equality}
\frac1p W_p(\mu,\nu)^p
 = \sup_{\psi \in C_b(X)} \left( \int_X \psi^{c_p} \,{\dd}\nu -\int_X \psi \,{\dd}\mu \right).
\end{equation}
Since the quantity in the RHS is invariant
when we add a constant to $\psi$, we may assume $\sup_X \psi=0$.
Given such $\psi$,
choose a sequence of increasing compact sets $K_i \subset X$ with $\bigcup_{i=1}^\infty K_i =X$
and consider $\psi_i:=\mathbbm{1}_{K_i} \cdot \psi$.
Each $\psi_i$ is bounded and $\tau$-lower semi-continuous,
and we have $\psi_i \downarrow \psi$ as well as $\psi_i^{c_p} \downarrow \psi^{c_p}$.
Thus, in the RHS of \eqref{Q1equality},
it is sufficient to consider $\tau$-lower semi-continuous $\psi$
whose support is included in some compact set $K$.

We shall use some results in Appendix \ref{hopflaxsemig} on the Hopf--Lax semigroup, defined by
\[
Q_t \psi(y) :=\inf_{x \in X} \biggl\{ \psi(x)+\frac{d^p(x,y)}{pt^{p-1}} \biggr\}.
\]
Note that $Q_1 \psi(y)=\psi^{c_p}(y)$.
By Propositions \ref{strongexpaap}, \ref{qtlowersemi} and the compactness of $K$,
for each $t>0$, $Q_t \psi$ is bounded, forward Lipschitz, and $\tau$-lower semi-continuous.
In particular, we have $t_*(y)=\infty$ for every $y \in X$ (see \eqref{txdefi}).

We claim that $(Q_t\psi)_{t\in [0,p]}$ is uniformly bounded in $L^p(X,\m)$.
We may assume $\min_K\psi<0$; otherwise, $\psi\equiv0$ and there is nothing to prove.
Observe first that
\[
\min_K \psi \leq Q_t\psi(y) \le \psi(y) \leq 0.
\]
On the other hand, for any $t \in (0,p]$ and $y \in X \setminus K$ with $d(K,y) \ge p|{\min_K \psi}|^{1/p}$,
we have
\[
Q_t\psi(y) = \inf_{x \in K} \biggl\{ \psi(x)+\frac{d^p(x,y)}{pt^{p-1}} \biggr\} \ge 0,
\]
which implies $Q_t\psi(y)=0$.
Thus, we obtain for $t\in (0,p]$ that
\[
\|Q_t\psi\|_{L^p(X,\m)}^p
 \leq \bigl| \min_K\psi \bigr|^p \m\Biggl[ \overline{B^+_K\Bigl( p\bigl| \min_K\psi \bigr|^{1/p} \Bigr)}^d \Biggr]
 <\infty,
\]
thereby $(Q_t\psi)_{t \in (0,p]}$ is uniformly bounded.
Moreover, we infer from Remark \ref{conatt=0} that
$Q_t\psi \uparrow\psi $ as $t\downarrow0$.
Hence, the monotone convergence theorem implies the uniform boundedness
of $(Q_t\psi)_{t \in [0,p]}$ in $L^p(X,\m)$.

Next, we fix $\varepsilon \in (0,1)$ and set $\varphi:=Q_{\varepsilon}\psi$.
For any $y \in X$, we deduce from \eqref{linftestQ} and \eqref{osccond} that
\begin{equation*}
\biggl\| \frac{\dd}{{\dd}t} [Q_t \varphi(y)] \biggr\|_{L^\infty(0,1)}
 =\biggl\| \frac{\dd}{{\dd}t} [Q_t \psi(y)] \biggr\|_{L^\infty(\varepsilon,1+\varepsilon)}
\leq \frac{1}{q\varepsilon^p}\| \mathfrak{d}^+(y,\cdot)^p \|_{L^\infty(\varepsilon,1+\varepsilon)}
\leq \frac{2^{p-1}(p-1)|{\min_K \psi}|}{\varepsilon^p}.
\end{equation*}
Thus,
\begin{equation}\label{lipspisocs}
\| Q_s \varphi -Q_t \varphi \|_{L^\infty}
 \leq \frac{2^{p-1}(p-1)|{\min_K \psi}|}{\varepsilon^p}|s-t|
 \quad\ \text{for all $s,t \in [0,1]$},
\end{equation}
and hence $t\mapsto Q_t \varphi$ is Lipschitz from $[0,1]$ to $L^\infty(X,\m)$,
and also to $L^2(X,\m)$ thanks to $\supp(Q_t \psi) \subset \overline{B^+_K(p|\min_K \psi|^{1/p})}^d$
and Remark \ref{meaningassmupt2}.
Note also that the derivative $\partial_t [Q_t \varphi]$ in the $L^2(X,\m)$-sense
coincides with ${\dd}[Q_t \varphi]/{\dd}t^+$ $\m$-a.e.\ for $\mathscr{L}^1$-a.e.\ $t$,
since
\[
\biggl\| \partial_t [Q_t \varphi] -\frac{{\dd}\ }{{\dd}t^+} [Q_t\varphi] \biggr\|_{L^2(X,\m)}
 =\lim_{\delta \to 0^+}
 \biggl\| \frac{Q_{t+\delta}\varphi -Q_t \varphi}{\delta} -\frac{{\dd}\ }{{\dd}t^+} [Q_t \varphi] \biggr\|_{L^2(X,\m)}
 =0
\]
by \eqref{lipspisocs} and the dominated convergence theorem.
Furthermore, by Proposition \ref{qtlowersemi},
${\dd}[Q_t \varphi]/{\dd}t^+$ is Borel in $X \times (0,\infty)$ and
$(x,t)\mapsto |D^\pm [Q_t \varphi]|(x)$ is $\mathcal{B}^*(X\times (0,\infty))$-measurable.

Now, given $0\leq s<t \leq 1$, we put $l:=t-s$ and apply Lemma \ref{L2caclu} with $p=q=2$ to
$a_r :=Q_{(l-r)/l}\varphi$ and $b_r :=f_{s+r}$ for $r \in (0,l]$
(recall Theorems \ref{existstheoremft}, \ref{realchgraevslop} for the Lipschitz continuity of $b_r$),
which yields
\begin{align*}
\frac{\dd}{{\dd}r} \biggl[ \int_X Q_{(l-r)/l} \varphi \cdot f_{s+r} \dm \bigg]
 =\int_X \biggl( -\frac{1}{l} \xi_{(l-r)/l} f_{s+r} +Q_{(l-r)/l}\varphi \cdot \Delta_q f_{s+r} \biggr) \dm
 \quad \text{for $\mathscr{L}^1$-a.e.\ $r$},
\end{align*}
where $\xi_t(y) :={\dd}[Q_t \varphi(y)]/{\dd}t^+$.
Recalling $\mu_t =f_t \m$ and by the Fubini theorem, we obtain
\[
\int_X \varphi \,{\dd} \mu_t -\int_X Q_1 \varphi \,{\dd}\mu_s
 = \int_0^l \int_X \biggl( -\frac{1}{l} \xi_{(l-r)/l} f_{s+r} +Q_{(l-r)/l} \varphi \cdot \Delta_q f_{s+r} \biggr) \,{\dm} \,{\dd}r.
\]
Combining this with \eqref{basicestimate1qqd}, we find
\begin{equation}\label{keyqq}
\int_X \varphi \,{\dd}\mu_t -\int_X Q_1 \varphi \,{\dd}\mu_s
\geq \int_0^l \int_X \biggl( \frac{|D^+ [Q_{(l-r)/l} \varphi]|^q}{ql} f_{s+r}
 +Q_{(l-r)/l}\varphi \cdot \Delta_q f_{s+r} \biggr) \,{\dm} \,{\dd}r.
\end{equation}
In the RHS, we deduce from
Proposition \ref{basciproerfoLapla}\eqref{proflap-1} and the Young inequality that
\begin{align*}
-\int_X Q_{(l-r)/l} \varphi \cdot \Delta_q f_{s+r} \dm
&\leq \int_X |D^+ [Q_{(l-r)/l} \varphi]|_{w,\I_p} |D^+ f_{s+r}|^{q-1}_{w,\I_p} \dm \\
&\leq \frac{1}{ql} \int_X |D^+[Q_{(l-r)/l} \varphi]|_{w,\I_p}^q \cdot f_{s+r} \dm
 +\frac{l^{p-1}}{p} \int_{\{f_{s+r}>0\}} \frac{|D^+ f_{s+r}|^q_{w,\I_p}}{f_{s+r}^{p-1}} \,{\dm}.
\end{align*}
Substituting this inequality into \eqref{keyqq}
and recalling $|D^+[Q_{(l-r)/l} \varphi]|_{w,\I_p} \le |D^+[Q_{(l-r)/l} \varphi]|$ furnishes
\[
\int_X Q_1 \varphi \,{\dd}\mu_s -\int_X \varphi \,{\dd}\mu_t
 \leq \frac{l^{p-1}}{p} \int^l_0 \int_{\{f_{s+r}>0\}} \frac{|D^+ f_{s+r}|_{w,\I_p}^q}{f_{s+r}^{p-1}} \dm {\dd}r.
\]

Since $\psi$ is $\tau$-lower semi-continuous,
by Remark \ref{conatt=0} we have $\varphi=Q_\varepsilon\psi\uparrow\psi$ as $\varepsilon\downarrow0$.
Moreover, since $t_*=\infty$, we find from Lemma \ref{properD}\eqref{Dii} that
$Q_1 \varphi =Q_{1+\varepsilon} \psi \to Q_1\psi$ as $\varepsilon\downarrow 0$.
Hence, letting $\varepsilon \downarrow 0$ in the above inequality, we have
\[
\int_X Q_1\psi \,{\dd}\mu_s -\int_X \psi \,{\dd}\mu_t
 \leq \frac{l^{p-1}}{p} \int^l_0 \int_{\{f_{s+r}>0\}} \frac{|D^+ f_{s+r}|_{w,\I_p}^q}{f_{s+r}^{p-1}} \dm {\dd}r.
\]
Finally, \eqref{Q1equality} yields
\[
W_p(\mu_t,\mu_s)^p
 \leq (t-s)^{p-1} \int_s^t \int_{\{f_r>0\}} \frac{|D^+ f_r|_{w,\I_p}^q}{f_r^{p-1}} \dm {\dd}r.
\]
Dividing both sides by $(t-s)^p$ and letting $s \uparrow t$ completes the proof.
\end{proof}

\begin{remark}\label{rm:revLap}
Since the $q$-Laplacian $\overleftarrow{\Delta}_q$ associated with the reverse metric
$\overleftarrow{d}(x,y)=d(y,x)$ is given by $\overleftarrow{\Delta}_q h=-\Delta_q (-h)$
(recall Remark \ref{rm:revheat}),
$h_t :=-\Htt(-f_0)$ satisfies $\partial_t h_t =\overleftarrow{\Delta}_q h_t$
and can be analyzed in the same way as in Lemma~\ref{wssvolemm2}
with respect to $\overleftarrow{d}$ (cf.\ \cite[Theorem~4.36]{OZ}), yielding
\[
|\nu'_+|^p(t) \leq \int_{\{h_t>0\}}\frac{|D^- h_t|^q_{w,\I_p}}{h_t^{p-1}} \,{\dm},
 \qquad \nu_t :=h_t \,{\m}.
\]
\end{remark}

We now generalize the \emph{moment-entropy estimates} with $\m[X]<\infty$
in \cite[Lemma~4.19]{AGS2} to $p \neq 2$ and to asymmetric distances;
see also \cite[Lemma~3.8]{Ke} for the case of $p \in (1,3) \setminus \{2\}$
 with symmetric distances. 

\begin{proposition}\label{keyabosplem}
For $p\in (1,\infty)$, define a convex function $U_p:[0,\infty)\rightarrow (-\infty,\infty]$ as follows:
\begin{equation}\label{functinalfun}
U_p(x):= \begin{cases}
 \frac{x^{3-p}}{(3-p)(2-p)}  & \text{ for } p \in (1,\infty) \setminus \{2,3\}, \\
 x\log x & \text{ for } p=2, \\
 -\log x &  \text{ for } p=3.
\end{cases}
\end{equation}
Let $f_t=\Ht(f_0)$, $t\geq 0$, be a gradient curve for $\Chc^+_{w,q}$
with $f_0 \ge 0$, $f_0 \in L^2(X,\m)$.
 When $p \in [3,\infty)$, we additionally assume $f_0 \ge c$ for some $c>0$. 
If $\m[X]<\infty$, then the function $t \mapsto  E(t):=\int_X U_p(f_t)\dm$ is absolutely continuous and satisfies
\begin{align}
E(0) &= E(t) +\int^t_0 \int_{X}   \frac{|D^+f_s|_{w,\I_p}^q}{f^{p-1}_s} \dm{\dd}s
 \quad\ \text{for all}\ t\geq 0,\label{infEconll1}\\
\frac{{\dd}}{{\dd}t}E(t)
 &= -\int_{X}\frac{|D^+f_s|_{w,\I_p}^q}{f^{p-1}_s} \dm \quad\ \text{for $\mathscr{L}^1$-a.e.}\ t\geq 0. \label{infEconll2}
 \end{align}
\end{proposition}

\begin{proof}
Given $\varepsilon>0$, Lemma \ref{potentialoflinearlity} and Theorem \ref{compprinconvr1}\eqref{specpro-1} imply $f_t+\varepsilon=\Htt(f_0+\varepsilon)$ with $f_t\geq 0$,  and $f_t \ge c$ for $p \in [3,\infty)$.
Set
\[
E_\varepsilon(t) :=\int_X U_p(f_t+\varepsilon) \dm=\int_X U_p \bigl( \Htt(f_0+\varepsilon)\bigr)\dm.
\]
Since $f_t\in L^2(X,\m)$ and $\m[X]<\infty$, we find that $E_\varepsilon(t), E(t)\in \mathbb{R}$ for all $p\in (1,\infty)$.
The rest of  the proof is divided into two  cases.

{\bf Case 1:} $p\in (1,\infty)\setminus\{3\}$.
Since $E(t)\in \mathbb{R}$, it suffices to show \eqref{infEconll1}. To do this,
fix $T>0$. By  Corollary \ref{heatflowequationproper-00}\eqref{heatflowequation} and Proposition \ref{basciproerfoLapla}\eqref{proflap-2}, we have for $\mathscr{L}^1$-a.e. $t\in (0,T)$
\[
\frac{\dd}{{\dd}t}\|f_t\|^2_{L^2}=2\int_X f_t \Delta_q f_t\dm=-2\int_X |D^+f_t|^q_{w,\I_p}\dm=-2q\Chc^+_{w,q}(f_t),
\]
which combined with Theorem \ref{existstheoremft}\eqref{graflowmaxislo1} implies
\[
\int^T_0 \Chc^+_{w,q}(f_s)\,{\dd}s=\frac{1}{2q}\left( \|f_0\|^2_{L^2}-\|f_T\|^2_{L^2}   \right)<\infty.
\]
Since  Lemma \ref{basicweakgradiprop}\eqref{weakcondpla1} combined with $f_s\geq 0$ yields
\[
 \int^T_0 \int_X \frac{|D^+(f_s+\varepsilon)|^q_{w,\I_p}}{(f_s+\varepsilon)^{p-1}} \dm {\dd}s
 \leq \frac{q}{ \varepsilon^{p-1}}\int^T_0 \Chc^+_{w,q}(f_s) \,{\dd}s <\infty,
\]
 Theorem \ref{compprinconvr1}\eqref{specpro-3} (for $p\in (1,2) \cup (3,\infty)$) and
Corollary \ref{longequationcrestim} (for $p\in [2,3)$) imply
\begin{equation*}
E_\varepsilon(t) + \int^t_0 \int_X  \frac{|D^+ f_s|^q_{w,\I_p}}{(f_s+\varepsilon)^{p-1}} \dm {\dd}s
 = E_\varepsilon(0) \quad\ \text{ for all } t \in [0,T].
\end{equation*}
By letting $\varepsilon\downarrow0$ and using the monotone convergence theorem (or the dominated convergence theorem for $p=2$), together with the fact that $E(t),E_\varepsilon(t)\in\mathbb{R}$, we obtain
\begin{equation*}
E(t) +\int^t_0 \int_X \frac{|D^+ f_s |^q_{w,\I_p}}{f_s^{p-1}} \dm {\dd}s
 =E(0) \quad \text{ for all } t\in [0,T].
\end{equation*}

\smallskip

{\bf Case 2:} $p=3$.
 In this case, recalling the hypothesis $f_t \ge c>0$,
we consider $h_t:=f_t -c>0$ and $U^c(x):=U_3(x+c)-U_3(c)$, $x \in [0,\infty)$.
Then $U^c$ is a nonpositive convex function with $U^c(0)=0$,
and satisfies $U^c(h_t)=U_3(f_t)-U_3(c)$.
Hence, applying the above argument (with Corollary \ref{longequationcrestim}) to $h_t$ and $U^c$ yields \eqref{infEconll1}.
 \end{proof}

Following \cite{AGS2,Ke}, one can further consider moment-entropy estimates for more general $\sigma$-finite measures $\m$, under additional conditions on $p$ and $V$.
Since such a discussion would be rather lengthy and the present paper is already long, we leave it to interested readers.

\section{Relaxed gradients, Cheeger energy and its $L^2$-gradient flow}\label{rexaghher}



In this section, let $(X,\tau,d,\m)$ be a $\APE$ and $q \in (1,\infty)$.
We introduce another kind of gradient without assuming Assumption \ref{strongerstassumptiontheta}.
For later convenience, we recall standard facts in the next lemma.

\begin{lemma}\label{weakconvergeunform}
Let $(G_i)_{i \ge 1}$ and $(H_i)_{i \ge 1}$ be sequences of nonnegative functions in $L^q(X,\m)$
such that $G_i \leq H_i$ and $H_i \to H$ weakly in $L^q(X,\m)$.
Then, we have the following.
\begin{enumerate}[{\rm (i)}]
\item\label{weakconvege}
$(G_i)_{i \ge 1}$ is uniformly bounded in $L^q(X,\m)$ and, in particular, has a weakly convergent subsequence.

\item\label{weakgracont}
If additionally $G_i \to G$ weakly in $L^q(X,\m)$, then $G \leq H$ $\m$-a.e.\ in $X$.
\end{enumerate}
\end{lemma}

\begin{proof}
\eqref{weakconvege}
By assumption we have $\sup_i \|H_i\|_{L^q} <\infty$,
then $0 \le G_i \le H_i$ implies the claim.

\eqref{weakgracont}
This readily follows from $H_i -G_i \to H-G$ weakly in $L^q(X,\m)$.
%
\end{proof}

\subsection{Minimal relaxed gradients}

Following \cite{Ch} and \cite[Definition~4.2]{AGS2}, \cite[Definition~4.2]{AGS4},
we introduce a relaxed notion of gradients.
Recall that, by Lemma \ref{borlstarmea},
$|D^+f|$ is $\mathcal{B}^*(X)$-measurable when $f$ is Borel (in $\tau$).
We remark that the following definition is slightly different from Kell's one \cite[Definition~1.3]{Ke}.

\begin{definition}[Relaxed gradients]\label{relaxedgradient}
Let $f \in L^q(X,\m)$.
\begin{enumerate}[(1)]
\item
A nonnegative function $G\in L^q(X,\m)$ is called
a {\it relaxed ascending} (resp., \emph{descending}) \emph{$q$-gradient} of $f$
if there are ($\tau$-)Borel forward (resp., backward) Lipschitz functions $f_i \in L^q(X,\m)$ such that
\begin{itemize}
\item $f_i \to f$ in $L^q(X,\m)$ and $|D^+f_i| \to \widetilde{G}$
(resp., $|D^-f_i| \to \widetilde{G}$) weakly in $L^q(X,\m)$,
\item $\widetilde{G} \leq G$ $\m$-a.e.\ in $X$.
\end{itemize}
Moreover, $G$ is called the {\it minimal relaxed ascending} (resp., \emph{descending}) \emph{$q$-gradient} of $f$
if it has the least $L^q(X,\m)$-norm among relaxed ascending (resp., descending) $q$-gradients,
and we denote such $G$ by $|D^+f|_{*,q}$ (resp., $|D^- f|_{*,q}$).

\item
Similarly, $G\in L^q(X,\m)$ is called a {\it relaxed $q$-gradient} of $f$ if there are Borel Lipschitz functions
$f_i \in L^q(X,\m)$ such that
\begin{itemize}
\item $f_i \to f$ in $L^q(X,\m)$ and $|Df_i| \to \widetilde{G}$ weakly in $L^q(X,\m)$,
\item $\widetilde{G}\leq G$  $\m$-a.e.\ in $X$.
\end{itemize}
We call $G$ with the minimal $L^q(X,\m)$-norm the {\it minimal relaxed $q$-gradient} of $f$,
and denote it by $|Df|_{*,q}$.
\end{enumerate}
\end{definition}



The minimal relaxed ascending $q$-gradient is indeed unique (if it exists) by the next two lemmas,
and it is also possible to obtain a minimal relaxed $q$-gradient as a strong limit
(cf.\ \cite[Lemma~4.3]{AGS2}, \cite[Proposition~4.3]{AGS4}).

\begin{lemma}\label{convxitycollgrade}
The collection of relaxed ascending $q$-gradients of $f \in L^q(X,\m)$ is a convex set $($possibly empty$)$.
\end{lemma}

\begin{proof}
We give a proof for completeness.
Let $G^\alpha$, $\alpha=1,2$, be two relaxed ascending $q$-gradients of $f$.
For each $\alpha$, we have a sequence $(f^\alpha_i)_{i \ge 1}$ of Borel forward Lipschitz functions such that
$f^\alpha_i \to f$ in $L^q(X,\m)$ and $|D^+f^\alpha_i| \to \widetilde{G}^\alpha$ weakly in $L^q(X,\m)$
with $\widetilde{G}^\alpha \leq G^\alpha$.
Given $t \in (0,1)$, $f_i := {t} f^1_i +(1-{t})f^2_i$ converges to $f$ in $L^q(X,\m)$ and,
by \eqref{postivegradnormineq}, $|D^+f_i| \leq {t}|D^+ f^1_i| +(1-{t})|D^+ f^2_i|$.
Then, it follows from Lemma \ref{weakconvergeunform} that,
by passing to a subsequence, $|D^+ f_i|$ weakly converges to some $G'\in L^q(X,\m)$
and $G' \leq {t} \widetilde{G}^1+(1-{t})\widetilde{G}^2$ $\m$-a.e.
Thus, ${t}{G}^1+(1-{t}){G}^2$ is a relaxed ascending $q$-gradient.
\end{proof}

\begin{lemma}\label{closurestrcrexlgra}
Let $f,G \in L^q(X,\m)$.
\begin{enumerate}[{\rm (i)}]

\item\label{closuer1}
If $G$ is a relaxed ascending $q$-gradient of $f$,
then there exist a sequence $(f_i)_{i \ge 1}$ of Borel forward Lipschitz functions
and a sequence $(G_i)_{i \ge 1}$ converging to $\widetilde{G}$ in $L^q(X,\m)$ such that
$f_i \to f$ in $L^q(X,\m)$, $|D^+ f_i| \leq G_i$ and $\widetilde{G} \leq G$.

\item\label{closuer2}
If $G^\alpha \in L^q(X,\m)$ is a relaxed ascending $q$-gradient of $f^\alpha \in L^q(X,\m)$
such that $f^\alpha \to f$ and $G^\alpha \to G$ both weakly in $L^q(X,\m)$ as $\alpha \to \infty$,
then $G$ is a relaxed ascending $q$-gradient of $f$.

\item\label{closuer3}
The collection of relaxed ascending $q$-gradients of $f$ is closed in $L^q(X,\m)$.
Moreover, if it is nonempty, then there are bounded Borel forward Lipschitz functions $f_i \in L^q(X,\m)$
such that $f_i \to f$ and $|D^+ f_i| \to |D^+f|_{*,q}$ in $L^q(X,\m)$.
\end{enumerate}
\end{lemma}

\begin{proof}
\eqref{closuer1}
Since $G$ is a relaxed ascending $q$-gradient,
there are Borel forward Lipschitz functions $g_k \in L^q(X,\m)$ such that $g_k \to f$ in $L^q(X,\m)$,
$|D^+g_k| \to \widetilde{G}$ weakly in $L^q(X,\m)$ and $\widetilde{G} \leq G$.
Then, by Mazur's lemma, we can take a sequence of convex combinations
\[
G_i =\sum_{k=i}^{m_i} a_{ik} |D^+ g_k|,\quad a_{ik}\geq 0,\quad \sum_{k=i}^{m_i} a_{ik}=1,
\]
converging to $\widetilde{G}$ in $L^q(X,\m)$.
Thanks to Lemma~\ref{convxitycollgrade},
we conclude the proof by letting $f_i :=\sum_{k=i}^{m_i} a_{ik} g_k$.

\eqref{closuer2}
The claim is equivalent to the weak closedness of the set
\[
S:=\{(f,G)\in L^q(X,\m) \times L^q(X,\m) \mid G \text{ is a relaxed ascending $q$-gradient of $f$} \}.
\]
Moreover, since $S$ is convex by (the proof of) Lemma \ref{convxitycollgrade},
it suffices to show that $S$ is strongly closed.

Suppose that $(f^\alpha,G^\alpha) \to (f,G)$ in $L^q(X,\m)\times L^q(X,\m)$.
For each $\alpha \ge 1$, by \eqref{closuer1} we can find $(f^\alpha_i)_i$ and $(G^\alpha_i)_i$ such that
$f_i^\alpha \to f^\alpha$ and $G_i^\alpha \to \widetilde{G}^\alpha$ in $L^q(X,\m)$,
$|D^+ f^\alpha_i| \leq G^\alpha_i$, and $\widetilde{G}^\alpha\leq G^\alpha$.
Since $\sup_\alpha \|\widetilde{G}^\alpha\|_{L^q}\leq \sup_\alpha \|G^\alpha\|_{L^q} <\infty$,
by passing to a subsequence, we may assume that $\widetilde{G}^\alpha$ weakly converges to some $\widetilde{G}$ in $L^q(X,\m)$ and $\widetilde{G} \leq G$ (by Lemma \ref{weakconvergeunform}\eqref{weakgracont}).
A diagonal argument then yields an increasing sequence $(i_\alpha)_\alpha$ such that
$f^\alpha_{i_\alpha} \to f$ in $L^q(X,\m)$, $G^\alpha_{i_\alpha} \to \widetilde{G}$ weakly in $L^q(X,\m)$,
and that $|D^+f^\alpha_{i_\alpha}|$ is uniformly bounded in $L^q(X,\m)$.
Again by passing to a subsequence,
we can assume that $|D^+f^\alpha_{i_\alpha}|$ weakly converges to some $H$.
Lemma \ref{weakconvergeunform}\eqref{weakgracont} yields $H\leq \widetilde{G}\leq G$,
thereby $G$ is a relaxed ascending $q$-gradient of $f$.

\eqref{closuer3}
Assuming that the collection of relaxed ascending $q$-gradients is nonempty,
we can find a unique minimal element by \eqref{closuer2}, denoted by $G:=|D^+f|_{*,q}$.
Take $(f_i)$, $(G_i)$ and $\widetilde{G}$ as in \eqref{closuer1}.
Since $|D^+f_i|$ is uniformly bounded in $L^q(X,\m)$,
we may assume that $|D^+f_i|$ weakly converges to some $H\in L^q(X,\m)$.
Then $H$ is also a relaxed ascending $q$-gradient,
and we have $0\leq H\leq \widetilde{G}\leq G$ (by Lemma \ref{weakconvergeunform}\eqref{weakgracont}).
This implies that $H=\widetilde{G}=G$ necessarily holds, and
\[
\|H\|^q_{L^q} \leq \liminf_{i \to \infty} \left\| |D^+f_i| \right\|^q_{L^q}
 \leq \limsup_{i \to \infty} \left\| |D^+f_i| \right\|^q_{L^q}
 \le  \|\widetilde{G}\|^q_{L^q} = \|H\|^q_{L^q}.
\]
Therefore, we obtain $|D^+ f_i| \to H=|D^+f|_{*,q}$ in $L^q(X,\m)$.
Finally, if $f_i$ is unbounded, then we can employ a suitable truncation $\widetilde{f}_i$
such that $\widetilde{f}_i \to f$ in $L^q(X,\m)$, since $|D^+\widetilde{f}_i| \leq |D^+f_i|$.
\end{proof}

Note that Lemmas \ref{convxitycollgrade} and \ref{closurestrcrexlgra} remain valid
for relaxed descending $q$-gradients and relaxed $q$-gradients.
The following lemma is straightforward.

\begin{lemma}\label{ascenddescednrealx}
\begin{enumerate}[{\rm (i)}]
\item\label{relaxedascgra-1} $f\in L^q(X,\m)$ has a relaxed ascending $q$-gradient if and only if $-f$ has a relaxed descending $q$-gradient, and then $|D^+f|_{*,q}=|D^-(-f)|_{*,q}$.

\item\label{relaxedascgra-2} If $f\in L^q(X,\m)$ has a relaxed $q$-gradient, then
$\max\{|D^\pm f|_{*,q}\}\leq |D f|_{*,q}=|D(-f)|_{*,q}$.
\end{enumerate}
\end{lemma}

The minimal relaxed $q$-gradient satisfies the following product (Leibniz) rule.

\begin{corollary}\label{gra*estimate}
If $f,g\in L^q(X,\m) \cap L^\infty(X,\m)$ have relaxed $q$-gradients,
then $fg$ has a relaxed $q$-gradient and
\[
|D(fg)|_{*,q}\leq |f| \cdot |Dg|_{*,q}+|g| \cdot |Df|_{*,q}
\]
holds.
Moreover, if $f$ and $g$ are nonnegative, then we have
\[
|D^+(fg)|_{*,q}\leq f|D^+g|_{*,q}+g|D^+f|_{*,q}.
\]
\end{corollary}

\begin{proof}
Owing to Lemma \ref{closurestrcrexlgra}\eqref{closuer3},
we find bounded Borel Lipschitz functions $f_i$, $g_i$ such that
$f_i \to f$, $|Df_i| \to |Df|_{*,q}$, $g_i \to g$, and $|Dg_i| \to |Dg|_{*,q}$ in $L^q(X,\m)$.
Moreover, since $f,g \in L^\infty(X,\m)$, by truncation, we may also assume that
$f_i$ and $g_i$ are uniformly bounded.
This implies that $f_i g_i$ is a Lipschitz function.
Then, it follows from Lemma \ref{derivatnorm}\eqref{derivatnorm2} and the dominated convergence theorem that
\[
|D(f_i g_i)| \leq |f_i||Dg_i|+|g_i||Df_i| \to |f||Dg|_{*,q}+|g||Df|_{*,q}
\]
weakly in $L^q(X,\m)$.
This shows the former assertion, since (a subsequence of) $|D(f_i g_i )|$ has a weak limit in $L^q(X,\m)$,
which is a relaxed $q$-gradient of $fg$.

The latter assertion is seen in the same way by employing the nonnegative part $f_i^+$ of $f_i$.
\end{proof}

\subsection{Properties of minimal relaxed gradients}\label{ssc-4.2}

Let us introduce another condition, now on the space $(X,d)$, rather than the measure $\m$
(compare it with Assumptions \ref{strongerstassumptiontheta}, \ref{newassumpt32liambdafinite}).

\begin{assumption}\label{compactreversbiliey}
Let $(X,\tau,d,\m)$ be a $\APE$. Suppose that one of the following conditions holds:
\begin{itemize}
\item The reversibility of $(X,d)$ is finite;
\item $(X,d)$ is a geodesic space such that every compact set $K\subset (X,\tau)$   has finite diameter.
\end{itemize}
\end{assumption}

In the rest of this subsection, we will assume Assumption \ref{compactreversbiliey}.
We discuss further properties of minimal relaxed $q$-gradients
along the lines of \cite[Lemma 4.4, Proposition 4.8]{AGS2}
(see also \cite[Lemma~5.1, Proposition~5.2]{AGS4}, \cite[Proposition~1.4, Corollary~1.5]{Ke}).

\begin{lemma}[Locality]\label{localmingrest}
Let $f \in L^q(X,\m)$.
\begin{enumerate}[{\rm (i)}]
\item\label{localmingrest1}
For any two relaxed ascending $q$-gradients $G_1,G_2$ of $f$,
$\min\{G_1,G_2\}$ and $\mathbbm{1}_B G_1+\mathbbm{1}_{X\setminus B}G_2$, $B\in \mathcal{B}(X)$,
are relaxed ascending $q$-gradients of $f$.

\item\label{localmingrest2}
For any relaxed ascending $q$-gradient $G$ of $f$, we have $|D^+f|_{*,q} \leq G$ $\m$-a.e.\ in $X$.
\end{enumerate}
\end{lemma}

\begin{proof}
\eqref{localmingrest1}
Note that $\min\{G_1,G_2\} =\mathbbm{1}_B G_1 +\mathbbm{1}_{X\setminus B} G_2$
for $B =\{x\in X \mid G_1(x)<G_2(x)\}$,
thereby it is sufficient to consider $\mathbbm{1}_B G_1+\mathbbm{1}_{X\setminus B}G_2$.
Moreover, by an approximation by compact subsets together with Lemma \ref{closurestrcrexlgra}\eqref{closuer2},
we may assume that $K:=X\setminus B$ is compact (in $\tau$).

Given $r>0$, take a non-increasing Lipschitz function $\phi_r:[0,\infty) \rightarrow [0,1]$ such that
$\phi_r \equiv 1$ in $[0,r/3]$ and $\phi_r \equiv 0$ in $[(2r)/3,\infty)$, and set $\chi_r(x):=\phi_r(d(K,x))$.
It follows from Lemma \ref{ref'nproofandother}\eqref{chiislips} and Assumption \ref{compactreversbiliey}
that $\chi_r$ is a Borel Lipschitz function.
Note that $\chi_r(x)=|D\chi_r|(x)=0$ for $x \in X$ with $d(K,x)>(2r)/3$,
$\chi_{r} \equiv 1$ in $K=X\setminus B$, and $\chi_{r}\to 0$ in $X\setminus K=B$ as $r \downarrow 0$.
Hence, the dominated convergence theorem yields that
$(1-\chi_{r}) G_1+\chi_{r}G_2$ converges to $\mathbbm{1}_B G_1+\mathbbm{1}_{X\setminus B}G_2$ in $L^q(X,\m)$.
Therefore, it suffices to show that $(1-\chi_{r}) G_1+\chi_{r}G_2$ is a relaxed ascending $q$-gradient for all $r>0$.

For each $\alpha=1,2$, let $(f_i^\alpha)_i$ be a sequence of bounded Borel forward Lipschitz functions
converging to $f$ in $L^q(X,\m)$ such that $|D^+ f_i^\alpha| \to \widetilde{G}_\alpha$
weakly in $L^q(X,\m)$, with $\widetilde{G}_\alpha \leq G_\alpha$.
Then, $f_i :=(1-\chi_{r}) f_i^1 +\chi_{r}f_i^2$ is forward Lipschitz by Lemma \ref{ref'nproofandother}\eqref{chilinercombinislip}, and we infer from Lemma \ref{derivatnorm}\eqref{derivatnorm3} that
\begin{align*}
|D^+ f_i|
&\leq (1-\chi_{r})|D^+f_i^1|+ \chi_{r}|D^+f_i^2|+{|D\chi_r| |f_i^1 -f_i^2|}\\
&\to (1-\chi_{r}) \widetilde{G}_1 +\chi_{r}\widetilde{G}_2
 \leq (1-\chi_{r}) {G}_1+\chi_{r} {G}_2.
\end{align*}
Therefore, $(1-\chi_{r}) {G}_1+ \chi_{r} {G}_2$ is a relaxed ascending $q$-gradient
(by Lemma \ref{weakconvergeunform}).

\eqref{localmingrest2}
Suppose that $G<|D^+f|_{*,q}$ on a Borel set $B$ with $\m[B] >0$.
Then, \eqref{localmingrest1} implies that $\mathbbm{1}_B G+\mathbbm{1}_{X\setminus B}|D^+f|_{*,q}$
is a relaxed ascending $q$-gradient, however,
whose $L^q$-norm is strictly less than that of $|D^+f|_{*,q}$.
This contradicts the minimality of $|D^+f|_{*,q}$.
\end{proof}

As a direct consequence of Lemma \ref{localmingrest}\eqref{localmingrest2},
if $f,g \in L^q(X,\m)$ have relaxed ascending $q$-gradients, then
\begin{equation}\label{triangeinequaofdf}
|D^+(f+g)|_{*,q}\leq |D^+f|_{*,q}+|D^+g|_{*,q} \quad \text{$\m$-a.e.\ in $X$}.
\end{equation}
Moreover, for any forward Lipschitz function $f \in L^q(X,\m)$,
we find from Lemmas~\ref{localmingrest}\eqref{localmingrest2} and \ref{Lipscontin}\eqref{Lispro1} that
\begin{equation}\label{lisprelgraest}
|D^+f|_{*,q} \leq |D^+f|\leq \Lip(f) \quad \text{$\m$-a.e.\ in $X$}.
\end{equation}

\begin{lemma}\label{dividationofminrelaxedgradient}
If $f\in L^q(X,\m)$ has a relaxed ascending $q$-gradient,
then $f^+$ and $-f^-$ also have relaxed ascending $q$-gradients and we have
\[
|D^+ f|_{*,q}=|D^+ f^+|_{*,q}+|D^+(- f^-)|_{*,q} \quad \text{$\m$-a.e.\ in $X$}.
\]
\end{lemma}

\begin{proof}
By Lemma \ref{closurestrcrexlgra}\eqref{closuer3}, there are bounded Borel forward Lipschitz functions
$f_i \in L^q(X,\m)$ such that $f_i \to f$ and $|D^+f_i| \to |D^+f|_{*,q}$ in $L^q(X,\m)$.
Since $|f_i^\pm -f^\pm| \le |f_i-f|$, we find $f^\pm_i \to  f^\pm$ in $L^q(X,\m)$.
Then, $|D^+f_i| =|D^+f^+_i| +|D^+(-f^-_i)|$ from \eqref{expressionofD+}
implies that $f^\pm$ have relaxed ascending $q$-gradients and,
together with Lemma \ref{localmingrest}\eqref{localmingrest2},
\[
|D^+ f|_{*,q}\geq |D^+ f^+|_{*,q}+|D^+(- f^-)|_{*,q} \quad \text{$\m$-a.e.}
\]
The converse inequality follows from \eqref{triangeinequaofdf}.
\end{proof}

\begin{proposition}[Chain rule]\label{chainrulerelax}
Let $f,g \in L^q(X,\m)$ have relaxed ascending $q$-gradients.
\begin{enumerate}[{\rm (i)}]
\item\label{Chainrule-1}
For any $\mathscr{L}^1$-negligible Borel set $N\subset \mathbb{R}$,
we have $|D^+f|_{*,q}=0$ $\m$-a.e.\ in $f^{-1}(N)$.

\item\label{Chainrule-2}
For any $c\in \mathbb{R}$, $|D^+f|_{*,q}=|D^+g|_{*,q}$ holds $\m$-a.e.\ in $(f-g)^{-1}(c)$.


\item\label{Chainrule-3}
For any non-decreasing Lipschitz function $\phi$ on an interval $I$ including the image of $f$
$($with $0\in I$ and $\phi(0)=0$ if $\m[X]=\infty)$,
we have $|D^+ [\phi(f)]|_{*,q}=\phi'(f)|D^+f|_{*,q}$ $\m$-a.e.\ in $X$.

\item\label{Chainrule-4}
If $\phi:\mathbb{R}\rightarrow \mathbb{R}$ is a non-decreasing contraction
$($with $\phi(0)=0$ if $\m[X] =\infty)$, then
\begin{equation}\label{2conveindd}
\bigl| D^+ \bigl( f+\phi(g-f) \bigr) \bigr|^q_{*,q} +\bigl| D^+ \bigl( g-\phi(g-f) \bigr) \bigr|_{*,q}^q
 \leq |D^+f|^q_{*,q}+|D^+g|^q_{*,q} \quad \text{$\m$-a.e.\ in $X$}.
\end{equation}
\end{enumerate}
\end{proposition}

Recall that $\phi$ being a contraction means $|\phi(x)-\phi(y)| \leq c|x-y|$ for some $c \in [0,1)$.
In \eqref{Chainrule-3}, we remark that,
since $\phi$ is differentiable $\mathscr{L}^1$-a.e.\ by the Rademacher theorem
and $|D^+f|_{*,q}= |D^+ [\phi(f)]|_{*,q} =0$ $\m$-a.e.\ in $f^{-1}(N)$ for the set $N \subset I$
on which $\phi$ is not differentiable by \eqref{Chainrule-1}, the assertion makes sense.

\begin{proof}
We first claim that, for any non-decreasing $C^1$-function $\phi:\mathbb{R}\rightarrow \mathbb{R}$
which is Lipschitz on the image of $f$ (with $\phi(0)=0$ if $\m[X]=\infty$), we have
\begin{equation}\label{compwekmingrades}
|D^+ [\phi(f)]|_{*,q}\leq \phi'(f) |D^+f|_{*,q} \quad \text{$\m$-a.e.\ in $X$}.
\end{equation}
By Lemma \ref{closurestrcrexlgra}\eqref{closuer3},
we can take a sequence $(f_i)_{i \ge 1}$ of bounded Borel forward Lipschitz functions
such that $f_i \to f$ and $|D^+f_i| \to |D^+ f|_{*,q}$ in $L^q(X,\m)$.
Note that $\phi(f_i)$ is forward Lipschitz since $\phi$ is non-decreasing.
Then, by the Lipschitz continuity of $\phi$, we find $\phi(f_i) \to \phi(f)$ in $L^q(X,\m)$.
Moreover, since $\phi'(f_i)$ is bounded,
\[
|D^+ [\phi(f_i)]| \leq \phi'(f_i) |D^+ f_i| \to \phi'(f) |D^+f|_{*,q} \quad \text{in $L^q(X,\m)$}.
\]
Hence, $|D^+ [\phi(f_i)]|$ is uniformly bounded in $L^q(X,\m)$
and has a subsequence weakly convergent to some $G$,
which is a relaxed $q$-gradient of $\phi(f)$.
Thus, Lemma \ref{localmingrest}\eqref{localmingrest2} yields
$|D^+ [\phi(f)]|_{*,q} \leq G \leq \phi'(f) |D^+f|_{*,q}$.

\eqref{Chainrule-1}
We first assume that $N$ is compact.
Let $A_i \subset \mathbb{R}$ be a sequence of decreasing open sets such that
$A_i \downarrow N$ and $\mathscr{L}^1(A_1)<\infty$,
and let $\psi_i: \mathbb{R}\rightarrow [0,1]$ be a continuous function satisfying
$\mathbbm{1}_N \leq \psi_i\leq \mathbbm{1}_{A_i}$.
Define $\phi_i: \mathbb{R}\rightarrow \mathbb{R}$ by $\phi_i(0)=0$ and $\phi'_i =1-\psi_i$.
Then, every $\phi_i$ is a non-decreasing, $1$-Lipschitz and $C^1$-function.
Moreover, since $N$ is $\mathscr{L}^1$-negligible,
$(\phi_i)_{i \ge 1}$ uniformly converges to the identity map on $\mathbb{R}$.
We infer from the dominated convergence theorem that $\phi_i(f) \to f$ in $L^q(X,\m)$,
then Lemma \ref{closurestrcrexlgra}\eqref{closuer2} yields
\[
\int_X |D^+f|_{*,q}^q \dm \leq \liminf_{i\to \infty} \int_X |D^+ [\phi_i(f)]|^q_{*,q} \,{\dm}.
\]
Combining this with \eqref{compwekmingrades}, $\phi'_i =0$ on $N$ and $0 \le \phi'_i \leq 1$, we obtain
\[
\int_X |D^+f|_{*,q}^q\dm
 \leq \liminf_{i \to \infty} \int_X \phi'_i(f)^q |D^+f|^q_{*,q} \dm
 \leq \int_{X\setminus f^{-1}(N)} |D^+f|^q_{*,q} \,{\dm}.
\]
Therefore, $|D^+f|_{*,q}=0$ $\m$-a.e.\ in $f^{-1}(N)$.

If $N$ is not compact, then we take a finite measure $\widetilde{\m}=\vartheta\m$
as in Lemma \ref{fintieboumeasure} and consider the push-forward measure $\mu:=f_\sharp \widetilde{\m}$.
Then, there is an increasing sequence $(K_i)_{i \ge 1}$ of compact subsets of $N$
with $\mu(K_i) \uparrow \mu(N) <\infty$, thereby $\mu(N \setminus \bigcup_i K_i)=0$,
equivalently, $\widetilde{\m}(f^{-1}(N \setminus \bigcup_i K_i))=0$.
Since $|D^+f|_{*,q}=0$ $\m$-a.e.\ in $\bigcup_i f^{-1}(K_i)$
and $\widetilde{\m}$ and $\m$ share the same negligible sets,
$|D^+f|_{*,q}=0$ holds $\m$-a.e.\ in $f^{-1}(N)$.

\eqref{Chainrule-2}
We deduce from \eqref{triangeinequaofdf} and \eqref{Chainrule-1} that
\[
|D^+f|_{*,q} \leq |D^+(f-g)|_{*,q} +|D^+g|_{*,q} =|D^+g|_{*,q} \quad \text{$\m$-a.e.\ in $(f-g)^{-1}(c)$}.
\]
Exchanging $f$ and $g$ yields the claim.


\eqref{Chainrule-3}
Observe that, by approximating $\phi$ with a sequence $(\phi_i)_{i \ge 1}$
of non-decreasing, equi-Lipschitz $C^1$-functions such that $\phi'_i \to \phi'$ a.e.\ on the image of $f$,
the inequality \eqref{compwekmingrades} holds.
Then, assuming $0 \leq \phi' \leq 1$ without loss of generality,
we deduce from \eqref{triangeinequaofdf} and \eqref{compwekmingrades} that
\[
|D^+f|_{*,q} \leq \bigl| D^+ \bigl( f-\phi(f) \bigr) \bigr|_{*,q} +|D^+ [\phi(f)]|_{*,q}
 \leq \bigl( 1-\phi'(f)+\phi'(f) \big) |D^+f|_{*,q} =|D^+f|_{*,q}.
\]
This yields $|D^+ [\phi(f)]|_{*,q}=\phi'(f)|D^+f|_{*,q}$.

\eqref{Chainrule-4}
Applying Lemma \ref{closurestrcrexlgra}\eqref{closuer3}, we find sequences $(f_i)_{i \ge 1}$, $(g_i)_{i \ge 1}$
of bounded Borel forward Lipschitz functions satisfying
$f_i \to f$, $|D^+ f_i| \to |D^+f|_{*,q}$, $g_i \to g$ and $|D^+ g_i| \to |D^+g|_{*,q}$ in $L^q(X,\m)$.
When $\phi$ is $C^1$ (with $0 \leq \phi' \leq 1$),
we infer from $\phi' \le 1$ that $f_i+\phi(g_i-f_i)$ and $g_i-\phi(g_i-f_i)$ are both forward Lipschitz.
Thus,
 by Lemma \ref{lipconvexfundd} with $\psi(x)=x^q$,
we have
\[
\bigl| D^+ \bigl( f_i +\phi(g_i -f_i) \bigr) \bigr|^q +\bigl| D^+ \bigl( g_i -\phi\bigl( g_i -f_i) \bigr) \bigr|^q
 \leq |D^+f_i|^q +|D^+g_i|^q,
\]
and letting $i \to \infty$ yields \eqref{2conveindd}.

In general, we approximate $\phi$ by a sequence $(\phi_i)_{i \ge 1}$
of non-decreasing $C^1$-contractions converging to $\phi$ pointwise (with $\phi_i(0)=0$ if $\m[X]=\infty$).
Then, on the one hand, it follows from the dominated convergence theorem that
\[
f+\phi_i(g-f) \to f+\phi(g-f), \qquad g-\phi_i(g-f) \to g-\phi(g-f) \quad \text{in $L^q(X,\m)$}.
\]
On the other hand, for each $i$, we have
\[
\bigl| D^+ \bigl( f +\phi_i (g-f) \bigr) \bigr|^q_{*,q} +\bigl| D^+ \bigl( g-\phi_i(g-f) \bigr) \bigr|_{*,q}^q
 \leq |D^+f|^q_{*,q} +|D^+g|^q_{*,q} \quad \text{$\m$-a.e.\ in $X$}.
\]
Combining these and using Lemma \ref{closurestrcrexlgra}\eqref{closuer2}, we obtain  \eqref{2conveindd}.
\end{proof}

%
%

We next see that relaxed ascending $q$-gradients are invariant under multiplicative deformations of $\m$
(cf.\ \cite[Lemma 4.11]{AGS2}).

\begin{lemma}\label{diffgradinvar1}
Let $\m'=\vartheta\m$ be a $\sigma$-finite Borel measure such that,
for every compact set $K \subset (X,\tau)$, there are positive constants $r,c,C$ such that
\begin{equation}\label{basiccontheta1}
0<c \leq \vartheta \leq C <\infty \qquad \text{$\m$-a.e.\ in } \overline{B^+_K(r)}^d.
\end{equation}
\begin{enumerate}[{\rm (i)}]
\item\label{multi-1}
For any $f\in L^q(X,\m) \cap L^q(X,\m')$ admitting relaxed ascending $q$-gradients for both $\m$ and $\m'$,
the minimal relaxed ascending $q$-gradient $|D^+f|'_{*,q}$ with respect to $\m'$ coincides with $|D^+f|_{*,q}$ $\m$-a.e.

\item\label{multi-2}
If the reversibility $\lambda_d(X)$ is finite
and $r>0$ in \eqref{basiccontheta1} can be taken uniformly in $K$, then,
for every $f\in L^q(X,\m) \cap L^q(X,\m')$ with $|D^+f|_{*,q} \in L^q(X,\m) \cap L^q(X,\m')$,
$|D^+f|'_{*,q}$ exists in $L^q(X,\m')$.
\end{enumerate}
\end{lemma}

\begin{proof}
\eqref{multi-1}
Since $\m$ and $\m'$ can be exchanged, it is sufficient to show $|D^+f|_{*,q} \geq |D^+f|'_{*,q}$.
Suppose to the contrary that $|D^+f|_{*,q}<|D^+f|_{*,q}'$ holds in a Borel set $B$ with $0<\m[B]<\infty$.
Take a compact set $K \subset B$ with $\m[K]>0$ and $r>0$ such that \eqref{basiccontheta1} holds,
and consider $\chi_r(x)=\phi_r(d(K,x))$ as in Lemma \ref{ref'nproofandother}.
Then, we have $\chi_r=1$ in $B^+_K(r/3)$ and $\chi_r=0$ in $X \setminus B^+_K((2r)/3)$.

Lemma \ref{closurestrcrexlgra}\eqref{closuer3} furnishes a sequence $(f_i)_{i \ge 1}$
of bounded Borel forward Lipschitz functions with $f_i \to f$ and $|D^+f_i| \to |D^+f|_{*,q}$ in $L^q(X,\m)$.
Then, by Lemma \ref{ref'nproofandother}\eqref{producetchi},
$f'_i :=\chi_r f_i$ is bounded Borel forward Lipschitz and satisfies
\[
|D^+ f'_i| \le \chi_r |D^+ f_i| +|f_i| \SL(\phi_r) \Theta_K(r).
\]
Note that $f'_i$ converges to $f':=\chi_r f$ in $L^q(X,\m')$
by \eqref{basiccontheta1} and $\chi_r=0$ in $X \setminus B^+_K((2r)/3)$,
and that $|D^+ f'_i| =0$ in $X \setminus B^+_K(r)$.
Since $|D^+ f'_i|$ is uniformly bounded in $L^q(X,\m')$ again by \eqref{basiccontheta1},
we have a subsequence weakly convergent to some $G' \ge |D^+ f'|'_{*,q}$.
On the other hand, by the choice of $f_i$, $G'=|D^+ f|_{*,q}$ holds $\m$-a.e.\ in $K$.
Thus, $|D^+ f|_{*,q} \ge |D^+ f'|'_{*,q}$ $\m$-a.e.\ in $K$,
which together with Proposition \ref{chainrulerelax}\eqref{Chainrule-2} implies
$|D^+ f|_{*,q} \ge |D^+ f|'_{*,q}$ $\m$-a.e.\ in $K$, a contradiction.


\eqref{multi-2}
Let $(K_k)_{k \ge 1}$ be a sequence of compact sets
such that $\mathbbm{1}_{K_k} \uparrow 1$ as $k \to \infty$ $\m$-a.e.\ in $X$.
Set $\chi_k(x):=\phi_r(d(K_k,x))$ and note that $\chi_k f \to f$ in $L^q(X,\m')$.
By Lemma \ref{closurestrcrexlgra}\eqref{closuer3}, we find a sequence $(f_i)_{i \ge 1}$ with
$f_i \to f$ and $|D^+f_i| \to |D^+f|_{*,q}$ in $L^q(X,\m)$.
Then, for fixed $k$, it follows from \eqref{basiccontheta1}, \eqref{df'nlipscontroll} and $\lambda_d(X)<\infty$ that
$\chi_k f_i \to \chi_k f$ as $i \to \infty$ in both $L^q(X,\m)$ and $L^q(X,\m')$, and
\begin{align*}
|D^+ (\chi_k f_i)| &\leq \chi_k |D^+f_i| +|f_i| \SL(\phi_r) \lambda_d(X) \mathbbm{1}_{B^+_{K_k}(r)} \\
&\to \chi_k |D^+f|_{*,q} +|f| \SL(\phi_r) \lambda_d(X) \mathbbm{1}_{B^+_{K_k}(r)}
 \quad \text{(as $i \to \infty$ in $L^q(X,\m)$ and $L^q(X,\m')$)} \\
&\leq |D^+f|_{*,q} +\frac{3}{r} \lambda_d(X) |f|.
\end{align*}
This implies
\[
\max\left\{|D^+ (\chi_k f)|_{*,q},\, |D^+ (\chi_k f)|'_{*,q} \right\}
 \leq |D^+f|_{*,q} +\frac{3}{r} \lambda_d(X) |f| \in L^p(X,\m)\cap L^p(X,\m').
\]
Thus, we infer from  \eqref{multi-1} that $|D^+(\chi_k f)|_{*,q} =|D^+(\chi_k f)|'_{*,q}$,
which is uniformly bounded in $L^q(X,\m')$ and, up to a subsequence,
converges weakly to some $G$ as $k \to \infty$ in $L^q(X,\m')$.
Hence, we deduce from Lemma \ref{closurestrcrexlgra}\eqref{closuer2} that
$|D^+f|'_{*,q}$ exists in $L^q(X,\m')$.
\end{proof}


\subsection{Cheeger energy of relaxed gradients}

We define an energy functional in terms of relaxed gradients
(compare this with $\Chc^+_{w,q}(f)$ in Definition \ref{weakcheegerenergy}).

\begin{definition}[$q$-Cheeger energy]\label{df:Cheeger}
The {\it forward $q$-Cheeger energy} of $f \in L^q(X,\m)$ is defined as
\[
\Chc^+_q (f) :=\frac1q \int_X |D^+ f|^q_{*,q} \dm
\]
if $f$ has a relaxed ascending $q$-gradient, and $\Chc^+_q (f) :=\infty$ otherwise.
The {\it backward $q$-Cheeger energy} $\Chc^-_q$ and the \emph{$q$-Cheeger energy} $\Chc_q$
are defined in the same way with $|D^- f|_{*,q}$ and $|Df|_{*,q}$, respectively.
\end{definition}

We refer to \cite[\S 8.5]{AGS4} and the references therein for the case of $q=1$.
We observe fundamental properties of $\Chc^+_q$ (cf.\ \cite[Theorem 4.5]{AGS2}).

\begin{theorem}\label{chlowersemicont}
Let $(X,\tau,d,\m)$ be a $\APE$.
Then we have the following.
\begin{enumerate}[{\rm (i)}]
\item\label{cheeg-1}
For any $f\in L^q(X,\m)$ and $\lambda>0$,
we have $\Chc^+_q(f)=\Chc^-_q(-f)$ and $\Chc^+_q(\lambda f) =\lambda^q \Chc^+_q(f)$.

\item\label{cheeg-2}
$\Chc^+_q$ is convex and lower semi-continuous with respect to the weak topology of $L^q(X,\m)$.

\item\label{cheeg-3}
If \eqref{Kcondition1} holds, then the domain $\mathfrak{D}(\Chc^+_q)$ of $\Chc^+_q$ is dense in $L^q(X,\m)$.
\end{enumerate}
\end{theorem}

\begin{proof}
\eqref{cheeg-1}
This is straightforward by definition.

\eqref{cheeg-2}
For any $f,g \in \mathfrak{D}(\Chc^+_q)$ and $\alpha,\beta \ge 0$,
Lemma~\ref{closurestrcrexlgra}\eqref{closuer3} together with \eqref{postivegradnormineq} yields
\begin{equation*}
|D^+(\alpha f+\beta g)|_{*,q}\leq \alpha|D^+f|_{*,q}+\beta |D^+g|_{*,q}.
\end{equation*}
Hence, for any $\lambda\in [0,1]$, we have
\begin{align*}
\Chc^+_q \bigl( \lambda f+(1-\lambda)g \bigr)
&\leq \frac1q \int_X \bigl( \lambda |D^+ f|_{*,q} +(1-\lambda)|D^+g|_{*,q} \bigr)^q \dm \\
&\leq \frac1q \int_X \bigl( \lambda |D^+ f|_{*,q}^q +(1-\lambda)|D^+g|_{*,q}^q \bigr) \dm \\
&= \lambda\Chc^+_q(f)+(1-\lambda)\Chc^+_q(g).
\end{align*}
Thus, $\Chc^+_q$ is convex.
To show the lower semi-continuity,
let $(f_i)_{i \ge 1}$ be a sequence weakly convergent to $f \in L^q(X,\m)$.
Without loss of generality, suppose that $\lim_{i \to \infty} \Chc^+_q(f_i)$ exists and is finite.
Then, $(|D^+f_i|_{*,q})_i$ is uniformly bounded in $L^q(X,\m)$,
and hence we can further assume that it weakly converges to some $G\in L^q(X,\m)$.
Lemma \ref{closurestrcrexlgra}\eqref{closuer2} shows that $G$ is a relaxed ascending $q$-gradient of $f$,
thereby $\Chc^+_q(f) \le (1/q) \|G\|^q_{L^q}$.
This completes the proof.

\eqref{cheeg-3}
This follows from Proposition \ref{densfolLIPS} with the help of \eqref{lisprelgraest}.
\end{proof}

In the remainder of this subsection,
let $(X,\tau,d,\m)$ be a $\APE$ satisfying Assumption \ref{compactreversbiliey}.
Then, along \cite[(4.17)]{AGS2} and \cite[\S 6]{AGS4},
we can extend the domain of $\Chc^+_q$ as follows.

\begin{definition}[Extension of $q$-Cheeger energy]\label{generalizedrelaxgrc}
For an $\m$-measurable function $f:X\rightarrow \mathbb{R}$ such that its truncations
$f_N:=\min\{\max\{f,-N\},N\}$ are in $\mathfrak{D}(\Chc^+_q) \subset L^q(X,\m)$ for all $N\in \mathbb{N}$,
we set
\begin{equation}\label{extendgradre}
|D^+ f|_{*,q} :=|D^+ f_N|_{*,q} \quad\ \text{on } A_N:=\{x\in X \mid |f(x)|<N \}.
\end{equation}
Then, we define
\[
\Chc^+_q(f) := \frac1q \int_X |D^+ f|^q_{*,q} \dm
\]
if $f_N \in \mathfrak{D}(\Chc^+_q)$ for all $N \geq 1$, and $\Chc^+_q(f):=\infty$ otherwise.
We extend the domains of $\Chc^-_q$ and $\Chc_q$ in the same way.
\end{definition}

Note that \eqref{extendgradre} is consistent by virtue of Proposition \ref{chainrulerelax}\eqref{Chainrule-2},
and this definition coincides with Definition \ref{df:Cheeger} for $f \in L^q(X,\m)$.
Moreover, the convexity of $\Chc^+_q$ follows from Theorem \ref{chlowersemicont}\eqref{cheeg-2}
by noticing
\[
\bigl| [(1-\lambda)f +\lambda g]_N \bigr| \le \bigl[ (1-\lambda)|f| +\lambda |g| \bigr]_N
\le \max\{ |f|_N,|g|_N \}.
\]

\begin{lemma}\label{generaCheeprop}
If $\m[X] <\infty$, then $\Chc^+_q$ is sequentially lower semi-continuous
with respect to the $\m$-a.e.\ pointwise convergence.
\end{lemma}

\begin{proof}
Let $(f_i)_{i \ge 1}$ be a sequence of $\m$-measurable functions $\m$-a.e.\ convergent to $f$
and, without loss of generality, suppose that $\liminf_{i \to \infty}\Chc^+_q(f_i)<\infty$.
Given $N \geq 1$, it follows from $\m[X]<\infty$ and the bounded convergence theorem that
$[f_i]_N \to f_N$ as $i \to \infty$ in $L^q(X,\m)$.
Moreover, by Lemma \ref{closurestrcrexlgra}\eqref{closuer2},
every weak limit $G$ of a subsequence of $(|D^+ ([f_i]_N)|_{*,q})_{i \ge 1}$ in $L^q(X,\m)$
is a relaxed ascending $q$-gradient of $f_N$.
Hence, we obtain
\[
\int_X |D^+ f_N|_{*,q}^q \dm \leq \int_X  G^q \dm
 \leq \liminf_{i \to \infty} \int_X |D^+ ([f_i]_N)|^q_{*,q} \dm
 \leq q \cdot \liminf_{i \to \infty} \Chc^+_q(f_i).
\]
We conclude the proof by letting $N \to \infty$.
\end{proof}

The following simple example points out two facts:
\begin{itemize}
\item the condition $\m[X] <\infty$ in Lemma \ref{generaCheeprop} is necessary
if only Assumption \ref{compactreversbiliey} holds;

\item the weak forward $q$-Cheeger energy $\Chc^+_{w,q}$ (recall Definition \ref{weakcheegerenergy})
is different from the forward $q$-Cheeger energy $\Chc^+_q$
even if both Assumptions \ref{newassumpt32liambdafinite} and  \ref{compactreversbiliey} hold.
\end{itemize}

\begin{example}\label{constantfunctonwithoutgradient}
Let $(X,d,\m)$ be the real line $\mathbb{R}$ endowed with $d(x,y)=|x-y|$ and $\m=\mathscr{L}^1$.
By setting $V(x):=\sqrt{\pi}|x|$,
Assumptions \ref{newassumpt32liambdafinite} and \ref{compactreversbiliey} clearly hold.
For $i \in \mathbb{N}$, consider a $1$-Lipschitz function
\[
f_i(x) :=\min\bigl\{ \max\{ i+1-|x|,0 \},1 \bigr\},\ x \in \mathbb{R}.
\]
Observe that $f_i \in L^q(X,\m)$, $\Chc^+_q(f_i)=2/q$, and that $f_i \to 1$ pointwise.
However, since $1 \not\in L^q(X,\m)$ (i.e., $\m[X]=\infty$), we have $\Chc^+_q(1)=\infty$,
which implies that $\Chc^+_q$ is not lower semi-continuous.
Moreover, we readily see that $|D^+1|_{w,\I_p}=0$ and $\Chc^+_{w,q}(1)=0$.
\end{example}

\subsection{$L^2$-gradient flow of Cheeger energy}\label{strongcheegerflow}

In this subsection, we introduce a stronger assumption to study
the $L^2$-gradient flow of $\Chc^+_q$ and the corresponding Laplacian.

\begin{assumption}\label{newassumpt32liambdafinite22}
Let $(X,\tau,d,\m)$ be a $\APE$ satisfying one of the following two conditions:
\begin{itemize}
\item $\lambda_d(X)<\infty$ and Assumption~\ref{strongerstassumptiontheta};
\item $\m[X]<\infty$, $(X,d)$ is a geodesic space,
and every compact set $K\subset (X,\tau)$ has finite diameter.
\end{itemize}
\end{assumption}

\begin{remark}\label{Ass4containAss2}
With the help of Remark~\ref{meaningassmupt2},
one can see that Assumption \ref{newassumpt32liambdafinite22} is stronger than
both Assumptions \ref{newassumpt32liambdafinite} and \ref{compactreversbiliey}.
As we saw in Example~\ref{constantfunctonwithoutgradient},
it is more difficult to analyze $\Chc^+_q$ than $\Chc^+_{w,q}$,
thus it is natural to study properties of $\Chc^+_q$ under a stronger condition.
Note also that, in the symmetric case, Assumption~\ref{newassumpt32liambdafinite22}
follows from Assumption~\ref{strongerstassumptiontheta}.
\end{remark}

Throughout this subsection, let $(X,\tau,d,\m)$ be a $\APE$ with Assumption \ref{newassumpt32liambdafinite22}.
In the same spirit as Section~\ref{L^2cheegerengerystrong},
we shall consider properties of the restriction of $\Chc^+_q$ to $L^2(X,\m)$.
By Proposition \ref{densfolLIPS} and \eqref{lisprelgraest} (see also Remark \ref{meaningassmupt2}),
we obtain a counterpart to Lemma~\ref{lm:Cheeger}.

\begin{lemma}\label{chcisdesninL222}
$\mathfrak{D}(\Chc^+_q) \cap L^2(X,\m)$ is dense in $L^2(X,\m)$.
\end{lemma}

The following lemma can be compared with Example~\ref{constantfunctonwithoutgradient}
(cf.\ \cite[Theorem~6.1]{AGS4} under $\m[X]<\infty$).

\begin{lemma}\label{constlowerseimcheegenerge}
For $q \in [2,\infty)$,
$\Chc^+_q$ restricted to $L^2(X,\m)$ is sequentially lower semi-continuous
with respect to the $\m$-a.e.\ pointwise convergence.
\end{lemma}

\begin{proof}
Since the case of $\m[X]<\infty$ is reduced to Lemma \ref{generaCheeprop}, we assume $\m[X]=\infty$;
then we have Assumption \ref{strongerstassumptiontheta} and $\lambda_d(X)<\infty$.
Let $V$ be as in Assumption \ref{strongerstassumptiontheta} and put $\m_0:={\ee}^{-V^2}\m$.
For each $k\in \mathbb{N}$, define
\[
V_k:=\min\{V,k\}, \qquad \m_k:= {\ee}^{V_k^2}\m_0,
\]
and denote by $\Chc^+_{q,\m_k}$ the $q$-Cheeger energy for $\m_k$.
Note that $\m_k\leq \m_{k+1}$ with $\lim_{k\to \infty}\m_k[B]=\m[B]$ for any $B\in \mathcal{B}(X)$.
For $f\in \mathfrak{D}(\Chc^+_q)$ and $N \in \mathbb{N}$,
recall from Definition~\ref{generalizedrelaxgrc} that
the minimal relaxed ascending $q$-gradient $|D^+f_N|_{*,q}$ exists.
Since $\m\geq \m_k$, the minimal relaxed ascending $q$-gradient of $f_N$ {\color{blue}for $\m_k$},
denoted by $|D^+f_N|_{*,q,k}$, also exists and satisfies
$|D^+f_N|_{*,q,k}\leq |D^+f_N|_{*,q}\in L^q(X,\m)\cap L^q(X,\m_k)$.
Thus, Lemma \ref{diffgradinvar1} furnishes $|D^+f_N|_{*,q,k}= |D^+f_N|_{*,q}$ $\m$-a.e.
Letting $N\to \infty$ and using \eqref{extendgradre}, we have
 \begin{equation}\label{gradequformeasurfunciton}
|D^+f|_{*,q,k}= |D^+f|_{*,q} \quad \text{$\m$-a.e.}
\end{equation}

Now, given $f\in L^2(X,\m)$,
let $(f_i)_{i \ge1}$ be a sequence in $L^2(X,\m)$ converging to $f$ $\m$-a.e.\ pointwise.
Without loss of generality, suppose that $\liminf_{i \to \infty}\Chc^+_q(f_i)<\infty$
and $\Chc^+_q(f_i)<\infty$ for all $i$.
For any $N,k \in \mathbb{N}$, since $\m_k[X]<\infty$,
$[f_i]_N=\min\{ \max\{f_i,-N\},N \}$ converges to $f_N$ as $i \to \infty$
$\m$-a.e.\ pointwise and strongly in $L^q(X,\m_k)$.
Note that
\[
|D^+([f_i]_N)|_{*,q,k} =|D^+([f_i]_N)|_{*,q} \quad \text{$\m$-a.e.}
\]
by $\Chc^+_q(f_i)<\infty$ and \eqref{gradequformeasurfunciton}.
Since $\m_k[X]<\infty$, Lemma \ref{generaCheeprop} yields
\begin{align*}
\liminf_{i \to \infty}\Chc^+_q(f_i)
&= \liminf_{i \to \infty} \frac{1}{q} \int_X |D^+f_i|^q_{*,q}\dm
 \geq \liminf_{i \to \infty} \frac{1}{q} \int_X |D^+([f_i]_N)|^q_{*,q} \dm_k \\
&= \liminf_{i \to \infty} \Chc^+_{q,\m_k}([f_i]_N)
 \geq \Chc^+_{q,\m_k}(f_N).
\end{align*}
Hence, $f_N \in \mathfrak{D}(\Chc^+_{q,\m_k})$ for all $k$,
and $g_N:=|D^+f_N|_{*,q,k}$ is independent of $k$ thanks to \eqref{gradequformeasurfunciton}.
Combining this with the monotone convergence theorem implies
\begin{equation}\label{limtNch}
\frac1q \int_X g_N^q\dm
 =\lim_{k\to \infty} \frac1q \int_X g_N^q\dm_k
 =\lim_{k\to \infty} \Chc^+_{q,\m_k}(f_N) \leq \liminf_{i\to \infty}\Chc^+_q(f_i) <\infty,
\end{equation}
thereby $g_N \in L^q(X,\m)$.

Since $f_N \in L^q(X,\m)$ (by $f \in L^2(X,\m)$ and $|f_N|^q \le |f|^2 N^{q-2}$),
it then follows from Remark \ref{meaningassmupt2} and Lemma~\ref{diffgradinvar1}
(with $\m'=\m_k$) that $g_N=|D^+f_{N}|_{*,q,k}=|D^+f_N|_{*,q}$,
which combined with \eqref{limtNch} furnishes
\[
\liminf_{i\to \infty}\Chc^+_q(f_i) \geq \frac1q \int_X g_N^q \dm =\Chc^+_q(f_N).
\]
Letting $N \to \infty$, we obtain $\liminf_{i \to \infty}\Chc^+_q(f_i) \geq \Chc^+_q(f)$.
\end{proof}

We remark that the condition $q \ge 2$ is necessary in the lemma above.

\begin{example}\label{rm:q<2}
Let $(X,d)$ be the real line endowed with $\m=\mathscr{L}^1$.
For $q \in (1,2)$, consider the function
\[ f(x):=\begin{cases}
 1 & \text{for}\ x \in [-1,1], \\
 |x|^{-1/q} & \text{otherwise}.
 \end{cases} \]
Then, $f \in L^2(X,\m)$ but $f \not\in L^q(X,\m)$.
Since $f' \in L^q(X,\m)$, by a standard cut-off argument,
we find a sequence $(f_i)_{i \ge 1}$ converging to $f$ in $L^2(X,\m)$
and $\limsup_{i \to \infty} \Chc^+_q(f_i)<\infty$;
however, $f \not\in L^q(X,\m)$ and $|f| \le 1$ imply $\Chc^+_q(f)=\infty$.
\end{example}

Owing to Lemmas~\ref{chcisdesninL222}, \ref{constlowerseimcheegenerge},
we can consider the $L^2$-gradient flow of $\Chc^+_q$ defined in the same way as Definition~\ref{df:GF},
and follow the argument in Subsection~\ref{L^2cheegerengerystrong}
to generalize Theorems~\ref{existstheoremft}, \ref{realchgraevslop} as follows.

\begin{definition}
Given $f_0\in L^2(X,\m)$, a curve $(f_t) \in \AC^2_{\loc}((0,\infty);L^2(X,\m))$
is called a \emph{gradient curve} for $\Chc^+_q$ if $f_t \in \mathfrak{D}(\Chc^+_q)$ for $t>0$,
$f_t\to f_0$ in $L^2(X,\m)$ as $t\to 0$ and
\[
\frac{\dd}{{\dd}t}f_t  \in -\partial^- \Chc^+_q(f_t) \quad \text{for $\mathscr{L}^1$-a.e.}\ t \in (0,\infty).
\]
\end{definition}

\begin{theorem}\label{chgradinetflwo}
Assume $q \ge 2$ or $\m[X]<\infty$.
\begin{enumerate}[{\rm (i)}]
\item\label{chcproflow1}
For any $f_0\in L^2(X,\m)$,
there exists a unique curve $(f_t)_{t \ge 0}$ of $2$-maximal slope for $\Chc^+_q$
with respect to $|D^-\Chc^+_q|$ satisfying the same properties as in Theorem~$\ref{existstheoremft}$.

\item\label{chcproflow2}
For any $f_0\in L^2(X,\m)$,
a curve $(f_t) \in \AC^2_{\loc}((0,\infty);L^2(X,\m))$ with $f_t \to f_0$ in $L^2(X,\m)$ as $t \to 0$
is a curve of $2$-maximal slope for $\Chc^+_q$
with respect to $|D^-\Chc^+_q|$ if and only if it is a gradient curve for $\Chc^+_q$.
In particular, for each $t>0$, we have $f_t\in \mathfrak{D}(|D^- \Chc^+_q|)$
and $-{\dd}f_t/{\dd}t^+$ is the unique element in $\partial^-\Chc^+_q(f_t)$
with the minimal $L^2(X,\m)$-norm.
\end{enumerate}
\end{theorem}

In view of Theorem~\ref{chgradinetflwo}\eqref{chcproflow2},
one can also introduce the $q$-Laplacian associated with $\Chc^+_q$ (as in Definition~\ref{lapldef}),
and the argument in Subsection~\ref{oldLaplcian} can be generalized.
In particular, we have the following along the lines of Theorem~\ref{compprinconvr1}, Lemma~\ref{wssvolemm2} and Proposition~\ref{keyabosplem}.

\begin{theorem}\label{compprinconvr12}
Assume $q \ge 2$ or $\m[X]<\infty$,
and let $(f_t)_{t \ge 0},(h_t)_{t \ge 0}$ be gradient curves for $\Chc^+_q$
starting from $f_0,h_0\in L^2(X,\m)$ with $-h_t\in \mathfrak{D}(\Chc^+_q)$.
Let ${e}:\mathbb{R}\rightarrow [0,\infty]$ be a convex, lower semi-continuous function
and set $E(f):=\int_X {e}(f) \dm$.
\begin{enumerate}[{\rm (i)}]
\item\label{specpro-12}
If $f_0\leq C$ $($resp., $f_0\geq C)$ for some $C \in \mathbb{R}$,
then $f_t\leq C$ $($resp., $f_t\geq C)$ for all $t\geq 0$.
Similarly, if $f_0\leq h_0+C$ for some $C\in \mathbb{R}$, then $f_t\leq h_t+C$ for all $t \ge 0$.

\item\label{specpro-22}
We have
\[ E(f_t)\leq E(f_0), \quad E(f_t-h_t)\leq E(f_0-h_0) \quad \text{ for all }t\geq 0. \]
In particular, if $f_0, h_0 \in L^\theta(X,\m)$ for some $\theta \in [1,\infty]$,
then $f_t, h_t \in L^\theta(X,\m)$ and
\[ \| f_t -h_t \|_{L^\theta}\leq \|f_0-h_0\|_{L^\theta} \quad \text{ for all }t\geq 0. \]

\item\label{specpro-32}
If ${e}'$ is locally Lipschitz and $E(f_0)<\infty$, then we have
\[ E(f_t) +\int^t_0 \int_X {e}''(f_s) |D^+f_s|^q_{*,q} \dm {\dd}s
 =E(f_0) \quad \text{ for all } t\geq 0. \]

\item\label{specpro-42}
When $\m[X]<\infty$, we also have
\[
\int_X f_t \dm=\int_X f_0\dm \quad \text{ for all } t\geq 0.
\]
\end{enumerate}

\end{theorem}



\begin{lemma}\label{wssvolemm}
Assume $q \ge 2$ or $\m[X]<\infty$,
and let $(f_t)_{t \ge 0}$ be a gradient curve for $\Chc^+_q$ in $L^2(X,\m)$ such that
\[
f_0 \ge 0, \qquad
\int_X f_t \dm =1, \qquad
 \int^t_0 \int_{\{f_s>0 \}} \frac{|D^+f_s|^q_{*,q}}{f_s^{p-1}} \dm {\dd}s <\infty \quad \text{for all}\ t\geq 0.
\]
Then, the curve $\mu_t:=f_t\m$ is in $\BAC^p_{\loc}((0,\infty);(\Po(X),W_p))$
and $|\mu'_-|$ satisfies
\[
|\mu'_-|^p(t) \leq \int_{\{f_t>0\}}\frac{|D^+f_t|^q_{*,q}}{f_t^{p-1}} \dm \quad
 \text{for $\mathscr{L}^1$-a.e.\ $t\in (0,\infty)$}.
\]
\end{lemma}

\begin{proposition}\label{keyabosplem2} For $p\in (1,\infty)$, let $U_p$ be defined as in  \eqref{functinalfun}.
Let $f_t$, $t\geq 0$, be a gradient curve for $\Chc^+_{q}$ with $f_0 \ge 0$, $f_0 \in L^2(X,\m)$.
 When $p \in [3,\infty)$, we additionally assume $f_0 \ge c$ for some $c>0$.
If $\m[X]<\infty$, then  the function $t \mapsto  E(t):=\int_X U_p(f_t)\dm$ is absolutely continuous and satisfies
\begin{align*}
&E(0) = E(t) +\int^t_0 \int_{X}   \frac{|D^+f_s|_{*,q}^q}{f^{p-1}_s} \dm{\dd}s
 \quad\ \text{for all}\ t\geq 0,\\
&\frac{{\dd}}{{\dd}t}E(t)
 =-\int_{X}\frac{|D^+f_s|_{*,q}^q}{f^{p-1}_s} \dm \quad\ \text{for $\mathscr{L}^1$-a.e.}\ t\geq 0. 
 \end{align*}
\end{proposition}

\subsection{Identification of relaxed and weak upper gradients}\label{indertitiyuppergraident}

In this subsection, along the lines of \cite[\S 6]{AGS2} ($q=2$) and
\cite[\S 7]{AGS4} ($q \in (1,\infty)$, $\m[X]<\infty$),
we investigate the relation between weak upper gradients
and relaxed gradients under Assumption~\ref{newassumpt32liambdafinite22}
(which is stronger than Assumptions~\ref{newassumpt32liambdafinite}, \ref{compactreversbiliey};
recall Remark \ref{Ass4containAss2}).

\begin{proposition}\label{bascieastime}
Let $(X,\tau,d,\m)$ be a $\APE$ satisfying Assumption~$\ref{strongerstassumptiontheta}$ and
${\I}$ be a stretchable collection of $p$-test plans satisfying \eqref{etbxcondition}.
If $f\in \mathfrak{D}(\Chc^+_q)$, then $f$ is Sobolev along ${\I}$-almost every curve
and $|D^+f|_{w,{\I}}\leq |D^+f|_{*,q}$ holds $\m$-a.e.\ in $X$.
\end{proposition}

\begin{proof}
We first assume that $f$ is bounded,
then $f\in L^q(X,\m)$ since $f\in \mathfrak{D}(\Chc^+_q)$ (recall Definition~\ref{generalizedrelaxgrc}).
By Lemma~\ref{closurestrcrexlgra}\eqref{closuer3},
there is a sequence $(f_i)_{i \ge 1}$ of Borel forward Lipschitz functions such that
$f_i \to f$ and $|D^+f_i| \to |D^+f|_{*,q}$ in $L^q(X,\m)$.
Then, $|D^+f_i|$ converges to $|D^+f|_{*,q}$ in $L^q(\{V\leq M\},\m)$ for all $M\geq 0$,
thereby Proposition \ref{stabilityofweakgradient} implies that
$|D^+f|_{*,q}$ is a ${\I}$-weak forward upper gradient of $f$.
Hence, $|D^+f|_{w,\I}\leq |D^+f|_{*,q}$ $\m$-a.e.

In the general case, for $N>0$, let $f_N:=\min\{\max\{f,-N\},N\}$.
Then, the above argument shows
\begin{equation}\label{weakuwgradient}
|D^+f_N|_{w,{\I}}\leq |D^+f_N|_{*,q}\leq |D^+f|_{*,q}\in L^q(X,\m).
\end{equation}
Thus, $(|D^+f_N|_{w,{\I}})_{N>0}$ is a uniformly bounded sequence in $L^q(X,\m)$,
and has a weak subsequential limit $H$.
Since $(f_N)_{N>0}$ converges pointwise to $f$ $\m$-a.e.,
it follows from Proposition \ref{stabilityofweakgradient} that
$H$ is a $\I$-weak forward upper gradient of $f$.
Together with \eqref{weakuwgradient},
we obtain $|D^+f|_{w,{\I}} \leq H \leq |D^+f|_{*,q}$ $\m$-a.e.
\end{proof}

\begin{corollary}\label{finserlnorm}
Let $(\X, d_F,\m)$ be a forward metric measure space induced from a forward complete Finsler manifold
satisfying either $\m[\X]<\infty$, or $\lambda_F(\X)<\infty$ and $\CD(K,\infty)$.
\begin{enumerate}[{\rm (i)}]
\item \label{finslersob1}
For any $f\in \Lip(\X)\cap \mathfrak{D}(\Chc^+_q)$, we have
\begin{equation}\label{equdfgrad}
|D^+f|=|D^+f|_{*,q}=F^*({\dd}f) \quad \text{$\m$-a.e.\ in $\X$}.
\end{equation}

\item \label{finslersob2}
For any $f\in \mathfrak{D}(\Chc^+_q)$,
there exists a sequence $(f_i)_{i \ge 1}$ in $\Lip(\X)\cap \mathfrak{D}(\Chc^+_q)$ such that
$f_i \to f$ $\m$-a.e.\ and $F^*({\dd}f_i) \to |D^+f|_{*,q}$ in $L^q(\X,\m)$.
\end{enumerate}
\end{corollary}

\begin{proof}
Note that Assumption~\ref{newassumpt32liambdafinite22} is satisfied by the hypothesis.

\eqref{finslersob1}
First, suppose that $f$ is bounded (thereby $f\in L^q(\X,\m)$).
Since $|D^+f| =F^*({\dd}f)$ $\m$-a.e.\ in $\X$, considering the trivial sequence $f_i \equiv f$,
we find $|D^+f|_{*,q}\leq  F^*({\dd}f)$.
To see the converse inequality, let $\I_p$ be the collection of all $p$-test plans
with bounded compression on the sublevels of $V$.
We deduce from Corollary~\ref{equivweakupp} (see also Example \ref{finslercaseweakupgra})
and Proposition \ref{bascieastime} that $F^*({\dd}f) =|D^+f|_{w,{\I_p}} \le |D^+f|_{*,q}$.
Hence, \eqref{equdfgrad} holds.

If $f$ is unbounded, then we obtain $|D^+f_N| =|D^+f_N|_{*,q} =F^*({\dd}f_N)$ for every $N>0$,
and letting $N\to \infty$ completes the proof.

\eqref{finslersob2}
Note that $f_N\in L^q(\X,\m)$, $f_N \to f$ $\m$-a.e., and $|D^+f_N|_{*,q} \to |D^+f|_{*,q}$ in $L^q(\X,\m)$.
For each $N>0$, Lemma~\ref{closurestrcrexlgra}\eqref{closuer3} combined with \eqref{finslersob1} above yields
a sequence $u_{N,k}\in \Lip(\X)\cap \mathfrak{D}(\Chc^+_q)$ such that
\[
u_{N,k}\overset{L^q}{\longrightarrow} f_N, \qquad
 F^*({\dd}u_{N,k})=|D^+u_{N,k}|\overset{L^q}{\longrightarrow}|D^+f_N|_{*,q} \quad \text{ as }k\to \infty.
\]
Passing to a subsequence, we may assume that $u_{N,k} \to f_N$ $\m$-a.e.
Then, the diagonal argument implies the claim.
\end{proof}

A converse implication to Proposition \ref{bascieastime} holds true as follows.

\begin{theorem}[Relaxed and weak upper gradients coincide]\label{twogradientcoindced}
Let $(X,\tau,d,\m)$ be a $\APE$ with Assumption~$\ref{newassumpt32liambdafinite22}$ and
$\I_p$ be the collection of all $p$-test plans with bounded compression on the sublevels of $V$.
Assume $q \ge 2$ or $\m[X]<\infty$,
and let $f:X\rightarrow \mathbb{R}$ be an $\m$-measurable function satisfying
$f_N \in L^q(X,\m)$ for every $N>0$.
Then, $f$ has the relaxed ascending $q$-gradient $|D^+f|_{*,q}$ as in \eqref{extendgradre} in $L^q(X,\m)$
if and only if $f$ is Sobolev along $\I_p$-almost every curve and $|D^+f|_{w,\I_p}\in L^q(X,\m)$,
in which case
\[
|D^+f|_{*,q}=|D^+f|_{w,\I_p} \quad \text{$\m$-a.e.\ in $X$}.
\]
\end{theorem}

\begin{proof}
Remark~\ref{inveraweakupper} implies that the minimal $\I_p$-weak forward upper gradient induced by $\m$ coincides with the one induced by $\ee^{-V^2}{\m}$; similarly, Lemma~\ref{diffgradinvar1} and Remark~\ref{meaningassmupt2} imply that the minimal relaxed ascending \(q\)-gradients induced by $\m$ and by $\ee^{-V^2}{\m}$ coincide.
Since $\ee^{-V^2}{\m}$ is a finite measure, for convenience we may assume $\m\in\Po(X)$ and choose $V\equiv 0$.
The rest of the proof is divided into two steps.

\smallskip

{\bf Step 1.} We first show the theorem when $f$ is bounded.
Then it costs no generality to assume $0<A \le f \le B<\infty$ with $\int_X f\dm=1$,
since  $|D^+(\lambda f+c)|_{w,\I_p}=\lambda |D^+ f|_{w,\I_p}$ and $|D^+(\lambda f+c)|_{*,q} =\lambda |D^+ f|_{*,q}$ for any $\lambda >0$ and $c \in \mathbb{R}$
by Lemma~\ref{basicweakgradiprop} and Proposition~\ref{chainrulerelax}.

Assume that $f$ is Sobolev along $\I_p$-almost every curve with $|D^+f|_{w,\I_p}\in L^q(X,\m)$
and consider the gradient curve $(h_t)_{t \ge 0}$ for $\Chc^+_q$ with $h_0=f$.
Since $\m[X]<\infty$, it follows from
Theorem~\ref{compprinconvr12}\eqref{specpro-12}, \eqref{specpro-42}
that $A \leq h_t \leq B$ and $\int_X h_t\dm=1$.
Then  $\int_X U_p(h_t) \dm<\infty$ for $U_p$ defined in \eqref{functinalfun}. Moreover,  Proposition~\ref{keyabosplem2} yields $\int^t_0 \int_{X}   \frac{|D^+h_s|_{*,q}^q}{h^{p-1}_s} \dm{\dd}s<\infty$ for all $t\geq 0$.

Then Lemma~\ref{wssvolemm} implies $\mu_t:=h_t\m\in \BAC^p([0,1];(\mathscr{P}(X),W_p))$.
Moreover, Theorem~\ref{chgradinetflwo} (cf.\ Theorem~\ref{existstheoremft}\eqref{graflowmaxislo1}) and Remark~\ref{totalvariationint} imply
that $\mu_t$ is uniformly continuous with respect to the total variation norm.
Thus, Theorem~\ref{Lisinitheorem} is applicable to the reverse curve $\bar{\mu}_t:=\mu_{1-t}\in \FAC^p([0,1];(\mathscr{P}(X),W_p))$.

By Proposition~\ref{keyabosplem2} again, we obtain
\begin{equation}\label{anestdouin}
\int^t_0 \int_X \frac{|D^+ h_s|^q_{*,q}}{h_s^{p-1}} \dm {\dd}s
=\int_X U_p(h_0) \dm -\int_X U_p(h_t) \,{\dm}.
\end{equation}
Now, set
\[
 g:=  |D^+ (U'_p \circ h_0)|_{w,\I_p}
 =U_p''(h_0)|D^+ h_0|_{w,\I_p} =\frac{|D^+ h_0|_{w,\I_p}}{h_0^{p-1}}.
\]
Then, the same argument as in \eqref{keyconstPhi} (thanks to Theorem~\ref{Lisinitheorem}) and Lemma \ref{wssvolemm} imply
\begin{align*}
\int_X \bigl( U_p(h_0) -U_p(h_t) \bigr) \dm
&\le \int_X U_p'(h_0) \,{\dd}\mu_0 -\int_X U_p'(h_0) \,{\dd}\mu_t
 \leq \int^t_{0} \biggl( \int_X g^q h_s \,{\dm} \biggr)^{1/q} |{\mu}'_-|(s) \,{\dd}s\\
&\leq \biggl( \int^t_0 \int_X g^q \,{\dd}\mu_s \,{\dd}s \biggr)^{1/q} \biggl( \int^t_0 |\mu'_-|^p(s) \,{\dd}s \biggr)^{1/p} \\
&\leq \frac1q \int^t_0 \int_X g^q \,{\dd}\mu_s \,{\dd}s +\frac1p \int^t_0 |\mu'_-|^p(s) \,{\dd}s \\
&\leq \frac1q \int^t_0 \int_X g^q \,{\dd}\mu_s \,{\dd}s
 +\frac1p \int^t_0 \int_X \frac{|D^+h_s|^q_{*,q}}{h_s^{p-1}} \dm {\dd}s.
\end{align*}
Combining this with \eqref{anestdouin}, we find
\[
\int^t_0 \int_X \frac{|D^+h_s|^q_{*,q}}{h_s^{p-1}} \dm {\dd}s
 \leq \int^t_0 \int_X g^q \,{\dd}\mu_s \,{\dd}s.
\]
Thus, we obtain, since $(h_0^{p-1})^q=h_0^p$,
\[
\int_X \frac{|D^+h_0|^q_{*,q}}{h_0^{p-1}} \dm
 \leq \int_X g^q \,{\dd}\mu_0
 =\int_X \frac{|D^+h_0|^q_{w,\I_p}}{h_0^{p-1}} \,{\dm}.
\]
Since $h_0=f \ge A>0$ and $|D^+f|^q_{*,q} \ge |D^+f|^q_{w,\I_p}$ $\m$-a.e.\ by Proposition \ref{bascieastime},
we obtain $|D^+f|_{*,q}=|D^+f|_{w,\I_p}$ $\m$-a.e.

\smallskip

{\bf Step 2.}
Next we consider a general function $f$ which is Sobolev along $\I_p$-almost every curve with $|D^+f|_{w,\I_p}\in L^q(X,\m)$.
For every $N>0$, set $f_N:=\min\{\max\{f,-N\},N\}$. It follows from \eqref{derivuppeest} and Corollary \ref{equivweakupp}   that $f_N$ is Sobolev along $\I_p$-almost every curve with $|D^+f_N|_{w,\I_p}\leq |D^+f|_{w,\I_p}$. 

Then Step 1 yields  $|D^+f_N|_{*,q}=|D^+f_N|_{w,\I_p}\leq |D^+f|_{w,\I_p}$. 
Hence, $(|D^+f_N|_{*,q})_{N>0}$ is a uniformly bounded sequence in $L^q(X,\m)$.
By the monotone property and \eqref{extendgradre}, we see that $(|D^+f_N|_{*,q})_{N>0}$ converges to $|D^+f|_{*,q}$ in $L^q(X,\m)$. Lemma \ref{weakconvergeunform}  then yields $|D^+f|_{*,q}\leq |D^+f|_{w,\I_p}$ $\m$-a.e., which combined with Proposition \ref{bascieastime} concludes the proof.
\end{proof}

\begin{corollary}\label{keyconidnce}
Let $(X,\tau,d,\m)$  be as in Theorem~$\ref{twogradientcoindced}$.
If $f:X\rightarrow \mathbb{R}$ is an $\m$-measurable function such that $|D f|_{*,q}$ exists,
then we have $|D f|_{*,q}=\max\{|D^\pm f|_{*,q}\}$ $\m$-a.e.\ in $X$.
\end{corollary}

\begin{proof}
The existence of $|D^\pm f|_{*,q}$ follows from that of $|Df|_{*,q}$
as in Lemma~\ref{ascenddescednrealx}\eqref{relaxedascgra-2}.
Then, the claim is a consequence of Lemma~\ref{lm:wD^pm} and Theorem \ref{twogradientcoindced}.
\end{proof}

\begin{corollary}
Let $(X,\tau,d,\m)$ be as in Theorem~$\ref{twogradientcoindced}$.
Then, we have the following.
\begin{enumerate}[{\rm (i)}]
\item\label{coninctwograd1}
For any $f\in L^2(X,\m)$, we have $\Chc^+_q(f)=\Chc^+_{w,q}(f)$.
In particular, $\mathfrak{D}(\partial^- \Chc^+_q) =\mathfrak{D}(\partial^- \Chc^+_{w,q})$.

\item\label{coninctwograd2}
For any $f \in \mathfrak{D}(\partial^- \Chc^+_q) =\mathfrak{D}(\partial^- \Chc^+_{w,q})$,
we have $\partial^- \Chc^+_q(f)=\partial^- \Chc^+_{w,q}(f)$.

\item\label{coninctwograd3}
For any $f \in L^2(X,\m)$, curves of $2$-maximal slope $($or gradient curves$)$ from $f$
for $\Chc^+_q$ and $\Chc^+_{w,q}$ coincide.
\end{enumerate}
\end{corollary}

\begin{proof}
\eqref{coninctwograd1}
Note that $f_N\in L^q(X,\m)$ for all $N> 0$ by the hypotheses.
Then Theorem~\ref{twogradientcoindced} shows the claim.

\eqref{coninctwograd2} and \eqref{coninctwograd3} immediately follow from \eqref{coninctwograd1}.
\end{proof}

\section{Sobolev spaces}\label{Sobolev-section}

\subsection{Coincidence  of variant gradients}

Following \cite[\S 4]{AGS4}, we introduce two more gradients.
The first one generalizes Cheeger's definition \cite{Ch}.

\begin{definition}[Relaxed upper gradients]\label{df:relug}
Let  $(X,\tau,d,\m)$ be a $\APE$ and $f\in L^q(X,\m)$.
\begin{enumerate}[(1)]
\item
A function $G\in L^q(X,\m)$ is called a \emph{relaxed forward} (resp., {\it backward}) \emph{$q$-upper gradient} of $f$
if there are a sequence $(f_i)_{i \ge 1}$ in $L^q(X,\m)$
and a sequence of strong upper gradients $g_i$ of $f_i$ (resp., $-f_i$)
such that
\begin{itemize}
\item $f_i \to f$ in $L^q(X,\m)$ and $g_i \rightarrow \widetilde{G}$ weakly in $L^q(X,\m)$,
\item $\widetilde{G} \leq G$ $\m$-a.e.\ in $X$.
\end{itemize}
We call a relaxed forward (resp., backward) $q$-upper gradient with the minimal $L^q(X,\m)$-norm
a \emph{minimal relaxed forward} (resp., \emph{backward}) \emph{$q$-upper gradient} of $f$,
and denote it by $|D^+f|_{C,q}$ (resp., $|D^-f|_{C,q}$).

\item
A function $G\in L^q(X,\m)$ is called a \emph{relaxed $q$-upper gradient} of $f$
if there are a sequence $(f_i)_{i \ge 1}$ in $L^q(X,\m)$ and a sequence of strong upper gradients $g_i$ of both $\pm f_i$
such that
\begin{itemize}
\item $f_i \to f$ in $L^q(X,\m)$ and $g_i \rightarrow \widetilde{G}$ weakly in $L^q(X,\m)$,
\item $\widetilde{G}\leq G$ $\m$-a.e.\ in $X$.
\end{itemize}
We call a relaxed $q$-upper gradient with the minimal $L^q(X,\m)$-norm
a \emph{minimal relaxed $q$-upper gradient} of $f$, and denote it by $|Df|_{C,q}$.
\end{enumerate}
\end{definition}

Recall Definition~\ref{gradefe} for the definition of strong upper gradients,
and note that $g$ is a strong upper gradient of $f$ with respect to $d$ if and only if
it is a strong upper gradient of $-f$ with respect to $\overleftarrow{d}$.

\begin{lemma}\label{ascenddescednrealx2}
\begin{enumerate}[{\rm (i)}]
\item\label{relaxedascgra2-1} $f\in L^q(X,\m)$ has a relaxed forward $q$-upper gradient
if and only if $-f$ has a relaxed backward $q$-upper gradient,
and then $|D^+f|_{C,q}=|D^-(-f)|_{C,q}$.

\item\label{relaxedascgra2-2}
If $f\in L^q(X,\m)$ has a relaxed $q$-upper gradient,
then $\max\{|D^\pm f|_{C,q}\} \le |D f|_{C,q}=|D(-f)|_{C,q}$ holds.

\item\label{relaxedascgra2-3}
If $f$ has a relaxed ascending $q$-gradient $($resp., relaxed $q$-gradient$)$,
then we have $|D^+f|_{C,q}\leq |D^+f|_{*,q} $ $($resp., $|Df|_{C,q}\leq |Df|_{*,q})$.
\end{enumerate}
\end{lemma}

\begin{proof}
\eqref{relaxedascgra2-1} and \eqref{relaxedascgra2-2} are straightforward by definition.

\eqref{relaxedascgra2-3}
By Lemma \ref{closurestrcrexlgra}\eqref{closuer3},
there is a sequence of bounded forward Lipschitz functions $f_i \in L^q(X,\m)$ such that
$f_i \to f$ and $|D^+f_i| \to |D^+f|_{*,q}$ in $L^q(X,\m)$.
Since $|D^+f_i|$ is a strong upper gradient of $f_i$ by Lemma \ref{uppergradofLispcf},
$|D^+f|_{*,q}$ is a relaxed forward $q$-upper gradient of $f$,
thereby $|D^+f|_{C,q}\leq |D^+f|_{*,q}$ holds.
We have $|Df|_{C,q}\leq |Df|_{*,q}$ in the same way.
\end{proof}

Next, to generalize a gradient as in \cite{Sh} (see also \cite{He,HKST} and \cite[\S 4.4]{AGS4}),
we define the \emph{$q$-modulus} $\Mod_q(\Gamma)$ of $\Gamma\subset \AC([0,1];(X,d))$ by
\[
\Mod_q(\Gamma) :=\inf\biggl\{ \int_X \varrho^q \,{\dm} \,\Big|\,
 \text{$\varrho$ is a nonnegative Borel function},\,
 \int_\gamma \varrho\geq 1\ \ \forall\gamma\in \Gamma \biggr\}.
\]
We say that $\Gamma$ is {\it $\Mod_q$-negligible} if $\Mod_q(\Gamma)=0$.
The next lemma can be seen in the same way as the symmetric case
(see \cite[Theorem~3(f)]{Fug}, \cite[Section 5]{HKST}).

\begin{lemma}[Fuglede's lemma]\label{Fuglede}
Let $(g_i)_{i\geq 1}$ be a sequence of Borel functions that converges in $L^q(X,\m)$.
Then, there are a subsequence $(g_{i_k})_{k\geq 1}$ and a $\Mod_q$-negligible set $\Gamma$ satisfying
\[
\lim_{k\to \infty}\int_\gamma |g_{i_k}-g|=0 \quad \text{for all }\gamma\in \AC \bigl( [0,1];(X,d) \bigr) \setminus\Gamma,
\]
where $g$ is any Borel representative of the $L^q$-limit of $(g_i)_{i\geq 1}$.
\end{lemma}

\begin{definition}[Upper gradients]\label{Moduppgra}
A Borel function $g:X\rightarrow [0,\infty]$ is called a \emph{forward} (resp., {\it backward}) \emph{$q$-upper gradient}
of $f$ if there are a function $\tilde{f}$ and a $\Mod_q$-negligible set $\Gamma$
such that $\tilde{f}=f$ $\m$-a.e. in $X$ and
\begin{equation}\label{weakmoddir}
\int_{\partial \gamma} \tilde{f} \leq \int_\gamma g \quad
 \left( \text{resp., } \int_{\partial \gamma} (-\tilde{f}) \leq \int_\gamma g \right) \quad
 \text{for all }\gamma\in \AC \bigl( [0,1];(X,d) \bigr) \setminus\Gamma.
\end{equation}
We call $g$ a \emph{$q$-upper gradient} if it is both forward and backward $q$-upper gradient of $f$.
\end{definition}

\begin{lemma}
The collection of all forward $q$-upper gradients of $f$ is convex and closed in $L^q(X,\m)$.
The same holds also for backward $q$-upper gradients and $q$-upper gradients.
\end{lemma}

\begin{proof}
The convexity is obvious.
To see the closedness, let $(g_i)_{i \ge 1}$ be a sequence of forward $q$-upper gradients of $f$
converging to some $g \in L^q(X,\m)$.
We infer from Lemma \ref{Fuglede} that there exist a subsequence $(g_{i_k})$
and a $\Mod_q$-negligible set $\Gamma$ such that
$\lim_{k\to \infty}\int_\gamma |g_{i_k}-g|=0$ for $\gamma \in \AC([0,1];(X,d))\setminus \Gamma$.
Let $\tilde{f}_{i_k}$ be a representative of $f$ corresponding to $g_{i_k}$ satisfying \eqref{weakmoddir}.
Let us consider
$\tilde{f}(x):=\limsup_{k\rightarrow \infty}\tilde{f}_{i_k}(x)$,
which coincides with $f$ except for an $\m$-negligible set.
Then, for all $\gamma \in \AC([0,1];(X,d)) \setminus \Gamma$,
\[
\int_{\partial \gamma} \tilde{f}
 \leq \limsup_{k \to \infty} \int_{\partial \gamma} \tilde{f}_{i_k}
 \leq \lim_{k\to \infty}\int_\gamma g_{i_k}=\int_\gamma g,
\]
which yields that $g$ is also a forward $q$-upper gradient of $f$.
\end{proof}

We call a forward (resp., backward) $q$-upper gradient of $f$ with the minimal $L^q(X,\m)$-norm
a \emph{minimal forward} (resp., \emph{backward}) \emph{$q$-upper gradient},
and denote it by $|D^+f|_{S,q}$ (resp., $|D^-f|_{S,q}$).
Note that $|D^+ f|_{S,q} =|D^- (-f)|_{S,q}$.
We similarly define $|Df|_{S,q}$, and see that $|Df|_{S,q} =\max\{|D^\pm f|_{S,q}\}$.
Now, we consider the relation of the four gradients.

\begin{lemma}\label{gradeeen1}
Every relaxed forward $q$-upper gradient is a forward $q$-upper gradient.
In particular, for any $f\in L^q(X,\m)$, we have
\[
|D^+f|_{S,q}\leq |D^+f|_{C,q} \quad \text{ $\m$-a.e.\ in $X$}.
\]
\end{lemma}

\begin{proof}
Let $G$ be a relaxed forward $q$-upper gradient of $f$.
By definition, there are a sequence $(f_i)_{i \ge 1}$ with $f_i \to f$ in $L^q(X,\m)$
and a sequence of strong upper gradients $g_i$ of $f_i$ such that
$g_i \to \widetilde{G}$ weakly in $L^q(X,\m)$ and $\widetilde{G} \leq G$ $\m$-a.e.
By Mazur's lemma we can take a convex combination of $g_i$, denoted by $\tilde{g}_i$,
converging to $\widetilde{G}$ in $L^q(X,\m)$.
For the corresponding convex combination $\tilde{f}_i$ of $f_i$,
$\tilde{g}_i$ is a strong upper gradient of $\tilde{f}_i$.
Moreover, since $\tilde{f}_i$ converges to $f$ in $L^q(X,\m)$,
we may assume that $\tilde{f}_i$ pointwise converges to $f$ $\m$-a.e.
Then, by Lemma~\ref{Fuglede},
there exist a subsequence $(\tilde{g}_{i_k})_{k \ge 1}$ and a $\Mod_q$-negligible set $\Gamma$ such that
\[
\int_{\partial \gamma} \tilde{f}
 \leq \lim_{k\to \infty} \int_\gamma \tilde{g}_{i_k} =\int_\gamma \widetilde{G} \leq \int_\gamma G \quad
\text{for all } \gamma\in \AC\bigl( [0,1];(X,d) \bigr) \setminus \Gamma,
\]
where we set $\tilde{f}:=\limsup_{k \rightarrow \infty}\tilde{f}_{i_k}$ similarly to the proof of the previous lemma.
Therefore, $G$ is a forward $q$-upper gradient of $f$.
\end{proof}

\begin{lemma}\label{gradeeen2}
Let $\mathfrak{T}_p$ be the collection of all $p$-test plans with bounded compression
on the sublevels of $V$.
Then, any forward $q$-upper gradient is a $\mathfrak{T}_p$-weak forward upper gradient.
In particular,
\[
|D^+f|_{w,\I_p} \leq |D^+f|_{S,q} \quad \text{ $\m$-a.e.\ in $X$}.
\]
\end{lemma}

\begin{proof}
It suffices to show that $\Mod_q$-negligible sets are $\I_p$-negligible.
Given $\eta\in \mathfrak{T}_p$, similarly to the proof of Proposition~\ref{stabilityofweakgradient},
we assume that
\[
\int_{C_\tau([0,1];X)} \mathcal{E}_p(\gamma)\,\eta({\dd}\gamma)
 =\int_{C_\tau([0,1];X)} \int^1_0 |\gamma'_+|^p(t) \,{\dd}t \,\eta({\dd}\gamma)<\infty
\]
and that $\eta$ is concentrated on curves contained in $\{V\leq M\}$.
For any $\Gamma\subset \AC([0,1];(X,d))$ and $\varrho:X\rightarrow [0,\infty)$
satisfying $\int_\gamma \varrho\geq 1$ for all $\gamma \in \Gamma$,
we have
\begin{align*}
\eta(\Gamma)
&\leq \int_\Gamma \biggl(  \int_\gamma \varrho \biggr) \,\eta({\dd}\gamma)
 =\int_{\Gamma} \int_0^1 \varrho(\gamma_t)|\gamma'_+|(t) \,{\dd}t \,\eta({\dd}\gamma) \\
&\leq \biggl( \int_{\Gamma} \int_0^1 \varrho^q(\gamma_t) \,{\dd}t \,\eta({\dd}\gamma) \biggr)^{1/q}
 \biggl( \int_{\Gamma} \int_0^1 |\gamma'_+|^p(t) \,{\dd}t \,\eta({\dd}\gamma) \biggr)^{1/p} \\
&= \biggl( \int^1_0 \int_{\Gamma} \varrho^q \bigl( e_t(\gamma) \bigr) \,\eta({\dd}\gamma) \,{\dd}t \biggr)^{1/q}
 \biggl( \int_\Gamma \mathcal{E}_p(\gamma) \,\eta({\dd}\gamma) \biggr)^{1/p} \\
&\leq c(\eta,M)^{1/q} \biggl( \int_X \varrho^q \dm \biggr)^{1/q}
 \biggl( \int_\Gamma \mathcal{E}_p(\gamma) \,\eta({\dd}\gamma) \biggr)^{1/p},
\end{align*}
where $c(\eta,M)$ is the constant satisfying \eqref{bouncomprecon}.
By taking the infimum in $\varrho$, we have
\[
\eta(\Gamma) \leq c(\eta,M)^{1/q} \Mod_q(\Gamma)^{1/q}
 \left(\int_\Gamma \mathcal{E}_p(\gamma) \,\eta({\dd}\gamma) \right)^{1/p}.
\]
Therefore, $\Mod_q$-negligible sets are $\I_p$-negligible.
\end{proof}

\begin{theorem}\label{conincdiegradee}
Let $q\in (1,\infty)$, $(X,\tau,d,\m)$ be a $\APE$  as in Theorem~$\ref{twogradientcoindced}$
and $\mathfrak{T}_p$ be the collection of all $p$-test plans with bounded compression on the sublevels of $V$.
For each $f\in L^q(X,\m)$, the following are equivalent.
\begin{enumerate}[{\rm (I)}]
\item $f$ has a $\mathfrak{T}_p$-weak forward upper gradient in $L^q(X,\m)$
$($see Definition~$\ref{weakuppgrad})$.
\item $f$ has a relaxed ascending $q$-gradient
$($see Definition~$\ref{relaxedgradient})$.
\item $f$ has a relaxed forward $q$-upper gradient
$($see Definition~$\ref{df:relug})$.
\item $f$ has a forward $q$-upper gradient in $L^q(X,\m)$
$($see Definition~$\ref{Moduppgra})$.
\end{enumerate}
In addition, we have
\[
|D^+f|_{w,\I_p} =|D^+f|_{*,q} =|D^+f|_{C,q} =|D^+f|_{S,q} \quad \text{$\m$-a.e.\ in $X$.}
\]
\end{theorem}

\begin{proof}
This follows from Theorem~\ref{twogradientcoindced} and
Lemmas~\ref{ascenddescednrealx2}\eqref{relaxedascgra2-3}, \ref{gradeeen1} and \ref{gradeeen2}.
\end{proof}

\begin{remark}\label{allnormsameremark}
Obviously, Theorem~\ref{conincdiegradee} is valid also for backward gradients.
Moreover, recall that $|Df|_{\varsigma}=\max\{ |D^\pm f|_{\varsigma}\}$ holds for $\varsigma=(w,\I_p), (*,q), (S,q)$
(see Lemma~\ref{lm:wD^pm}, Corollary~\ref{keyconidnce}).
Combining this with $|Df|_{S,q} \le |Df|_{C,q}$ seen in the same way as Lemma~\ref{gradeeen1},
we observe the following:
\begin{enumerate}[(a)]
\item\label{itesame2} $|Df|_{w,\I_p}=|Df|_{*,q}=|Df|_{C,q}=|Df|_{S,q}$;
\item\label{itesame3} $|Df|_{C,q}=\max\{ |D^\pm f|_{C,q}\}$.
\end{enumerate}
\end{remark}

\subsection{Sobolev spaces from relaxed gradients}

Let $(X,\tau,d,\m)$ be a $\APE$.
We define the (absolute) \emph{$q$-Cheeger energy} by minimal relaxed $q$-gradients:
\[
\Chc_{q}(f) :=\frac1q \int_X |Df|^q_{*,q} \dm
\]
if $f\in L^q(X,\m)$ has a relaxed $q$-gradient, and $\Chc_q(f):=\infty$ otherwise.
Arguing as in the proof of Theorem \ref{chlowersemicont}, we see that $\Chc_{q}$ is convex
and lower semi-continuous with respect to the weak topology of $L^q(X,\m)$.

\begin{definition}\label{standardsobolespaces}
For $f\in \mathfrak{D}({\Chc}_{q})$, define the norm
\[
\|f\|_{W^{1,q}_*}:=\bigl( {\|f\|^q_{L^q}+\bigl\| |Df|_{*,q} \bigr\|^q_{L^q}} \bigr)^{1/q}
 =\bigl(\|f\|^q_{L^q}+q\Chc_q(f) \bigr)^{1/q}.
\]
The \emph{Sobolev space} $W^{1,q}_*(X,d,\m)$ is defined as the closure of
$\mathfrak{D}({\Chc}_{q})$ with respect to $\|\cdot\|_{W^{1,q}_*}$.
\end{definition}

Observe that $\|\cdot\|_{W^{1,q}_*}$ is indeed a norm:
For any $f,g\in \mathfrak{D}({\Chc}_{q})$,
\begin{itemize}
\item $\|f\|_{W^{1,q}_*}\geq 0$ with equality if and only if $f=0$,
\item $\|\lambda f\|_{W^{1,q}_*} =|\lambda| \cdot \|f\|_{W^{1,q}_*}$ for all $\lambda \in \mathbb{R}$,
\item $\|f+g\|_{W^{1,q}_*}\leq \|f\|_{W^{1,q}_*}+\|g\|_{W^{1,q}_*}$.
\end{itemize}
Moreover, we have the following (cf.\ \cite[Theorem 2.7]{Ch}).

\begin{theorem}\label{sobolevspacedefinion}
Let $(X,\tau,d,\m)$ be a $\APE$ and $q \in (1,\infty)$.
Then  $W^{1,q}_*(X,d,\m)$ is a Banach space with 
$W^{1,q}_*(X,d,\m)= \mathfrak{D}({\Chc}_{q})$.
\end{theorem}

\begin{proof}
Note that $\mathfrak{D}({\Chc}_{q})$ is a linear space by the above observations.
To see the completeness, let $(f_i)_{i \ge 1}$ be a Cauchy sequence with respect to $\|\cdot\|_{W^{1,q}_*}$.
Then, $f_i$ converges to some $f$ in $L^q(X,\m)$,
and $|Df_i|_{*,q}$ is uniformly bounded in $L^q(X,\m)$.
Thus, we may assume that $|Df_i|_{*,q}$ weakly converges to some $G \in L^q(X,\m)$.
We deduce from Lemma~\ref{closurestrcrexlgra}\eqref{closuer2} that
$G$ is a relaxed $q$-gradient of $f$, and hence $f\in W^{1,q}_*(X,d,\m)$.

It remains to show $\|f_i -f\|_{W^{1,q}_*}\to 0$.
Put $u_i :=f_i -f$, and suppose to the contrary that $c:=\liminf_{i \to \infty}\|u_i\|_{W^{1,q}_*}>0$.
Since $(u_i)_{i \ge 1}$ is a Cauchy sequence with respect to $\|\cdot\|_{W^{1,q}_*}$,
we can choose $i$ such that $\|u_i -u_j\|_{W^{1,q}_*}<c/q$ for all $j>i$.
Then, we have $\limsup_{j \to \infty}\|u_i -u_j\|_{W^{1,q}_*}\leq c/q$.
However, it follows from the lower semi-continuity of ${\Chc_{q}}$ (in the weak topology of $L^q(X,\m)$) that
\[
\|u_i\|_{W^{1,q}_*} \leq \liminf_{j \to \infty}\|u_i -u_j\|_{W^{1,q}_*}
 \leq \limsup_{j \to \infty}\|u_i -u_j\|_{W^{1,q}_*} \leq \frac{c}q,
\]
which contradicts $\liminf_{i \to \infty}\|u_i\|_{W^{1,q}_*}=c>0$.
 Hence, $W^{1,q}_*(X,d,\m)$ is a Banach space.

Then $W^{1,q}_*(X,d,\m)= \mathfrak{D}({\Chc}_{q})$ follows from the lower semi-continuity of ${\Chc}_{q}$.
\end{proof}

\begin{remark}
Thanks to Remark \ref{allnormsameremark}\eqref{itesame2},
in the situation of Theorem~\ref{twogradientcoindced},
employing other gradients leads to the same Sobolev space.
\end{remark}

The following example shows difficulty in employing $|D^+f|_{*,q}$ to define a Sobolev space.

\begin{example}\label{Examphalfsobolevspace}
If we define
\[
\|f\|_{W^{1,q}_{+}} :=\bigl({\|f\|^q_{L^q}+ \bigl\| |D^+f|_{*,q} \bigr\|^q_{L^q}} \bigr)^{1/q}
 =\bigl(\|f\|^q_{L^q}+q\Chc^+_q(f) \bigr)^{1/q},
\]
then we have
\begin{itemize}
\item $\|f\|_{W^{1,q}_+}\geq 0$ with equality if and only if $f=0$,
\item $\|\lambda f\|_{W^{1,q}_+}= \lambda \|f\|_{W^{1,q}_+}$ for all $\lambda>0$,
\item $\|f+g\|_{W^{1,q}_+}\leq \|f\|_{W^{1,q}_+}+\|g\|_{W^{1,q}_+}$.
\end{itemize}
Due to the asymmetry of $d$, $\|f\|_{W^{1,q}_+} \neq \|{-f}\|_{W^{1,q}_+}$ may occur.
For instance, we consider the Funk metric space $(\mathbb{B}^n,F)$ as in Example \ref{funkmetricsapce}
with the Busemann--Hausdorff measure $\m_{\mathrm{BH}}$.
Then, $\m_{\mathrm{BH}}[\mathbb{B}^n]<\infty$ (see Krist\'aly--Rudas \cite{KR}) and
\[
\|f\|_{W^{1,q}_{+}} =\bigl( {\|f\|^q_{L^q}+\|F^*({\dd}f)\|^q_{L^q}} \bigr)^{1/q} \qquad
 \text{for all } f\in \Lip(\X)\cap \mathfrak{D}(\Chc^+_{q})
\]
(recall Corollary \ref{finserlnorm}).
It was shown in \cite[\S 3]{KR} that the function $f(x)=-\sqrt{1-|x|}$
satisfies $f \in \mathfrak{D}(\Chc^+_{q})$ but $-f \notin \mathfrak{D}(\Chc^+_{q})$,
thus $\mathfrak{D}(\Chc^+_{q})$ is not a linear space.
\end{example}

We also remark that,
though the construction in this subsection makes sense in general forward metric measure spaces,
it may turn out trivial if no additional assumptions are made.
For instance, in the same way as \cite[Remark~4.12]{AGS2} (with $q=2$),
we can construct $(X,d,\m)$ such that $\Chc_q(f)=0$ for all $f \in L^q(X,\m)$,
which implies $W^{1,q}_*(X,d,\m)=L^q(X,\m)$.

\subsection{Sobolev spaces over Finsler manifolds}

This subsection is devoted to the special case of Finsler manifolds.
Let $(\X,F,\m)$ be a forward complete Finsler manifold endowed with a smooth positive measure $\m$.
For $f \in C^\infty_0(\X)$, by a similar argument to the proof of Corollary~\ref{finserlnorm}\eqref{finslersob1},
we have
\begin{equation}\label{eq:|Df|}
|Df| =|Df|_{*,q} =|Df|_{w,\I_p} =\max\{F^*(\pm{\dd}f)\}
\end{equation}
under $\m[\X]<\infty$ or $\lambda_F(\X)<\infty$ and $\CD(K,\infty)$, and hence
\[
\|f\|_{W^{1,q}_*} =\biggl( \int_{\X} |f|^q\dm +\int_{\X} \max\{F^*(\pm {\dd}f)^q \} \dm \biggr)^{1/q}.
\]
Define the Sobolev space $W^{1,q}_0(\X,F,\m)$ in the standard way
as the completion of $C^\infty_0(\X)$ under $\|\cdot\|_{W^{1,q}_*}$.
Note that, equivalently, $\max\{F^*(\pm{\dd}f)^q\}$ can be replaced with $F^*({\dd}f)^q +F^*(-{\dd}f)^q$.
We also remark that, in certain contexts, $\max\{F^*(\pm{\dd}f)\}$ arises more naturally as a gradient norm than either $F^*({\dd}f)$ or $F^*(-{\dd}f)$.
For example, the Heisenberg--Pauli--Weyl, Caffarelli--Kohn--Nirenberg and Hardy inequalities can be extended to the Finsler setting  via  $\max\{F^*(\pm{\dd}f)\}$ but may fail for $F^*(\pm{\dd}f)$;
see Huang--Krist\'aly--Zhao \cite{HKZ}, Krist\'aly--Rudas \cite{KR}, Krist\'aly--Li--Zhao \cite{KLZ} and Zhao \cite{ZhaoHardy} for details.


\begin{theorem}\label{sobolevfinsler}
Let $(\X, d_F,\m)$ be a forward metric measure space induced from a forward complete Finsler manifold $(\X, F)$
with a smooth positive measure $\m$.
Suppose that one of the following holds:
\begin{enumerate}[\rm \ \ (a)]
\item \label{Finslersob1} $\lambda_F(\X)<\infty$ and $\CD(K,\infty)$;
\item \label{Finslersob2} $\m[\X] <\infty$ and $\liminf_{r\to \infty} \lambda_F(B^+_o(r))^q \m[B^+_o(r)\setminus B^+_o(r-1)]=0$
for some $o \in \X$.
\end{enumerate}
Then, we have $W^{1,q}_0(\X,F,\m)=W^{1,q}_*(\X,d_F,\m)$.
\end{theorem}

\begin{proof}
It follows from Lemma~\ref{closurestrcrexlgra}\eqref{closuer3} (for relaxed $q$-gradients) that
\[
W^{1,q}_*(\X,d_F,\m)=\overline{\SL_b(\X)\cap \mathfrak{D}(\Chc_{q})}^{\|\cdot\|_{W_*^{1,q}}},
\]
where $\SL_b(\X)$ denotes the collection of bounded Lipschitz functions.
Moreover,
\[
\overline{C^\infty_0(\X)}^{\|\cdot\|_{W^{1,q}_*}} =\overline{\SL_0(\X)}^{\|\cdot\|_{W^{1,q}_*}}
\]
by \cite[Proposition~3.3]{KLZ}.
Thus, it suffices to show that
\begin{equation}\label{tworelatonLipsct}
\overline{\SL_0(\X)}^{\|\cdot\|_{W^{1,q}_*}}
 =\overline{\SL_b(\X)\cap \mathfrak{D}(\Chc_{q})}^{\|\cdot\|_{W^{1,q}_*}}.
\end{equation}
Since ``$\subset$'' clearly holds, we shall show that,
for every $f\in \SL_b(\X)\cap \mathfrak{D}(\Chc_{q}) \subset L^q(\X,\m)$,
there exists a sequence $(f_i)_{i \ge 1}$ in $\SL_0(\X)$ such that $\|f_i -f\|_{W^{1,q}_*} \to 0$.

Now we define an increasing sequence $(r_i)_{i \ge 1}$ divergent to infinity. 
Under Condition \eqref{Finslersob1}, set $r_i:=i$.
For Condition \eqref{Finslersob2}, there exists a sequence $(r_i)_{i \ge 1}$ such that 
\begin{equation}\label{conditionb2}
\lim_{i\to \infty} r_i=\infty, \qquad
\lim_{i\to \infty} \lambda_F \bigl(B^+_o(r_i)\bigr)^q \m[B^+_o(r_i)\setminus B^+_o(r_i-1)]=0.
\end{equation}
Define $\chi_i(x):= \phi(d_F(o,x)-r_i)$,
where $o$ is from \eqref{Finslersob2} or an arbitrarily fixed point in $\X$ in the case \eqref{Finslersob1},
and 
\[
\phi(t) :=\begin{cases}
1 & \text{for } t<0, \\
1-t & \text{for } t\in [0,1], \\
0 &\text{for } t>1.
\end{cases}
\]
Note that $\chi_i \in \SL_0(\X)$ with $\supp(\chi_i) \subset \overline{B^+_o(r_i+1)}$.
Moreover, since $(\X,d_F)$ is a geodesic space, one can show  
\[
|\chi_i(x)-\chi_i(y)| \leq  \lambda_F\bigl( B^+_o(r_i+1) \bigr) d_F(x,y) \quad\   \text{for all}\ x,y\in \X.
\]     
Hence,
\[
\max\{ F^*(\pm {\dd}\chi_i)  \} \leq
 \lambda_F \bigl( B^+_o(i+1) \bigr) \cdot
 \mathbbm{1}_{\overline{B^+_o(r_i+1)} \setminus B^+_o(r_i)}.
\]
Since $f$ is in $L^q(X,\m)$ and  bounded, both $\lambda_F(\X)<\infty$ in Condition \eqref{Finslersob1} and \eqref{conditionb2} imply
\begin{equation}\label{limitlocalfin}
\lim_{i \to \infty} \|F^*(\pm f{\dd}\chi_i)\|_{L^q}=0.
\end{equation}     
Put $f_i:=\chi_i\cdot f$, which converges to $f$ in $L^q(\X,\m)$ by the dominated convergence theorem.
Moreover, we infer from \eqref{limitlocalfin} that
\[
\bigl\| F^* \bigl( \pm{\dd}(f_i-f) \bigr) \bigl\|_{L^q}
 \leq \| F^*(\pm f{\dd}\chi_i) \|_{L^q}
 +\bigl\| F^*\bigl(  \mp (1-\chi_i) {\dd}f \bigr) \bigr\|_{L^q} \to 0
\]
as $i \to \infty$, thereby $\|f_i -f\|_{W^{1,q}_*} \to 0$.
Hence, \eqref{tworelatonLipsct} holds.
\end{proof}


As we have already seen in Example~\ref{Examphalfsobolevspace},
the finiteness of the reversibility plays an important role in Sobolev spaces.
Recall \S \ref{Finslergeometry} for the $S$-curvature $\textbf{S}$ and
the weighted Ricci curvature $\mathbf{Ric}_\infty$.

\begin{example}\label{ex-1}
For $\alpha\in [0,1]$, we consider an interpolation between the hyperbolic metric and the Funk metric
given by $(\mathbb{B}^n,F_\alpha)$,
\[
F_\alpha(x,y) :=\frac{\sqrt{\|y\|^2-(\|x\|^2\|y\|^2-\langle x,y\rangle^2)}}{1-\|x\|^2}
 +\alpha\frac{\langle x,y\rangle}{1-\|x\|^2},\quad
  x \in \mathbb{B}^n,\ y\in T_x\mathbb{B}^n=\mathbb{R}^n.
\]
According to    \cite{KR}, $\lambda_{F_\alpha}(\mathbb{B}^n)<+\infty$ if and only if $\alpha\neq 1$.
Moreover, in the 2-dimensional case $(\mathbb{B}^2,d_{F_\alpha},\m_{\mathrm{BH}})$
with the Busemann--Hausdorff measure $\m_{\mathrm{BH}}$,
Kaj\'ant\'o--Krist\'aly \cite{KK} proved that $(\mathbb{B}^2,F_\alpha)$ satisfies
\[
-\frac{1}{(1-\alpha)^2} \leq {\bf Ric} \leq -\frac{1}{(1+\alpha)^2}, \qquad
 0\leq \textbf{S}\leq \frac{\alpha}{2(1-\alpha^2)},
\]
provided $\alpha \in [0,1)$.
In addition, $\textbf{Ric}_\infty$ is bounded from below by $-1/(1-\alpha)^2$,
therefore $(\mathbb{B}^2,d_{F_\alpha},\m_{\mathrm{BH}})$ satisfies $\CD(-1/(1-\alpha)^2,\infty)$.
\end{example}

We next see that the latter condition
\begin{equation}\label{strongsoboreverconditon}
\m[\X]<\infty, \qquad
 \liminf_{r\to \infty} \lambda_F \bigl( B^+_o(r) \bigr)^q \m[B^+_o(r)\setminus B^+_o(r-1)] =0
\end{equation}
in Theorem~\ref{sobolevfinsler} does not imply $\lambda_F(\X)<\infty$.

\begin{example}\label{ex-2}
Let $\Omega\subset \mathbb{R}^n$ be a bounded strongly convex domain with $\mathbf{0}\in \Omega$
and $F$ be the Funk metric on $\Omega$ given by
\[
x+\frac{y}{F(x,y)}\in \partial\Omega \quad
  \text{for } x \in \Omega,\ y\in T_x\Omega=\mathbb{R}^n.
\]
Then, $(\Omega,d_F)$ is a forward complete forward metric space with $\lambda_F(\Omega)=\infty$.
We set $\m:={\ee}^{-r^2}\m_{\mathrm{BH}}$,
where $\m_{\mathrm{BH}}$ is the Busemann--Hausdorff measure and $r(x):=d_F(\mathbf{0},x)$.
By  \cite[Section 6]{KLZ}, we have
\[
\lambda_F \bigl( B^+_{\mathbf{0}}(r) \bigr) \leq C{\ee}^r-1, \qquad
 \tau\bigl( \gamma'_y(t) \bigr) =\tau_{\mathrm{BH}}(y)+\frac{n+1}{2}t+t^2,
\]
where $C\geq 2$ is a constant,
$\gamma_y$ is the unit speed geodesic with initial velocity $y\in T_\mathbf{0}\mathbb{B}^n$,
and $\tau$ is the distortion as in \S \ref{Finslergeometry}.
It follows from the volume estimate in Zhao--Shen \cite[Theorem 3.4, Remark 3.8]{ZS} that
\begin{align*}
\m[B^+_\mathbf{0}(r)]
&= a_{n-1} \int^r_0 \exp\biggl( -\frac{n+1}{2}t -t^2 \biggr) \mathfrak{s}_{-1/4}^{n-1}(t) \,{\dd}t
 \leq a_{n-1} \int^r_0 {\ee}^{-t-t^2} \,{\dd}t \leq  \sqrt{\pi}  a_{n-1},
\end{align*}
where $a_{n-1}$ denotes the Euclidean area of $\mathbb{S}^{n-1}$
and $\mathfrak{s}_{-1/4}(t):=2\sinh(t/2)$.
We similarly deduce that
\[
\m [B^+_\mathbf{0}(r) \setminus B^+_\mathbf{0}(r-1)]
 \leq a_{n-1} \int^r_{r-1} {\ee}^{-t^2} \,{\dd}t \leq a_{n-1} {\ee}^{-(r-1)^2},
\]
which implies that $(\Omega,d_F,\m)$ also satisfies the latter condition in \eqref{strongsoboreverconditon}.
\end{example}


We finally summarize applications of our results to $q$-heat flow on Finsler manifolds.
Though we do not pursue such a direction,
we expect that some of the assumptions can be removed (or weakened)
by directly working on this setting.

\begin{corollary}
Let $(\X, d_F,\m)$ be a forward metric measure space induced from a forward complete Finsler manifold
with a smooth positive measure $\m$.
Suppose that one of the following conditions holds:
\begin{enumerate}[{\rm (a)}]
\item\label{case:a} $q \ge 2$, $\lambda_F(\X)<\infty$ and $\CD(K,\infty)$;
\item $\m[\X]<\infty$.
\end{enumerate}
Then, for any $f \in L^2(\X,\m)$, there exists a unique solution
$u_t \in \mathfrak{D}(\Chc^+_q) \cap L^2(\X,\m)$ to the $q$-heat equation
\begin{equation}\label{eq:q-heat}
\partial_t u_t=\Delta_q u_t \,\ \text{for $\mathscr{L}^1$-a.e.\ $t\in (0,\infty)$}, \qquad
u_t \to f \text{ in } L^2(\X,\m)  \text{ as } t \to 0 .
\end{equation}
Moreover, in the case of \eqref{case:a} with $f \in L^2(\X,\m) \cap L^q(\X,\m)$,
we have $u_t\in W^{1,q}_0(\X,F,\m)$.
\end{corollary}

\begin{proof}
Given $f \in L^2(\X,\m)$, it follows from Corollary~\ref{heatflowequationproper-00} that
there exists a unique solution $u_t=\Ht(f)$ to the $q$-heat equation \eqref{eq:q-heat}.
Note that $|[u_t]_N|^q \in L^q(\X,\m)$ for all $N>0$;
this is obvious when $\m[\X]<\infty$, and follows from $|[u_t]_N|^q \le |u_t|^2 N^{q-2}$ when $q \ge 2$
as we saw in the proof of Lemma~\ref{constlowerseimcheegenerge}.
Hence, thanks to Theorem~\ref{twogradientcoindced},
$(u_t)_{t>0}$ coincides with the gradient curve for $\Chc^+_q$ given by Theorem~\ref{chgradinetflwo}.

If $f \in L^2(\X,\m) \cap L^q(\X,\m)$, then the contraction property \eqref{lpl2consgra}
(with $\theta=q$ and $h_0 \equiv 0$) implies that $u_t \in L^q(\X,\m)$ for all $t>0$.
Hence, the latter assertion follows from
Theorems~\ref{sobolevspacedefinion}, \ref{sobolevfinsler}
and $\Chc^-_q(u_t)\leq \lambda_F(\X)^q\Chc^+_q(u_t)$.
\end{proof}


\appendix

\section{Lisini's theorem for forward extended Polish metric spaces}\label{Lisinitheoremproof}

This  appendix is devoted to the proof of Theorem \ref{Lisinitheorem}, which is similar to that of  \cite[Theorem 3.1]{Li2} (see also \cite{Li,ZZ}), for the sake of self-containedness.
Let $(X,\tau,d)$ be a forward extended Polish space.
Recalling Definition \ref{weaktopolgytau}\eqref{weakt1-2},
we introduce a bounded symmetric  metric  on $X$ by
\[
\rho_1(x_1,x_2):=\min\left\{\rho(x_1,x_2),1\right\}.
\]
The topology induced by $\rho_1$ is also $\tau$.
 Let $\mathcal {M}([0,1];X)$ be the collection of Lebesgue equivalent classes of Lebesgue measurable maps from $[0,1]$ to $X$, i.e., for a Lebesgue measurable map $\gamma:[0,1] \to X$,
$[\gamma]\in \mathcal {M}([0,1];X)$ is given by
\[
[\gamma]:=\bigl\{\zeta:[0,1]\rightarrow X\,\big|\, \zeta \text{ is Lebesgue measurable and $\zeta(t)=\gamma(t)$ for $\mathscr{L}^1$-a.e.\ $t\in [0,1]$}  \bigr\}.
\]
For brevity, we will use $\gamma$ to denote $[\gamma]$ in the sequel.
Equip $\mathcal {M}([0,1];X)$ with the symmetric distance
\[
\varrho(
\gamma_1,\gamma_2):=\int^1_0\rho_1 \bigl( \gamma_1(t),\gamma_2(t) \bigr) \,{\dd}t.
\]
Then $(\mathcal {M}([0,1];X),\varrho)$ is a Polish space.
Moreover, it follows from Chebyshev's inequality that  if a sequence $(\gamma_n)_{n \ge 1}$ converges to $\gamma$ in $\varrho$, then it converges to $\gamma$ in measure, i.e.,
\begin{equation}\label{conveginmeasure}
\lim_{n\rightarrow \infty}\mathscr{L}^1\bigl[ \bigl\{t\in [0,1]\,\big|\, \rho\bigl( \gamma(t),\gamma_n(t) \bigr) >\varepsilon \bigr\} \bigr] =0 \quad \text{for all $\varepsilon>0$}.
\end{equation}
Recall the following from \cite[Theorem 2.2]{Li2} (see also Rossi--Savar\'e \cite[Theorem 2]{RS} for a general version).

\begin{theorem}\label{tightcompreversible}
A family  $\mathscr{A}\subset \mathcal {M}([0,1];X)$ is precompact if the following two conditions hold:
\begin{itemize}

\item[(1)] There exists a function $\psi:X\rightarrow [0,\infty]$ whose sublevels $\lambda_c(\psi):=\{x\in X\mid \psi (x)\leq c\}$ are compact for every $c\geq 0$ with $\sup_{\gamma\in \mathscr{A}}\int^1_0 \psi(\gamma(t)) \,{\dd}t<\infty$;

\item[(2)] There exists a $(\tau\times \tau)$-lower semi-continuous function $g:X\times X\rightarrow [0,\infty]$ such that 
\begin{itemize}
\item[$\bullet$] $\ds g(x,y)=0$ implies $x=y$;
\item[$\bullet$] $\lim_{h\rightarrow 0^+} \sup_{\gamma\in \mathscr{A}} \int^{1-h}_0 g(\gamma(t+h),\gamma(t))\,{\dd}t=0$.
\end{itemize}
\end{itemize}
\end{theorem}

We also recall the following result from \cite[Lemma 2.3]{Li2}.

\begin{lemma}\label{mcontress}
Let $Y$ be a Polish space and $(\mu_n)_{n\ge 1}\subset \Po(Y)$ be a sequence weakly convergent to $\mu\in \Po(Y)$.
Let $\ell_n:Y\rightarrow [0,\infty)$ be a sequence of $\mu_n$-measurable functions such that 
\[
\sup_{n\in \mathbb{N}}\int_Y \ell_n \,{\dd}\mu_n <\infty.
\]
Then we have a subsequence $(\mu_{n_k})_{k \ge 1}$ such that for $\mu$-a.e.\ $\bar{y}\in \supp(\mu)$, there exists $y_{n_k}\in \supp(\mu_{n_k})$ with
\[
\lim_{k\rightarrow \infty}y_{n_k}=\bar{y}, \qquad \sup_{k\in \mathbb{N}}\ell_{n_k}(y_{n_k})<\infty.
\]
\end{lemma}

The following lemma is a direct consequence of \cite[Lemma 2.1]{Li2}.

\begin{lemma}\label{strongcontrall}
 Suppose that $\gamma:[0,1]\rightarrow (X,d)$ is a map satisfying the following:
\begin{itemize}

\item $\gamma$ is right-continuous at every $t\in [0,1]$ and continuous except a countable set in $[0,1]$;

\item there is a bounded non-decreasing function $v:[0,1]\rightarrow [0,\infty)$ such that
\begin{equation}\label{contbounded}
d\bigl( \gamma(s),\gamma(t) \bigr) \leq v(t)-v(s)\quad \text{for all } [s,t]\subset [0,1];
\end{equation}

\item by extending $\gamma(t):=\gamma(1)$ for $t>1$, we have
\begin{equation}\label{normcontrall}
\limsup_{h\rightarrow 0^+}\left\|  \frac{d(\gamma(\cdot),\gamma(\cdot+h))}{h}   \right\|_{L^p([0,1])}<\infty.
\end{equation}

\end{itemize}
Then $\gamma\in \AC^p([0,1];X)$.
\end{lemma}

\begin{proof}
Setting $C:=v(1)$, we deduce from \eqref{contbounded} that
$\gamma([0,1])\subset \overline{B^+_{\gamma(0)}(C)}^{d}$.
Then for the symmetrized metric $\hat{d}$ defined by \eqref{symmmetricde},
we have $\hat{d}(\gamma(t+h),\gamma(t))\leq \theta d(\gamma(t),\gamma(t+h))$ with
\[
\theta:=\lambda_d \left( \overline{B^+_{\gamma(0)}(C)}^{d}  \right)<\infty.
\]
Together with \eqref{normcontrall}, we obtain
\begin{equation}\label{normcontrall2}
\limsup_{h\rightarrow 0^+}\left\|  \frac{\hat{d}(\gamma(\cdot+h),\gamma(\cdot))}{h}   \right\|_{L^p([0,1])}<\infty.
\end{equation}
Recall from Remark \ref{exposlipspaceexplain}\eqref{rmpolish-a} that $(X,\hat{d})$ is a complete symmetric extended metric space and its metric topology is the same as that of $(X,d)$.
Thus, \cite[Lemma 2.1]{Li2} (choosing $\psi(r):=r^p$) yields $\gamma\in \AC^p([0,1];(X,\hat{d}))=\AC^p([0,1];X)$.
\end{proof}

Now we prove Theorem \ref{Lisinitheorem}.

\begin{proof}[Proof of Theorem $\ref{Lisinitheorem}$]
Given any $N\in \mathbb{N}$, we set $t^i:=i/2^N$, $i=0,1,\ldots, 2^N$,
and, letting $X_i$, $i= 0,1,\ldots, 2^N$, be $2^N+1$ copies of $X$, 
$\mathbf{X}:=X_0\times \cdots\times X_{2^N}$.
Every $X_i$ is equipped with the Polish topology $\tau$. 
Choose optimal transference plans $\pi_{i,i+1}\in \Pi(\mu_{t^i},\mu_{t^{i+1}})$ for $i=0,\ldots,2^{N}-1$. By Definition \ref{weaktopolgytau}\eqref{weakt1-2} and  \cite[Lemma 5.3.4]{AGS}, there exists $\pi_N\in \mathscr{P}(\mathbf{X})$  such that
\begin{equation}\label{distingragofpi_N}
\mathfrak{p}^i_\sharp \pi_N=\mu_{t^i},\qquad  \mathfrak{p}^{i,i+1}_\sharp \pi_N=\pi_{i,i+1},
\end{equation}
where $\mathfrak{p}^i:\mathbf{X}\rightarrow X_i$ and $\mathfrak{p}^{i,i+1}:\mathbf{X}\rightarrow X_i\times X_{i+1}$  are the natural projections.

Now define a map $\sigma:\mathbf{X}\rightarrow \mathcal {M}([0,1];X)$ by
\begin{equation}\label{defsigma}
\mathbf{x}=(x_0,\ldots,x_{2^N}) \longmapsto \sigma_t(\mathbf{x}):=x_i \text{ for }t\in [t^i,t^{i+1}),
\end{equation}
and set
\begin{equation}\label{etandef}
\widetilde{\eta}_N:=\sigma_\sharp\pi_N\in \mathscr{P} \bigl( \mathcal {M}([0,1];X) \bigr).
\end{equation}
The rest of the proof is divided into {six} steps.

\smallskip

\noindent \textbf{Step 1.} In this step, we prove that for any $\alpha\in [1,p]$,
\begin{align}
\sup_{2^{-N}\leq h<1}\int^{1-h}_0 \left(\frac{{d}(\sigma_t(\mathbf{x}),\sigma_{t+h}(\mathbf{x})) }{h }\right)^\alpha \,{\dd}t&\leq 2^{\alpha+N(\alpha-1)}\sum_{i=0}^{2^N-1}{d}(x_i,x_{i+1})^\alpha \quad
 \text{for all $\mathbf{x}\in \mathbf{X}$};\label{lapclsigmestimate}\\
\sup_{0<h<1}\int^{1-h}_0 \frac{{d}(\sigma_t(\mathbf{x}),\sigma_{t+h}(\mathbf{x}))^\alpha}{h} \,{\dd}t &\leq \left( 2^\alpha+\frac{2^{\alpha N}}{2^N-1}  \right)\sum_{i=0}^{2^N-1} {d}(x_i,x_{i+1})^\alpha \quad
 \text{for all $\mathbf{x}\in \mathbf{X}$};\label{forprobb}\\
\int_{\mathbf{X}}d(x_i,x_{i+1})^\alpha \,\pi_N({\dd}\mathbf{x})&\leq W_p(\mu_{t^i},\mu_{t^{i+1}})^\alpha.\label{forprobb22}
\end{align}
For \eqref{lapclsigmestimate}, if $2^{-N} \leq h<1$, then choose the integer $k\geq 1$ such that
$t^k \le h <t^{k+1}$, which implies
\begin{equation}\label{khcontroll}
\frac{1}{2^N}\leq \min\left\{\frac{h}k,\ 2^{N(\alpha-1)}\frac{h^\alpha}{k^\alpha}\right\}.
\end{equation}
The triangle inequality of $d$ combined with \eqref{defsigma} yields, for $t \in (0,1-h)$,
\begin{align}
d\bigl( \sigma_t(\mathbf{x}),\sigma_{t+h}(\mathbf{x})\bigr)^{\alpha}
\leq  \left[\sum_{i=0}^k  d\bigl(\sigma_{t+t^i}(\mathbf{x}),\sigma_{t+t^{i+1}}(\mathbf{x})\bigr)\right]^{\alpha}
\leq (k+1)^{{\alpha}-1}\sum_{i=0}^k d\bigl(\sigma_{t+t^i}(\mathbf{x}),\sigma_{t+t^{i+1}}(\mathbf{x})\bigr)^{\alpha},
\label{sumqqesti}
\end{align}
where we set $\sigma_{t+t^{k+1}}(\mathbf{x}):=x_{2^N}$ if $t+t^{k+1} \ge 1$.
By $1-t^k=(2^N-k)/2^N$, \eqref{sumqqesti} and \eqref{khcontroll}, we have
\begin{align}
&\int^{1-h}_0 d\bigl(\sigma_t(\mathbf{x}),\sigma_{t+h}(\mathbf{x})\bigr)^{\alpha} \,{\dd}t
\leq \int_0^{1-t^k} d\bigl(\sigma_t(\mathbf{x}),\sigma_{t+h}(\mathbf{x})\bigr)^{\alpha} \,{\dd}t\notag\\
&\leq (k+1)^{\alpha-1} \sum_{i=0}^k \int^{1-t^k}_0 d\bigl(\sigma_{t+t^i}(\mathbf{x}),\sigma_{t+t^{i+1}}(\mathbf{x})\bigr)^{\alpha} \,{\dd}t\notag\\
&= (k+1)^{\alpha-1}\sum_{i=0}^k \sum_{j=0}^{2^N-k-1} \int^{t^{j+1}}_{t^j} d\bigl(\sigma_{t+t^i}(\mathbf{x}),\sigma_{t+t^{i+1}}(\mathbf{x})\bigr)^{\alpha} \,{\dd}t\notag\\
&= (k+1)^{\alpha-1}\sum_{i=0}^k \frac1{2^N} \sum_{j=0}^{2^N-k-1}{d}( {x_{i+j}}, x_{i+j+1} )^{\alpha}\leq \frac{(k+1)^{\alpha}}{2^N}\sum_{j=0}^{2^N-1}{d}(x_j,x_{j+1})^{\alpha}\label{khNinequality}\\
&\leq  2^{N(\alpha-1)}\frac{(k+1)^{\alpha}}{k^\alpha}h^\alpha\sum_{j=0}^{2^N-1}{d}(x_j,x_{j+1})^{\alpha}\leq 2^{N(\alpha-1)+\alpha} h^\alpha\sum_{j=0}^{2^N-1}{d}(x_j,x_{j+1})^{\alpha},\notag
\end{align}
which is \eqref{lapclsigmestimate}. 
For \eqref{forprobb}, it follows from a similar calculation but with $2^{-N} \le h/k$ in \eqref{khcontroll} that
\begin{equation*}
\int^{1-h}_0  d\bigl(\sigma_t(\mathbf{x}),\sigma_{t+h}(\mathbf{x})\bigr)^\alpha \,{\dd}t
 \leq h\frac{(k+1)^{\alpha}}{k}\sum_{j=0}^{2^N-1} {d}(x_j,x_{j+1})^\alpha,
\end{equation*}
which together with $\frac{(k+1)^\alpha}{k}\leq 2^\alpha+\frac{2^{N\alpha}}{2^N-1}$
implies
\begin{align}\label{bacisinequl}
\sup_{2^{-N} \leq h<1}\int^{1-h}_0 \frac{ {d}(\sigma_t(\mathbf{x}),\sigma_{t+h}(\mathbf{x}))^\alpha}{h} \,{\dd}t 
\leq \left( 2^\alpha+\frac{2^{\alpha N}}{2^N-1}  \right)\sum_{i=0}^{2^N-1} {d}(x_i,x_{i+1})^\alpha.
\end{align}
On the other hand, for $h\in (0,2^{-N})$, since $\sigma_{t+h}(\mathbf{x})=\sigma_t(\mathbf{x})$ if $t\in [t^i,t^{i+1}-h)$, we have
\begin{align*}
\int^{1-h}_0 d\bigl(\sigma_t(\mathbf{x}),\sigma_{t+h}(\mathbf{x})\bigr)^\alpha \,{\dd}t
=\sum_{i=0}^{2^N-2}\int^{t^{i+1}}_{t^i} d\bigl(\sigma_t(\mathbf{x}),\sigma_{t+h}(\mathbf{x})\bigr)^\alpha \,{\dd}t
=h\sum_{i=0}^{2^N-2} {d}(x_i,x_{i+1})^\alpha,
\end{align*}
which combined with \eqref{bacisinequl} yields \eqref{forprobb}.

Finally, \eqref{forprobb22} follows from H\"older's inequality and $\mathfrak{p}^{i,i+1}(\mathbf{x})=(x_i,x_{i+1})$:
\begin{align*}
\int_{\mathbf{X}} {d}(x_i,x_{i+1})^\alpha \,\pi_N({\dd}\mathbf{x})
\leq  \biggl(  \int_{X_i\times X_{i+1}} {d}(x_i,x_{i+1})^p \,(\mathfrak{p}^{i,i+1}_\sharp \pi_N)({\dd}x_i \, {\dd}x_{i+1}) \biggr)^{\alpha/p}
= {W_p(\mu_{t^i},\mu_{t^{i+1}})^\alpha}.
\end{align*}

\smallskip

\noindent \textbf{Step 2.}
In this step, we show the tightness of $(\widetilde{\eta}_N)_{N \ge 1}$ in $\mathcal {M}([0,1];X)$.
According to Prokhorov's theorem, it suffices to construct a function $\Phi:\mathcal {M}([0,1];X)\rightarrow [0,\infty]$ satisfying the following two properties:

\begin{enumerate}[{\rm {\quad} (1)}]

\item\label{basicapoer1} For any $c\geq 0$, the sublevel $\lambda_c(\Phi):=\{\gamma\in \mathcal {M}([0,1];X) \mid \Phi(\gamma)\leq c \}$ is compact in $\mathcal {M}([0,1];X)$;

\item\label{basicapoer2} $\sup_{N\in \mathbb{N}}\int_{\mathcal {M}([0,1]; {X})}\Phi(\gamma) \,\widetilde{\eta}_N({\dd}\gamma)<\infty$.
\end{enumerate}

\smallskip

\textbf{Construction of $\Phi$:}
Firstly, we claim that $\mathscr{A}:=\{\mu_t \mid t\in [0,1]\}$ is a compact set   in $\mathscr{P}(X)$
(with respect to the weak topology induced by $\tau$).
This is clear when Condition \eqref{lisptotalvariation} holds.
Under Condition \eqref{backwwardabs},  
 $(\mu_t) \in  \AC([0,1];(\Po(X),W_1))$ since $W_1\leq W_p$.
Hence,   $t\mapsto\mu_t$ is continuous in $\Po(X)$ by Lemma \ref{wasstopologequaivel},
which shows the claim.

The compactness of $\mathscr{A}$ in $\mathscr{P}(X)$ furnishes a function $\psi:X\rightarrow [0,\infty]$ satisfying
\begin{itemize}
\item for any $c\geq 0$, the sublevel $\lambda_c(\psi):=\{x\in X\mid \psi(x)\leq c \}$ is compact in  $(X,\tau)$;
\item $C_1:=\sup_{t\in [0,1]}\int_{X}\psi \,{\dd}\mu_t <\infty$
\end{itemize}
(see \cite[Remark 5.1.5]{AGS}).
The former condition implies that $\psi$ is lower semi-continuous.
Now we define $\Phi:\mathcal {M}([0,1];X)\rightarrow [0,\infty]$ by
\begin{equation*}
\Phi(\gamma):=\int^1_0 \psi \bigl(\gamma(t) \bigr) \,{\dd}t +\sup_{0<h<1}\int^{1-h}_0 \frac{d(\gamma(t),\gamma(t+h))}{h} \,{\dd}t.
\end{equation*}
It follows from Fatou's lemma  that $\Phi$ is lower semi-continuous in $\varrho$.

\smallskip

\textbf{Property \eqref{basicapoer1}:}
To show \eqref{basicapoer1} by using Theorem \ref{tightcompreversible}, we observe for any $c \ge 0$ 
\[
\lim_{h\rightarrow0^+}\sup_{\gamma\in  \lambda_c(\Phi)}
\int^{1-h}_0 d\bigl(\gamma(t),\gamma(t+h)\bigr) \,{\dd}t
\leq \lim_{h\rightarrow0^+} ch=0.
\]
Furthermore, $\lambda_c(\psi)$ is compact and
\[
\sup_{\gamma\in \lambda_c(\Phi)} \int^1_0 \psi\bigl( \gamma(t) \bigr) \,{\dd}t\leq c<\infty.
\]
Hence, by choosing $g:=\overleftarrow{d}$, the precompactness of $\lambda_c(\Phi)$ in $\mathcal {M}([0,1];X)$ follows from Theorem \ref{tightcompreversible}.
Combining this with the lower semi-continuity of $\Phi$ shows that $\lambda_c(\Phi)$ is compact in $\mathcal {M}([0,1];X)$.

\smallskip

 \textbf{Property \eqref{basicapoer2}:}
 In order to show \eqref{basicapoer2}, we first claim
\begin{equation}\label{boundsupNdoublinter}
\sup_{N\in \mathbb{N}}\int_{\mathcal {M}([0,1];X)}\int^1_0  \psi\bigl(\gamma(t)\bigr) \,{\dd}t \,\widetilde{\eta}_N({\dd}\gamma)<\infty.
\end{equation}
For $t\in [t^i,t^{i+1})$, by \eqref{defsigma} and \eqref{distingragofpi_N} we have
\begin{align*}
\int_{\mathbf{X}} \psi\bigl(\sigma_t(\mathbf{x})\bigr) \,\pi_N({\dd}\mathbf{x})
=\int_{\mathbf{X}} \psi\bigl(\mathfrak{p}^i(\mathbf{x})\bigr)  \,\pi_N({\dd}\mathbf{x})
=\int_{X}  \psi \,{\dd}\mu_{t^i},
\end{align*}
which combined with \eqref{etandef}  yields
\begin{align*}
\int^1_0 \int_{\mathcal {M}([0,1];X)} \psi\bigl(\gamma(t)\bigr) \,\widetilde{\eta}_N({\dd}\gamma) \,{\dd}t
=\int^1_0 \int_{\mathbf{X}} \psi\bigl(\sigma_t(\mathbf{x})\bigr) \,\pi_N({\dd}\mathbf{x}) \,{\dd}t
=\sum_{i=0}^{2^N-1}\int^{t^{i+1}}_{t^i}\int_{X} \psi \,{\dd}\mu_{t^i} \,{\dd}t \le C_1. 
\end{align*}
Then, \eqref{boundsupNdoublinter} follows from Fubini's theorem.

Secondly, since $\mu_t\in \FAC^p([0,1];(\mathscr{P}(X),W_p))$,
by \eqref{forprobb} and \eqref{forprobb22}, we have
\begin{align}
&\int_{\mathcal {M}([0,1];X)}\sup_{0<h<1}\int^{1-h}_0 \frac{ d(\gamma(t),\gamma(t+h)) }{h} \,{\dd}t \,\widetilde{\eta}_N({\dd}\gamma) \notag\\
&= \int_{\mathbf{X}}\sup_{0<h<1}\int^{1-h}_0 \frac{ d(\sigma_t(\mathbf{x}),\sigma_{t+h}(\mathbf{x}))}{h} \,{\dd}t \,\pi_N({\dd}\mathbf{x})
 \leq  \left( 2 +\frac{2^{  N }}{2^N-1}  \right)\sum_{i=0}^{2^N-1}\int_{\mathbf{X}} d(x_i,x_{i+1}) \,\pi_N({\dd}\mathbf{x}) \notag\\
&\leq 4 \sum_{i=0}^{2^N-1}W_p(\mu_{t^i},\mu_{t^{i+1}})
\leq 4 \sum_{i=0}^{2^N-1} \int^{t^{i+1}}_{t^i}|\mu'_+|(t) \,{\dd}t
= 4 \int^1_0 |\mu'_+|(t) \,{\dd}t<\infty, \label{keylatboelone}
\end{align}
which together with \eqref{boundsupNdoublinter} gives Property \eqref{basicapoer2},
and then the tightness of $(\widetilde{\eta}_N)_{N \ge 1}$ in $\mathcal {M}([0,1];X)$.

\smallskip

\noindent \textbf{Step 3.}
As $(\widetilde{\eta}_N)_{N \ge 1}$ is tight, there is a subsequence $(\widetilde{\eta}_{N_k})_{k \ge 1}$ weakly convergent to some $\widetilde{\eta}$ in $\mathscr{P}(\mathcal {M}([0,1];X))$.
In this step, we show that $\widetilde{\eta}$ is concentrated on BV right-continuous curves.

Given $\gamma:[0,1]\rightarrow X$, let its pointwise variation and essential variation be denoted by
\begin{align*}
\Vv(\gamma;[0,1])&:=\sup\left\{   \sum_{i=0}^{n-1} d\bigl(\gamma(s_i),\gamma(s_{i+1})\bigr)\,\Big|\,0=s_0<s_2<\cdots<s_{n}=1    \right\}, \\
\eV(\gamma;[0,1])&:=\inf\bigl\{ \Vv(\zeta;[0,1]) \,\big|\, \zeta(s)=\gamma(s) \text{ for $\mathscr{L}^1$-a.e.\ $s\in [0,1]$}   \bigr\}.
\end{align*}
Also define a function $L_N:\mathcal {M}([0,1];X)\rightarrow [0,\infty)$ by
\[
L_N(\gamma) :=
\begin{cases}
\ds \eV(\gamma;[0,1]) & \text{ if }\gamma\in \supp(\widetilde{\eta}_N),\\
0 & \text{ otherwise. }
\end{cases}
\]
If $\gamma(t)=\sigma_t({\mathbf{x}})$ for $\mathscr{L}^1$-a.e.\ $t\in [0,1]$ for some $\mathbf{x} \in \mathbf{X}$, by \eqref{defsigma} we obtain
\[
\eV(\gamma;[0,1])=\Vv(\sigma(\mathbf{x});[0,1])=\sum_{j=0}^nd(x_j,x_{j+1}),
\]
which combined with \eqref{etandef} and \eqref{forprobb22} implies
\begin{align*}
\sup_{N\in \mathbb{N}}\int_{\mathcal {M}([0,1];X)}L_N(\gamma) \,\widetilde{\eta}_N({\dd}\gamma)
&= \sup_{N\in \mathbb{N}}\int_{\mathbf{X}}\sum_{j=0}^{2^N-1}d(x_j,x_{j+1}) \,\pi_N({\dd}\mathbf{x})
 \leq  \sup_{N\in \mathbb{N}} \sum_{j=0}^{2^N-1} W_p(\mu_{t^j},\mu_{t^{j+1}}) \\
&= \int^1_0|\mu'_+|(t) \,{\dd}t <\infty. 
\end{align*}

By passing to a subsequence and using Lemma \ref{mcontress} (with $Y=\mathcal {M}([0,1];X)$, $\mu_N=\widetilde{\eta}_{N}$ and $\ell_N=L_{N}$),
we assume that for $\widetilde{\eta}$-a.e.\ $\gamma\in \supp(\widetilde{\eta})$, there is $\gamma_{N_k}\in \supp(\widetilde{\eta}_{N_k})$ with
\[
\lim_{k\rightarrow \infty}\varrho(\gamma,\gamma_{N_k})=0,\qquad C_2:=\sup_{k\in \mathbb{N}}L_{N_k}(\gamma_{N_k})<\infty.
 \]
By \eqref{conveginmeasure} and extracting a further subsequence if necessary, we may also assume that $\gamma_{N_k}(t)$ converges to $\gamma(t)$ (in $\tau$) for $\mathscr{L}^1$-a.e.\ $t\in [0,1]$.
By the construction of $\widetilde{\eta}_{N_k}$, we can choose the piecewise constant right-continuous representative of $\gamma_{N_k}$, which is still denoted by $\gamma_{N_k}$.
Then,
\begin{equation}\label{vargammank}
\eV(\gamma_{N_k})=\Vv(\gamma_{N_k})\leq C_2.
\end{equation}

For each $k$, define a non-decreasing function $v_k:[0,1]\rightarrow [0,C_2]$ by $v_k(t):=\Vv(\gamma_{N_k};[0,t])$. 
By extracting a subsequence, it follows from \eqref{vargammank} and Helly's selection theorem that $(v_k)$ pointwisely converges to a non-decreasing function $v : [0,1]\rightarrow [0,C_2]$. 
Since the set of discontinuity points of $v$ is at most countable, we can define a right-continuous non-decreasing bounded function $\bar{v}$ by $\bar{v}(t):=\lim_{s\rightarrow t^+}v(s)$.
Since
\begin{equation*}
d\bigl(\gamma_{N_k}(s),\gamma_{N_k}(t) \bigr) \leq v_k(t)-v_k(s) \quad \text{for all $0 \le s<t \le 1$}
\end{equation*}
by the definition of $v_k$, the lower semi-continuity of $d$ (in $\tau\times \tau$) yields 
\begin{equation*}
d\bigl(\gamma(s),\gamma(t)\bigr) \leq \bar{v}(t)-\bar{v}(s) \quad \text{ for $\mathscr{L}^1$-a.e.\ $s,t\in [0,1]$ with $s\leq t$}.
\end{equation*}
We can choose a  representative of $\gamma$ defined by $\bar{\gamma}(t):=\lim_{s\rightarrow t^+}\gamma(s)$.
Clearly $\bar{\gamma}$ is $d$-right-continuous and 
\begin{equation}\label{boundedvariationcrv}
d\bigl(\bar{\gamma}(s),\bar{\gamma}(t)\bigr) \leq \bar{v}(t)-\bar{v}(s) \quad \text{for all $0 \le s<t \le 1$}.
\end{equation}

Consequently, for $\widetilde{\eta}$-a.e.\ $\gamma \in \supp(\widetilde{\eta})$, there is a right-continuous representative $\bar{\gamma}$ with \eqref{boundedvariationcrv}, which is $d$-continuous except at most a countable set.

\smallskip

\noindent \textbf{Step 4.}
In this step, we show the first assertion \eqref{cond1measre1}.
Define a function $f_N:\mathcal {M}([0,1];X)\rightarrow [0,\infty]$ by
\[
f_N(\gamma):=\sup_{2^{-N}\leq h<1}\int^{1-h}_0 \left(\frac{ d(\gamma(t),\gamma(t+h)) }{h }\right)^p \,{\dd}t,
\]
which is lower semi-continuous by the $(\tau \times \tau)$-lower semi-continuity of $d$ 
and enjoys the monotonicity $f_N \le f_{N+1}$ by definition.
We claim that there is a constant $C_3>0$ with
\begin{equation}\label{lpestimatefn}
 \sup_{N\in \mathbb{N}}\int_{\mathcal {M}([0,1];X)} f_N(\gamma) \,\widetilde{\eta}_N({\dd}\gamma)\leq C_3.
\end{equation}
A modification of \eqref{keylatboelone} together with \eqref{lapclsigmestimate}, \eqref{forprobb22} and H\"older's inequality  yields
\begin{align}
& \int_{\mathcal {M}([0,1];X)} f_N(\gamma) \,\widetilde{\eta}_N({\dd}\gamma)
= \int_{\mathbf{X}}\sup_{2^{-N}\leq h<1}\int^{1-h}_0 \left(\frac{ d(\sigma_t(\mathbf{x}),\sigma_{t+h}(\mathbf{x}))}{h}\right)^p \,{\dd}t \,\pi_N({\dd}\mathbf{x}) \notag\\
&\leq  2^{p+N(p-1)}\sum_{i=0}^{2^N-1}\int_{\mathbf{X}} d(x_i,x_{i+1})^p \,\pi_N({\dd}\mathbf{x})
\leq 2^{p+N(p-1)} \sum_{i=0}^{2^N-1} W_p(\mu_{t^i},\mu_{t^{i+1}})^{p} \notag\\
&\leq 2^{p+N(p-1)} \sum_{i=0}^{2^N-1} \biggl( \int_{t^i}^{t^{i+1}} |\mu'_+|(t) \,{\dd}t \biggr)^p
\leq 2^p \int^1_0 |\mu'_+|^p(t) \,{\dd}t <\infty. \label{derivatproof}
\end{align}
Hence, \eqref{lpestimatefn} holds with $C_3:=(2  \||\mu'_+|\|_{L^p([0,1])})^p$.
Then, for $N_k\geq N$, \eqref{lpestimatefn} and $f_N \le f_{N+1}$ imply
\[
\int_{\mathcal {M}([0,1];X)}f_N(\gamma) \,\widetilde{\eta}_{N_k}({\dd}\gamma)
\leq \int_{\mathcal {M}([0,1];X)}f_{N_k}(\gamma) \,\widetilde{\eta}_{N_k}({\dd}\gamma) \leq  C_3.
\]
Combining this with the lower semi-continuity of $f_N$ and \cite[Lemma 5.1.7]{AGS}, we have
\[
\int_{\mathcal {M}([0,1];X)}f_N(\gamma) \,\widetilde{\eta}({\dd}\gamma)\leq C_3 \quad
 \text{for all $N\in \mathbb{N}$}.
\]
The monotone convergence theorem indicates
$\int_{\mathcal {M}([0,1];X)}\sup_{N\in \mathbb{N}}f_N(\gamma) \,\widetilde{\eta}({\dd}\gamma)\leq C_3$ and hence,
\begin{equation*}\label{fnfinite}
\sup_{0<h<1}\int^{1-h}_0 \left(\frac{ d(\gamma(t),\gamma(t+h)) }{h }\right)^p \,{\dd}t
=\sup_{N\in \mathbb{N}}f_N(\gamma)<\infty \quad \text{for $\widetilde{\eta}$-a.e.\ $\gamma\in \mathcal {M}([0,1];X)$},
\end{equation*}
which implies
\[
\limsup_{h\rightarrow 0^+}\left\| \frac{d(\gamma(\cdot),\gamma(\cdot+h))}{h}    \right\|_{L^p([0,1])}<\infty
 \quad \text{for $\widetilde{\eta}$-a.e.\ $\gamma\in \mathcal {M}([0,1];X)$}.
\]

Therefore, thanks to Step 3 and Lemma \ref{strongcontrall}, for $\widetilde{\eta}$-a.e.\ $\gamma \in \supp\widetilde{\eta}$, we find a representative $\bar{\gamma}\in \AC^p([0,1];X)$.
Let $T:C_{\tau}([0,1];X)\rightarrow \mathcal {M}([0,1];X)$ denote the canonical immersion, which is continuous (see \cite{Li,Li2}).
Then the Borel measure
\begin{equation}\label{defwideeta}
\eta:= \widetilde{\eta}\circ T\in \mathscr{P}\bigl(C_{\tau}([0,1];X) \bigr)
\end{equation}
is concentrated on $\AC^p([0,1];X)$, completing the proof of \eqref{cond1measre1}.

\smallskip

\noindent \textbf{Step 5.}
In this step, we prove \eqref{cond1measre2}.
It is sufficient to show that for every $t\in [0,1]$ and $\varphi \in C_b(X)$,
\begin{equation}\label{lastidentityee}
\int_{C_\tau([0,1];X)}\varphi \bigl(e_t(\gamma)\bigr) \,\eta({\dd}\gamma)
=\int_{X}\varphi \,{\dd}\mu_t.
\end{equation}
We set $g(t):=\int_X \varphi \,{\dd}\mu_t$ and claim that $g$ is uniformly continuous.
If Condition \eqref{lisptotalvariation} holds, then 
\begin{equation*}\label{uniformlyconverge}
|g(t)-g(s)| =\left| \int_X \varphi \,{\dd}\mu_t -\int_X \varphi \,{\dd}\mu_s   \right|
 \leq \sup_X |\varphi| \cdot |\mu_t-\mu_s|(X) \quad \text{for all}\ s,t \in [0,1],
\end{equation*}
which yields the uniform continuity of $g$.
Under Condition \eqref{backwwardabs}, 
it follows from Lemma \ref{wasstopologequaivel} that $g$ is continuous,
and hence uniformly continuous by the compactness of $[0,1]$.

Define a sequence of piecewise constant functions $g_N:[0,1]\rightarrow \mathbb{R}$ by 
$g_N(t):=g(t^i)=\int_{X} \varphi \,{\dd}\mu_{t^i}$ for $t\in [t^i,t^{i+1})$,
which uniformly converges to $g$  as $N\rightarrow \infty$ due to the uniform continuity. Thus, for every test function $\zeta\in C_b([0,1])$, we have
\begin{equation}\label{fistzetalimit}
\lim_{N\rightarrow \infty}\int^1_0 \zeta(t)g_N(t) \,{\dd}t
=\int^1_0 \zeta(t)g(t) \,{\dd}t
=\int^1_0\zeta(t)\left(\int_{X} \varphi \,{\dd}\mu_t \right){\dd}t.
\end{equation}
On the other hand, we deduce from Fubini's theorem that
\begin{align}\label{secondlimit}
\int^1_0 \zeta(t) g_N(t) \,{\dd}t
=\int_{\mathcal {M}([0,1];X)}\left(\int^1_0\zeta(t)\varphi \bigl(e_t(\gamma)\bigr) \,{\dd}t\right) \,\widetilde{\eta}_N({\dd}\gamma).
\end{align}

Since $\varphi \in C_b(X)$, the function
\[
\gamma\in \mathcal {M}([0,1];X)\longmapsto
 \int^1_0\zeta(t)\varphi\bigl(e_t(\gamma)\bigr) \,{\dd}t\in \mathbb{R}
\]
is bounded and continuous. 
The weak convergence of $(\widetilde{\eta}_{N_k})_{k \ge 1}$ to $\widetilde{\eta}$ in $\mathscr{P}(\mathcal {M}([0,1];X))$ (see the beginning of Step 3) combined with Fubini's theorem and \eqref{defwideeta}  yields
\begin{align*}\label{lastlimi}
\lim_{k\rightarrow \infty}\int_{\mathcal {M}([0,1];X)}\left(\int^1_0 \zeta(t)\varphi\bigl(e_t(\gamma)\bigr) \,{\dd}t\right) \widetilde{\eta}_{N_k}({\dd}\gamma)
&=\int_{\mathcal {M}([0,1];X)}\left(\int^1_0 \zeta(t)\varphi\bigl(e_t(\gamma)\bigr) \,{\dd}t\right) \,\widetilde{\eta}({\dd}\gamma)\\
&=\int^1_0 \zeta(t) \left( \int_{C_\tau([0,1];X)} \varphi\bigl(e_t(\gamma)\bigr) \,\eta({\dd}\gamma)\right){\dd}t,
\end{align*}
which together with \eqref{fistzetalimit} and \eqref{secondlimit} furnishes
\[
\int^1_0\zeta(t)\left(\int_X \varphi \,{\dd}\mu_t \right){\dd}t
=\int^1_0 \zeta(t) \left(\int_{C_\tau([0,1];X)} \varphi\bigl(e_t(\gamma)\bigr) \,\eta({\dd}\gamma)\right){\dd}t.
\]
Since $\zeta \in C_b([0,1])$ was arbitrary, we find that \eqref{lastidentityee} holds for $\mathscr{L}^1$-a.e.\ $t\in (0,1)$.
Then, by noticing that both sides are continuous in $t$, \eqref{lastidentityee} holds for all $t\in [0,1]$.

\smallskip

\noindent \textbf{Step 6.}
Finally, we prove \eqref{cond1measre3}.
First, we claim that for all $0 \le s_1<s_2 \le 1$ and $h\in (0,1-s_2)$,
\begin{equation}\label{s1s2+hestim}
\int_{\mathcal {M}([0,1];X)}\int^{s_2}_{s_1}\left(\frac{d(\gamma(t),\gamma(t+h)) }{h }\right)^p {\dd}t \,\widetilde{\eta}({\dd}\gamma)
\leq  \int^{s_2+h}_{s_1}|\mu'_+|^p(t) \,{\dd}t.
\end{equation}
Choose $N\in \mathbb{N}$ satisfying $2^{-N} \leq h$ and then $k\geq 1$ as in Step 1, and set
\[
s^N_1:=\frac{\underline{i}}{2^N}:=\max\left\{ \frac{i}{2^N}\,\bigg|\, \frac{i}{2^N}\leq s_1 \right\},\qquad s^N_2:=\frac{\bar{i}}{2^N}:=\min\left\{ \frac{i}{2^N}\,\bigg|\, \frac{i}{2^N}\geq s_2 +h \right\}.
\]
Calculating as in \eqref{khNinequality}, we obtain
\begin{equation}\label{dfsing}
\int^{s_2}_{s_1}d\bigl(\sigma_t(\mathbf{x}),\sigma_{t+h}(\mathbf{x})\bigr)^p \,{\dd}t
\leq \frac{(k+1)^p}{2^N}\sum_{j=\underline{i}}^{\bar{i}-1}d(x_j,x_{j+1})^p
\leq h^p \left(  \frac{k+1}{k} \right)^p 2^{N(p-1)} \sum_{j=\underline{i}}^{\bar{i}-1} d(x_j,x_{j+1})^p,
\end{equation}
where we used \eqref{khcontroll} in the latter inequality.
Moreover, \eqref{forprobb22} and H\"older's inequality give, as in \eqref{derivatproof},
\begin{equation}\label{sigmalable1}
\int_{\mathbf{X}}d(x_j,x_{j+1})^p \,\pi_N({\dd}\mathbf{x})
 \leq W_p(\mu_{t^j},\mu_{t^{j+1}})^p\leq \frac{1}{2^{N(p-1)}}\int^{t^{j+1}}_{t^j}|\mu'_+|^p(t) \,{\dd}t.
\end{equation}
It follows from \eqref{dfsing} and \eqref{sigmalable1} that
\begin{align*}
\int_{\mathcal {M}([0,1];X)}\int^{s_2}_{s_1}\left(\frac{d(\gamma(t),\gamma(t+h)) }{h }\right)^p {\dd}t \,\widetilde{\eta}_N({\dd}\gamma)
&\leq \left( \frac{k+1}{k} \right)^p 2^{N(p-1)} \int_{\mathbf{X}} \sum_{j=\underline{i}}^{\bar{i}-1}d(x_j,x_{j+1})^p \,\pi_N({\dd}\mathbf{x}) \\
&\leq \left( \frac{k+1}{k} \right)^p \int_{s^N_1}^{s^N_2}|\mu'_+|^p(t) \,{\dd}t.
\end{align*}
Letting $N\rightarrow \infty$ and hence $k\rightarrow \infty$, we obtain \eqref{s1s2+hestim} from \cite[Lemma 5.1.7]{AGS}.

Secondly, by \eqref{defwideeta} and \eqref{s1s2+hestim}, we find
\begin{align*}
\int_{C_{\tau}([0,1];X)}\int^{s_2}_{s_1}\left(\frac{d(\gamma(t),\gamma(t+h)) }{h }\right)^p {\dd}t \,\eta({\dd}\gamma)
\leq  \int^{s_2+h}_{s_1}|\mu'_+|^p(t) \,{\dd}t.
\end{align*}
Since $\eta$ is concentrated on $\AC^p([0,1];X)$, by letting $h\rightarrow 0^+$, Fatou's lemma yields
\begin{align}\label{energeyidenetii}
\int_{C_{\tau}([0,1];X)}\int^{s_2}_{s_1}{|\gamma'_+|^p(t)} \,{\dd}t \,\eta({\dd}\gamma)
\leq  \int^{s_2}_{s_1}|\mu'_+|^p(t) \,{\dd}t
\end{align}
for every $0 \le s_1<s_2 \le 1$.
Then Fubini's theorem and the Lebesgue differentiation theorem imply
\begin{equation}\label{bascimulargeerinf}
 |\mu'_+|^p(t)\geq  \int_{C_{\tau}([0,1];X)} |\gamma'_+|^p(t) \,\eta({\dd}\gamma) \quad
  \text{ for $\mathscr{L}^1$-a.e.\ $t\in (0,1)$}.
\end{equation}

In order to show the reverse inequality of \eqref{bascimulargeerinf},
choose $t\in (0,1)$ such that $|\mu'_+|(t)$ exists.
Then, by H\"older's inequality and Fubini's theorem,
\begin{align*}
&|\mu'_+|^p(t)
=\lim_{h\rightarrow 0^+}\left(\frac{W_p(\mu_t,\mu_{t+h}) }{h}\right)^p 
\leq \limsup_{h\rightarrow 0^+}\int_{C_\tau([0,1];X)}\left(  \frac{d(\gamma(t),\gamma(t+h))}{ h} \right)^p \eta({\dd}\gamma) \\
&\le \limsup_{h\rightarrow 0^+}\int_{C_\tau([0,1];X)}\left( \frac{1}{h} \int_t^{t+h} |\gamma'_+|(s) \,{\dd}s \right)^p \eta({\dd}\gamma)
\le  \limsup_{h\rightarrow 0^+} \frac{1}{h} \int_t^{t+h} \int_{C_\tau([0,1];X)} |\gamma'_+|^p(s) \,\eta({\dd}\gamma) \,{\dd}s.
\end{align*}
Combining this with \eqref{bascimulargeerinf} completes the proof of \eqref{cond1measre3}.
\end{proof}

\section{Hopf--Lax semigroup in forward extended Polish metric spaces}\label{hopflaxsemig}

This appendix is devoted to studying the \emph{Hopf--Lax semigroup} in the asymmetric setting.
We refer to Lott--Villani \cite{JV} and Ambrosio--Gigli--Savar\'e \cite[\S 3]{AGS2}, \cite[\S 3]{AGS4} for the symmetric case.
Throughout,
let $(X,\tau,d)$ be a forward extended Polish space and $f:X\rightarrow (-\infty,\infty]$ be a proper function
(i.e., $\mathfrak{D}(f) =f^{-1}(\mathbb{R}) \neq \emptyset$).
We define, for $x,y \in X$, $p \in (1,\infty)$ and $t>0$,
\[
\Phi(t,x,y):=f(x)+\frac{d(x,y)^p}{pt^{p-1}},\qquad
Q_t f(y):=\inf_{x \in X}\Phi(t,x,y).
\]

\begin{lemma}\label{uppersemiconQ}
The map $(y,t)\mapsto Q_t f(y)$, $X\times (0,\infty)\rightarrow [-\infty,\infty]$, is $d$-upper semi-continuous.
\end{lemma}

\begin{proof}
For any sequence $(y_i,t_i)\to (y,t)$ such that $y_i \to y$ in $d$ and $z\in X$, we have
\[
Q_{t_i} f(y_i) =\inf_{x \in X} \biggl[ f(x)+\frac{d(x,y_i)^p}{pt_i^{p-1}} \biggr]
 \leq f(z)+\frac{d(z,y_i)^p}{pt^{p-1}_i} \to f(z)+\frac{d(z,y)^p}{pt^{p-1}}
\]
as $i \to \infty$, which implies
\[
\limsup_{i\to \infty} Q_{t_i} f(y_i)\leq \inf_{z\in X} \biggl[ f(z)+\frac{d(z,y)^p}{pt^{p-1}} \biggr] =Q_t f(y)
\]
as desired.
\end{proof}

We shall consider the behavior of $Q_tf$ in the set
\[
\mathscr{D}(f):=\{y \in X\mid d(x,y)<\infty \text{ for some $x$ with $f(x)<\infty$} \}\supset \mathfrak{D}(f).
\]
If $\mathscr{D}(f)\neq \emptyset$ and $d$ is finite, then $\mathscr{D}(f)=X$.
For $y \in \mathscr{D}(f)$, we have $Q_tf(y)\in [-\infty,\infty)$ for all $t>0$ and set
\begin{equation}\label{txdefi}
t_*(y):=\sup\{t>0\mid Q_t f(y)>-\infty \},
\end{equation}
while $t_*(y):=0$ if $Q_tf(y)=-\infty$ for all $t>0$.
The behavior of $t_*$ on equivalent classes $X_{[y]}=\overleftarrow{X}_{[y]}$
(recall \eqref{equasetfintex}) is as follows.

\begin{lemma}
If $Q_t f(y)>-\infty$ for some $t>0$ and $y\in X$,
then we have $Q_s f(z)>-\infty$ for all $s\in (0,t)$ and $z\in X_{[y]} =\overleftarrow{X}_{[y]}$.
\end{lemma}

\begin{proof}
Note first that, for any $s\in (0,t)$, we have $Q_sf(y) \geq Q_tf(y)>-\infty$.
Given $z \in X_{[y]}$ and any $x \in X_{[y]} \cap \mathfrak{D}(f)$ (otherwise $\Phi(s,x,z)=\infty$),
we deduce from the triangle inequality and the convexity of the function $r \mapsto r^p$ for $r \ge 0$ that
\begin{align*}
Q_s f(y)
&\leq f(x) +\frac{1}{ps^{p-1}} \bigl( d(x,z)+d(z,y) \bigr)^p \\
&\leq f(x) +\frac{1}{ps^{p-1}} \bigl\{ d(x,z)^p +p\bigl( d(x,z)+d(z,y) \bigr)^{p-1} d(z,y) \bigr\} \\
&= \Phi(s,x,z) +\frac{1}{s^{p-1}} \bigl( d(x,z)+d(z,y) \bigr)^{p-1} d(z,y).
\end{align*}
Taking a sequence $(x_i)_{i \ge 1}$ such that $\Phi(s,x_i,z) \to Q_s f(z)$ as $i \to \infty$,
we infer from Lemma \ref{properD}\eqref{Diii} below that $Q_s f(z)>-\infty$.
\end{proof}

We also introduce the following functions:
for $y \in \mathscr{D}(f)$  with $t_*(y)>0$  and $t \in (0,t_*(y))$,
\begin{equation}\label{DDcont}
\mathfrak{d}^+(y,t) :=\sup_{(x_i)} \limsup_{i \to \infty} d(x_i,y), \qquad
 \mathfrak{d}^-(y,t) :=\inf_{(x_i)} \liminf_{i \to \infty} d(x_i,y),
\end{equation}
where in both cases $(x_i)_{i \ge 1}$ runs over all minimizing sequences of $\Phi(t,\cdot,y)$, i.e.,
$\lim_{i \to \infty} \Phi(t,x_i,y)=Q_tf(y)$.

\begin{lemma}\label{properD}
Given $y \in \mathscr{D}(f)$, we have the following.
\begin{enumerate}[{\rm (i)}]
\item\label{Di} $\mathfrak{d}^+(y,t)$ is finite for every $t \in (0,t_*(y))$.
\item\label{Dii} If $\lim_{i\to \infty} d(y,y_i)=0$ and $\lim_{i\to \infty}t_i=t\in (0,t_*(y))$,
then $\lim_{i\to \infty}Q_{t_i}f(y_i)=Q_t f(y)$.
\item\label{Diii} $\ds \sup\{\mathfrak{d}^+(z,t) \,|\, d(y,z)\leq R,\, 0<t<t_*(y)-\varepsilon \}<\infty$
for any $R>0$ and $\varepsilon>0$.
\end{enumerate}
\end{lemma}

\begin{proof}
We remark that $Q_t f(y)$ coincides with $\Phi_t(y)$ in \cite{OZ} with respect to the reverse metric
$\overleftarrow{d}$.
Then, \eqref{Di} and \eqref{Diii} follow from  (3.3) in  \cite[Lemma 3.13]{OZ},
and \eqref{Dii} follows from \cite[Lemma 3.18(i)]{OZ}.
\end{proof}

A simple diagonal argument shows that the supremum (and the infimum) in \eqref{DDcont} is attained,
that is, there exists $(x_i)_{i \ge 1}$ satisfying $\lim_{i \to \infty} \Phi(t,x_i,y)=Q_tf(y)$
and $\lim_{i \to \infty} d(x_i,y) =\mathfrak{d}^+(y,t)$.
Note also that $\mathfrak{d}^-(y,t)\leq \mathfrak{d}^+(y,t)$ clearly holds.

\begin{proposition}[Monotonicity of $\mathfrak{d}^\pm$]\label{monotonicity}
For all $y \in \mathscr{D}(f)$ and $0<s<t<t_*(y)$, we have $\mathfrak{d}^+(y,s)\leq \mathfrak{d}^-(y,t)$.
In particular, $\mathfrak{d}^+(y,\cdot)$ and $\mathfrak{d}^-(y,\cdot)$ are both non-decreasing
and coincide in $(0,t_*(y))$ with at most countably many exceptions.
\end{proposition}

\begin{proof}
Given $y \in \mathscr{D}(f)$ and $0<s<t<t_*(y)$,
choose minimizing sequences $(x^s_i)_{i \ge 1}$ and $(x^t_i)_{i \ge 1}$ for $\Phi(s,\cdot,y)$ and $\Phi(t,\cdot,y)$, respectively,
such that $d(x^s_i,y) \to \mathfrak{d}^+(y,s)$ and $d(x^t_i,y)\rightarrow\mathfrak{d}^-(y,t)$.
Then $f(x^s_i)$ and $f(x^t_i)$ converge as $i \to \infty$:
\[
f(x^s_i) =\Phi(s,x^s_i,y) -\frac{d(x^s_i,y)^p}{ps^{p-1}}
 \to Q_s f(y) -\frac{\mathfrak{d}^+(y,s)^p}{ps^{p-1}}, \qquad
f(x^t_i) \to Q_t f(y) -\frac{\mathfrak{d}^-(y,t)^p}{pt^{p-1}}.
\]
Note also that
\begin{align*}
Q_s f(y) &\le \lim_{i \to \infty} \biggl[ f(x^t_i) +\frac{d(x^t_i,y)^p}{ps^{p-1}} \biggr]
 =Q_t f(y) -\frac{\mathfrak{d}^-(y,t)^p}{pt^{p-1}} +\frac{\mathfrak{d}^-(y,t)^p}{ps^{p-1}}, \\
Q_t f(y) &\le \lim_{i \to \infty} \biggl[ f(x^s_i) +\frac{d(x^s_i,y)^p}{pt^{p-1}} \biggr]
 =Q_s f(y) -\frac{\mathfrak{d}^+(y,s)^p}{ps^{p-1}} +\frac{\mathfrak{d}^+(y,s)^p}{pt^{p-1}}.
\end{align*}
Combining these yields
\[
\frac{1}{p} \biggl( \frac{1}{s^{p-1}}-\frac{1}{t^{p-1}} \biggr) \mathfrak{d}^+(y,s)^p
 \le Q_s f(y) -Q_t f(y)
 \le \frac{1}{p} \biggl( \frac{1}{s^{p-1}}-\frac{1}{t^{p-1}} \biggr) \mathfrak{d}^-(y,t)^p.
\]
Since $s<t$, we obtain $\mathfrak{d}^+(y,s) \leq \mathfrak{d}^-(y,t)$ as desired.

Comparing the above inequality with $\mathfrak{d}^- \leq \mathfrak{d}^+$,
we see that both functions are non-decreasing.
Moreover, for a point $s$ of right continuity of $\mathfrak{d}^-(y,\cdot)$, we have
\[
\mathfrak{d}^+(y,s) \geq \mathfrak{d}^-(y,s) =\lim_{t \to s^+}\mathfrak{d}^-(y,t) \geq \mathfrak{d}^+(y,s),
\]
which shows $\mathfrak{d}^+(y,s)= \mathfrak{d}^-(y,s)$.
Hence, $\mathfrak{d}^+(y,\cdot)= \mathfrak{d}^-(y,\cdot)$ except for at most countably many points.
\end{proof}

Next, we examine the semi-continuity of $\mathfrak{d}^\pm$.

\begin{proposition}[Semi-continuity of $\mathfrak{d}^\pm$]\label{continofD}
If $y_i \to y \in \mathscr{D}(f)$ in $d$ and $t_i \to t\in (0,t_*(y))$, then
\[
\mathfrak{d}^-(y,t) \leq \liminf_{i \to \infty} \mathfrak{d}^-(y_i,t_i), \qquad
 \mathfrak{d}^+(y,t) \geq \limsup_{i \to \infty} \mathfrak{d}^+(y_i,t_i).
\]
In particular, for every $y \in \mathscr{D}(f)$, $t\mapsto \mathfrak{d}^-(y,t)$ is left-continuous
and $t\mapsto \mathfrak{d}^+(y,t)$ is right-continuous in $(0,t_*(y))$.
\end{proposition}

\begin{proof}
For each $i$, take a sequence $(x^i_k)_{k \ge 1}$ such that
$\Phi(t_i,x^i_k,y_i) \to Q_{t_i}f(y_i)$ and $d(x^i_k,y_i) \to \mathfrak{d}^-(y_i,t_i)$ as $k \to \infty$.
Choose $k(i)$ satisfying 
\[ \Phi(t_i,x^i_{k(i)},y_i) \le Q_{t_i}f(y_i) +\frac{1}{i}, \qquad
 |d(x^i_{k(i)},y_i) -\mathfrak{d}^-(y_i,t_i)| \le \frac{1}{i}. \]
Then, since $y_i \to y$ and $t_i \to t$, we have
\[
\limsup_{i \to \infty} \Phi(t,x^i_{k(i)},y) =\limsup_{i \to \infty} \Phi(t_i,x^i_{k(i)},y_i)
 \leq \limsup_{i \to \infty} Q_{t_i}f(y_i) \leq Q_t f(y),
\]
where the last inequality follows from Lemma \ref{uppersemiconQ}.
This implies that $(x^i_{k(i)})_{i \ge 1}$ is a minimizing sequence of $\Phi(t,\cdot,y)$,
and hence
\[
\mathfrak{d}^-(y,t) \leq \liminf_{i \to \infty} d(x^i_{k(i)},y)
 =\liminf_{i \to \infty}d(x^i_{k(i)},y_i) =\liminf_{i \to \infty} \mathfrak{d}^-(y_i,t_i).
\]
We can prove the upper semi-continuity of $\mathfrak{d}^+$ in the same way.
Combining them with the monotonicity as in Proposition \ref{monotonicity} yields the left/right-continuity.
\end{proof}

As a corollary to the previous proposition  (together with $\mathfrak{d}^- \le \mathfrak{d}^+$),
if $\mathfrak{d}^+(y,t)=\mathfrak{d}^-(y,t)$, then $\mathfrak{d}^\pm$ are continuous at $(y,t)$.



\begin{proposition}[Time derivative of $Q_tf$]\label{timederivative}
For any $y \in \mathscr{D}(f)$, $t \mapsto Q_tf(y)$ is locally Lipschitz in $(0,t_*(y))$. 
Moreover, for all $t \in (0,t_*(y))$, we have
\begin{equation}\label{eq:dQ}
\frac{{\dd}\ }{{\dd}t^-} [Q_t f(y)] =-\frac{1}{q} \frac{\mathfrak{d}^-(y,t)^p}{t^p}, \qquad
\frac{{\dd}\ }{{\dd}t^+} [Q_t f(y)] =-\frac{1}{q} \frac{\mathfrak{d}^+(y,t)^p}{t^p}.
\end{equation}
In particular, $t \mapsto Q_t f(y)$ is differentiable at $t \in (0,t_*(y))$
if and only if $\mathfrak{d}^-(y,t)=\mathfrak{d}^+(y,t)$.
\end{proposition}

\begin{proof}
Given $0<s<t<t_*(y)$, let $(x^s_i)_{i \ge 1}$ and $(x^t_i)_{i \ge 1}$ be minimizing sequences
for $\Phi(s,\cdot,y)$ and $\Phi(t,\cdot,y)$ such that
$\lim_{i \to \infty} d(x^s_i,y) =\mathfrak{d}^+(y,s)$ and
$\lim_{i \to \infty} d(x^t_i,y) =\mathfrak{d}^-(y,t)$.
Then, as in the proof of Proposition~\ref{monotonicity}, we find
\[ 
\frac{\mathfrak{d}^+(y,s)^p}{p(t-s)} \biggl( \frac1{s^{p-1}}-\frac1{t^{p-1}} \biggr)
 \leq \frac{Q_s f(y)-Q_t f(y)}{t-s}
 \leq \frac{\mathfrak{d}^-(y,t)^p}{p(t-s)} \biggl( \frac1{s^{p-1}}-\frac1{t^{p-1}} \biggr).
\]
Combining this with
$\liminf_{s\to t^-}\mathfrak{d}^+(y,s)\geq \lim_{s\to t^-}  \mathfrak{d}^-(y,s) =\mathfrak{d}^-(y,t)$
from Proposition \ref{continofD}, we obtain
 \[
\frac{{\dd}\ }{{\dd}t^-} [Q_t f(y)] =-\frac{p-1}{t^p}\frac{\mathfrak{d}^-(y,t)^p}{p}
 =-\frac{1}{q}\frac{\mathfrak{d}^-(y,t)^p}{t^p}.
\]
One can similarly prove the latter equation in \eqref{eq:dQ} by considering $s>t$ and letting $s\to t^+$.

Moreover, since $\mathfrak{d}^\pm (y,\cdot) $ are locally bounded
by Lemma~\ref{properD}\eqref{Di} and Proposition \ref{monotonicity},
the local Lipschitz continuity of $t \mapsto Q_t f(y)$ follows with the bound
\begin{equation}\label{linftestQ}
\biggl\| \frac{{\dd}}{{\dd}t}[Q_t f(y)]  \biggr\|_{L^\infty(\tau,\tau')}
 \leq \frac{1}{q\tau^p}\| \mathfrak{d}^+(y,\cdot)^p\|_{L^\infty(\tau,\tau')}
\end{equation}
for any $0<\tau<\tau'<t_*(y)$.
%
\end{proof}

Recall from Definition \ref{forwadefsepaere} that $\lambda_{d}(y):=\lim_{r\to 0^+}\lambda_d(B^+_y(r))$.

\begin{proposition}[Slopes and upper gradients of $Q_tf$]\label{slopeSofQ}
For any $y \in \mathscr{D}(f)$ and $t \in (0,t_*(y))$, we have
\[
|D^+ [Q_t f]|(y) \leq \biggl( \frac{\mathfrak{d}^-(y,t)}{t} \biggr)^{p-1}, \qquad
 |D^- [Q_t f]|(y) \leq \lambda_d(y) \biggl( \frac{\mathfrak{d}^+(y,t)}{t} \biggr)^{p-1}.
\]
Moreover, $(\mathfrak{d}^-(\cdot,t)/t)^{p-1}$ is a strong upper gradient of $Q_t f$
on $X_{[y]}=\overleftarrow{X}_{[y]}$ for all $t\in (0,t_*(y))$.
\end{proposition}

\begin{proof}
We first prove that, for any $z \in X$ with $d(y,z)<\infty$,
\begin{equation}\label{qdifferece}
Q_t f(z)-Q_t f(y) \leq \frac{1}{t^{p-1}} \bigl( \mathfrak{d}^-(y,t)+d(y,z) \bigr)^{p-1} d(y,z).
\end{equation}
To this end, take a minimizing sequence $(x_i)_{i \ge 1}$ for $\Phi(t,\cdot,y)$
such that $\lim_{i \to \infty}d(x_i,y)=\mathfrak{d}^-(y,t)$.
Then we have, by the triangle inequality,
\begin{align*}
Q_t f(z)-Q_t f(y) &\leq \liminf_{i \to \infty} \biggl[ \frac{d(x_i,z)^p}{pt^{p-1}} -\frac{d(x_i,y)^p}{pt^{p-1}} \biggr]\\
&\le \frac{1}{pt^{p-1}} \liminf_{i \to \infty} \Bigl[ \bigl( d(x_i,y)+d(y,z) \bigr)^p -d(x_i,y)^p \Bigr] \\
&\le \frac{1}{t^{p-1}} \liminf_{i \to \infty} \bigl( d(x_i,y)+d(y,z) \bigr)^{p-1} d(y,z) \\
&= \frac{1}{t^{p-1}} \bigl( \mathfrak{d}^-(y,t)+d(y,z) \bigr)^{p-1} d(y,z),
\end{align*}
which is exactly \eqref{qdifferece}.
This implies that, on the one hand,
\[
|D^+[Q_t f]|(y) =\limsup_{z \to y} \frac{[Q_t f(z)-Q_t f(y)]^+}{d(y,z)}
 \leq \frac{1}{t^{p-1}} \mathfrak{d}^-(y,t)^{p-1}
\]
as desired.
On the other hand, exchanging $y$ and $z$ in \eqref{qdifferece} yields
\[
Q_t f(y)-Q_t f(z) \leq \frac{1}{t^{p-1}} \bigl( \mathfrak{d}^-(z,t)+d(z,y) \bigr)^{p-1} d(z,y).
\]
Hence, together with Proposition~\ref{continofD}, we obtain
\begin{align*}
|D^-[Q_tf]|(y) &=\limsup_{z \to y} \frac{[Q_t f(y)-Q_t f(z)]^+}{d(y,z)}
 \leq \frac{1}{t^{p-1}} \limsup_{z \to y} \biggl[ \frac{d(z,y)}{d(y,z)} \mathfrak{d}^-(z,t)^{p-1} \biggr] \\
&\leq \frac{1}{t^{p-1}} \lambda_d(y) \mathfrak{d}^+(y,t)^{p-1}.
\end{align*}

To show the latter assertion, fix $t\in (0,t_*(y))$ and let $\gamma:[0,1]\rightarrow X_{[y]}$
be a (forward) absolutely continuous curve with constant speed $c$ (i.e., $|\gamma'_+| \equiv c$).
For $r,s\in [0,1]$,
it follows from \eqref{qdifferece} that
\begin{equation}\label{eq:Qtf}
Q_t f(\gamma_s)-Q_t f(\gamma_r) \leq \frac{1}{t^{p-1}} \bigl( \mathfrak{d}^-(\gamma_r,t)+d(\gamma_r,\gamma_s) \bigr)^{p-1} d(\gamma_r,\gamma_s).
\end{equation}
Moreover, for the function $\Theta$ as in Lemma~\ref{pfms-d} and the symmetric metric $\hat{d}$ as in \eqref{symmmetricde}, we have
\[
 d(\gamma_r,\gamma_s)\leq \Theta_{\gamma(0)}(c), \quad d(\gamma_r,\gamma_s)\leq 2\hat{d}(\gamma_r,\gamma_s), \quad \AC\bigl( [0,1];(X,d) \bigr) =\AC\bigl( [0,1];(X,\hat{d}) \bigr).
\]
Then, thanks to  Lemmas~\ref{uppersemiconQ} and \ref{properD}\eqref{Diii},
we deduce from \cite[Corollary 2.10]{AGS2} that $s\mapsto Q_t f(\gamma_s)$ is absolutely continuous.
Therefore, by letting  $s \rightarrow r^+$, it follows from \eqref{eq:Qtf} that
\[
\frac{{\dd}}{{\dd}r}[Q_tf(\gamma_r)] \leq \biggl( \frac{\mathfrak{d}^-(\gamma_r,t)}{t} \biggr)^{p-1}|\gamma'_+|(r)
 \quad \text{ for $\mathscr{L}^1$-a.e.}\ r\in [0,1],
\] 
which concludes the proof.
\end{proof}

\begin{theorem}[Subsolution of Hamilton--Jacobi equation]\label{subslohh}
For $y\in \mathscr{D}(f)$ and $t\in (0,t_*(y))$, we have
\[
\frac{{\dd}\ }{{\dd}t^+}[Q_t f(y)] +\frac{1}{q} \biggl(\frac{|D^-[Q_t f]|(y)}{\lambda_d (y)} \biggr)^q \leq 0,
 \qquad \frac{{\dd}\ }{{\dd}t^-}[Q_t f(y)] +\frac{|D^+[Q_t f]|(y)^q}{q} \leq 0.
\]
In particular,
\begin{equation}\label{basicestimate1qqd}
\frac{{\dd}}{{\dd}t}[Q_t f(y)] +\frac{|D^+[Q_t f]|(y)^q}{q} \leq 0
\end{equation}
with at most countably many exceptions in $(0,t_*(y))$.
\end{theorem}

\begin{proof}
According to Propositions \ref{timederivative} and \ref{slopeSofQ}, we have
\begin{align*}
&\frac{{\dd}\ }{{\dd}t^+}[Q_t f(y)] =-\frac{1}{q} \biggl(\frac{\mathfrak{d}^+(y,t)}{t} \biggr)^p
 \leq -\frac{1}{q} \biggl( \frac{|D^-[Q_t f]|(y)}{\lambda_d (y)} \biggr)^q, \\
&\frac{{\dd}\ }{{\dd}t^-}[Q_t f(y)] =-\frac{1}{q} \biggl(\frac{\mathfrak{d}^-(y,t)}{t} \biggr)^p
 \leq -\frac{1}{q} |D^+[Q_t f]|(y)^q.
\end{align*}
Moreover, \eqref{basicestimate1qqd} follows from the second inequality
and Propositions \ref{monotonicity} and \ref{timederivative}.
\end{proof}


\begin{proposition}\label{strongexpaap}
Let $f:X\rightarrow \mathbb{R}$ be a bounded function.
Then, for any $t>0$, $Q_t f:X \rightarrow \mathbb{R}$ is bounded and forward Lipschitz.
More precisely, we have
\begin{equation}\label{boundedQtf}
\inf_X f\leq \inf_X Q_tf\leq \sup_X Q_tf\leq \sup_X f, \qquad
\Lip(Q_t f) \leq 2^{p-1} \biggl( \frac{p\osc(f)}{t} \biggr)^{1/q},
\end{equation}
where $\osc(f):=\sup_X f-\inf_X f$.
\end{proposition}

\begin{proof}
Observe from the definition of $Q_t f$ that $\inf_X f \le \inf_X Q_t f$ and $Q_t f(y) \le f(y)$ for all $y \in X$,
thus we have the boundedness as in \eqref{boundedQtf}.

Next, given $y \in X$, choose a minimizing sequence $(x_i)_{i \ge 1}$ for $\Phi(t,\cdot,y)$
such that $\lim_{i \to \infty}d(x_i,y) =\mathfrak{d}^+(y,t)$.
Note that
\[
\frac{\mathfrak{d}^+(y,t)^p}{pt^{p-1}} -\osc(f)
 \leq \lim_{i \to \infty} \biggl[ f(x_i)+\frac{d(x_i,y)^p}{pt^{p-1}} \biggr] -f(y)
 =Q_t f(y)-f(y) \leq 0,
\]
which implies
\begin{equation}\label{osccond}
\mathfrak{d}^+(y,t) \leq \bigl( pt^{p-1} \osc(f) \bigr)^{1/p}=:R.
\end{equation}
For any $z \in X$ with $0<d(y,z)<R/p$,
we infer from \eqref{qdifferece} and $\mathfrak{d}^-(y,t)\leq \mathfrak{d}^+(y,t)$ that
\[
\frac{Q_t f(z)-Q_t f(y)}{d(y,z)}
 \le \frac{1}{t^{p-1}} \bigl( \mathfrak{d}^-(y,t)+d(y,z) \bigr)^{p-1}
 \le \frac{(2R)^{p-1}}{t^{p-1}}
 =2^{p-1} \biggl( \frac{p \osc(f)}{t} \biggr)^{1/q}.
\]
For $z \in X$ with $d(y,z)\geq R/p$, since $\osc(Q_tf)\leq \osc(f)$ by the former assertion in \eqref{boundedQtf},
we have
\[
\frac{Q_t f(z)-Q_t f(y)}{d(y,z)} \leq \frac{\osc(Q_tf)}{R/p}
 \leq \frac{p\osc(f)}{(pt^{p-1}\osc(f))^{1/p}}
 = \biggl( \frac{ {p\osc(f)}}{t} \biggr)^{1/q}.
\]
This completes the proof of the latter assertion in \eqref{boundedQtf}.
\end{proof}

\begin{remark}[Continuity of $Q_t$ at $t=0$]\label{conatt=0}
If $f$ is bounded and $\tau$-lower semi-continuous, then
  for all $y \in X$
\[
f(y) \ge Q_tf(y) \ge
 \inf\biggl\{ f(x)+\frac{d(x,y)^p}{pt^{p-1}} \,\bigg|\, d(x,y)^p \le pt^{p-1}\Bigl( f(y)-\inf_X f \Bigr) \biggr\}
\]
and hence $Q_t f(y) \uparrow f(y)$ as $t\downarrow 0$ (cf.\ \cite[Lemma 3.18(v)]{OZ}).
\end{remark}

We finally study the Borel measurability (in $\tau$) of the slopes $|D^\pm [Q_t f]|(y)$
along the lines of \cite[Proposition~3.8]{AGS2}.

\begin{proposition}\label{qtlowersemi}
Let $(X,\tau,d)$ be a forward extended Polish space.
\begin{enumerate}[{\rm (i)}]
\item\label{cporK1}
Given $g \in C(K)$ on a compact set $K\subset (X,\tau)$ and $M \geq \max_K g$, set
\[
f(x) :=\begin{cases}
g(x) & \text{ if } x \in K, \\
M & \text{ if } x \in X \setminus K.
\end{cases}
\]
Then, $Q_t f$ is $\tau$-lower semi-continuous in $X$ for all $t>0$.
\item\label{cporK2}
Suppose that $\mathscr{D}(f)=X$, $t_*(y)\geq T>0$ for all $y \in X$,
and $Q_t f$ is Borel measurable for all $t \in (0,T)$ $($in $\tau)$.
Then, $(y,t) \mapsto {\dd}[Q_t f(y)]/{\dd}t^+$ is Borel measurable in $X \times (0,T)$
and both slopes $|D^\pm [Q_t f]|(y)$ are $\mathcal{B}^*(X\times (0,T))$-measurable in $X\times (0,T)$.
 \end{enumerate}
 \end{proposition}

\begin{proof}
\eqref{cporK1}
By the choice of $M$, we have for all $y \in X$
\[
Q_t f(y) =\min\biggl\{ \min_{x \in K} \biggl[ g(x)+\frac{d(x,y)^{p}}{pt^{p-1}} \biggr], M \biggr\}.
\]
Thus, it suffices to show the $\tau$-lower semi-continuity of
\[
\Phi(y) :=\min_{x \in K} \biggl[ g(x)+\frac{d(x,y)^p}{pt^{p-1}} \biggr].
\]
Given a sequence $(y_i)_{i \ge 1} \overset{\tau}{\rightarrow} y \in X$,
since $d^p$ is $\tau$-lower semi-continuous, we can find a point $x_i \in K $ satisfying
\[
\Phi(y_i) =g(x_i)+\frac{d(x_i,y_i)^p}{pt^{p-1}}.
\]
Take a subsequence $(y_{i_k})_{k \ge 1}$ such that $\lim_{k \to \infty} \Phi(y_{i_k}) =\liminf_{i \to \infty}\Phi(y_i)$.
Since $K$ is $\tau$-compact, by passing to a further subsequence,
we also assume that $(x_{i_k})$ $\tau$-converges to some $x_\infty \in K$.
Then we have
\begin{align*}
\lim_{k \to \infty} \Phi(y_{i_k})
 =\lim_{k \to \infty} \biggl( g(x_{i_k}) +\frac{d(x_{i_k},y_{i_k})^p }{pt^{p-1}} \biggr)
 \geq g(x_\infty)+\frac{d(x_\infty,y)^p}{pt^{p-1}} \geq \Phi(y),
\end{align*}
which shows the $\tau$-lower semi-continuity.

\eqref{cporK2}
It follows from \eqref{cporK1} and the (local Lipschitz) continuity of $t \mapsto Q_t f(y)$ for each $y \in \mathscr{D}(f)$ (Proposition~\ref{timederivative}) that $(y,t) \mapsto Q_t f(y)$ is Borel measurable.
The Borel measurability of $(y,t)\mapsto {\dd}[Q_t f(y)]/{\dd}t^+$ follows from
Propositions \ref{continofD}, \ref{timederivative}.
The measurability of slopes is seen as in \cite[Lemma~2.6]{AGS2} (recall Lemma~\ref{borlstarmea}).
\end{proof}

\begin{remark}[Further problems]\label{rm:further}
Having the Hamilton--Jacobi inequality (Theorem~\ref{subslohh}) in asymmetric metric measure spaces
for the Hopf--Lax semigroup in hand, a natural question arises concerning the hypercontractivity estimate
in the sense of Bakry \cite{Bakry}.
A suitable framework seems to be the class of asymmetric $\CD(0,N)$-spaces $(N>1)$,
studied in detail in  \cite{KZ}.
Besides the Hamilton--Jacobi inequality,
the validity of the sharp logarithmic Sobolev inequality is needed for such an estimate,
which requires a suitable rearrangament argument
(together with a P\'olya--Szeg\H o inequality, based on a sharp isoperimetric inequality),
and further regularity properties of the Hopf--Lax semigroup.
Although some of the aforementioned ingredients are already available in the literature
(see, e.g., Manini \cite{Manini} for the sharp isoperimetric inequality on forward complete measured Finsler manifolds
satisfying the $\CD(0,N)$ condition for some $N > 1$), further fine analysis will be needed
--- as regularity of the Hopf--Lax semigroup and a suitable P\'olya--Szeg\H o inequality
involving appropriate gradient-type term --- to describe the asymmetric hypercontractivity estimate
for the Hopf--Lax semigroup.
\end{remark}

\bigskip

\noindent {\bf Data Availability.} Data sharing is not applicable to this article as no data sets were generated or analyzed during
the current study.

\smallskip

\noindent{\bf Conflicts of Interest.} There is no conflict of interest.


\end{document}